\newif{\ifarxiv}
\arxivtrue

\newif{\ifta}
\tafalse

\ifta
\documentclass[preprint,12pt]{elsarticle}
\else
\documentclass[twocolumn]{article}
\fi
\pdfoutput=1 

\usepackage[inline]{enumitem}
\usepackage{latexsym,amsmath,amssymb,stmaryrd,graphicx,orcidlink}
\usepackage{color}
\usepackage[all]{xy}
\usepackage{url}

\newcommand\eqdef{\mathrel{\buildrel \text{def}\over=}}

\newcommand\pow{\mathbb{P}}
\newcommand\Pfin{\pow_{\mathrm{fin}}}
\newcommand{\rat}{\mathbb{Q}}
\newcommand{\real}{\mathbb{R}}
\newcommand{\creal}{\overline{\mathbb{R}}_+}
\newcommand{\Rp}{\mathbb{R}_+}

\newcommand\diff{\smallsetminus}
\DeclareMathOperator{\upc}{\uparrow\!}
\DeclareMathOperator{\dc}{\downarrow\!}
\newcommand\nat{\mathbb{N}}

\newcommand\limp{\mathrel{\Rightarrow}}

\newcommand\Lform{\mathcal L}

\newcommand\uuarrow{\rlap{$\uparrow$}\raise.5ex\hbox{$\uparrow$}}
\newcommand\ddarrow{\rlap{$\downarrow$}\raise.5ex\hbox{$\downarrow$}}

\newcommand\identity[1]{\mathrm{id}_{#1}}
\newcommand\dotleq{\mathrel{\dot\leq}}
\newcommand\one{{\mathbf 1}}

\newcommand\Smyth{{\mathcal Q}}
\newcommand\cvx{{\mathrm{cvx}}}
\newcommand\Smythc{\Smyth^\cvx}
\newcommand\SV{\Smyth_\Vt}
\newcommand\SVc{\SV^\cvx}

\newcommand\Vt{\mathsf{V}}

\newcommand\PV\Plotkinn 
\newcommand\Val{{\mathbf V}}
\newcommand\Sober{{\mathcal S}}
\newcommand\Open{{\mathcal O}}
\newcommand\pw{\mathrm{\lowercase{p}}}
\newcommand\fin{\mathrm{\lowercase{f}}}
\newcommand\wS{\mathit{\lowercase{w}S}}

\newcommand\Demon{{\mathtt{D}}}
\newcommand\Nature{{\mathtt{P}}}

\newcommand\DN{{\Demon\Nature}}

\newcommand\Topcat{\mathbf{Top}}
\newcommand\Setcat{\mathbf{Set}}

\newcommand{\interior}[1]{\text{int} ({#1})} 
\newcommand\dG{{\text{\textsf{d}}}}
\newcommand\patch{{\text{\textsf{patch}}}}

\newcommand\Min{\mathop{\mathrm{Min}}}

\newcommand\Latt{\mathcal{L}}

\newcommand\cb[1]{\mathbf{#1}} 

\newcommand\conv{\mathop{\mathrm{conv}}}
\newcommand\clconv{\mathop{\overline{\conv}}}

\newcommand\Alg{\mathfrak{A}}
\newcommand\B{\mathfrak{B}}
\newcommand\C{\mathfrak{C}}
\newcommand\Cb{\mathfrak{D}}
\newcommand\Iu{\mathbb I}
\newcommand\quot[2]{{#1}/{#2}}
\newcommand\etaS{\eta^{\Sober}}

\newcommand\conify{\mathrm{cone}}
\newcommand\etac{\eta^{\conify}}

\newcommand\leqc{\leq^{\conify}}
\newcommand\cext{\leftslice} 
\newcommand\dbl{\mathrm{dbl}}
\newcommand\tscope{\mathrm{tscope}}
\newcommand\etaca{\eta^{\tscope, \alpha}}
\newcommand\leqca{\leq_{\alpha}}
\newcommand\etaps{\eta^{\astar*}}
\newcommand\etass{\eta^{**}}
\newcommand\scope[1]{((#1))}
\newcommand\rast{\circledast} 
\newcommand\astar{\star}

\newcommand\Cone{\cb{Cone}}
\newcommand\BA{\cb{BA}}

\newcommand\sTBA{\cb{sTBA}}

\newcommand\TBA{\cb{TBA}}

\newtheorem{theorem}{Theorem}[section]
\newtheorem{proposition}[theorem]{Proposition}

\newtheorem{deflem}[theorem]{Def.\ and Lemma}
\newtheorem{defprop}[theorem]{Def.\ and Proposition}
\newtheorem{corollary}[theorem]{Corollary}
\newtheorem{lemma}[theorem]{Lemma}
\ifta
\newproof{proof}{Proof}
\else

\newcommand\qed{\hfill$\Box$}
\fi

\newtheorem{problem}[theorem]{Problem}
\newtheorem{fact}[theorem]{Fact}

\newtheorem{definition}[theorem]{Definition}
\newtheorem{example}[theorem]{Example}

\newtheorem{remark}[theorem]{Remark}

\newcommand\ForAuthors[1]
 {\par\smallskip                     
  \begin{center}
   \fbox
   {\parbox{0.9\linewidth}
    {\raggedright--- #1}
   }
  \end{center}
  \par\smallskip                     %
 }        

 \ifta
\journal{Topology and its Applications}
\fi

\newcommand\abs{%
  Barycentric algebras are an abstraction of the notion of convex
  sets, defined by a set of equations.  We study semitopological and
  topological barycentric algebras, in the spirit of a previous study
  by Klaus Keimel on semitopological and topological cones (2008),
  which are special cases of semitopological and topological
  barycentric algebras.  For example, the space of all continuous
  valuations (a very close cousin of measures) over a topological
  space is a topological cone, while probability valuations form a
  topological barycentric algebra, and subprobability valuations form
  a pointed topological barycentric algebra.  Among other results, we
  show the existence of free semitopological cones over
  semitopological barycentric algebras and over pointed
  semitopological algebras, we investigate which semitopological
  barycentric algebras embed into semitopological cones and which
  pointed semitopological barycentric algebras embed strictly into
  semitopological cones.  We study notions of local convexity, which
  split into weak local convexity, local convexity, local affineness
  and local linearity.  We show that the weakly locally convex
  topological barycentric algebras are exactly the affine retracts of
  locally affine topological barycentric algebras.  On locally convex
  barycentric algebras, we show sandwich theorems, extending theorems
  by Roth and Keimel on cones.  A running theme of this paper is the
  notion of barycenters, which we progressively generalize until we
  reach a general notion of barycenters of continuous (resp.,
  subprobability, probability) valuations, inspired by a definition of
  Choquet.  We conclude with a general barycenter existence theorem,
  whose proof relies on the study of the Smyth poweralgebra, namely
  the topological barycentric algebra of all non-empty convex compact
  saturated subsets of a topological barycentric algebra.  }

\ifta
\begin{document}
\begin{frontmatter}



\title{Semitopological Barycentric Algebras}


\author{Jean Goubault-Larrecq\orcidlink{0000-0001-5879-3304}}

\address{Universit\'e Paris-Saclay, CNRS, ENS Paris-Saclay,
  Laboratoire M\'ethodes
  Formelles, 91190, Gif-sur-Yvette, France.\\
  \texttt{jgl@lml.cnrs.fr}
}

\begin{abstract}
  \abs
\end{abstract}

\begin{keyword}
  \MSC[2020] 52A01 \sep 46A55 \sep 52A07 \sep 54E99
  
\end{keyword}

\end{frontmatter}
\else
\title{Semitopological Barycentric Algebras}
\author{Jean Goubault-Larrecq\\
  Universit\'e Paris-Saclay, CNRS, ENS Paris-Saclay,\\
  Laboratoire M\'ethodes
  Formelles, 91190, Gif-sur-Yvette, France.\\
  \texttt{jgl@lml.cnrs.fr}}
\date{\relax}

\begin{document}
\maketitle

\begin{abstract}
  \abs
  
  MSC classification [2020]: 52A01, 46A55, 52A07, 54E99
\end{abstract}
\fi


\ifta
\noindent
\begin{minipage}{0.25\linewidth}
  \includegraphics[scale=0.2]{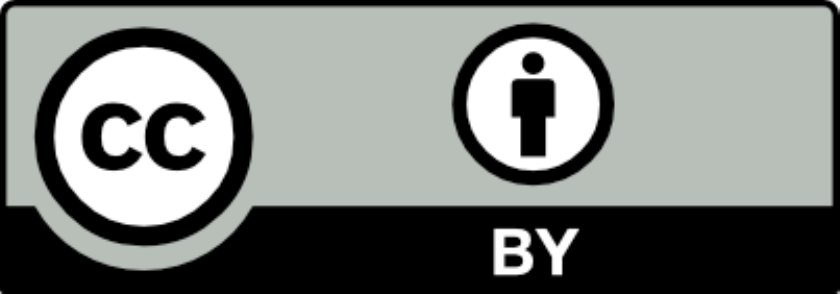}
\end{minipage}
\begin{minipage}{0.74\linewidth}
  \scriptsize
  For the purpose of Open Access, a CC-BY public copyright licence has
  been applied by the authors to the present document and will be
  applied to all subsequent versions up to the Author Accepted
  Manuscript arising from this submission.
\end{minipage}
\fi

\tableofcontents

\section{Introduction}
\label{sec:intro}

Functional analysis is mostly done on topological vector spaces.
Klaus Keimel set out the basis for functional analysis in a $T_0$
setting in \cite{Keimel:topcones2}, by studying topological and
semitopological \emph{cones}.  In a nutshell, a cone is a structure
that obeys all the axioms of a real vector space, except without
opposites or multiplication by negative scalars.  Many standard
theorems still hold in this setting, and the space of all (possibly
unbounded) measures on a space $X$, is a prime example of a cone that
is not embeddable in a vector space.

The space of probability measures on $X$ is not a cone, because you
cannot multiply a probability measure by any scalar except $1$ in
order to obtain a probability measure, but it still has some intrinsic
structure.  We may say that it is a convex subset of a cone, but there
is a more intrinsic view to it: it is a \emph{barycentric algebra}
with a compatible topology.  Barycentric algebras are sets with binary
operations $+_a$ with $a \in [0, 1]$ that model taking a convex
combination of two elements with weight $a$ and $1-a$ respectively,
and obey specific axioms.  Barycentric algebras were studied by
Marshall H. Stone \cite{Stone:bary} and Hellmuth Kneser
\cite{Kneser:convex} independently, then by Walter D. Neumann
\cite{neumann:bary} and Tadeusz {\'S}wirszcz
\cite{Swirszcz:bary,Swirszcz:monadic,Swirszcz:monadic:convex}.  Joe
Flood \cite{Flood:semiconvex}, Lev Anatol'evich Skornyakov
\cite{Skornyakov:convexor}, Viktor Viktorovich Ignatov
\cite{Ignatov:convexor:qvar,Ignatov:convexor:struct}, Anna
B. Romanowska and Jonathan D. H. Smith studied them in more depth
\cite{RS:bary:alg}, as well as other authors, which we will not cite
(see \cite[Remarks~2.9]{KP:mixed} for a somewhat more complete
bibliography).  Barycentric algebras are more than convex subsets of
real vector spaces, as sup-semilattices are barycentric algebras, too.
In general, every barycentric algebra decomposes uniquely into a
semi-structural band of open affine sets
\cite[Teorema~2]{Ignatov:convexor:struct}, and open affine sets are
convex subsets of real vector spaces up to isomorphism; every
barycentric algebra also embeds in what is known as a P\l onka sum of
convex subsets of real vector spaces, indexed by a sup-semilattice
\cite{RS:bary:alg}; none of this will play any r\^ole in this paper.

Apart from a (crucial) notion of algebraically open sets, no topology
is involved in the papers mentioned above.  Klaus Keimel and Gordon
D. Plotkin studied ordered barycentric algebras, and then barycentric
algebras with a partial ordering that is directed-complete and makes
the algebraic operations (Scott-)continuous \cite{KP:mixed}.  We
develop this further, and, imitating \cite{Keimel:topcones2} somewhat,
produce a theory of semitopological and topological barycentric
algebras.

\paragraph{Outline.}  After setting out a few preliminary notions and
results in Section~\ref{sec:preliminaries}, we embark on a study of
barycentric algebras in Section~\ref{sec:baryc-algebr-preord}.  We
recall what they are, then proceed to preordered, ordered, then
semitopological and topological barycentric algebras.  
We turn to pointed barycentric algebras in
Section~\ref{sec:point-baryc-algebr}: those are simply barycentric
algebras with a distinguished element $\bot$, which one may see as a
zero.  In such pointed barycentric algebras, one can multiply by
scalars in $[0, 1]$, as in the space $\Val_{\leq 1} X$ of
subprobability valuations on a space $X$, where $\bot$ is the zero
valuation.

These sections contain the main novelty of this paper: free
semitopological cone constructions, equivalent notions of
embeddability (resp.\ strict embeddability) of (resp.\ pointed)
semitopological barycentric algebras in semitopological cones, showing
that the space $\Val_b X$ of bounded continuous valuations over a
space $X$ is the free semitopological cone over the pointed
semitopological barycentric algebra $\Val_{\leq 1} X$ and over the
semitopological barycentric algebras $\Val_1 X$ of probability
valuations on $X$, for example.

The rest of the paper adapts results on semitopological cones, mostly
due to Keimel, to the setting of semitopological barycentric algebras
and pointed semitopological barycentric algebras
\cite{Keimel:topcones2}.  Section~\ref{sec:convexity} is about
convexity, convex hulls, closed convex hulls, saturated convex hulls.
We study notions of local convexity in
Section~\ref{sec:local-convexity}.
In Section~\ref{sec:cons-baryc-algebr}, we turn to consistent and
strongly consistent semitopological barycentric algebras, generalizing
a notion of cones where addition is almost open due to Keimel.
While the theory of locally convex vector spaces relies heavily on the
Hahn-Banach theorem, the theory of semitopological cones relies a lot
on a sandwich theorem due to Keimel \cite{Keimel:topcones2}, and
derived from a deep result due to Roth \cite{Roth:sandwich}.  We will
establish a similar sandwich theorem on semitopological barycentric
algebras, by reduction to the case of cones.
This will lead us to the study of convex compact saturated subsets,
and then to locally convex-compact semitopological
barycentric algebras, adapting the eponymous cone-theoretic notion due
to Keimel.
We call \emph{convenient} those semitopological barycentric algebras
that are linearly separated, locally compact, sober and topological,
and
we will then investigate the barycentric algebra-theoretic equivalent of
Keimel's Smyth powercone over a topological cone,
namely its space of non-empty convex compact saturated subsets.  All
this will be needed in the final Section~\ref{sec:barycenters-part-6},
where we will study a general notion of barycenter of continuous
(resp.\ subprobability, probability) valuations on semitopological
cones (resp.\ pointed semitopological algebras, resp.\ semitopological
algebras), in the style of Choquet \cite{Choquet:analysis:2}.
We will conclude with a theorem stating, inter alia, that every
probability valuation on convenient barycentric algebra has a unique
barycenter, and that the barycenter map is continuous, generalizing a
similar, recent theorem on cones \cite{GLJ:bary}.

\section{Preliminaries}
\label{sec:preliminaries}

Cones will play a prominent role in this paper.  By cone, we do not
mean cones embedded in real vector spaces, rather the following more
abstract version, used in \cite{Keimel:topcones2,KP:mixed}, but
already studied by Benno Fuchssteiner and Wolfgang Lusky
\cite{FL:cones}: a \emph{cone} is a commutative monoid $(\C, +, 0)$
(or just $\C$ for short) with an action $a, x \mapsto a \cdot x$ of
the semi-ring $\Rp$ on $\C$, that is:
\begin{equation}
  \label{eq:cone}
  \ifta
  \begin{array}{ccc}
    0 \cdot x = 0 & (ab) \cdot x = a \cdot (b \cdot x) & 1 \cdot x=x\\
    a \cdot 0=0 & a \cdot (x+y) = a \cdot x+a \cdot y & (a+b) \cdot x = a \cdot x+b \cdot x
  \end{array}
  \else
  \begin{array}{ccc}
    0 \cdot x = 0 & (ab) \cdot x = a \cdot (b \cdot x) \\
    1 \cdot x=x & a \cdot 0=0 \\
    a \cdot (x+y) = a \cdot x+a \cdot y & (a+b) \cdot x = a \cdot x+b \cdot x
  \end{array}
  \fi
\end{equation}
for all $a, b \in \Rp$ and $x, y \in \C$.  We will often write $ax$ for $a \cdot x$.

An \emph{preordered cone} (resp., \emph{ordered cone}) is a cone $\C$
together with a preordering (resp., an ordering) $\leq$, such that
addition and scalar multiplication are monotonic in each argument.  A
\emph{semipreordered cone} (resp., \emph{semiordered cone}) is a cone
with a preordering (resp., an ordering) such that addition is
monotonic, and for every $a \in \Rp$, $a \cdot \_$ is monotonic.  In a
preordered cone $\C$, $0$ is a least element, since for every
$x \in \C$, $0 = 0 \cdot x \leq 1 \cdot x = x$; not so in a
semipreordered cone.  We will not consider semipreordered and
semiordered cones, except for references to the literature.

\begin{remark}
  \label{rem:semiordered}
  With these definitions, we stray from the literature: e.g., Keimel
  \cite{Keimel:topcones2}, Keimel and Plotkin \cite{KP:mixed}, and
  several others, call \emph{ordered} cone what we call semiordered
  cone.  Our ordered cones are their \emph{pointed} ordered cones,
  namely the ordered cones (in their sense) such that $0$ is a least
  element; this is equivalent to requiring that $\_ \cdot x$ is
  monotonic for every $x \in \C$ \cite[Lemma~3.5,
  Definition~3.6]{Keimel:topcones2}.  All our ordered, preordered or
  semitopological cones will be pointed, and therefore we prefer to
  follow \cite{tix01} here, who assumes monotonicity both in $a$ and
  in $x$.
\end{remark}

A function $f \colon \C \to \Cb$ between cones is \emph{positively
  homogeneous} if and only if $f (a \cdot x) = a \cdot f (x)$ for all
$a \in \Rp$ and $x \in \C$.  It is \emph{additive} if and only if
$f (x+y) = f (x) + f (y)$ for all $x, y \in C$, and \emph{linear} if
it is positively homogeneous and additive.  This applies notably when
$\Cb$ is the cone $\creal \eqdef \Rp \cup \{\infty\}$.  When $\Cb$ is
a preordered cone (as $\creal$ is, for example), $f$ is
\emph{superadditive} (resp.\ \emph{subadditive}) if and only if
$f (x + y) \geq f (x) + f (y)$ (resp.\ $f (x+y) \leq f (x) + f (y)$)
for all $x, y \in C$, and is \emph{superlinear} (resp.\
\emph{sublinear}) if and only if it is superadditive (resp.\
subadditive) and positively homogeneous.

Cones and linear maps form a category, which we write as $\Cone$.

Let us say \emph{poset} for partially ordered set.  We can equip every
poset $(P, \leq)$ (or just $P$ for short) with various topologies.  We
will be particularly interested in the \emph{Scott topology}, whose
open subsets are the subsets $U$ that are upwards-closed (for every
$x \in U$, for every $y \in P$ such that $x \leq y$, $y$ is in $U$)
and such that that every directed family ${(x_i)}_{i \in I}$ whose
supremum exists in $P$ and is in $U$, some $x_i$ is already in $U$.
When this is needed, we write $P_\sigma$ for $P$ with its Scott
topology.  We will dispense with the $\sigma$ subscript when $P=\Rp$:
we will always equip $\Rp$ with its Scott topology, and its open
subsets are $]a, \infty[$ where $a \in \Rp$, $\Rp$ and the empty set.
We will do the same with $\Rp \diff \{0\}$ and with
$\creal \eqdef \Rp \cup \{\infty\}$, where $\infty$ is a fresh element
larger than every element of $\Rp$.  The Scott topology is
particularly important on \emph{dcpos} (short for ``directed-complete
partial order''), which are posets in which every directed family has
a supremum; $\creal$ is a dcpo, $\Rp$ is not.  Given two posets $P$
and $Q$, the continuous maps $f \colon P_\sigma \to Q_\sigma$ are
exactly the \emph{Scott-continuous} maps, namely the functions
$f \colon P \to Q$ that are monotonic and preserve directed suprema.
We refer to \cite{JGL-topology}, to \cite{GHKLMS:contlatt} and to
\cite{AJ:domains} for notions of domain theory and non-Hausdorff
topology.

We write $\interior A$ for the interior of $A$ and $cl (A)$ for the
closure of $A$.  Compact sets are defined without any requirement of
separation, namely they are those sets such that every open cover has
a finite subcover.  Every finite set $A$ is compact, as well as $\upc
A$ with $A$ finite.

Every topological space $X$ comes with a specialization preordering
$\leq$, defined by $x \leq y$ if and only if every open neighborhood
of $x$ contains $y$.  It is a partial ordering if and only if $X$ is
$T_0$.  The specialization preordering of a poset $P$ is its original
ordering.  A subset of a space $X$ is \emph{saturated} if and
only if it is equal to the intersection of its open neighborhoods,
equivalently if it is upwards-closed in the specialization
preordering.  The \emph{saturation} of $A$ in $X$ is the intersection
of the open neighborhoods of $A$, and is equal to
$\upc A \eqdef \{y \in X \mid \exists x \in A, x \leq y\}$.

A \emph{c-space} is a topological space $X$ such that every point has
a basis of neighborhoods of the form $\upc y$, where
$\upc y \eqdef \{z \in X \mid y \leq z\}$; in other words, for every
$x \in X$, for every open neighborhood $U$ of $x$, there is a point
$y \in U$ such that $x \in \interior {\upc y}$ \cite{Erne:ABC}.  For
example, $\Rp$ is a c-space in its Scott topology.

In general, all continuous posets from domain theory are c-spaces in
their Scott topology.  The definition goes as follows.  In every poset
$P$, the \emph{way-below} relation $\ll$ is defined by $x \ll y$ if
and only if every directed subset $D$ of $P$ whose supremum exists and
is larger than or equal to $y$ contains an element larger than or
equal to $x$.  Let $\ddarrow x \eqdef \{y \in P \mid y \ll x \}$.  $P$
is \emph{continuous} if and only if $\ddarrow x$ is directed for every
$x \in P$, and $x$ is its supremum.  If so, the Scott topology has a
base of open subsets of the form
$\uuarrow x \eqdef \{y \in P \mid x \ll y\}$, $x \in P$; such sets are
equal to the interior $\interior {\upc x}$ in the Scott topology.  For
example, in $\Rp$ and in $\creal$, $s \ll t$ if and only if $s=0$ or
$s < t$.

We will use various notions of embeddings.  An \emph{order embedding}
of a poset $P$ in a poset $Q$ is a function $f \colon P \to Q$ that
preserves and reflects orders, namely $f (x) \leq f (y)$ if and only
if $x \leq y$.  Every order embedding is injective.  We decide to call
\emph{preorder embedding} any order-preserving and reflecting map
between preordered sets, despite the fact that such a map may fail to
be injective.

A \emph{topological embedding} $f \colon X \to Y$ between topological
spaces is a homeomorphism of $X$ onto the image of $f$, with the
subspace topology induced by the inclusion in $Y$.  Equivalently, it
is an injective continuous map that is \emph{full}, in the sense that
for every open subset $U$ of $X$, there is an open subset $V$ of $Y$
such that $U = f^{-1} (V)$.  The term ``full'' originates from
\cite{GLHL:inf:words}.  A full continuous map is injective, hence a
topological embedding, if $X$ is $T_0$.

A \emph{semitopological cone} is a cone with a topology that makes $+$
and $\cdot$ separately continuous---the topology on $\Rp$ being the
Scott topology \cite[Definition~4.2]{Keimel:topcones2}.  It is
\emph{topological} if $+$ and $\cdot$ are jointly continuous
\cite[Definition~4.2]{Keimel:topcones2}.  It turns out that if $\cdot$
is separately continuous, then it is jointly continuous
\cite[Corollary~6.9(a)]{Keimel:topcones2}, because of the following,
which we will call Ershov's observation
\cite[Proposition~2]{Ershov:aspace:hull}: given a c-space $X$, and two
topological spaces $Y$ and $Z$, and function
$f \colon X \times Y \to Z$ is jointly continuous if and only if it is
separately continuous.  (In all fairness, a more general statement had
been obtained by Jimmie D. Lawson \cite[Theorem~2]{Lawson:pointwise},
extending and fixing a previous result due to Banaschewski
\cite[Proposition~6]{Banaschewski:ext}, which says that the class of
spaces $X$ such that the latter property holds is the class of locally
finitary compact spaces---which include the c-spaces.  A space is
\emph{locally finitary compact} if and only if every point has a base
of neighborhoods of the form $\upc E \eqdef \bigcup_{y \in E} \upc y$
with $E$ finite.)  It also follows from Ershov's observation that
every \emph{c-cone}, namely every semitopological cone whose
underlying topological space is a c-space, is topological.

Given a semitopological cone $\C$, every positively homogeneous lower
semicontinuous function $f \colon \C \to \creal$ gives rise to a
proper open subset $f^{-1} (]1, \infty])$ of $\C$.  Conversely, every
proper open subset $U$ of $\C$ gives rise to a positively homogeneous
lower semicontinuous function $M^U$, called its \emph{upper Minkowski
  functional}:
$M^U (x) \eqdef \sup \{r \in \Rp \diff \{0\} \mid (1/r) \cdot x \in
U\}$ \cite[Section~7]{Keimel:topcones2}.  The two constructions are
inverse of each other, and are monotonic: if $U \subseteq V$ then
$M^U \leq M^V$ (pointwise), and if $f \leq g$ then
$f^{-1} (]1, \infty]) \subseteq g^{-1} (]1, \infty])$.  This bijective
correspondence specializes to one between superlinear, resp.\
sublinear, resp.\ linear lower semicontinuous functions and proper
convex open sets, resp.\ proper concave open sets, resp.\ proper open
half-spaces; a set is \emph{convex} if and only if it is closed under
all operations $+_a$, $a \in [0, 1]$, \emph{concave} if and only if
its complement is convex, and a \emph{half-space} if and only if it is
both convex and concave.

Here are a few examples of semitopological cones.
\begin{example}
  \label{exa:cone:R}
  $\creal$, with its usual notions of addition, multiplication, and
  the Scott topology of the usual ordering $\leq$, is a c-cone, hence
  a topological cone.
\end{example}

\begin{example}[Example~3.5 in \cite{GLJ:Valg}]
  \label{exa:dcone:LX}
  For every space $X$, let $\Lform X$ be the set of all lower
  semicontinuous maps from $X$ to $\creal$, namely the set of all
  continuous maps from $X$ to the space $\creal$ with its Scott
  topology.  We equip it with pointwise addition, pointwise
  multiplication by a scalar, and this yields a semitopological cone
  when given the Scott topology of the pointwise ordering.  When $X$
  is core-compact, namely when the lattice $\Open X$ of open subsets
  of $X$ is a continuous dcpo with respect to inclusion, $\Lform X$ is
  a continuous dcpo, hence a c-cone, hence a topological cone.

  There are other topologies that make $\Lform X$ a semitopological
  cone, for example the topology of pointwise convergence (since then
  $\Lform X$ appears as a subpace of the product ${\creal}^X$), or the
  weak$^*$upper topology (see \cite[Example~5.3]{Keimel:topcones2});
  the latter is always topological.
\end{example}

\begin{example}[Section~6.1 in \cite{heckmann96}]
  \label{exa:semilatt:cone}
  (See also Example~4.5 in \cite{Keimel:topcones2} and Example~3.6 in
  \cite{GLJ:Valg}.)  Every sup-semilattice $L$ that is pointed in the
  sense that it has a least element $0$ is canonically an ordered cone
  with supremum $\vee$ as addition and scalar multiplication defined
  by $a \cdot x \eqdef x$ if $a \neq 0$, $a \cdot x \eqdef 0$
  otherwise.  With the Scott topology, it is a semitopological cone,
  and a topological cone if $L$ is a continuous poset
  \cite{heckmann96}.
\end{example}

The next example relies on a notion that is close to measures, and
which is in a sense our central object of study: \emph{continuous
  valuations}
\cite{saheb-djahromi:meas,Lawson:valuation,jones89,Jones:proba}.  Let
$\Open X$ denote the lattice of open subsets of a space $X$.  This is
in particular a dcpo; so is $\creal$.  A \emph{continuous valuation}
on a space $X$ is a map $\nu \colon \Open X \to \creal$ that is
\emph{strict} ($\nu
(\emptyset)=0$), 
\emph{modular} (for all $U, V \in \Open X$,
$\nu (U) + \nu (V) = \nu (U \cup V) + \nu (U \cap V)$) and
Scott-continuous.  Among all continuous valuations, we find the
\emph{simple valuations} $\sum_{i=1}^n a_i \delta_{x_i}$, where
$n \in \nat$, each $a_i$ is in $\Rp$ and each $x_i$ is a point of $X$;
the \emph{Dirac valuation} $\delta_x$ maps every $U \in \Open X$ to
$1$ if $x \in U$, to $0$ otherwise.

We say that $\nu$ is a \emph{probability} valuation
if and only if $\nu (X)=1$, and a \emph{subprobability} valuation if
and only if $\nu (X) \leq 1$.

We will need to touch upon Reinhold Heckmann's notion of
\emph{point-continuous} valuation \cite{heckmann96} from time to time.
Let $\Open_\pw X$ denote the set of open subsets of $X$ with the
\emph{pointwise topology}, whose subbasic open sets are
$[x \in] \eqdef \{U \in \Open X \mid x \in U\}$.  A point-continuous
valuation is a strict, modular, continuous map from $\Open_\pw X$ to
$\creal$ ($\creal$ has its Scott topology, as usual).  Since every
open subset of $\Open_\pw X$ is Scott-open, every point-continuous
valuation is a continuous valuation.  Every simple valuation if a
point-continuous.  There are point-continuous valuations that are not
simple, e.g. $\infty.\delta_x$ or
$\sum_{n \in \nat} \frac 1 {2^n} \delta_n$ on $\nat$ (with its
discrete topology).  Point-continuous valuations and continuous
valuations coincide on continuous dcpos
\cite[Theorem~6.9]{heckmann96}, but not on general spaces, not even on
dcpos \cite{GLJ:minval}.

Continuous valuations are an alternative to measures that have become
popular in domain theory \cite{jones89,Jones:proba}.  There are
several results that link continuous valuations to measures, starting
with \cite{saheb-djahromi:meas} and \cite{Lawson:valuation}.  To say
it briefly, if slightly inaccurately, continuous valuations and
measures are the same thing, on reasonable enough spaces, and a pretty
complete paper on this topic is \cite{KL:measureext}.  More precisely,
there is a two-way correspondence between continuous valuations and
measures.  In one direction, every measure on the Borel
$\sigma$-algebra of $X$ induces a continuous valuation on $X$ by
restriction to the open sets, if $X$ is hereditarily Lindel\"of
(namely, if every directed family of open sets contains a cofinal
monotone sequence).  This is an easy observation, and one half of
Adamski's theorem \cite[Theorem~3.1]{Adamski:measures}, which states
that a space is hereditary Lindel\"of if and only if every measure on
its Borel $\sigma$-algebra restricts to a continuous valuation on its
open sets.  In the other direction, every continuous valuation on a
space $X$ extends to a measure on the Borel sets provided that $X$ is
an LCS-complete space \cite[Theorem~1]{dBGLJL:LCS}.  An
\emph{LCS-complete} space is a space homeomorphic to a $G_\delta$
subset of a locally compact sober space.  Every continuous dcpo in its
Scott topology, being locally compact and sober, is LCS-complete.
Every Polish space, and more generally every quasi-Polish space, is
LCS-complete; see \cite{dBGLJL:LCS}.

We write $\Val X$ for the space of continuous valuations on a space
$X$, with the \emph{weak topology}, whose subbasic open sets are
defined by $[U > r] \eqdef \{\nu \in \Val X \mid \nu (U) > r\}$, where
$U \in \Open X$ and $r \in \Rp \diff \{0\}$.  It will sometimes be practical
to realize that the sets $[r \ll U] \eqdef \{\nu \in \Val X \mid r \ll
\nu (U)\}$ with $U \in \Open X$ and $r \in \Rp$ form a subbase of the
topology, too---simply because $[r \ll U]$ is equal to $[U > r]$ when
$r \neq 0$ and to the whole of $\Val X$ if $r=0$.

There is a notion of integral $\int_{x \in X} h (x) \,d\nu$, or
$\int h \,d\nu$ for short, where $h \in \Lform X$ and
$\nu \in \Val X$, which was originally defined in a similar manner to
the usual Lebesgue integral by Jones \cite{jones89,Jones:proba}; see
\cite[Section~2.2]{GLJ:Valg} for a quick survey.  A practical
definition is by using a Choquet formula
$\int h \,d\nu \eqdef \int_0^\infty \nu (h^{-1} (]t, \infty]))\,dt$,
where the integral on the right is a Riemann integral.  The integral
$\int h \,d\nu$ is linear and Scott-continuous in both $h$ and $\nu$.
Additionally, the weak topology also has a subbase of open sets
$[h > r] \eqdef \{\nu \in \Val X \mid \int h \,d\nu > r\}$, where
$h \in \Lform X$ and $r \in \Rp$ (or $r \in \Rp \diff \{0\}$, since
$[h > 0] = \bigcup_{r > 0} [h > r]$).

$\Val_1 X$ is its subspace of probability valuations and
$\Val_{\leq 1} X$ is its subspace of subprobability valuations.  We
will also write $\Val_b X$ for the subspace of \emph{bounded}
continuous valuations, namely those $\nu \in \Val X$ such that
$\nu (X) < \infty$.  We will write $\Val_\fin X$ for the subspace of
\emph{simple} valuations, $\Val_\pw X$ for the subspace of
\emph{point-continuous} valuations, $\Val_{\leq 1, \fin} X$ for simple
subprobability valuations, $\Val_{1, \fin} X$ for simple probability
valuations, $\Val_{\leq 1, \pw}$ for point-continuous subprobability
valuations, $\Val_{1, \pw}$ for point-continuous probability
valuations.  The specialization ordering in each case is given by the
\emph{stochastic ordering}: $\mu \leq \nu$ if and only if
$\mu (U) \leq \nu (U)$ for every $U \in \Open X$.

\begin{example}
  \label{exa:VX:top:cone}
  With the obvious addition and scalar multiplication, $\Val X$ is a
  $T_0$ topological cone for every space $X$
  \cite[Proposition~3.8]{GLJ:Valg}.  With the stochastic ordering, it
  is an ordered cone.
\end{example}

Let us mention some results from \cite{AMJK:scs:prob}, which should
help understand how the topology of $\Val X$, or rather of
$\Val_{\leq 1} X$ and of $\Val_1 X$ relates to the vague topology on
the corresponding spaces of measures.  When $X$ is a compact pospace
(a compact space with an ordering which is compatible with the
topology, in the sense that its graph is closed in $X \times X$), the
space $\mathfrak M_{\leq 1} (X)$ of all subprobability measures on
$X$, with the vague topology and the stochastic ordering (defined
here, as before, by $\mu \leq \mu'$ if and only if
$\mu (U) \leq \mu' (U)$ for every $U \in \Open X$) is also a compact
pospace \cite[Theorem~31]{AMJK:scs:prob}.  There is a well-known
one-to-one correspondence between compact pospaces and certain spaces
called the \emph{stably compact} spaces: given a compact pospace $X$,
the space $X^\uparrow$ obtained as $X$ with the topology obtained by
just keeping the open upwards-closed subsets as open sets is stably
compact, and one can recover the original topology as the \emph{patch
  topology} (the coarsest topology containing the open sets and the
compact saturated sets), and the original ordering as the
specialization ordering of $X$.  Then
$\mathfrak M_{\leq 1} (X)^\uparrow$ is isomorphic, as a topological
cone, with $\Val_{\leq 1} (X^\uparrow)$, and therefore
$\mathfrak M_{\leq 1} (X)$ is isomorphic, as a compact pospace, to the
patch space of $\Val_{\leq 1} (X^\uparrow)$
\cite[Theorem~36]{AMJK:scs:prob}.  A similar result holds for spaces
of probability measures or valuations.


\section{Barycentric algebras}
\label{sec:baryc-algebr-preord}

\subsection{Plain barycentric algebras}
\label{sec:plain-baryc-algebr}

A \emph{barycentric algebra} is a set $\B$ with a family of binary
operations $+_a$, for every $a \in [0, 1]$, satisfying:
\begin{align*}
  x +_1 y & = x \\
  x +_a x & = x \\
  x +_a y & = y +_{1-a} x \\
  (x +_a y) +_b c & = x +_{ab} (y +_{\frac {(1-a)b} {1-ab}} z)
          & \text{if $a < 1$ and $b < 1$},
\end{align*}
for all $x, y, z \in \B$ and $a, b \in [0, 1]$.  An \emph{affine map}
$f \colon \B \to \Alg$ between barycentric algebras is a function such
that $f (x +_a y) = f (x) +_a f (y)$ for all $x, y \in \B$ and
$a \in [0, 1]$.  Those form a category, which we write as $\BA$.

Every cone is a barycentric algebra provided that we define $x +_a y$
as $a \cdot x + (1-a) \cdot y$, and every linear map between cones is
then affine.  Every convex subset of a cone is a barycentric algebra,
with the induced operations $+_a$.  There are stranger examples of
barycentric algebras.  For example, every sup-semilattice $X$ is a
barycentric algebra, with $x +_a y \eqdef x \vee y$ if
$a \in {]0, 1[}$ (and $x +_1 y \eqdef x$, $x +_0 y \eqdef y$): see
\cite[Example~1.3]{Flood:semiconvex}; Flood considers an
inf-semilattice instead, and while this is immaterial for now, the
preordered and semitopological variants will require that the
semilattice operation $\vee$ be a supremum, not an infimum.

There is a free cone on every barycentric algebra $\B$ \cite[Standard
construction~2.4]{KP:mixed}, which Keimel and Plotkin write as
$C_{\B}$, but which we prefer to write as $\conify (\B)$.  This is
obtained as the collection of pairs $(r, x)$ where
$r \in \Rp \diff \{0\}$ and $x \in \B$, plus a distinguished element
$0$.  It is a cone with the following operations $+$ and $\cdot$:
\begin{align*}
  a \cdot (r, x) & \eqdef (ar, x)
  & (r, x) + (s, y) & \eqdef (r+s, x +_{\frac r {r+s}} y) \\
  a \cdot 0 & \eqdef 0
  & (r, x) + 0 & \eqdef (r, x) \\
  0 \cdot (r, x) & \eqdef 0
  & 0 + (s, y) & \eqdef (s, y) \\
  0 \cdot 0 & \eqdef 0
  & 0 + 0 & \eqdef 0
\end{align*}
where $a \in \Rp \diff \{0\}$.  There is a function
$\etac_{\B} \colon \B \to \conify (\B)$ defined by
$\etac_{\B} (x) \eqdef (1, x)$ for every $x \in \B$, and $\etac_{\B}$
is injective and affine.  For every function $f \colon \B \to \C$ to a
cone $\C$, there is a unique positively homogeneous map
$f^\cext \colon \conify (\B) \to \C$ such that
$f^\cext \circ \etac_{\B} = f$: just define $f^\cext (0) \eqdef 0$ and
$f^\cext (r, x) \eqdef r f (x)$ for all
$(r, x) \in \conify (\B) \diff \{0\}$.  (This function is written as
$\overline f$ in \cite{KP:mixed}.)  Additionally, if $f$ is affine
then $f^\cext$ is linear, and this shows that $\conify (\B)$ is the
free cone over the barycentric algebra $\B$.  This is due to Flood
\cite[Section~1.4]{Flood:semiconvex}.  Since $\etac_{\B}$ is injective
and affine, we can think of $\B$ as being affinely embedded in the
cone $\conify (\B)$, and this can be used to simplify computations a
lot; in fact, we will almost never use the defining axioms of the
operations $+_a$ in a barycentric algebra.

For example, this trick can be used to prove the following, originally
due to Stone \cite[Definition~1, Lemma~2, Lemma~3]{Stone:bary}, see
also Skornyakov \cite[Section~1]{Skornyakov:stoch:alg}.  We let
$\Delta_n \eqdef \{(\alpha_1, \cdots, \alpha_n) \in \Rp^n \mid
\sum_{i=1}^n \alpha_i=1\}$ be the standard simplex, and we abbreviate
$(\alpha_1, \cdots, \alpha_n)$ as $\vec\alpha$.
\begin{defprop}[Barycenters, part 1]
  \label{prop:bary:1}
  Let $\B$ be a barycentric algebra.  For every $n \geq 1$, for all
  points $x_1$, \ldots, $x_n$ of $\B$ and for every
  $(a_1, \cdots, a_n) \in \Delta_n$, there is a unique point of $\B$,
  which we write as $\sum_{i=1}^n a_i \cdot x_i$, whose image under
  $\etac_{\B}$ is equal to $\sum_{i=1}^n a_i \cdot \etac_{\B} (x_i)$.
  This is the \emph{barycenter} of the points $x_i$ with weights
  $a_i$.
  \begin{enumerate}
  \item It is equal to $x_n$ if $a_n=1$ and to
    $(\sum_{i=1}^{n-1} \frac {a_i} {1-a_n} \cdot x_i) +_{1-a_n} x_n$
    otherwise,
  \item it is invariant under permutations, namely for every bijection
    $\sigma$ of $\{1, \cdots, n\}$ onto itself, $\sum_{i=1}^n a_i
    \cdot x_i = \sum_{j=1}^n a_{\sigma (i)} \cdot x_{\sigma (i)}$,
  \item if $a_n=0$, then
    $\sum_{i=1}^n a_i \cdot x_i = \sum_{i=1}^{n-1} a_i \cdot x_i$,
  \item and $\sum_{i=1}^n a_i \cdot x_i$ is mapped by every affine
    function $f$ from $\B$ to a barycentric algebra $\Alg$ to
    $\sum_{i=1}^n a_i \cdot f (x_i)$.
  \end{enumerate}
\end{defprop}
\begin{proof}
  Uniqueness comes from the injectivity of $\etac_{\B}$.  For
  existence, we give an explicit description of this point, by
  induction on $n \geq 1$.  If $n=1$, then $a_1=1$, and we define
  $\sum_{i=1}^n a_i \cdot x_i$ as $x_1$.  In general, if $a_n=1$, then
  we define $\sum_{i=1}^n a_i \cdot x_i$ as $x_n$.  Otherwise, and if
  $n \geq 2$, then $1-a_n$ is non-zero, and we define
  $\sum_{i=1}^n a_i \cdot x_i$ as
  $(\sum_{i=1}^{n-1} \frac {a_i} {1-a_n} \cdot x_i) +_{1-a_n} x_n$.
  This fits our needs, since $\etac_{\B}$ is affine.  This also shows
  item~1.  Item~2 is because $\sum_{i=1}^n a_i \cdot x_i$ and
  $\sum_{j=1}^n a_{\sigma (i)} \cdot x_{\sigma (i)}$ are both mapped
  by $\etac_{\B}$ to
  $\sum_{i=1}^n a_i \cdot \etac_{\B} (x_i) = \sum_{j=1}^n a_{\sigma
    (i)} \cdot \etac_{\B} (x_{\sigma (i)})$, and since $\etac_{\B}$ is
  injective.

  3. If $a_n=0$, then $\etac_{\B} (\sum_{i=1}^n a_i \cdot x_i) =
  \sum_{i=1}^n a_i \cdot \etac_{\B} (x_i) = \sum_{i=1}^{n-1} a_i \cdot
  \etac_{\B} (x_i) = \etac_{\B} (\sum_{i=1}^{n-1} a_i \cdot x_i)$, and
  we conclude since $\etac_{\B}$ is injective.
  
  4.  Given any affine function $f$ from $\B$ to a barycentric algebra
  $\Alg$, we verify that
  $f (\sum_{i=1}^n a_i \cdot x_i) = \sum_{i=1}^n a_i \cdot f (x_i)$ by
  induction on $n \geq 1$.  If $n=1$ or if more generally $a_n=1$,
  then both sides are equal to $f (x_n)$.  Otherwise, the left-hand
  side is equal to
  $f (\sum_{i=1}^{n-1} \frac {a_i} {1-a_n} \cdot x_i) +_{1-a_n} x_n)$,
  hence to
  $f (\sum_{i=1}^{n-1} \frac {a_i} {1-a_n} \cdot x_i) +_{1-a_n} f
  (x_n)$ since $f$ is affine.  This is equal to
  $(\sum_{i=1}^{n-1} \frac {a_i} {1-a_n} \cdot f (x_i)) +_{1-a_n} f
  (x_n)$ by induction hypothesis, hence to
  $ \sum_{i=1}^n a_i \cdot f (x_i)$.  \qed
\end{proof}

The same embedding trick is also helpful in proving the following
\emph{entropic law}:
\begin{align}
  \label{eq:entropic}
  (x +_a y) +_b (z +_a t) & = (x +_b t) +_a (y +_b t)
\end{align}
for all $x, y, z, t \in \B$ and $a, b \in [0, 1]$ \cite[Section~1.5,
identity~(7)]{Flood:semiconvex}; see also \cite[paragraph after
Lemma~2.3]{KP:mixed}.

\subsection{Preordered and ordered barycentric algebras}
\label{sec:preord-order-baryc}

Before we define the semitopological and topological variants of
barycentric algebras, we need to talk about the preordered and ordered
variants.  A \emph{preordered barycentric algebra} is a barycentric
algebra $\B$ with a preordering $\leq$ such that $+_a$ is monotonic
for every $a \in [0, 1]$.  An \emph{ordered barycentric algebra} is
one where $\leq$ is antisymmetric \cite[Definition 2.6]{KP:mixed}.

Every (pre)ordered cone is a (pre)ordered barycentric algebra.  Every
sup-semilattice is an ordered barycentric algebra, with
$x +_a y \eqdef x \vee y$ if $a \in {]0, 1[}$.

For every barycentric algebra $\B$ and every preordered barycentric
algebra $\Alg$, we say that a function $f \colon \B \to \Alg$ is
\emph{concave} (resp.\ \emph{convex}) if and only if
$f (x +_a y) \geq f (x) +_a f (y)$ (resp.\ $\leq$) for all
$x, y \in \B$ and $a \in [0, 1]$.  Hence an affine map is one that is
both concave and convex.

\begin{lemma}
  \label{lemma:f*:conv:conc}
  For every 
  barycentric algebra $\B$ and every preordered cone
  $\C$, for every function $f \colon \B \to \C$, $f^\cext$ is
  positively homogeneous; if $f$ is concave then $f^\cext$ is
  superlinear, if $f$ is convex then $f^\cext$ is sublinear, and if
  $f$ is affine then $f^\cext$ is linear.
\end{lemma}
\begin{proof}
  We have already mentioned that $f^\cext$ is positively homogeneous,
  and that it is linear if $f$ is affine.  If $f$ is concave, then we
  check that $f^\cext$ is superadditive, hence superlinear.  We have
  $f^\cext (u+0) = f^\cext (u) = f^\cext (u) + f^\cext (0)$ for every
  $u \in \conify (\B)$, and it remains to verify that
  $f^\cext ((r, x) + (s, y)) \geq f^\cext (r, x) + f^\cext (s, y)$.
  Indeed,
  $f^\cext ((r, x) + (s, y)) = f^\cext (r+s, x +_{\frac r {r+s}} y) =
  (r+s) f (x +_{\frac r {r+s}} y) \geq (r+s) (\frac r {r+s} f (x) +
  \frac s {r+s} f (y)) = r f (x) + s f (y) = f^\cext (r, x) + f^\cext
  (s, y)$.  Similarly, if $f$ is convex, then $f^\cext$ is subadditive
  hence sublinear.  \qed
\end{proof}

Keimel and Plotkin build the free semiordered cone (the free ordered
cone, in the usual terminology, see Remark~\ref{rem:semiordered}) on
an ordered barycentric algebra $\B$ \cite[Standard
construction~2.7]{KP:mixed}, and this is simply $\conify (\B)$, with
ordering defined by $0 \leq 0$ and $(r, x) \leq (s, y)$ if and only if
$r=s$ and $x \leq y$ in $\B$.  We are more interested in the free
ordered, in general the free preordered, cone.

To this end, we need a preordering $\leq$ on $\conify (\B)$ that
contains the Keimel-Plotkin (pre)ordering, and that satisfies
$u \leq u+v$ for all $u, v \in \conify (\B)$.  Indeed, in a preordered
cone, $0 \leq v$, whence $u \leq u + v$.  Temporarily calling $\leq_1$
the Keimel-Plotkin preordering and $\leq_2$ the smallest preordering
such that $u \leq_2 u+v$ for all $u, v \in \conify (\B)$, we will
build the reflexive transitive closure of $\leq_1 \cup \leq_2$.  This
simplifies to $\leq_2 ; \leq_1$, where $u \leq_2 ; \leq_1 v$ if and
only if $u \leq_2 w \leq_1 v$ for some $w \in \conify (\B)$, and we
arrive at the following definition.
\begin{definition}[$\leqc$]
  \label{defn:bary:alg:conify:ord}
  For every preordered barycentric algebra $\B$, the relation $\leqc$
  on $\conify (\B)$ is defined by letting $u \leqc v$ if and only if:
  \begin{itemize}
  \item either $u=v=0$,
  \item or $v$ can be written as $(s, y)$ and there are
    $u' \in \conify (\B)$ and $y_1 \in \B$ such that $u+u' = (s,y_1)$
    and $y_1 \leq y$.
  \end{itemize}
\end{definition}

\begin{proposition}
  \label{prop:bary:alg:conify:ord}
  For every preordered barycentric algebra $\B$:
  \begin{enumerate}
  \item the relation $\leqc$ is a preordering, and turns
    $\conify (\B)$ into a preordered cone;
  \item $\etac_{\B}$ is an injective affine preorder embedding;
  \item for every monotonic map $f \colon \B \to \C$ to a preordered
    cone $\C$, there is a unique monotonic positively homogeneous map
    $f^\cext \colon \conify (\B) \to \C$ such that
    $f^\cext \circ \etac_{\B} = f$; if $f$ is affine (resp.\ concave,
    convex) then $f^\cext$ is linear (resp.\ superlinear, sublinear).
  \item In particular, $\conify (\B)$ with the preordering $\leqc$ is
    the free preordered cone over the preordered barycentric algebra
    $\B$.
  \end{enumerate}
\end{proposition}
\begin{proof}
  1.  The relation $\leqc$ just defined on $\conify (\B)$ is reflexive.
  We also observe that $0 \leqc v$ for every $v \in \conify (\B)$: when
  $v$ is of the form $(s, y)$, this follows by taking
  $u' \eqdef (s, y)$ and $y_1 \eqdef y$.

  Let us show that $\leqc$ is transitive.  We assume three points
  $u, v, w \in \conify (\B)$ such that $u \leqc v \leqc w$.  If $u=v$
  or $v=w$, then $u \leqc w$ is trivial, so we assume $u \neq v$ and
  $v \neq w$.  If $u=0$, then $u \leqc w$, since $0$ is below every
  point, as we have seen.  If $v=0$, then $u \leqc v$ is only possible
  if $u=0$, whence $u \leqc w$.  Otherwise, none of $u, v, w$ is equal
  to $0$; hence we write $u$ as $(r, x)$, $v$ as $(s, y)$, and $w$ as
  $(t, z)$.  By Definition~\ref{defn:bary:alg:conify:ord}, there are
  points $u', v' \in \conify (\B)$ and $y_1, z_1 \in \B$ such that
  $u+u' = (s, y_1)$, $y_1 \leq y$, $v+v' = (t,z_1)$, and $z_1 \leq z$.
  Since $u \neq v$, $u' \neq 0$, so we can write $u'$ as $(r',x')$;
  similarly we can write $v'$ as $(s', y')$.  Then
  $u+(u'+v') = (u+u') + v'$ (since $\conify (\B)$ is a cone)
  $= (s, y_1) + (s', y') = (s+s', y_1 +_{\frac s {s+s'}} y')$.  Now
  $v+v' = (s, y) + (s', y') = (s+s', y +_{\frac s {s+s'}} y')$ is
  equal to $(t, z_1)$, so $t=s+s'$ and
  $y +_{\frac s {s+s'}} y' = z_1$.  Since $t=s+s'$,
  $u+(u'+v') = (t, y_1 +_{\frac s {s+s'}} y')$.  Then
  $y_1 +_{\frac s {s+s'}} y' \leq y +_{\frac s {s+s'}} y'$ (since
  $y_1 \leq y$ and $\B$ is preordered) $= z_1 \leq z$.  This shows
  that $u \leqc (t, z) = w$.

  We verify that the cone operations are monotonic on $\conify (\B)$.
  We start with $\_ \cdot w$, where $w \in \conify (\B)$.  If $w=0$,
  this is the constant $0$ map, which is monotonic, so we assume that
  $w$ is of the form $(t, z)$.  For all $a \leq b$ in $\Rp$, if $a=0$
  then $a \cdot w = 0$, which is less than or equal to every element
  of $\conify (\B)$, hence to $b \cdot w$.  If $a=b$, the inequality
  $a \cdot w \leqc b \cdot w$ is clear, too.  Otherwise, $0 < a < b$.
  Let $u' \eqdef ((b-a)t, z)$.  Then
  $(a \cdot w) + u' = (at, z) + ((b-a)t, z) = (bt, z +_{a/b} z) = (bt,
  z) = b \cdot w$, so $a \cdot w \leqc b \cdot w$.

  We proceed with $a \cdot \_$, where $a \in \Rp$.  Let $u \leq v$ in
  $\conify (\B)$.  We need to show that $a \cdot u \leq a \cdot v$.
  This is clear if $u=v=0$ or if $a=0$.  Hence we assume that $a > 0$,
  and that $u+u' = (s, y_1)$ and $v=(s, y)$ for some
  $u' \in \conify (\B)$ and $y_1 \leq y$ in $\B$.  Then
  $(a \cdot u) + (a \cdot u') = a \cdot (u+u') = (as, y_1)$,
  $y_1 \leq y$, and $a \cdot v = (as, y)$, so
  $a \cdot u \leq a \cdot v$.

  Finally, we show that $\_ + w$ is monotonic for every
  $w \in \conify (\B)$.  If $w=0$, this is the identity map, for which
  the claim is clear, so we assume that $w$ is of the form $(t, z)$.
  Let $u \leqc v$ in $\conify (\B)$.  If $u=v=0$, then
  $u+w = w = v +w$, so $u + w \leqc v + w$.  Otherwise,
  $u+u' = (s, y_1)$ for some $u' \in \conify (\B)$ and some
  $y_1 \leq y$ such that $v = (s, y)$.  Then
  $(u+w) + u' = (u+u') + w = (s, y_1) + (t, z) = (s+t, y_1 +_{s/(s+t)}
  z)$ and $v+w = (s, y) + (t, z) = (s+t, y+_{s/(s+t)} z)$.  Since
  $+_{s/(s+t)}$ is monotonic and $y_1 \leq y$, we have
  $y_1 +_{s/(s+t)} z \leq y +_{s/(s+t)} z$, so $u+w \leqc v+w$.
  
  2.  We have seen that the map $\etac_{\B}$ is affine and injective.
  For all $x \leq y$ in $\B$, we have
  $\etac_{\B} (x) \leqc \etac_{\B} (y)$, namely $(1, x) \leqc (1, y)$:
  just take $u' \eqdef 0$, $y_1 \eqdef x$ and $s \eqdef 1$ in
  Definition~\ref{defn:bary:alg:conify:ord}.  Conversely, if
  $\etac_{\B} (x) \leqc \etac_{\B} (y)$, then there are
  $u' \in \conify (\B)$ and $y_1 \in \B$ such that
  $(1, x) + u' = (1, y_1)$ and $y_1 \leq y$.  The point $u'$ cannot be
  of the form $(r', x')$, since then $(1,x)+(r',x')$ would be equal to
  $(1+r', x +_{\frac 1 {1+r'}} x')$, whose first component is not
  equal to $1$.  Hence $u'=0$; then $y_1 = x$, and therefore
  $x \leq y$, since $y_1 \leq y$.

  3.  Let $f$ be any monotonic map from $\B$ to $\C$.  If $f^\cext$
  exists, because $\conify (\B)$ is the free cone on $\B$.  In order
  to show that $f^\cext$ exists, we build it as before, letting
  $f^\cext (0) \eqdef 0$ and $f^\cext (a, x) \eqdef a f (x)$; then
  $f^\cext$ is positively homogeneous and
  $f^\cext \circ \etac_{\B} = f$.  We claim that $f^\cext$ is
  monotonic.  Let $u \leqc v$ in $\conify (\B)$.  If $u=v=0$, then
  clearly $f^\cext (u) \leq f^\cext (v)$.  Otherwise, there are
  $u' \in \conify (\B)$ and $y_1 \in \B$ such that $u+u' = (s, y_1)$,
  $v = (s, y)$, and $y_1 \leq y$.  Then
  $f^\cext (u+u') = f^\cext (u) + f^\cext (u')$ since $f^\cext$ is
  linear.  Since $\C$ is preordered (in our sense, see
  Remark~\ref{rem:semiordered}), $0 \leq f^\cext (u')$, so
  $f^\cext (u) \leq f^\cext (u+u')$.  Also,
  $f^\cext (u+u') = s f (y_1) \leq s f (y)$ (since $f$ is monotonic)
  $= f^\cext (v)$.  Hence $f^\cext (u) \leq f^\cext (v)$.

  The fact that $f^\cext$ is superlinear (resp.\ sublinear, affine) If
  $f$ is concave (resp.\ convex, affine), is by
  Lemma~\ref{lemma:f*:conv:conc}.
  
  4.  Obvious consequence of item~3.  \qed
\end{proof}
Since $\etac_{\B}$ is an injective affine preorder embedding, it
follows that, up to affine bijective preorder embeddings, the
preordered barycentric algebras are exactly the convex subsets of
preordered cones.

\begin{remark}
  \label{rem:bary:alg:ord}
  Given a preordered barycentric algebra $\B$, we have $u \leqc u+v$
  for all $u, v \in \conify (\B)$, just as in every preordered cone.
  In particular, for all $r, s \in \Rp \diff \{0\}$ and for every
  $x \in \B$, $(r, x) \leqc (r+s, x)$.  Indeed,
  $(r, x) + (s, x) = (r+s, x +_{\frac r {r+s}} x) = (r+s, x)$.
\end{remark}

We also notice the following.
\begin{lemma}
  \label{lemma:bary:alg:ord:ord}
  Given an ordered (not just preordered) barycentric algebra $\B$, the
  preordering $\leqc$ on $\conify (\B)$ is antisymmetric, hence a
  partial ordering.
\end{lemma}
\begin{proof}
    Let us assume that $u \leqc v \leqc u$ in $\conify (\B)$.  Apart from
  the trivial case where $u=v=0$, $u \leqc v$ states that there are
  $u' \in \conify (\B)$ and $y_1 \in \B$ such that $u+u' = (s,y_1)$,
  $v = (s, y)$, and $y_1 \leq y$.  Similarly, $v \leqc u$ states that
  there are $v' \in \conify (\B)$ and $x_1 \in \B$ such that
  $v+v' = (r, x_1)$, $u = (r, x)$ and $x_1 \leq x$.

  Since $u+u' = (s, y_1)$ and $u = (r, x)$, we have $r \leq s$: either
  $u'=0$ and this is clear, or $u'$ is of the form $(r', x')$, and
  then $u+u' = (r+r', x +_{\frac r {r+r'}} x')$, with $r < r+r'$.
  Similarly, since $v+v' = (r, x_1)$ and $v = (s, y)$, we have
  $s \leq r$.  Therefore $r=s$.

  Then, $u+u' = (s, y_1)$ and $u = (r, x)$ with $r=s$ imply that
  $u'=0$ (otherwise $s = r+r' > r$, where $u' = (r', x')$).
  Similarly, $v+v' = (r, x_1)$ and $v = (s, y)$ with $s=r$ imply that
  $v'=0$.

  Therefore $u = (s, y_1) = (r, y_1)$, so $y_1=x$; since $y_1 \leq y$,
  it follows that $x \leq y$.  Similarly, $v = (r, x_1) = (s, x_1)$,
  so $x_1= y$; since $x_1 \leq x$, $y \leq x$.  Since $\B$ is ordered,
  it follows that $x=y$.  \qed
\end{proof}

\subsection{Semitopological and topological barycentric algebras}
\label{sec:semit-topol-baryc}

\begin{definition}
  \label{defn:bary:alg:top}
  A barycentric algebra $\B$ with a topology is:
  \begin{itemize}
  \item \emph{semitopological} if and only if the function
    $(x, a, y) \mapsto x +_a y$ is separately continuous from
    $\B \times [0, 1] \times \B$ to $\B$,
  \item \emph{quasitopological} if and only if the map
    $(x, a) \mapsto x +_a y$ is (jointly) continuous from
    $\B \times [0, 1]$, for every fixed $y \in \B$,
  \item \emph{topological} if and only if the function
    $(x, a, y) \mapsto x +_a y$ is jointly continuous from
    $\B \times [0, 1] \times \B$ to $\B$,
  \end{itemize}
  where in each case $[0, 1]$ is given its usual, metric topology.
\end{definition}
There are corresponding categories $\sTBA$ of semitopological
barycentric algebras and $\TBA$ of topological barycentric algebras,
with affine continuous maps as morphisms.  Quasitopological
barycentric algebras will almost play no role in the sequel, except in
Lemma~\ref{lemma:bary:in:cone} below.

For barycentric algebras, the implications topological $\limp$
quasitopological $\limp$ semitopological hold.
\begin{lemma}
  \label{lemma:bary:in:cone}
  Given any convex subset $A$ of a cone $\C$, $A$ is a barycentric
  algebra with the induced operation
  $x +_a y \eqdef a \cdot x + (1-a) \cdot y$.

  If $\C$ is a preordered cone, then $A$ is a preordered barycentric
  algebra.

  If $\C$ is a semitopological cone, then $A$ is a quasitopological,
  hence a semitopological barycentric algebra in the subspace
  topology.

  If $\C$ is a topological cone, then $A$ is a topological barycentric
  algebra.
\end{lemma}
\begin{proof}
  The first part is clear, as the second part on the preordered case.
  If $\C$ is semitopological, then let us fix $y \in A$, and let us
  define $f \colon \C \times \Rp \times \Rp \to \C$ by
  $f (x, a, b) \eqdef a \cdot x + b \cdot y$.  The function $f$ is
  separately continuous, hence jointly continuous by two applications
  of Ershov's observation.  Let $g \colon [0, 1] \to \Rp \times \Rp$
  be defined by $g (a) \eqdef (a, 1-a)$.  This is a continuous map,
  since $g^{-1} ({]s, \infty[} \times {]t, \infty[})$ is equal to
  $]s, 1-t[$ if $s < 1-t \leq 1$, to $]s, 1]$ if $s < 1 < 1-t$,
  and is empty otherwise.  Letting $i$ be the inclusion of $A$ into
  $\C$, it follows that
  $f \circ (i \times g) \colon A \times [0, 1] \to \C$, which maps
  $(x, a)$ to $x +_a y$, is (jointly) continuous.

  If $\C$ is topological, then the function
  $f' \colon \C \times \Rp \times \Rp \times \C \to \C$ defined by
  $f' (x, a, b, y) \eqdef a \cdot x + b \cdot y$ is (jointly)
  continuous.  We see $f'$ as a function from
  $\C \times (\Rp \times \Rp) \times \C$ to $\C$.  Then
  $f' \times (i \times g \times i) \colon A \times [0, 1] \times A \to
  \C$ is (jointly) continuous, and maps $(x, a, y)$ to $x +_a y$.  \qed
\end{proof}

Since every cone is a barycentric algebra, there is a risk of conflict
between the uses of ``semitopological'' and ``topological'' in each
case.  This risk is in fact nonexistent because of the following.
\begin{lemma}
  \label{lemma:bary:in:cone:iff}
  Let $\C$ be a cone with a topology that makes scalar multiplication
  separately continuous.  Seen as a barycentric algebra, $\C$ is
  semitopological (resp.\ topological) if and only if $\C$ is a
  semitopological (resp.\ topological) cone.
\end{lemma}
\begin{proof}
  If $\C$ is a semitopological (resp.\ topological) cone, then by
  taking $A \eqdef \C$ in Lemma~\ref{lemma:bary:in:cone}, we obtain
  that $\C$ is a semitopological (resp.\ topological) barycentric
  algebra.  Conversely, let us assume that $\C$ is semitopological
  (resp.\ topological) as a barycentric algebra.  The function $+$ is
  the composition of $(x, y) \mapsto x +_{1/2} y$ with
  $z \mapsto 2 \cdot z$.  The latter is continuous by assumption, and
  the former is separately (resp.\ jointly) continuous.  Hence $\C$ is
  a semitopological (resp.\ topological) cone.  \qed
\end{proof}

\begin{example}
  \label{exa:V1X:top:bary:alg}
  By Lemma~\ref{lemma:bary:in:cone} and Example~\ref{exa:VX:top:cone},
  $\Val_{\leq 1} X$, $\Val_1 X$ and $\Val_b X$ are topological
  barycentric algebras, being convex subspaces of the topological cone
  $\Val X$.  The same holds of $\Val_\fin X$, $\Val_{\leq 1, \fin} X$,
  $\Val_{1, \fin} X$, $\Val_\pw X$, $\Val_{\leq 1, \pw} X$, $\Val_{1,
    \pw} X$.  For the latter three, we need to observe that
  point-continuous valuations are closed under finite linear
  combinations \cite[Section~3.2]{heckmann96}.
\end{example}

\begin{example}
  \label{exa:lattice:top:bary:alg}
  Every sup-semilattice $L$ is a quasitopological barycentric algebra,
  with its Scott topology, and where $x +_a y$ is defined as the
  supremum $x \vee y$ for all $a \in {]0, 1[}$ (as $x$ when $a=1$, as
  $y$ when $a=0$).  It is topological if $L$ is a continuous poset.
  Indeed, the lift $L_\bot$, obtained by adding a fresh element $\bot$
  below every element of $L$, is a pointed sup-semilattice, hence a
  semitopological cone by Example~\ref{exa:semilatt:cone}, which is
  even topological if $L$ is a continuous poset.  Additionally, $L$ is
  convex in $L_\bot$, so we conclude by
  Lemma~\ref{lemma:bary:in:cone}.
\end{example}

\begin{example}
  \label{exa:cont:s-baryalg}
  Let us call \emph{s-barycentric algebra} any ordered barycentric
  algebra $\B$ such that $+_a$ is Scott-continuous for every
  $a \in [0, 1]$ and such that for all $x, y \in \B$, the map
  $a \mapsto x +_a y$ is continuous from $[0, 1]$, with its usual
  metric topology, to $\B_\sigma$.  (An \emph{s-cone} is an ordered
  cone that is an s-barycentric algebra \cite[Definition
  6.1]{Keimel:topcones2}.)  By definition, every s-barycentric algebra
  is a semitopological barycentric algebra in its Scott topology.
    
  Every continuous s-barycentric algebra is a topological barycentric
  algebra in its Scott topology, by Ershov's observation.
\end{example}
\begin{proposition}
  \label{prop:c:bary:alg}
  Every semitopological barycentric algebra $\B$ whose underlying
  topological space is locally finitary compact (e.g., a c-space, a
  continuous dcpo) is a topological barycentric algebra.
\end{proposition}
\begin{proof}
  We use the Banaschewski-Lawson theorem
  \cite[Theorem~2]{Lawson:pointwise}; Ershov's observation is enough
  if $\B$ is a c-space.  For each $y \in \B$, the map
  $(x, a) \mapsto x +_a y$ is separately continuous, hence jointly
  continuous since $\B$ is locally finitary compact, from
  $\B \times [0, 1]$ to $\B$.  Hence $\B$ is quasitopological.  The
  map $(x, a, y) \mapsto x +_a y$ is then separately continuous in
  $(x, a)$ and in $y$, hence jointly continuous since $\B$ is locally
  finitary compact.  \qed
\end{proof}

We have seen that every (preordered) barycentric algebra $\B$ embeds
in a (preordered) cone, namely in its free (preordered) cone
$\conify (\B)$.  The situation is different with semitopological
barycentric algebras.
\begin{definition}[Embeddable]
  \label{defn:bary:alg:embed}
  An \emph{embedding} of a semitopological barycentric algebra $\B$ in
  a semitopological cone $\C$ is an affine topological embedding of
  $\B$ in $\C$.  $\B$ is \emph{embeddable} if and only if there is an
  embedding of $\B$ in some semitopological cone $\C$.
\end{definition}
Not all semitopological barycentric algebras are embeddable: we will
give a counterexample in Example~\ref{exa:bary:alg:notembed}.  We will
see that $\B$ is embeddable if and only if $\etac_{\B}$ is an
embedding in $\conify (\B)$, which will be the free semitopological
cone over $\B$.  We need to give $\conify (\B)$ a particular topology.
\begin{definition}[Cone topology]
  \label{defn:cone:top}
  For every semitopological barycentric algebra $\B$, the \emph{cone
    topology} on $\conify (\B)$ has a subbase of open subsets
  consisting of those open subsets of
  $(\Rp \diff \{0\})_\sigma \times \B$ (with the product topology)
  that are upwards-closed with respect to the preordering $\leqc$ of
  Definition~\ref{defn:bary:alg:conify:ord}.
\end{definition}

\begin{remark}
  \label{rem:cone:top:spec}
  The specialization preordering of $\conify (\B)$, with the cone
  topology of Definition~\ref{defn:cone:top}, is \emph{not} $\leqc$ in
  general: we will give a counterexample in
  Example~\ref{exa:bary:alg:notembed}.  This may seem paradoxical,
  since we have defined the topology on $\conify (\B)$ as certain open
  sets that are upwards-closed with respect to $\leqc$, and this can
  lead to subtle errors.  The best we can say is that for all
  $u, v \in \conify (\B)$, if $u \leqc v$ then $u$ is less than or
  equal to $v$ in the specialization ordering of $\conify (\B)$.
\end{remark}

Studying the cone topology will be easier with the following notion.
\begin{definition}[Semi-concave]
  \label{defn:qaff}
  A \emph{semi-concave map} from a semitopological barycentric algebra
  $\B$ to a preordered cone $\C$ is a function $h \colon \B \to \C$
  such that for all $x, y \in \B$ and $a \in [0, 1]$,
  $h (x +_a y) \geq a \, h (x)$.
\end{definition}
We will almost exclusively use this notion when $\C=\creal$.

The following is immediate.
\begin{fact}
  \label{fact:qaff}
  Let $\B$ be a semitopological barycentric algebra.  Every concave map
  from $\B$ to $\creal$, and in particular every affine map from $\B$ to
  $\creal$ is semi-concave.  The collection of semi-concave lower
  semicontinuous maps on $\B$ is closed under scalar multiplication by
  non-negative numbers, addition, pointwise binary infima and
  pointwise arbitrary suprema.
\end{fact}

\begin{fact}
  \label{fact:qaff:homog}
  Every positively homogeneous lower semicontinuous map $f$ from a
  semitopological cone $\C$ to $\creal$ is semi-concave.
\end{fact}
Indeed,
$f (x +_a y) = f (a \cdot x + (1-a) \cdot y) \geq f (a \cdot x) = a \,
f (x)$, since $(1-a) \cdot y \geq 0$, and $+$ and $f$ are monotonic.

For every function $h$ from a barycentric algebra $\B$ to $\creal$,
there is a unique positively homogeneous map
$h^\cext \colon \conify (\B) \to \creal$ defined by
$h^\cext (0) \eqdef 0$, $h^\cext (r, x) \eqdef r h (x)$.  We use this
in the following.
\begin{proposition}
  \label{prop:bary:alg:preorder:product}
  For every semitopological barycentric algebra $\B$, the open subsets
  of $\conify (\B)$ in its cone topology are exactly the subsets of the
  form $(h^\cext)^{-1} (]1, \infty])$, where $h$ ranges over the
  semi-concave lower semicontinuous maps from $\B$ to $\creal$, plus
  $\conify (\B)$ itself.
\end{proposition}
\begin{proof}
  Let us call the \emph{pointed product topology} the topology on
  $\conify (\B)$ consisting of the open subsets of
  $(\Rp \diff \{0\})_\sigma \times \B$ with the product topology, plus
  the whole of $\conify (\B)$.  The cone topology consists of the open
  subsets in the pointed product topology that are upwards-closed.  We
  call the \emph{semi-concave topology} the topology consisting of the
  sets $(h^\cext)^{-1} (]1, \infty])$, where $h$ ranges over the
  semi-concave lower semicontinuous maps from $\B$ to $\creal$, plus
  the whole set $\conify (\B)$.

  For every open subset $V$ of $\conify (\B)$ in the semi-concave
  topology, we claim that $V$ is open in the cone topology.  If
  $V = \conify (\B)$, this is clear, so we assume that $V$ is proper.
  We write $V$ as $(h^\cext)^{-1} (]1, \infty[)$ for some semi-concave
  lower semicontinuous map $h$ from $\B$ to $\creal$.  For every
  $u \in V$, necessarily of the form $(r, x)$ with $r > 0$,
  $h^\cext (u) = r h (x) > 1$.  Multiplication is Scott-continuous on
  $\creal$, and $(r, h (x))$ is the supremum of some directed family
  of points $(a, b) \ll (r, h (x))$, equivalently $a \ll r$ and
  $b \ll h (x)$.  Hence there are two numbers $a, b \in \Rp$ such that
  $a \ll r$, $b \ll h (x)$ and $ab > 1$; in particular, both $a$ and
  $b$ are non-zero, so $r > a$ and $h (x) > b$.  The open rectangle
  ${]a, \infty[} \times h^{-1} (]b, \infty])$ then contains $u$, and
  for every point $(s, y)$ in that open rectangle,
  $h^\cext (s, y) = s h (y) > ab > 1$, so the open rectangle is included
  in $(h^\cext)^{-1} (]1, \infty]) = V$.  This shows that $V$ is open in
  the pointed product topology.  Additionally, for every
  $u \eqdef (r, x)$ in $V$, for every $v \in \conify (\B)$ such that
  $u \leqc v$, by definition (see
  Definition~\ref{defn:bary:alg:conify:ord}) we can write $v$ as
  $(s, y)$ so that for some $u' \in \conify (\B)$ and some $y_1 \in \B$,
  $u+u' = (s, y_1)$ and $y_1 \leq y$.  Since
  $(r, x) \in V = (h^\cext)^{-1} (]1, \infty])$, $h^\cext (u) = r h (x) > 1$.
  If $u'=0$, then $h^\cext (u+u') = h^\cext (u) > 1$.  Otherwise, we write
  $u'$ as $(r', x')$, and then
  $h^\cext (u+u') = h^\cext (r+r', x +_{\frac r {r+r'}} x') = (r+r') h (x
  +_{\frac r {r+r'}} x') \geq r h (x)$ (because $h$ is semi-concave)
  $= h^\cext (u) > 1$.  In any case, $h^\cext (u+u') > 1$.  Since
  $h^\cext (u+u') = h^\cext (s, y_1) = s h (y_1)$, since $y_1 \leq y$ and $h$,
  being lower semicontinuous, is monotonic, we also have
  $s h (y) > 1$, namely $h^\cext (v) > 1$.  In other words,
  $v \in (h^\cext)^{-1} (]1, \infty]) = V$, for every $v$ such that
  $u \leqc v$.  Hence $V$ is upwards-closed.  It follows that $V$ is
  open in the cone topology.
  
  Conversely, let $V$ be an open subset of $\conify (\B)$ in the cone
  topology.  If $0 \in V$, then $V = \conify (\B)$, which is open in
  $\conify (\B)$.  Henceforth, we will assume that $0 \not\in V$,
  namely that $V$ is proper.  Let
  $h (x) \eqdef \sup \{r \in \Rp \diff \{0\} \mid (1/r, x) \in V\}$,
  where the the supremum is taken in $\creal$, and the supremum of the
  empty set if $0$.

  We claim that the function $h$ is lower semicontinuous.  Let
  $x \in \B$ and let us assume that $h (x) > t$, where $t \in \Rp$.
  There is an $r \in \Rp \diff \{0\}$ such that $r > t$ and
  $(1/r, x) \in V$.  By definition of the cone topology, there is an
  open rectangle ${]a, \infty[} \times U$, with $a > 0$ and $U$ open
  in $\B$, that contains $(1/r, x)$ and is included in $V$.  Then
  $1/r > a$, $x \in U$, and for every $x' \in U$, $(1/r, x')$ is in
  ${]a, \infty[} \times U \subseteq V$, so $h (x') \geq r > t$.  Hence
  $U$ is an open neighborhood of $x$ included in
  $h^{-1} (]t, \infty])$.  Since $x$ is arbitrary in
  $h^{-1} (]t, \infty])$, $h^{-1} (]t, \infty])$ is open in $\B$.

  Ve verify that $h$ is semi-concave.  For all $x, y \in \B$ and
  $a \in [0, 1]$, we show that $h (x +_a y) \geq a h (x)$.  This is
  clear if $a=0$ or if $a=1$, so we assume that $0 < a < 1$.  Then, it
  suffices to show that for every $t < a h (x)$, $t < h (x +_a y)$.
  From $t < a h (x)$, we obtain that there is an
  $r \in \Rp \diff \{0\}$ such that $ar > t$ and $(1/r, x) \in V$.
  Then
  $(1/r, x) + ((1-a)/(ar), y) = (1/(ar), x +_{\frac {1/r} {1/(ar)}} y) =
  (1/(ar), x +_a y)$ is larger than or equal to $(1/r, x)$ by
  Remark~\ref{rem:bary:alg:ord}.  By definition of the cone topology,
  $V$ is upwards-closed; since $(1/r, x) \in V$, $(1/(ar), x +_a y)$
  is also in $V$.  By definition of $h$, $h (x +_a y) \geq ar > t$.

  For every point $u \in \conify (\B)$, $h^\cext (u) > 1$ if and only if
  $u$ is of the form $(s, x)$ with $s > 0$ and $s h (x) > 1$.  Then,
  $s h (x) > 1$ if and only if there is an $r \in \Rp \diff \{0\}$
  such that $rs > 1$ and $(1/r, x) \in V$, if and only if
  $(1/r, x) \in V$ for some $r \in \Rp \diff \{0\}$ such that
  $1/r < s$.  If so, then $(1/r, x) \leqc (s, x)$ by
  Remark~\ref{rem:bary:alg:ord}, since $s > 1/r$, and since $V$ is
  upwards-closed, it follows that $u = (s, x)$ is in $V$.  Conversely,
  if $u \in V$, we claim that $h^\cext (u) > 1$.  Since $0 \not\in V$, we
  may write $u$ as $(s, x)$ with $s > 0$.  By definition of the cone
  topology, there is an open rectangle ${]a, \infty[} \times U$ that
  contains $(s, x)$ and is included in $V$.  In particular, $s > a$
  and $x \in U$.  Let $r \in \Rp \diff \{0\}$ be such that
  $s > 1/r > a$, for example $\frac 2 {s+a}$.  Then $h (x) \geq r$, so
  $h^\cext (u) = s h (x) \geq rs > 1$.  It follows that
  $V = {(h^\cext)}^{-1} (]1, \infty])$.  Hence $V$ is open in the cone
  topology on $\conify (\B)$.  \qed
\end{proof}

\begin{lemma}
  \label{lemma:semi-concave:lsc}
  For every semi-concave lower semicontinuous map $h$ on a
  semitopological barycentric algebra $\B$, for all $x, y \in \B$, the
  map $a \mapsto \frac 1 {1-a} h (x +_a y)$ is Scott-continuous from
  $[0, 1[$ to $\creal$.
\end{lemma}
\begin{proof}
  Let $f$ be this map.  We first show that $f$ is monotonic.  Let
  $0 \leq a < b < 1$.  If $a=0$, then
  $\frac 1 {1-a} h (x +_a y) \leq \frac 1 {1-b} h (x +_b y)$ reduces
  to $h (y) \leq \frac 1 {1-b} h (x +_b y)$, namely to
  $(1-b) h (y) \leq h (x +_b y)$, or equivalently to
  $(1-b) h (y) \leq h (y +_{1-b} x)$, which follows from the fact that
  $h$ is semi-concave.  Hence let us assume $0 < a < b < 1$.  Let
  $z \eqdef y +_{1-a} x$.  We have
  $z +_{(1-b)/(1-a)} x = (y +_{1-a} x) +_{(1-b)/(1-a)} x = y +_{(1-a)
    \cdot (1-b)/(1-a)} (x +_{\frac {a (1-b)/(1-a)} {1- (1-a)\cdot
      (1-b)/(1-a)}} x) = y +_{1-b} x$, and since $h$ is semi-concave,
  $h (z +_{(1-b)/(1-a)} x) \geq \frac {1-b} {1-a} h (z)$.  In other
  words, $h (y +_{(1-b} x) \geq \frac {1-b} {1-a} h (y +_{1-a} x)$.
  Therefore
  $\frac 1 {1-b} h (y +_{1-b} x) \geq \frac 1 {1-a} h (y +_{1-a} x)$,
  or equivalently
  $\frac 1 {1-b} h (x +_b y) \geq \frac 1 {1-a} h (x +_a y)$.

  The function $f$ is continuous from $[0, 1[$ with its standard
  topology to $\creal$.  This is because $a \mapsto x +_a y$ is
  continuous from $[0, 1]$ to $\B$, $h$ is continuous from $\B$ to
  $\creal$, and multiplication by $\frac 1 {1-a}$ is
  Scott-continuous.  But the continuous monotonic maps $f$ from
  $[0, 1[$ to $\creal$ are exactly the Scott-continuous maps: if $f$
  is continuous and monotonic, for every $t \in \Rp$,
  $f^{-1} (]t, \infty])$ is upwards-closed, hence of the form $[r, 1[$
  or $]r, 1[$ with $0 \leq r \leq 1$, and the first form is open in
  $[0, 1[$ only if $r=0$.  \qed
\end{proof}

\begin{theorem}
  \label{thm:conify:semitop}
  For every semitopological barycentric algebra $\B$,
  \begin{enumerate}
  \item $\conify (\B)$ is a semitopological cone in its cone topology.
  \item The map $\etac_{\B}$ is injective, affine 
    and continuous.
  \item For every semi-concave (resp., concave, affine) continuous map
    $f \colon \B \to \C$ to a semitopological cone $\C$, there is a
    unique positively homogeneous (resp., superlinear, linear)
    continuous map $f^\cext \colon \conify (\B) \to \C$ such that
    $f^\cext \circ \etac_{\B} = f$.

    In particular, $\conify (\B)$ with the cone topology is the free
    semitopological cone over $\B$.
  \item The following are equivalent:
    \begin{enumerate}[label=(\alph*)]
    \item $\B$ is embeddable, that is, there is an affine topological
      embedding of $\B$ in a semitopological cone;
    \item $\etac_{\B} \colon \B \to \conify (\B)$ is a topological
      embedding;
    \item the open subsets of $\B$ are exactly the sets
      $h^{-1} (]1, \infty])$ where $h$ ranges over the semi-concave
      lower semicontinuous maps from $\B$ to $\creal$;
    \item the sets $h^{-1} (]1, \infty])$ where $h$ ranges over the
      semi-concave lower semicontinuous maps from $\B$ to $\creal$ form
      a subbase of its topology.
    \end{enumerate}
  \end{enumerate}
\end{theorem}
\begin{proof}
  1.  We know that $\conify (\B)$ is a cone.  For every
  $u \in \conify (\B)$, we claim that $\_ \cdot u$ is continuous from
  $\Rp$ to $\conify (\B)$.  Using
  Proposition~\ref{prop:bary:alg:preorder:product}, it suffices to
  show that
  $(\_ \cdot u)^{-1} ((h^\cext)^{-1} (]1, \infty]))$---alternatively,
  $\{a \in \Rp \mid h^\cext (a \cdot u) > 1\}$---is Scott-open in
  $\Rp$ for every semi-concave lower semicontinuous map
  $h \colon \B \to \creal$.  If $u=0$, then
  $h^\cext (a \cdot u) = h^\cext (0) = 0$, so
  $\{a \in \Rp \mid h^\cext (a \cdot u) > 1\}$ is empty; if $u$ is of
  the form $(r, x)$ with $r > 0$, then
  $h^\cext (a \cdot u) = ar \, h (x)$, so
  $\{a \in \Rp \mid h^\cext (a \cdot u) > 1\}$ is empty if $h (x)=0$,
  the whole of $]0, \infty[$ if $h (x) = \infty$, and
  $]1/(r h (x)), \infty[$ otherwise.

  Let $a \in \Rp$.  Similarly, in order to see that $a \cdot \_$ is
  continuous, it suffices to show that
  $(a \cdot \_)^{-1} ((h^\cext)^{-1} (]1, \infty]))$ is open for every
  semi-concave lower semicontinuous map $h \colon \B \to \creal$.  But
  $(a \cdot \_)^{-1} ((h^\cext)^{-1} (]1, \infty])) = \{u \in \conify (\B)
  \mid h^\cext (a \cdot u) > 1\} = \{(r, x) \in \conify (\B) \diff \{0\}
  \mid ar \, h (x) > 1\} = ((ah)^\cext)^{-1} (]1, \infty])$.  Since $ah$
  is semi-concave and lower semicontinuous by Fact~\ref{fact:qaff},
  the latter set is open by
  Proposition~\ref{prop:bary:alg:preorder:product}.

  Let $v \in \conify (\B)$.  In order to show that $\_ + v$ is
  continuous, we once again rely on
  Proposition~\ref{prop:bary:alg:preorder:product}, and we show that
  $V \eqdef (\_ + v)^{-1} ((h^\cext)^{-1} (]1, \infty]))$ is open in
  $\conify (\B)$, for every semi-concave lower semicontinuous map
  $h \colon \B \to \creal$.  This is clear if $v=0$, since in that case
  $V = (h^\cext)^{-1} (]1, \infty])$ is open in the cone topology, by
  Proposition~\ref{prop:bary:alg:preorder:product}.  Hence we assume
  that $v\neq 0$, and we write $v$ as $(s, y)$, with $s > 0$.

  If $0 \in V$, then $h^\cext (v) > 1$.  For every $u \in \conify (\B)$,
  $v \leqc u+v$ (see Remark~\ref{rem:bary:alg:ord}), and since
  $(h^\cext)^{-1} (]1, \infty])$ is open in the cone topology, as we have
  just recalled, it is upwards-closed, and therefore $h^\cext (u+v) > 1$.
  This shows that every $u \in \conify (\B)$ is in
  $V = (\_ + v)^{-1} ((h^\cext)^{-1} (]1, \infty]))$, so
  $V = \conify (\B)$, which is open.

  We therefore assume that $0 \not\in V$.  For every point
  $(r, x) \in \conify (\B) \diff \{0\}$, $(r, x) \in V$ if and only if
  $h^\cext ((r, x) + (s, y)) > 1$, equivalently
  $(r+s) h (x +_{\frac r {r+s}} y) > 1$.  Let
  $f \colon (\Rp \diff \{0\})_\sigma \times \B \to \creal$ map
  $(r, x)$ to $(r + s) h (x +_{\frac {r} {r+s}} y)$: we have just
  shown that $V = f^{-1} (]1, \infty])$.  We claim that $V$ is open in
  the product topology on $(\Rp \diff \{0\})_\sigma \times \B$, and
  upwards-closed in $\conify (\B)$.

  By Proposition~\ref{prop:bary:alg:preorder:product},
  $(h^\cext)^{-1} (]1, \infty])$ is open in the cone topology, hence
  upwards-closed.  The function $\_ + v$ is monotonic by
  Proposition~\ref{prop:bary:alg:conify:ord}, item~1, so
  $V = (\_ + v)^{-1} ((h^\cext)^{-1} (]1, \infty]))$ is upwards-closed.

  With $x$ fixed, $f (\_, x)$ is the composition of
  $a \mapsto \frac 1 {1-a} h (x +_a y)$ (which is Scott-continuous by
  Lemma~\ref{lemma:semi-concave:lsc}) with
  $r \mapsto \frac {r} {r+s}$, which is also Scott-continuous from
  $\Rp \diff \{0\}$ to $]0, 1[$ (indeed, it is continuous with respect
  to the standard topology on the reals, and monotonic since we can
  write $\frac {r} {r+s}$ as $1 - \frac s {r+s}$); hence $f (\_, x)$
  is Scott-continuous.  With $r$ fixed, $f (r, \_)$ is continuous from
  $\B$ to $\creal$, too, since $h$ is, scalar multiplication by $r+s$
  is, and $x \mapsto x +_{\frac {r} {r+s}} y$ is continuous, owing to
  the fact that $\B$ is semitopological.  Therefore $f$ is separately
  continuous from $(\Rp \diff \{0\}) \times \B$ to $\creal$.  By
  Ershov's observation, and since $\Rp \diff \{0\}$ is a continuous
  poset (with $<$ as way-below relation), $f$ is jointly continuous.
  Hence $V = f^{-1} (]1, \infty])$ is open in the product topology on
  $(\Rp \diff \{0\})_\sigma \times \B$.

  2.  The map $\etac_{\B}$ is affine and injective.  In order to show
  that it is continuous, it suffices to show that
  $(\etac_{\B})^{-1} ((h^\cext)^{-1} (]1, \infty])$ is open in $\B$
  for every semi-concave lower semicontinuous map
  $h \colon \B \to \creal$, by
  Proposition~\ref{prop:bary:alg:preorder:product}.  Now
  $(\etac_{\B})^{-1} ((h^\cext)^{-1} (]1, \infty]) = (h^\cext \circ
  \etac_{\B})^{-1} (]1, \infty]) = h^{-1} (]1, \infty])$ (by
  Proposition~\ref{prop:bary:alg:conify:ord}, item~3), which is open
  since $h$ is lower semicontinuous.

  3.  Let $f \colon \B \to \C$ be semi-concave and continuous.  Since
  every semitopological barycentric algebra (resp., cone) is a
  preordered barycentric algebra (resp., cone) in its specialization
  preordering, and since continuous maps are monotonic with respect to
  the specialization preorderings, the uniqueness of $f^\cext$ follows
  from Proposition~\ref{prop:bary:alg:conify:ord}, item~3.  Similarly,
  if $f^\cext$ exists, then if is positively homogeneous, and if $f$ is
  concave (resp.\ affine), then it is superlinear (resp.\ linear).

  We recall that $f^\cext$ is defined by $f^\cext (0) \eqdef 0$ and
  $f^\cext (r, x) \eqdef r f (x)$.  It remains to show that $f^\cext$
  is continuous.  Let $V$ be an open subset of $\C$.  If $0 \in V$,
  then $V = \C$ since $0$ is least in $\C$, and then
  $(f^\cext)^{-1} (V) = \conify (\B)$ is open.  Hence we assume that
  $0 \not\in V$.  Then $(f^\cext)^{-1} (V)$ does not contain $0$, and
  we show that $(f^\cext)^{-1} (V)$ is open in
  $(\Rp \diff \{0\})_\sigma \times \B$ and upwards-closed in
  $\conify (\B)$.

  For every $u \in (f^\cext)^{-1} (V)$, for every $v \in \conify (\B)$
  such that $u \leqc v$, we verify that $v \in (f^\cext)^{-1} (V)$.
  Since $u \neq 0$, $v$ is also different from $0$, and we write $v$
  as $(s, y)$.  There is a point $u' \in \conify (\B)$ and an element
  $y_1 \in \B$ such that $u+u' = (s, y_1)$ and $y_1 \leq y$.  Then
  $f^\cext (v) = s f (y) \geq s f (y_1)$ (since $f$, being continuous,
  is monotonic) $= f^\cext (u+u') \geq f^\cext (u)$ (since $f$ is
  monotonic and $u+u' \geq u$ by Remark~\ref{rem:bary:alg:ord}).
  Since $V$ is upwards-closed, $f^\cext (v) \in V$.  Therefore
  $v \in (f^\cext)^{-1} (V)$.

  Next, we check that $(f^\cext)^{-1} (V)$ is open in
  $(\Rp \diff \{0\})_\sigma \times \B$.  The restriction of $f^\cext$
  to $(\Rp \diff \{0\})_\sigma \times \B$ is the function
  $(r, x) \mapsto r f (x)$.  This is Scott-continuous in $r$ and
  continuous in $x$, since $f$ is continuous.  By Ershov's
  observation, this restriction of $f^\cext$ is jointly continuous
  from $(\Rp \diff \{0\}) \times \B$ to $\creal$, so
  $(f^\cext)^{-1} (V)$ is open in $(\Rp \diff \{0\})_\sigma \times \B$
  with the product topology.

  4.  $(a) \limp (b)$: Let $i \colon \B \to \C$ be an affine
  topological embedding of $\B$ in a semitopological cone $\C$.  For
  every open subset $U$ of $\B$, there is an open subset $V$ of $\C$
  such that $U = i^{-1} (V)$.  By item~3, there is a linear continuous
  map $i^\cext \colon \conify (\B) \to \C$ such that
  $i^\cext \circ \etac_{\B} = i$.  Hence
  $U = (i^\cext \circ \etac_{\B})^{-1} (V) = (\etac_{\B})^{-1}
  ((i^\cext)^{-1} (V))$, showing that $U = (\etac_{\B})^{-1} (V')$
  where $V'$ is the open set $(i^\cext)^{-1} (V)$.  It follows that
  $\etac_{\B}$ is full, and therefore a topological embedding, using
  item~2.

  $(b) \limp (c)$.  By item $(b)$, the open subsets of $\B$ are exactly
  the sets of the form $(\etac_{\B})^{-1} (V)$ where $V$ is open in
  $\conify (\B)$.  If $V$ is the whole of $\conify (\B)$, then
  $(\etac_{\B})^{-1} (V) = \B$, which we can write as
  $h^{-1} (]1, \infty])$, where $h$ is the constant function $2$, say;
  clearly, this function is semi-concave and lower semicontinuous.  If
  $V$ is proper, then $V = (h^\cext)^{-1} (]1, \infty])$ for some
  semi-concave lower semicontinuous map $h \colon \B \to \creal$ by
  Proposition~\ref{prop:bary:alg:preorder:product}; then
  $(\etac_{\B})^{-1} (V) = (h^\cext \circ \etac_{\B})^{-1} (]1, \infty]) =
  h^{-1} (]1, \infty])$.

  $(c) \limp (d)$: this is clear.

  $(d) \limp (b) \limp (a)$.  We assume that the sets
  $h^{-1} (]1, \infty])$ where $h$ ranges over the semi-concave lower
  semicontinuous maps from $\B$ to $\creal$ form a subbase of the
  topology on $\B$.  Since $h^\cext \circ \etac_{\B} = h$, every
  subbasic open set $h^{-1} (]1, \infty])$ is equal to
  ${(\etac_{\B})}^{-1} ((h^\cext)^{-1} (]1, \infty]))$.  Hence
  $\etac_{\B}$ is full, and using item~2, $\etac_{\B}$ is a
  topological embedding.  This shows $(b)$, hence $(a)$, thanks to
  item~1.  \qed
\end{proof}

\begin{remark}
  \label{rem:bary:alg:f*}
  Let $\C$ be a semitopological cone, for example $\creal$.  Using
  Theorem~\ref{thm:conify:semitop}, item~3, there is a bijection
  $f \mapsto f^\cext$, with inverse $M \mapsto M \circ \etac_{\B}$,
  between semi-concave, resp.\ concave, resp.\ affine lower
  semicontinuous maps $f$ from $\B$ to $\C$ and positively
  homogeneous, resp.\ superlinear, resp.\ sublinear lower
  semicontinuous maps $M$ from $\conify (\B)$ to $\C$.  The only
  missing argument is that $M \circ \etac_{\B}$ is semi-concave,
  resp.\ concave, resp.\ affine if $M$ is positively homogeneous,
  resp.\ superlinear, resp.\ sublinear, and this is easily verified,
  using Fact~\ref{fact:qaff:homog} notably.  Additionally, by
  definition of $f^\cext$, $f \mapsto f^\cext$ is monotonic, preserves
  pointwise suprema, binary pointwise infima, sums and scalar products
  of semi-concave lower semicontinuous maps, and so does its inverse
  $M \mapsto M \circ \etac_{\B}$.
\end{remark}

We now have enough to exhibit our prime example of a non-embeddable
semitopological barycentric algebra.
\begin{example}
  \label{exa:bary:alg:notembed}
  Let $\B^- \eqdef {]-\infty, 0]}$, with the Scott topology of its
  usual ordering.  We define $x +_a y$ for all $x, y \in \B^-$ and
  $a \in [0, 1]$ as $a \cdot x + (1-a) \cdot y$.
  \begin{enumerate}
  \item With these operations, $\B^-$ is a topological cone.
  \item The only semi-concave lower semicontinuous maps on $\B^-$ are
    the constant maps, and therefore $\B^-$ is not embeddable.
  \item The injective affine continuous map $\etac_{\B}$ is not a
    topological embedding.
  \item The open subsets in the cone topology on $\conify (\B^-)$ are
    $\conify (\B^-)$, the empty set, and the sets
    $\Open_t \eqdef \{(s, x) \in \Rp \times \B^- \mid s > t\}$,
    $t \in \Rp$.
  \item The specialization preordering of $\conify (\B^-)$ is given by:
    $u$ is less than or equal to $v$ if and only if
    $\ell (u) \leq \ell (v)$, where the level $\ell (u)$ of an element
    $u$ is defined by $\ell (0) \eqdef 0$, $\ell (r, x) \eqdef r$.
  \item $\conify (\B)$ is not $T_0$, although $\B^-$ is.
  \end{enumerate}
\end{example}
\begin{proof}
  1.  The barycentric algebra equations are obvious.  Let
  $(x_0, a_0, y_0) \in \B^- \times [0, 1] \times \B^-$ and let us
  consider a basic open neighborhood ${]b, 0]}$ of $x_0 +_{a_0} y_0$
  in $\B^-$.  We have $a_0x_0+(1-a_0)y_0 > b$.  Let $\epsilon > 0$ be
  such that: $(a)$
  $a_0 x_0 + (1-a_0) y_0 + \epsilon (x_0 + y_0 -1) > b$.

  For every $a \in [0, 1]$ such that $|a-a_0|< \epsilon$, for every
  $x > x_0-\epsilon$, for every $y > y_0 - \epsilon$, we have
  $ax+(1-a)y > a (x_0-\epsilon) + (1-a) (y_0-\epsilon)$ (the large
  inequality $\geq$ is clear, and the inequality is strict because at
  least one of $a$, $1-a$ is non-zero)
  $= a x_0 + (1-a) y_0 - \epsilon \geq (a_0+\epsilon) x_0 +
  (1-a_0+\epsilon) y_0 - \epsilon$ (taking into account that
  $x_0, y_0 \leq 0$)
  $= a_0 x_0 + (1-a_0) y_0 + \epsilon (x_0 + y_0 -1) > b$ (by
  condition $(a)$).

  Hence the triples $(x, a, y)$ in $\B^- \times [0, 1] \times \B^-$
  such that $a \in {]a_0 - \epsilon, a_0 + \epsilon[}$,
  $x > x_0 - \epsilon$ and $y > y_0 - \epsilon$, which form an open
  set, are all mapped to ${]b, 0]}$ by the $\_ +_\_ \_$ operation: the
  $\_ +_\_ \_$ operation is jointly continuous.

  2.  Let $C$ be a closed subset of $\B^-$ of the form
  $h^{-1} ([0, 1])$ with $h$ semi-concave and lower semicontinuous.
  We claim that $C$ is trivial, namely empty or equal to the whole of
  $\B^-$.  To this end, we assume that $C$ is non-empty.  By
  definition of the Scott topology, it must be of the form
  $]-\infty, a_0]$ for some $a_0 \in {]-\infty, 0]}$.  Since
  $a_0 \in C$, $h (a_0) \leq 1$.  For every $n \in \nat$,
  $a_0 = 0 +_{1-1/2^n} (a_0 2^n)$, so $h (a_0) \geq (1-1/2^n) h (0)$,
  since $h$ is semi-concave.  Taking suprema over $n \in \nat$, $h
  (a_0) \geq h (0)$, so $h (0) \leq 1$.  It follows that $0 \in C$,
  hence $a_0 = 0$ and $C = \B^-$.

  Let $h \colon \B^- \to \creal$ be any semi-concave lower
  semicontinuous map.  If $h$ is not constant, we can find two points
  $x$ and $y$ such that $h (x) < h (y)$.  Let $a \in \Rp$ be such that
  $h (x) < a < h (y)$.  Then $(1/a)h$ is semi-concave and lower
  semicontinuous by Remark~\ref{fact:qaff}.  It follows that
  $((1/a)h)^{-1} ([0, 1])$ is trivial.  But $((1/a)h)^{-1} ([0, 1])$
  contains $x$ and not $y$, and this is impossible.  Therefore $h$ is
  constant.  Conversely, every constant map is semi-concave and lower
  semicontinuous by Remark~\ref{fact:qaff}.

  Any non-trivial closed subset, such as $]-\infty, a]$ with $a < 0$,
  shows that item~4$(d)$ of Theorem~\ref{thm:conify:semitop} fails.
  Hence item~4$(a)$ fails, too: $\B^-$ is not embeddable.

  3.  Because item~4$(b)$ of Theorem~\ref{thm:conify:semitop} must
  also fail.

  4.  The only semi-concave lower semicontinuous maps from $\B^-$ to
  $\creal$ are the constant maps.  For a constant map $h$ with value
  $a$, $h^\cext$ maps $0$ to $0$ and $(r, x)$ to $ra$.  Hence
  ${(h^\cext)}^{-1} (]1, \infty])$ is
  $\{(r, x) \in \conify (\B) \diff \{0\} \mid r > 1/a\} = \Open_{1/a}$
  if $a \neq 0, \infty$, the empty set if $a=0$, and $\Open_0$ if
  $a=\infty$.

  5.  By item~4, $u$ is less than or equal to $v$ in the
  specialization preordering of $\conify (\B^-)$ if and only if for
  every $a \in \creal$, $u \in \Open_a$ implies $v \in \Open_a$.  If
  $u=0$, this is always true.  Let therefore $u$ be of the form
  $(r, x)$, with $r \in \Rp \diff \{0\}$.  Then $u$ is less than or
  equal to $v$ if and only if for every $a \in \creal$ such that
  $a < r$, $v$ is of the form $(s, y)$ with $a < s$; that is
  equivalent to $r \leq s$.

  6.  By item~5, any two elements with the same level are less than or
  equal to each other.  \qed
\end{proof}

\begin{theorem}
  \label{thm:bary:alg:topo}
  For every topological barycentric algebra $\B$, $\conify (\B)$ is a
  topological cone.
\end{theorem}
\begin{proof}
  The cone $\conify (\B)$ is semitopological by
  Theorem~\ref{thm:conify:semitop}, item~1.  It remains to show that
  $+$ is jointly continuous on $\conify (\B)$.

  Let $W$ be an open subset of $\conify (\B)$.  We claim that
  $+^{-1} (W)$ is open in $\conify (\B) \times \conify (\B)$.  This is
  clear if $W = \conify (\B)$, since then
  $+^{-1} (W) = \conify (\B) \times \conify (\B)$.  Hence we will assume
  that $W$ is proper.

  Let $(u, v) \in +^{-1} (W)$.  If $u=0$, then $v \in W$, and for
  every $u' \in \conify' (\B)$ and for every $v' \in W$,
  $u'+v' \geq 0+v'$ (since $0$ is least and $+$ is monotonic in a
  preordered cone) $= v' \in W$, so $\conify (\B) \times W$ is an open
  neighborhood of $(u, v)$ that is included in $+^{-1} (W)$.  The
  argument is similar if $v=0$.

  Let therefore $u$ be written as $(r, x)$ and $v$ as $(s, y)$, where
  $r, s > 0$.  Since $u+v \in W$ and $W$ is proper, there is an open
  rectangle ${]a, \infty[} \times A$ containing $u+v$ and included in
  $W$, where $a > 0$ and $A$ is open in $\B$.  Since
  $u+v = (r+s, x +_{\frac r {r+s}} y)$, we have $r+s > a$ and
  $x +_{\frac r {r+s}} y \in A$.  $\B$ is topological, so there are an
  $\epsilon > 0$, an open neighborhood $U$ of $x$ and an open
  neighborhood $V$ of $y$ such that for every
  $b \in {]\frac r {r+s} - \epsilon, \frac r {r+s} + \epsilon[}$, for
  every $x' \in U$ and for every $y' \in V$, $x' +_b y \in A$.  We
  take $\epsilon$ so small that $0 < \frac r {r+s} - \epsilon$ and
  $\frac r {r+s} + \epsilon < 1$.  The function
  $(r', s') \mapsto \frac {r'} {r'+s'}$ is (jointly) continuous from
  ${]0, \infty[} \times {]0, \infty[}$ to $]0, \infty[$, where
  $]0, \infty[$ is given its usual metric topology, and it takes the
  value $\frac r {r+s}$ at $(r', s') \eqdef (r, s)$.  Hence there is
  an $\eta > 0$ such that for every $r' \in {]r-\eta, r+\eta[}$ and
  every $s' \in {]s-\eta, s+\eta[}$, $\frac {r'} {r'+s'}$ is in
  ${]\frac r {r+s} - \epsilon, \frac r {r+s} + \epsilon[}$.  We can
  take $\eta$ so small that $r - \eta > 0$, $s - \eta > 0$,
  $r + \eta < 1$, $s + \eta < 1$, and $2\eta < r+s-a$, too.  We
  consider the open rectangle
  $({]r-\eta, \infty[} \times U) \times ({]r-\eta, \infty[} \times
  V)$.  By construction, this open rectangle contains
  $((r, x), (s, y))$.

  For every pair of points $(r', x')$ and $(s', y')$ in the open
  rectangle
  $({]r-\eta, \infty[} \times U) \times ({]r-\eta, \infty[} \times
  V)$, we claim that $(r', x') + (s', y') \in +^{-1} (W)$.

  \emph{Case~1.}  We verify that $(r', x') + (s', y') \in +^{-1} (W)$
  when $(r', x') \in {]r-\eta, r+\eta[} \times U$ and
  $(s', y') \in {]r-\eta, r+\eta[} \times V$.  In that case,
  $(r', x') + (s', y') = (r'+s', x' +_{\frac {r'} {r'+s}} y)$ is such
  that $r'+s' > r+s-2\eta > a$, since $2\eta < r+s-a$; since
  $r' \in {]r-\eta, r+\eta[}$ and $s' \in {]s-\eta, s+\eta[}$,
  $\frac {r'} {r'+s'}$ is in
  ${]\frac r {r+s} - \epsilon, \frac r {r+s} + \epsilon[}$, and using
  the fact that $x' \in U$ and $y' \in V$,
  $x' +_{\frac {r'} {r'+s'}} y$ is in $A$.  Therefore
  $(r', x') + (s', y') \in {]a, \infty[} \times A \subseteq W$.

  \emph{Case~2.}  We verify that $(r', x') + (s', y') \in +^{-1} (W)$
  when $(r', x') \in {[r, \infty[} \times U$ and
  $(s', y') \in {]s-\eta, s+\eta[} \times V$.  Since $r' \geq r$, we
  can write $(r', x')$ as $(r, x') + (r'-r, x')$ (see Remark~\ref{rem:bary:alg:ord}).
  Then $(r', x') + (s', y') = ((r, x') + (s', y')) + (r'-r, x')$ is
  larger than or equal to $(r, x') + (s', y')$, and the latter is in
  $W$ by Case~1.  Since $W$ is upwards-closed,
  $(r', x') + (s', y') \in W$.

  Together with Case~1, we have shown that $(r', x') + (s', y') \in +^{-1} (W)$
  for all $(r', x') \in {]r-\eta, \infty[} \times U$ and
  $(s', y') \in {]s-\eta, s+\eta[} \times V$.

  \emph{Case~3.}  For every $(r', x') \in {]r-\eta, \infty[} \times U$
  and for every $(s', y') \in {[s, \infty[} \times V$, we write
  $(s', y')$ as $(s, y') + (s'-s, y)$, using the fact that $s' \geq s$
  (see Remark~\ref{rem:bary:alg:ord}).  Then
  $(r', x') + (s', y') = ((r', x') + (s, y')) + (s'-s, y')$ is larger
  than or equal to $(r', x') + (s, y')$, which is in $W$ by Cases~1
  and~2.  Since $W$ is upwards-closed, $(r', x') + (s', y')$ is in
  $W$.  \qed
\end{proof}

\subsection{The free semitopological cone over $\Val_1 X$}
\label{sec:free-semit-cone}

Since one of our primary examples of barycentric algebras is
$\Val_1 X$, it is interesting to note that
$\Val_b X \cong \conify (\Val_1 X)$.  This is less trivial than it may
seem, and in order to prove it, we will need some preparation, and
notably a decomposition lemma due to Matthias Schr\"oder and Alex
Simpson.  This lemma is implicit in \cite[paragraph after
Theorem~4.2]{SchSimp:obs}, and is stated in
\cite[Slide~38]{Simpson:2009}.  Since no proof seems to exist in
print, we include one here, as Proposition~\ref{prop:schsimp:decomp}
below; but the result and the proof must be credited to Matthias
Schr\"oder and Alex Simpson.

In order to deal with it in full generality, we extend the notion of
valuation slightly.  A \emph{lattice of subsets} $\Latt$ of $X$ is a
collection of subsets of $X$ that is closed under finite intersection
and finite union.  Extending our previous notion, a \emph{valuation}
on $(X, \Latt)$ is a strict, modular and monotonic (Scott-continuity
is not required) map from $\Latt$ to $\creal$; it is \emph{bounded} if
$\nu (X) < \infty$.
An \emph{algebra of subsets} if a lattice of subsets that is closed
under complements.  Let $\mathcal A (\Latt)$ be the smallest algebra
of subsets of $X$ containing $\Latt$; its elements are the finite
disjoint unions of \emph{crescents}, where a crescent on $\Latt$ is a
difference $U \diff V$ of two elements $U, V \in \Latt$.  The
\emph{Smiley-Horn-Tarski theorem} states that every bounded valuation
$\nu$ on $(X, \Latt)$ extends to a unique valuation on
$(X, \mathcal A (\Latt))$---still written $\nu$---, which is bounded,
and satisfies
$\nu (U \diff V) = \nu (U) - \nu (U \cap V) = \nu (U \cup V) - \nu
(V)$ for all $U, V \in \Latt$.  This is sometimes attributed to Billy
J. Pettis \cite{Pettis:ext}, but similar results had been proved by
Malcolm F. \cite{smiley44} and Alfred Horn and Alfred Tarski
\cite{HornTarski48:ext}.

Given any bounded valuation $\nu$ on $(X, \Latt)$ and crescent $C$ on
$\Latt$, one can then define a new bounded valuation $\nu_{|C}$ on $(X,
\Latt)$ by $\nu_{|C} (U) \eqdef \nu (U \cap C)$ for every $U \in
\Latt$.  This is sometimes called the restriction of $\nu$ to $C$, but
we prefer to use the term \emph{constriction}: we wish to avoid
confusing this with, say, the restriction of $\nu$ to a smaller
lattice of subsets.

When $X$ is a topological space, $\Latt = \Open X$ and $\nu$ is a
bounded \emph{continuous} valuation, then $\nu_{|C}$ is a bounded
continuous valuation.  Indeed, let us write $C$ as $U \diff V$, where
$U, V \in \Open X$.  For every directed family ${(U_i)}_{i \in I}$ of
open subsets of $X$,
$\nu_{|C} (\bigcup_{i \in I} U_i) = \nu (\bigcup_{i \in I} U_i \cap U
\diff V) = \nu ((\bigcup_{i \in I} U_i \cap U) \cup V) - \nu (V) =
\sup_{i \in I} \nu ((U_i \cap U) \cup V) - \nu (V) = \sup_{i \in I}
\nu_{|C} (U_i)$.

Here is the announced Schr\"oder-Simpson decomposition lemma.
\begin{proposition}[Schr\"oder-Simpson]
  \label{prop:schsimp:decomp}
  Let $X$ be a set with a lattice of subsets $\Latt$, and $\mu$ and
  $\nu$ be two bounded valuations on $(X, \Latt)$.  Let
  $L \eqdef \{U_1, \cdots, U_n\}$ be a finite lattice of subsets of
  $X$ included in $\Latt$, and let us assume that
  $\mu (U_i) \leq \nu (U_i)$ for every $i$, $1\leq i\leq n$.  Then
  there are two bounded valuations $\nu_1$ and $\nu_2$ on $(X, \Latt)$
  such that:
  \begin{enumerate}
  \item $\nu = \nu_1+\nu_2$;
  \item $\mu (U_i) \leq \nu_1 (U_i)$ for every $i$, $1\leq i\leq n$;
  \item and $\mu (X) = \nu_1 (X)$.
  \end{enumerate}
  Additionally, $\nu_1$ and $\nu_2$ are finite linear combinations of
  constrictions $\nu_{|C}$ of $\nu$ to pairwise disjoint crescents $C$.
\end{proposition}
\begin{proof}
  We extend $\mu$ and $\nu$ uniquely to bounded valuations on
  $\mathcal A (\Latt)$, using the Smiley-Horn-Tarski theorem.

  Let $M$ be $\pow (\{1, \cdots, n\})$, ordered by inclusion.  We will
  silently see $M$ as a topological space, with the Alexandroff
  topology of inclusion, that is, its open sets are the upwards-closed
  subsets of $M$.  For each $I \in M$, let $C_I$ be the crescent
  $\bigcap_{i \in I} U_i \diff \bigcup_{i \not\in I} U_i$, and let us
  note that the sets $C_I$ are pairwise disjoint.

  We define a map $f \colon X \to M$ by letting $f (x)$ be the set of
  indices $i$, $1\leq i\leq n$, such that $x \in U_i$, and we note
  that $C_I = f^{-1} (\{I\})$ for every $I \in M$.  It follows that
  for every $A \subseteq M$, $f^{-1} (A)$ is the disjoint union
  $\bigcup_{I \in A} C_I$.

  Given any fixed element $I$ of $M$, let us show that
  $f^{-1} (\upc I) = \bigcap_{i \in I} U_i$.  For every
  $x \in f^{-1} (\upc I)$, $f (x)$ contains $I$, so for every
  $i \in I$, $i \in f (x)$, meaning that $x$ is in $U_i$: in other
  words, $x \in \bigcap_{i \in I} U_i$.  Conversely, for every
  $x \in \bigcap_{i \in I} U_i$, $f (x)$ must contain at least all the
  indices in $I$, so $I \subseteq f (x)$, meaning that
  $x \in f^{-1} (\upc I)$.

  Since $L$ is closed under finite intersections, $f^{-1} (\upc I)$ is
  in $L$.

  For every upwards-closed subset $A$ of $M$, $A$ is equal to the
  union of all sets $\upc I$, $I \in A$, so
  $f^{-1} (A) = \bigcup_{I \in A} f^{-1} (\upc I)$.  Since $L$ is
  closed under finite unions, $f^{-1} (A) \in L$ for every
  upwards-closed subset $A$ of $M$.

  There is an image valuation $f [\mu]$, defined by
  $f [\mu] (A) \eqdef \mu (f^{-1} (A))$ for every open
  (upwards-closed) subset $A$ of $M$, and it is bounded; similarly
  with $f [\nu]$.  We observe that $f [\mu] \leq f [\nu]$.  Indeed,
  for every upwards-closed subset $A$ of $M$,
  $f [\mu] (A) = \mu (f^{-1} (A))$ and
  $f [\nu] (A) = \nu (f^{-1} (A))$.  Since $f^{-1} (A)$ is in $L$, it
  is equal to $U_i$ for some $i$, and the claim follows from our
  assumption that $\mu (U_i) \leq \nu (U_i)$.

  Since $M$ is finite, $f [\mu]$ and $f [\nu]$ are simple valuations.
  Explicitly, $f [\mu] = \sum_{I \in M} a_I \delta_I$ where
  $a_I \eqdef \mu (C_I)$.  Indeed, for every (upwards-closed) subset
  $A$ of $M$, we write $f^{-1} (A)$ as the disjoint union of crescents
  $C_I$, $I \in A$, so $\mu (f^{-1} (A)) = \sum_{I \in A} \mu (C_I)$.
  Similarly, let us write $f [\nu]$ as $\sum_{J \in M} b_J \delta_J$.
  
  We can now use Claire Jones' splitting lemma.  This was first stated
  in \cite[Theorem 4.10]{Jones:proba} (for subprobability valuations,
  although that is not immediately apparent), and we will use the
  slightly more general version of \cite[Proposition
  IV-9.18]{GHKLMS:contlatt}.  Two simple valuations, such as $f [\mu]$
  and $f [\nu]$ above, as such that the first one is less than or
  equal to the second one in the stochastic ordering if and only if
  there is a \emph{transport matrix} ${(t_{IJ})}_{I, J \in M}$ of
  non-negative real numbers such that: $t_{IJ}\neq 0$ implies
  $I \subseteq J$, $\sum_{J \in M} t_{IJ} = a_I$ for every $I \in M$,
  and $\sum_{I \in M} t_{IJ} \leq b_J$ for every $J \in M$.  We define
  $\nu_1$ as $\sum_{J \in M} c_J \nu_{|C_J}$ and $\nu_2$ as
  $\sum_{J \in M} (1-c_J) \nu_{|C_J}$, where
  $c_J \eqdef \frac {\sum_{I \in M} t_{IJ}} {b_J}$ for every
  $J \in M$, provided that $b_J \neq 0$, and $c_J = 0$ otherwise.
  Those are bounded valuations, and they are continuous if $\Latt$ is
  a topology and $\mu$ and $\nu$ are continuous.

  1. By construction, $\nu_1 + \nu_2 = \sum_{J \in M} \nu_{|C_J}$, and
  the latter equals $\nu$: for every $U \in \Latt$,
  $\sum_{J \in M} \nu_{|C_J} (U) = \sum_{J \in M} \nu (U \cap C_J) =
  \nu (\bigcup_{J \in M} U \cap C_J) = \nu (U \cap \bigcup_{J \in M}
  C_J) = \nu (U)$, because $X$ is the disjoint union of all $C_J$s
  ($\bigcup_{J \in M} C_J = \bigcup_{J \in M} f^{-1} (\{J\}) = f^{-1}
  (M) = X$).

  2. For every $i$, $1\leq i\leq n$, we check that
  $\mu (U_i) \leq \nu_1 (U_i)$.  Let
  $A \eqdef \{I \in M \mid i \in I\} = \upc \{i\}$.  This is
  upwards-closed in $M$, and $U_i$ is the disjoint union
  $\bigcup_{I \in A} C_I$, because the latter is equal to
  $\bigcup_{I \in A} f^{-1} (\{I\}) = f^{-1} (\upc \{i\}) = U_i$.
  Also, $\mu (C_I) = \mu (f^{-1} (\{I\})) = a_I$.  Then
  $\mu (U_i) = \sum_{I \in A} \mu (C_I) = \sum_{I \in A} a_I = \sum_{I
    \in A, J \in M} t_{IJ}$.  The terms $t_{IJ}$ with $I \in A$ and
  $J \not\in A$ are zero, and can be removed from the sum; indeed, if
  $t_{IJ}$ were non-zero, then $I \subseteq J$, and since $I \in A$
  and $A$ is upwards-closed, $J$ would also be in $A$.  Therefore
  $\mu (U_i) = \sum_{I, J \in A} t_{IJ}$.  We compare this to
  $\nu_1 (U_i)$, which is equal to
  $\sum_{J \in A} c_J \nu (U_i \cap C_J)$.  For every $J \in A$, since
  $i \in J$ we have $C_J \subseteq U_i$, so
  $\nu_1 (U_i) = \sum_{J \in A} c_J \nu (C_J) = \sum_{J \in A} c_J b_J
  \geq \sum_{I \in M, J \in A} t_{IJ}$, and that is larger than or
  equal to $\sum_{I, J \in A} t_{IJ} = \mu (U_i)$.

  3. We use a similar argument with $A \eqdef M$:
  $\mu (X) = \sum_{I \in M} \mu (C_I) = \sum_{I \in M} a_I = \sum_{I,
    J \in M} t_{IJ} = \sum_{J \in M} c_J b_J$ by definition ot $c_J$.
  We have $b_J = \nu (C_J)$, so
  $\mu (X) = \sum_{J \in M} c_J \nu (C_J) = \nu_1 (X)$.  \qed
\end{proof}

\begin{theorem}
  \label{thm:V1->Vb}
  For every topological space $X$, $\Val_b X$ is the free cone over
  the barycentric algebra $\Val_1 X$ and $\Val_b X$ is the free
  semitopological cone over the semitopological barycentric algebra
  $\Val_1 X$.

  As a consequence, $\Val_b X$ and $\conify (\Val_1 X)$ are naturally
  isomorphic semitopological cones.

  Similarly, $\Val_{b, \pw} X$ is the free semitopological cone over
  the semitopological barycentric algebra $\Val_{1, \pw} X$.
\end{theorem}
The first part of the proposition states, equivalently, that for every
cone $\C$, every affine map $f \colon \Val_1 X \to \C$ extends to a
unique linear map $\hat f$ from $\Val_b X$ to $\C$, which is
continuous if $f$ is (assuming that $\C$ is a semitopological cone,
and that $\Val_1 X$ and $\Val_b X$ are given the weak topology).

\begin{proof}
  The second part of the proposition follows from the first part and
  the fact that free constructions are always unique up to natural
  isomorphisms.  We focus on the first part.
  
  If $\hat f$ exists, then for every $\nu \in \Val_b X$, either $\nu$
  is the zero valuation and $\hat f (\nu)$ must be equal to $0$, or
  $a \eqdef \nu (X)$ is non-zero, and in $\Rp$ since $\nu$ is bounded.
  In that case, $\frac 1 a \cdot \nu$ is in $\Val_1 X$, and since
  $\hat f$ is positively homogeneous we must have
  $\hat f (\nu) = a \cdot \hat f \Big(\frac 1 a \cdot \nu\Big) = a
  \cdot f \Big(\frac 1 a \cdot \nu\Big)$.  This shows that $\hat f$ is
  unique if it exists.

  Let us define $\hat f$ as mapping the zero valuation to $0$ and
  every non-zero element $\nu \in \Val_b X$ to
  $\nu (X) \cdot f \Big(\frac 1 {\nu (X)} \cdot \nu\Big)$.  The
  function $\hat f$ is strict by definition.  For every
  $a \in \Rp \diff \{0\}$, for every $\nu \in \Val_b X$, it is easy to
  see that $\hat f (a \cdot \nu) = a \cdot \hat f (\nu)$.  Let us show
  that $\hat f$ is additive.  For all $\mu, \nu \in \Val_b X$, if
  $\mu$ is the zero valuation then
  $\hat f (\mu + \nu) = \hat f (\nu) = 0 + \hat f (\nu) = \hat f (\mu)
  + \hat f (\nu)$, and similarly if $\nu$ is the zero valuation.
  Otherwise, let $a \eqdef \mu (X)$, $b \eqdef \nu (X)$, then:
  \begin{align*}
    \hat f (\mu + \nu)
    & = (a+b) \cdot f \left(\frac 1 {a+b} \cdot (\mu + \nu)\right) \\
    & = (a+b) \cdot f \left(\frac a {a+b} \cdot \frac 1 a \cdot \mu +
      \frac b {a+b} \cdot \frac 1 b \cdot \nu\right) \\
    \ifta
    & = (a+b) \cdot \left(\frac a {a+b} \cdot f \left(\frac 1 a \cdot
      \mu\right)
      + \frac b {a+b} \cdot f \left(\frac 1 b \cdot \nu\right)\right)
    & \text{since $f$ is affine} \\
    \else
    & = (a+b) \cdot \biggl(\frac a {a+b} \cdot f \left(\frac 1 a \cdot
      \mu\right) \\
    & \qquad
      + \frac b {a+b} \cdot f \left(\frac 1 b \cdot \nu\right)\biggr)
    \quad \text{since $f$ is affine} \\
    \fi
    & = a \cdot f \left(\frac 1 a \cdot \mu\right) + b \cdot f \left(\frac 1 b
      \cdot \nu\right)
    \\ &
         = \hat f (\mu) + \hat f (\nu).
  \end{align*}

  \emph{The semitopological case.}  If $\C$ is a semitopological cone
  and $f$ is continuous from $\Val_1 X$ to $\C$, we claim that
  $\hat f$ is continuous from $\Val_b X$ to $\C$.  This is the hard
  part.  Let $V$ be an open subset of $\C$.  We aim to show that
  ${\hat f}^{-1} (V)$ is open.  This is obvious if $V=\C$, so let us
  assume that $V$ is proper.  Equivalently, $0$ is not in $V$.  We fix
  $\mu$ in ${\hat f}^{-1} (V)$, and we will find a basic open set
  $\mathcal U$ containing $\mu$ and included in ${\hat f}^{-1} (V)$.

  The set $\{\alpha \in \Rp \mid \alpha \cdot \hat f (\mu) \in V\}$ is
  open in $\Rp$, since scalar multiplication is continuous in its
  first argument, and it contains $1$.  Therefore we can find a value
  $\alpha \in {]0, 1[}$ such that $\alpha \cdot \hat f (\mu)$ is in in
  $V$.
    
  Let $a \eqdef \mu (X)$.  If $a$ were equal to $0$, then $\mu$ would
  be the zero valuation, so $\hat f (\mu)$ would be equal to $0$, but
  $\hat f (\mu) \in V$ and $0 \not\in V$.  Hence $a \neq 0$, and
  $\hat f (\mu) = a \cdot f \Big(\frac 1 a \cdot \mu\Big)$.  Then
  $\alpha a \cdot f \Big(\frac 1 a \cdot \mu\Big)$ is in $V$, so
  $f \Big(\frac 1 a \cdot \mu\Big)$ is in
  $\frac 1 {\alpha a} \cdot V \eqdef \Big\{\frac 1 {\alpha a} \cdot y
  \mid y \in V\Big\}$.  The latter is open, as the inverse image of
  $V$ under the map $y \mapsto \alpha a \cdot y$.  It follows that
  $\frac 1 a \cdot \mu$ is in the open set
  $f^{-1} \Big(\frac 1 {\alpha a} \cdot V\Big)$.

  By definition of the weak topology, $\frac 1 a \cdot \mu$ is in some
  finite intersection $\bigcap_{i=1}^n [r_i \ll U_i]_1$ of subbasic
  open sets of $\Val_1 X$, itself included in
  $f^{-1} \Big(\frac 1 {\alpha a} \cdot V\Big)$ (we write
  $[r \ll U]_1$ for $[r \ll U] \cap \Val_1 X$, in order to avoid
  confusion with $[r \ll U]$, which we keep as a notation for an open
  subset of $\Val_b X$).

  Let $L$ be the smallest lattice of sets that contains every $U_i$.
  $L$ consists of all unions of intersections of sets among $U_1$,
  \ldots, $U_n$.  It is easy to see that $L$ is finite (and contains
  at most $2^{2^n}$ elements).  By taking the intersection of
  $\bigcap_{i=1}^n [r_i \ll U_i]_1$ with the extra sets $[0 \ll U]_1$,
  where $U \in L \diff \{U_1, \cdots, U_n\}$, and replacing
  $\{U_1, \cdots, U_n\}$ by $L$, we may assume that
  $\{U_1, \cdots, U_n\}$ is itself a lattice of subsets of $X$.
  
  Let $\mathcal U$ be the basic open set
  $\bigcap_{i=1}^n [\alpha \mu (U_i) \ll U_i]$ in $\Val_b X$.  Since
  $\alpha < 1$, $\mu$ is in $\mathcal U$.
  
  We check that $\mathcal U$ is included in ${\hat f}^{-1} (V)$.  For
  every $\nu \in \mathcal U$, $\alpha\mu (U_i) \leq \nu (U_i)$ for
  every $i$, $1\leq i\leq n$.  By the Schr\"oder and Simpson
  decomposition lemma (Proposition~\ref{prop:schsimp:decomp}), there
  are two bounded valuations $\nu_1$ and $\nu_2$ such that
  $\nu = \nu_1+\nu_2$, $\alpha\mu (U_i) \leq \nu_1 (U_i)$ for every
  $i$, and $\alpha\mu (X) = \nu_1 (X)$---namely,
  $\nu_1 (X) = \alpha a$.  Then
  $\hat f (\nu) = \hat f (\nu_1) + \hat f (\nu_2)$ (since $\hat f$ is
  additive)
  $\geq \hat f (\nu_1) = (\alpha a) \cdot f \Big(\frac 1 {\alpha a}
  \cdot \nu_1\Big)$.

  For every $i$, $\alpha\mu (U_i) \leq \nu_1 (U_i)$ and $r_i \ll \frac
  1 a \cdot \mu (U_i)$ (since $\frac 1 a \cdot \mu$ is in $[r_i \ll
  U_i]_1$), so $r_i \ll \frac 1 {\alpha a} \cdot \nu_1
  (U_i)$.  Hence $\frac 1 {\alpha a} \cdot
  \nu_1$ is in $\bigcap_{i=1}^n [r_i \ll
  U_i]_1$, and the latter is included in $f^{-1} \Big(\frac 1 {\alpha
    a} \cdot V\Big)$.  In other words, $f \Big(\frac 1 {\alpha a}
  \cdot \nu_1\Big)$ is in $\frac 1 {\alpha a} \cdot
  V$, and therefore the larger value $\hat f (\nu)$ is in
  $V$.  Since $\nu$ is arbitrary in $\mathcal U$, $\mathcal U
  \subseteq {\hat f}^{-1} (V)$.

  In the case of $\Val_{b, \pw} X$ and $\Val_{1, \pw} X$, we reason in
  exactly the same way.  We only have to pay attention that the two
  valuations $\nu_1$ and $\nu_2$ obtained by using Schr\"oder and
  Simpson's decomposition lemma
  (Proposition~\ref{prop:schsimp:decomp}) are point-continuous.  By
  Proposition~\ref{prop:schsimp:decomp}, $\nu_1$ and $\nu_2$ are
  obtained as finite linear combinations of constrictions $\nu_{|C}$
  of $\nu$ to pairwise disjoint crescents $C$.  Such constrictions are
  point-continuous: this is an immediate consequence of the two
  bulleted items of \cite[Section~3.3]{heckmann96}.  Additionally,
  finite linear combinations of point-continuous valuations are
  point-continuous by \cite[Section~3.2]{heckmann96}.
\end{proof}

\subsection{Barycenters, part 2}
\label{sec:barycenters-part-2}

Let us return to the barycenter construction
$\sum_{i=1}^n a_i \cdot x_i$ of Definition and
Proposition~\ref{prop:bary:1}.  We topologize the standard simplex
$\Delta_n$ with the subspace topology induced by its inclusion in
$\real^n$, where $\real$ has its usual metric topology.
\begin{proposition}
  \label{prop:bary:1:cont}
  For every topological barycentric algebra $\B$, the map
  $(a_1, \cdots, \allowbreak a_n, x_1, \cdots, x_n) \mapsto
  \sum_{i=1}^n a_i \cdot x_i$ is (jointly) continuous from
  $\Delta_n \times \B \times \cdots \times \B$ to $\B$.
\end{proposition}
\begin{proof}
  Let $\beta_n$ be this map.  We prove the claim by induction on
  $n \geq 1$.  When $n=1$, this is clear.  Otherwise, let us fix
  $(a_1, \cdots, a_n, x_1, \cdots, x_n) \in \Delta_n \times \B \times
  \cdots \times \B$ and let us assume that
  $\sum_{i=1}^n a_i \cdot x_i$ is in some open subset $W$ of $\B$.  We
  will find an open neighborhood of
  $(a_1, \cdots, a_n, x_1, \cdots, x_n)$ whose image under $\beta_n$
  is included in $W$.  By invariance under permutations (Definition
  and Proposition~\ref{prop:bary:1}, item~2), we may assume that
  $a_n \neq 1$: since $n \geq 2$ and $\sum_{i=1}^n a_i=1$, some $a_i$
  is different from $1$, and we may permute indices $i$ and $n$ if
  necessary.  Then
  $\beta_n (a_1, \cdots, a_n, x_1, \cdots, x_n) = \beta_{n-1} (\frac
  {a_1} {1-a_n}, \cdots, \frac {a_{n-1}} {1-a_n}, x_1, \cdots,
  x_{n-1}) +_{1-a_n} x_n$ by Definition and
  Proposition~\ref{prop:bary:1}, item~1.  Since $\B$ is topological,
  there is an open neighborhood $V$ of
  $\beta_{n-1} (\frac {a_1} {1-a_n}, \cdots, \frac {a_{n-1}} {1-a_n},
  \allowbreak x_1, \cdots, x_{n-1})$, an open neighborhood $U_n$ of
  $x_n$, and an open neighborhood $I$ of $a_n$ such that for all
  $y \in V$, $x'_n \in U_n$ and $a'_n \in I_n$,
  $y +_{1-a'_n} x'_n \in W$.  By taking a smaller neighborhood if
  necessary, and since $a_n < 1$, we may assume that $I$ is included
  in $[0, 1[$.  By induction hypothesis, and by definition of the
  topology on $\Delta_{n-1} \times \B^{n-1}$, there is an open
  rectangle $I'_1 \times \cdots \times I'_{n-1}$ containing
  $(\frac {a_1} {1-a_n}, \cdots, \frac {a_{n-1}} {1-a_n}, x_1, \cdots,
  x_{n-1})$, and there are open neighborhoods $U_1$ of $x_1$, \ldots,
  $U_{n-1}$ of $x_{n-1}$ such that the image of
  $((I'_1 \times \cdots \times I'_{n-1}) \cap \Delta_{n-1}) \times
  (U_1 \times \cdots \times U_{n-1})$ under $\beta_{n-1}$ is included
  in $V$.  Let $f \colon (s, t) \mapsto \frac s {1-t}$: since $f$ is
  continuous from $\real \times {[0, 1[}$ to $\real$, there are open
  rectangles $I_1 \times J_1$ containing $(a_1, a_n)$, \ldots,
  $I_{n-1} \times J_{n-1}$ containing $(a_{n-1}, a_n)$ such that $f$
  maps every pair in $I_1 \times J_1$ to $I'_1$, \ldots, every pair in
  $I_{n-1} \times J_{n-1}$ to $I'_{n-1}$.  Letting
  $I_n \eqdef J_1 \cap \cdots \cap J_{n-1} \cap I$,
  $((I_1 \times \cdots \times I_n) \cap \Delta_n) \times (U_1 \times
  \cdots \times U_n)$ is an open neighborhood of
  $(a_1, \cdots, a_n, x_1, \cdots, x_n)$ that is mapped by $\beta_n$
  to $W$.  \qed
\end{proof}

Before we close this section, we mention the following, out of
curiosity; this will be useful later.  The topology of $\Delta_n$ is
defined as the subspace topology from $\real^n$, where $\real$ has its
usual metric topology.  The topology of $\real^n$ is very different
from the Scott topology of the pointwise ordering.  Despite this, the
following holds.
\begin{lemma}
  \label{lemma:Delta:scott}
  For every $n \geq 1$, the (Scott) topology on $(\creal)^n$ and the
  product topology on $\real^n$ (where $\real$ has its metric
  topology) induce the same topology on the subspace $\Delta_n$.  With
  this topology, $\Delta_n$ is compact Hausdorff.
\end{lemma}
\begin{proof}
  Let $\tau_1$ be the Scott topology on $\creal^n$.  Since $\creal$ is
  a continuous dcpo, $\tau_1$ is also the product of the Scott
  topologies on each copy of $\creal$.  The Scott topology on $\creal$
  is coarser than the standard topology, so $\tau_1$ is coarser than
  the product topology on $\real^n$; we will call $\tau_2$ the latter.
  In order to show the first part of the lemma, it is then enough to
  show that that for every basic open subset
  $U \eqdef \prod_{i=1}^n {]a_i, b_i[}$ of $\tau_2$, $U \cap \Delta_n$
  is open in the subspace topology induced by $\tau_1$.  The elements
  $\vec\alpha$ of $U \cap \Delta_n$ are exactly those such that
  $\alpha_i > a_i$ and
  $\sum_{1 \leq j\leq n, j \neq i} \alpha_j > 1-b_i$ (because
  $\sum_{1\leq j \leq n, j \neq i} \alpha_j = 1-\alpha_i$).  Hence
  $U \cap \Delta_n = V \cap \Delta_n$, where
  $V \eqdef \{\vec\alpha \in \creal^n \mid \forall i \in \{1, \cdots,
  n\}, \alpha_i > a_i \text{ and }\sum_{1 \leq j\leq n, j \neq i}
  \alpha_j > 1-b_i\}$, which is easily seen to be Scott-open.

  Since $\real^n$ is Hausdorff, its subspace $\Delta_n$ is Hausdorff.
  It is closed in $\real^n$, because it is equal to the inverse image
  of the closed set $\{1\}$ under the continuous map
  $\vec \alpha \mapsto \sum_{i=1}^n \alpha_i$.  It is included in
  $[0, 1]^n$, which is compact by Tychonoff's theorem, hence
  $\Delta_n$ is a compact subset of $\real^n$, and therefore also a
  compact subspace of $\real^n$.  \qed
\end{proof}

\section{Pointed barycentric algebras}
\label{sec:point-baryc-algebr}

\subsection{Pointed barycentric algebras and preordered pointed
  barycentric algebras}
\label{sec:point-baryc-algebr-1}

Among the preordered barycentric algebras, some have a least element
$\bot$.  This is the case of $\Val_{\leq 1} X$ and of $\Val_b X$,
where $\bot$ is the zero valuation, but not of $\Val_1 X$, for
example.

\begin{definition}
  \label{defn:bary:alg:pointed}
  A \emph{pointed barycentric algebra} is a barycentric algebra $\B$
  with a distinguished element.

  A \emph{preordered pointed barycentric algebra} is a pointed
  barycentric algebra and a preordered barycentric algebra in which
  the distinguished element is least.  An \emph{ordered pointed
    barycentric algebra} is a preordered pointed barycentric algebra
  whose preordering is antisymmetric.

  A \emph{pointed semitopological barycentric algebra} is a
  semitopological barycentric algebra and a semitopological
  barycentric algebra such that the distinguished element is least in
  the specialization preordering.

  In all cases, we will write $\bot$ for the distinguished element,
  and $a \cdot x$ for $x +_a \bot$, for all $a \in [0, 1]$ and
  $x \in \B$.
\end{definition}

\begin{remark}
  \label{rem:bary:alg:pointed:bot:uniq}
  Beware that there may be more than one least element in a preordered
  barycentric algebra.  In Definition~\ref{defn:bary:alg:pointed}, we
  only require the distinguished element $\bot$ to \emph{one} of the
  least elements.  There is no such ambiguity if $\B$ is ordered, or
  if $\B$ is semitopological and $T_0$.
\end{remark}

The operation $a \cdot x$ obeys almost the same properties as in
cones, except that $a$ has to be smaller than or equal to $1$.
\begin{lemma}
  \label{lemma:bary:alg:pointed}
  In a pointed barycentric algebra $\B$ with distinguished element
  $\bot$, the following hold:
  \[
    \ifta
    \begin{array}{ccc}
      0 \cdot x = \bot & (ab) \cdot x = a \cdot (b \cdot x)
      & 1 \cdot x = x \\
      a \cdot \bot = \bot & a \cdot (x +_b y) = a \cdot x +_b a \cdot y
                            & a \cdot x +_r b \cdot x = (ra+(1-r)b) \cdot x
    \end{array}
    \else
    \begin{array}{ccc}
      0 \cdot x = \bot & (ab) \cdot x = a \cdot (b \cdot x) \\
      1 \cdot x = x & a \cdot \bot = \bot \\
      \multicolumn{3}{c}{a \cdot (x +_b y) = a \cdot x +_b a \cdot y} \\
      \multicolumn{3}{c}{a \cdot x +_r b \cdot x = (ra+(1-r)b) \cdot x}
    \end{array}
    \fi
  \]
  for all $a, b, r \in [0, 1]$ and $x, y \in \B$.

  If $\B$ is preordered or semitopological, additionally,
  $(a, x) \mapsto a \cdot x$ is monotonic in both $a$ and $x$.  If
  $\B$ is a pointed semitopological barycentric algebra, it is
  Scott-continuous in $a$ and continuous in $x$, hence jointly
  continuous from $[0, 1]_\sigma \times \B$ to $\B$.
\end{lemma}
\begin{proof}
  First, $0 \cdot x = x +_0 \bot = \bot$,
  $1 \cdot x = x +_1 \bot = x$, and
  $a \cdot \bot = \bot +_a \bot = \bot$.  Second,
  $(ab) \cdot x = x +_{ab} \bot$, and we compute
  $a \cdot (b \cdot x)$: if $a=1$, this is equal to
  $b \cdot x = x +_b \bot = x +_{ab} \bot$; if $b=1$, this is equal to
  $a \cdot x = x +_a \bot = x +_{ab} \bot$; otherwise, it is equal to
  $(b \cdot x) +_a \bot = (x +_b \bot) +_a \bot = x +_{ab} (\bot
  +_{\frac {(1-b)a} {1-ab}} \bot) = x +_{ab} \bot$.  Third, using the
  entropic law (\ref{eq:entropic}),
  $a \cdot x +_b a \cdot y = (x +_a \bot) +_b (y +_a \bot) = (x +_b y)
  +_a (\bot +_b \bot) = (x +_b y) +_a \bot = a \cdot (x +_b y)$.  In
  order to show the final law
  $a \cdot x +_r b \cdot x = (ra+(1-r)b) \cdot x$, we use the affine
  embedding $\etac_{\B}$ of $\B$ in the free cone $\conify (\B)$.
  Beware that $\etac_{\B} (a \cdot x)$ is not equal to
  $a \cdot \etac_{\B} (x)$, rather to
  $\etac_{\B} (x +_a \bot) = (1, x +_a \bot) = (a, x) + (1-a, \bot)$.
  Then, using this and the fact that $\etac_{\B}$ is affine,
  $\etac_{\B} (a \cdot x +_r b \cdot x) = r \cdot ((a, x) + (1-a,
  \bot)) + (1-r) \cdot ((b, x) + (1-b, \bot)) = (ra, x) + (r(1-a),
  \bot) + ((1-r) b, x) + ((1-r)(1-b), \bot) = (ra + (1-r)b, x) +
  (r(1-a)+(1-r)(1-b), \bot)$, while
  $\etac_{\B} ((ra+(1-r)b) \cdot x) = (ra+(1-r)b, x) + (1-ra-(1-r)b,
  \bot)$, which is the same value.  We conclude since $\etac_{\B}$ is
  injective.

  Let now $\B$ be preordered or semitopological.
  The map $(a, x) \mapsto a \cdot x = x +_a \bot$ is monotonic in $x$
  since $+_a$ is monotonic.  In order to show that it is monotonic in
  $a$, let $a, b \in [0, 1]$ with $a < b$.
  Let $c \eqdef \frac {1-b} {1-a}$, so that $1-c(1-a)=b$; also,
  $c < 1$ since $a < b$.  Then
  $a \cdot x = x +_a \bot = \bot +_{1-a} x \leq (\bot +_c x) +_{1-a}
  x$ (since $+_{1-a}$ is monotonic and $\bot$ is least)
  $= \bot +_{c(1-a)} (x +_{\frac {(1-c)(1-a)} {1-c(1-a)}} x) = \bot
  +_{c (1-a)} x = x +_{1-c(1-a)} \bot = x +_b \bot = b \cdot x$.

  Let now $\B$ be a pointed semitopological barycentric algebra.  Then
  $a \cdot x = x +_a \bot$ is continuous in $x$ since $+_a$ is
  separately continuous.  For every $x \in \B$, the inverse image
  $(\_ \cdot x)^{-1} (U)$ of an open subset $U$ of $\B$ is open in
  $[0, 1]$ with its usual metric topology because
  $a \mapsto x +_a \bot$ is continuous, and is upwards-closed: if
  $a \in (\_ \cdot x)^{-1} (U)$ and $a \leq b$ in $[0, 1]$, then
  $a \cdot x \in U$, $a \cdot x \leq b \cdot x$ as we have proved
  above, and since $U$ is upwards-closed, $b \cdot x$ is also in $U$.
  The upwards-closed open subsets of $[0, 1]$ are the Scott-open
  subsets of $[0, 1]$, so $(\_ \cdot x)^{-1} (U)$ is open in
  $[0, 1]_\sigma$, hence $\_ \cdot x$ is continuous, for every
  $x \in \B$.  This shows that $\cdot$ is separately continuous from
  $[0, 1]_\sigma \times \B \to \B$, hence jointly continuous by
  Ershov's observation, since $[0, 1]$ is a continuous poset in its
  usual ordering.  \qed
\end{proof}
Preordered pointed barycentric algebras form a category $\BA_\bot$,
whose morphisms are the \emph{strict} affine monotonic maps, namely
the affine monotonic functions mapping $\bot$ to $\bot$.  We will not
be so much interested in pointed barycentric algebras that are not at
least preordered.  Similarly, we write $\sTBA_\bot$ for the category
of pointed semitopological barycentric algebras and strict affine
continuous maps, and $\TBA_\bot$ for its full subcategory of pointed
topological barycentric algebras.

We will use the following additional classes of maps.  We reuse the
names of the corresponding classes of maps on cones.
\begin{definition}
  \label{defn:bary:alg:pointed:linear}
  A function $f \colon \B \to \Alg$ between preordered pointed
  barycentric algebras is:
  \begin{itemize}
  \item \emph{$\Iu$-homogeneous} if and only if
    $f (a \cdot x) = a \cdot f (x)$ for all $a \in [0, 1]$ and
    $x \in \B$ ($\Iu$ is a common name for $[0, 1]$);
  \item \emph{superlinear} if and only if it is $\Iu$-homogeneous and concave;
  \item \emph{sublinear} if and only if it is $\Iu$-homogeneous and convex;
  \item \emph{linear} if and only if it is $\Iu$-homogeneous and
    affine.
  \end{itemize}
\end{definition}
There is no risk of confusion with the similar notions on cones: on
cones, the similarly named notions mean the same thing, as the
following remark states.
\begin{remark}
  \label{rem:bary:alg:pointed:linear}
  Let $f \colon \C \to \Cb$ be a function between preordered cones.
  Then:
  \begin{itemize}
  \item $f$ is $\Iu$-homogeneous if and only if it is positively
    homogeneous.  Indeed, if $f (a \cdot x) = a \cdot f (x)$ for all
    $a \in [0, 1]$ and $x \in \C$, then
    $f (a \cdot x) = a \cdot f (x)$ also for all $a \in \Rp$: when
    $a > 1$, this is due to the fact that $1/a < 1$, so
    $f ((1/a) \cdot (a \cdot x))$, which is equal to $f (x)$, is also
    equal to $(1/a) \cdot f (a \cdot x)$ by
    Definition~\ref{defn:bary:alg:pointed:linear}, whence
    $f (a \cdot x) = a \cdot f (x)$ by multiplying both sides by $a$.
  \item Hence $f \colon \C \to \Cb$ is superlinear (resp.\ sublinear,
    linear) as a function between preordered cones if and only if it
    is superlinear (resp.\ sublinear, linear) as a function between
    preordered pointed barycentric algebras
    (Definition~\ref{defn:bary:alg:pointed:linear}).  Indeed, the only
    if direction is easy, while in the if direction, we assume $f$
    $\Iu$-homogenous (hence positively homogeneous, by the previous
    point) and concave (resp.\ convex, affine).  For all
    $x, y \in \C$,
    $f (x + y) = f (2 \cdot (x +_{1/2} y)) = 2 \cdot f (x +_{1/2} y)
    \geq 2 \cdot (f (x) +_{1/2} f (y))$ (resp. $\leq$, $=$)
    $= f (x) + f (y)$.
  \end{itemize}
\end{remark}

The morphisms of the category of pointed barycentric algebras are the
strict affine maps or equivalent the linear maps, thanks to the
following proposition.
\begin{proposition}
  \label{prop:strict:aff}
  A function $f \colon \B \to \Alg$ between pointed barycentric
  algebras is linear if and only if it is strict and affine.
\end{proposition}
\begin{proof}
  In one direction, every linear map is affine, and being
  $\Iu$-homogeneous, is strict: we have
  $f (0 \cdot x) = 0 \cdot f (x)$ for every point $x$, in particular
  for $x \eqdef \bot$; then $0 \cdot \bot = \bot +_0 \bot = \bot$, and
  $0 \cdot f (x) = f (x) +_0 \bot = \bot$.

  In the other direction, it suffices to show that every strict affine
  map $f$ is $\Iu$-homogeneous.  Indeed,
  $f (a \cdot x) = f (x +_a \bot) = f (x) +_a f (\bot)$ (since $f$ is
  affine) $= f (x) +_a \bot$ (since $f$ is strict) $= a \cdot f (x)$.  \qed
\end{proof}

The following shows how close cones and pointed barycentric algebras
are: the only thing missing in the latter to become cones is a
doubling map.  (Multiplication by any number strictly larger than $1$
would work, too.)
\begin{proposition}
  \label{prop:ba->cone}
  Let $\B$ be a pointed barycentric algebra with a linear map
  $\dbl \colon \B \to \B$ (the \emph{doubling map}) such that
  $\dbl (\frac 1 2 \cdot x) = x$ for every $x \in \B$.  Then $\B$ is a
  cone with the following operations:
  \begin{itemize}
  \item addition $x+y$ is defined as $\dbl (x +_{1/2} y)$ for all $x,
    y \in \B$;
  \item zero is defined as the distinguished element $\bot$ of $\B$;
  \item scalar multiplication is defined by $a \cdot x \eqdef \dbl^k
    (\frac a {2^k} \cdot x)$ for all $a \in \Rp$ and $x \in \B$, where
    $k$ is any natural number large enough that $\frac a {2^k} \leq
    1$, and $\dbl^k$ is $\dbl$ composed with itself $k$ times.
  \end{itemize}
  If $\B$ is a preordered barycentric algebra and $\dbl$ is monotonic,
  this construction turns $\B$ into a preordered cone.

  If $\B$ is a semitopological, resp.\ topological barycentric algebra
  and $\dbl$ is continuous, this construction turns $\B$ into a
  semitopological, resp.\ topological cone.
\end{proposition}
\begin{proof}
  We first verify that scalar multiplication is well-defined.  In
  order to clarify, the term $\frac a {2^k} \cdot x$ is equal to
  $x +_{\frac a {2^k}} \bot$, which makes sense in any pointed
  barycentric algebra with distinguished element $\bot$.  For all
  $k, k' \in \nat$ such that $\frac a {2^k}, \frac a {2^{k'}} \leq 1$,
  we observe that
  $\dbl^k (\frac a {2^k} \cdot x) = \dbl^{k'} (\frac a {2^{k'}} \cdot
  x)$, so the definition of $a \cdot x$ is indeed independent of the
  chosen $k$.  Indeed, and assuming $k \leq k'$ without loss of
  generality,
  $\dbl^k (\frac a {2^k} \cdot x) = \dbl^k (\dbl (\frac 1 2 \cdot
  \frac a {2^k} \cdot x))$ (since $\dbl (\frac 1 2 \cdot \_)$ is the
  identity map) $= \dbl^{k+1} (\frac a {2^{k+1}} \cdot x)$ (by the
  second equation of Lemma~\ref{lemma:bary:alg:pointed}); iterating
  this shows the claim.

  We have one notation $a \cdot x$ for potentially two different
  things when $a \in [0, 1]$: either the scalar multiplication that
  already exists in $\B$, namely $x +_a \bot$, or the new scalar
  multiplication defined in Proposition~\ref{prop:ba->cone}.  They
  turn out to be the same: since $a \leq 1$, our newly-defined notion
  $a \cdot x$ is equal to $\dbl^0 (\frac a {2^0} \cdot x)$ by picking
  $k \eqdef 0$, which is the old notion of $a \cdot x$.

  Let us check the cone equations.  We start with the properties of
  $+$.  Addition is commutative, since $x +_{1/2} y = y +_{1/2} x$ for
  all $x, y \in \B$.  For every $x \in \B$,
  $x + 0 = \dbl (x +_{1/2} \bot)$ (since $0$ is defined as $\bot$)
  $= \dbl (\frac 1 2 \cdot x) = x$.  In order to show associativity,
  let $x, y, z, \in \B$.  Then
  $(x+y)+z = \dbl (\dbl (x +_{1/2} y) +_{1/2} z) = \dbl (\dbl (x
  +_{1/2} y) +_{1/2} \dbl (z +_{1/2} \bot))$ (since
  $z = \dbl (\frac 1 2 \cdot z)$ and
  $\frac 1 2 \cdot z = z +_{1/2} \bot$)
  $= \dbl^2 ((x +_{1/2} y) +_{1/2} (z +_{1/2} \bot))$, since $\dbl$ is
  linear hence affine.  Similarly,
  $x+(y+z) = \dbl^2 ((x +_{1/2} \bot) +_{1/2} (y +_{1/2} z))$.  It
  then suffices to show that
  $(x +_{1/2} y) +_{1/2} (z +_{1/2} \bot) = (x +_{1/2} \bot) +_{1/2}
  (y +_{1/2} z)$, and this follows from the fact that they are both
  mapped to
  $\frac 1 4 \etac_{\B} (x) + \frac 1 4 \etac_{\B} (y) + \frac 1 4
  \etac_{\B} (z) + \frac 1 4 \etac_{\B} (\bot)$ by the injective
  affine map $\etac_{\B}$, or by the entropic law (\ref{eq:entropic}).

  We now check the six explicitly listed equations (\ref{eq:cone}),
  see Section~\ref{sec:preliminaries}.  First, for every $x \in \B$,
  $0 \cdot x = \bot = 0$ by the first equation of
  Lemma~\ref{lemma:bary:alg:pointed}.  Second, for all $a, b \in \Rp$,
  let $k$ be so large that $\frac a {2^k}, \frac b {2^k} \leq 1$,
  hence also $\frac {ab} {2^{2k}} \leq 1$.  Then
  $a \cdot (b \cdot x) = \dbl^k (\frac a {2^k} \cdot \dbl^k (\frac b
  {2^k} \cdot x)) = \dbl^{2k} (\frac a {2^k} \cdot \frac b {2^k} \cdot
  x)$ (since $\dbl$ is linear hence $\Iu$-homogeneous)
  $= \dbl^{2k} (\frac {ab} {2^{2k}} \cdot x)$ (by the second equation
  of Lemma~\ref{lemma:bary:alg:pointed}) $= (ab) \cdot x$.  Third, we
  have $1 \cdot x = x$ by the third equation of
  Lemma~\ref{lemma:bary:alg:pointed}.  Fourth,
  $a \cdot 0 = \dbl^k (\frac a {2^k} \cdot \bot)$ for some $k$ large
  enough, and that is equal to $\dbl^k (\bot)$ by the
  \ifta
  first equation
  of the second row of equations of
  \else
  fourth equation of
  \fi
  Lemma~\ref{lemma:bary:alg:pointed}, which is equal to $\bot = 0$,
  since $\dbl$ is linear hence strict.  Fifth, for all $x, y \in \B$,
  $a \cdot (x+y) = \dbl^k (\frac a {2^k} \cdot (\dbl (x +_{1/2} y)))$
  where $k$ is large enough that $\frac a {2^k} \leq 1$.  Since $\dbl$
  is linear hence $\Iu$-homogeneous, that is equal to
  $\dbl^{k+1} (\frac a {2^k} \cdot (x +_{1/2} y))$, which in turn is
  equal to
  $\dbl^{k+1} ((\frac a {2^k} \cdot x) +_{1/2} (\frac a {2^k} \cdot
  y))$ (by the \ifta
  second equation of the second row of equations of
  \else
  fifth equation of
  \fi
  Lemma~\ref{lemma:bary:alg:pointed}), which is equal to
  $\dbl ((\dbl^k (\frac a {2^k} \cdot x) +_{1/2} (\dbl^k (\frac a
  {2^k} \cdot y)))$ since $\dbl$, hence also $\dbl^k$, is affine; and
  what we wrote is just $(a \cdot x) + (a \cdot y)$, so
  $a \cdot (x+y) = (a \cdot x) + (a \cdot y)$.  Sixth, for all
  $a, b \in \Rp$, let $k \in \nat$ be such that
  $\frac a {2^k}, \frac b {2^k} \leq 1$; in particular,
  $\frac {a+b} {2^{k+1}} \leq 1$.  Then
  $(a+b) \cdot x = \dbl^k (\frac {a+b} {2^k} \cdot x)$
  $a \cdot x + b \cdot x = \dbl (\dbl^k (\frac a {2^k} \cdot x)
  +_{1/2} \dbl^k (\frac b {2^k} \cdot x)) = \dbl^{k+1} (\frac a {2^k}
  \cdot x +_{1/2} \frac b {2^k} \cdot x)$ (since $\dbl$, hence
  $\dbl^k$, is affine) $= \dbl^{k+1} (\frac {a+b} {2^k} \cdot x)$ (by
  the last equation of Lemma~\ref{lemma:bary:alg:pointed})
  $= (a+b) \cdot x$.

  If $\B$ is a preordered barycentric algebra and $\dbl$ is monotonic,
  then addition is monotonic as a composition of monotonic maps, zero
  is least by definition, and $a \cdot \_$ is monotonic for every
  $a \in \Rp$, as the composition of the monotonic maps $\dbl^k$ and
  $\frac a {2^k} \cdot \_$ (the latter being monotonic by
  Lemma~\ref{lemma:bary:alg:pointed}).  This makes $\B$ a pointed
  semipreordered cone, hence a preordered cone.  Explicitly,
  $\_ \cdot x$ is also monotonic: if $a \leq b$, then
  $b \cdot x = a \cdot x + (b-a) \cdot x \geq a \cdot x + 0 = a \cdot
  x$, where the middle inequality is because $0$ is least.

  We now assume that $\B$ is a semitopological (resp.\ topological)
  barycentric algebra and that $\dbl$ is continuous.  It is clear that
  addition is separately (resp.\ jointly) continuous, and that
  $a \cdot \_$ is continuous for every $a \in \Rp$.  It remains to
  show that $a \mapsto a \cdot x$ is continuous from $\Rp$ to $\B$
  for every $x \in \B$.  Let $a_0 \in \Rp$, and let $V$ be an open
  neighborhood of $a_0 \cdot x$.  If $V$ is the whole of $\B$, then
  its inverse image under $\_ \cdot x$ is the whole of $\Rp$, which is
  Scott-open, so we will assume that $V$ is proper, or equivalently
  that $\bot = 0$ is not in $V$.  We fix $k \in \nat$ so that
  $\frac {a_0} {2^k} \leq 1$.  The map
  $a \mapsto \dbl^k \circ (\frac a {2^k} \cdot x)$ is continuous from
  $[0, 2^k]_\sigma$ to $\B$, in particular because
  $b \mapsto b \cdot x$ is continuous by
  Lemma~\ref{lemma:bary:alg:pointed}.  Hence there is a Scott-open
  subset $U$ of $[0, 2^k]$ that contains $a_0$ and whose image under
  $a \mapsto \dbl^k \circ (\frac a {2^k} \cdot x)$ is included in $V$.
  Since $\bot \not\in V$, $0$ is not in $U$.  Therefore $U$ is an
  interval of the form $]r, 2^k]$, with $0 \leq r < a_0$.  It remains
  to see that $]r, \infty[$ is a Scott-open neighborhood of $a_0$ such
  that $a \cdot x \in V$ for every $a \in {]r, \infty[}$.  This is
  true for $a \in {]r, 2^k]}$ by construction.  Seeing $\B$ as a
  preordered barycentric algebra, with its specialization preordering,
  we have see that $\_ \cdot x$ is monotonic, so this is also true for
  larger values of $a$, since $V$ is upwards-closed.  \qed
\end{proof}

\subsection{Telescopes}
\label{sec:telescopes}

There are several ways of building the free preordered cone over a
preordered pointed barycentric algebra $\B$.  The cone $\conify (\B)$
itself cannot be the free preordered cone over $\B$ as a
\emph{pointed} barycentric algebra, because $\etac_{\B}$ is not a
morphism in the category of pointed barycentric algebras: it is not
strict, since it maps the the distinguished element $\bot$ of $\B$ to
$(1, \bot)$, which is not the distinguished element $0$ of
$\conify (\B)$.  One way of building the desired free preordered cone
is by taking a quotient of $\conify (\B)$, under the smallest
equivalence relation $\equiv_\bot$ such that $(r, \bot) \equiv_\bot 0$
and $(r, a \cdot x) \equiv_\bot (ra, x)$ for all
$r \in \Rp \diff \{0\}$, $a \in {]0, 1[}$ and $x \in \B$.  We will
leave this as an exercise to the reader.  A more interesting
construction, although it appears to depend on a spurious parameter
$\alpha \in {]0, 1[}$, is as follows.  We build the colimit, taken in
the category of sets, or of preordered sets, or of topological spaces,
of the diagram:
\begin{align*}
  \xymatrix{\B \ar[r]^{\alpha \cdot \_}
  & \B \ar[r]^{\alpha \cdot \_}
  & \cdots \ar[r]^{\alpha \cdot \_}
  & \B \ar[r]^{\alpha \cdot \_}
  & \cdots}
\end{align*}
where $\alpha$ is a fixed real number such that $0 < \alpha < 1$.
We first build the coproduct $\coprod_{n \in \nat} \B$, whose elements
are pairs $(n, x)$ with $x \in \B$, and then we take a quotient by an
appropriate equivalence relation.  Here is an explicit definition.

\begin{definition}[Telescope]
  \label{defn:tscope}
  For every pointed barycentric algebra $\B$ and every given number
  $\alpha \in {]0, 1[}$, let $\equiv_\alpha$ be the smallest equivalence
  relation on $\coprod_{n \in \nat} \B$ such that $(n, x)
  \equiv_\alpha (n+1, \alpha \cdot x)$ for all $n \in \nat$ and $x \in \B$.

  The \emph{$\alpha$-telescope} on $\B$ is the quotient
  $\tscope_\alpha (\B) \eqdef \quot {(\coprod_{n \in \nat} \B)}
  {\equiv_\alpha}$.  We write $[(n, x)]_\alpha$ for the equivalence
  class of $(n, x)$. 
\end{definition}

Every element of $\tscope_\alpha (\B)$ can be written
$[(n, x)]_\alpha$, for each sufficiently large natural number $n$.
Explicitly, one can write it as $[(n_0, x_0)]_\alpha$ for some
$n_0 \in \nat$ and $x_0 \in \B$, and therefore also as
$[(n, \alpha^{n-n_0} \cdot x_0)]_\alpha$ for every $n \geq n_0$.  It
follows that, given two points of $\tscope_\alpha (\B)$, we can write
them as $[(n, x)]_\alpha$ and $[(n, x)]_\alpha$ with the same value of
$n \in \nat$.  We will use these facts often.

\begin{lemma}
  \label{lemma:equiv:alpha}
  Let $\B$ be a pointed barycentric algebra.  For every
  $\alpha \in {]0, 1[}$, for all $m, n \in \nat$ and $x, y \in \B$,
  $(m, x) \equiv_\alpha (n, y)$ if and only if
  $\alpha^{k-m} \cdot x = \alpha^{k-n} \cdot y$ for some
  $k \geq m, n$.
\end{lemma}
\begin{proof}
  Let us define the relation $\cong_\alpha$ on
  $\coprod_{n \in \nat} \B$ by by $(m, x) \cong_\alpha (n, y)$ if and
  only if $\alpha^{k-m} \cdot x = \alpha^{k-n} \cdot y$ for some
  $k \geq m, n$.

  If $(m, x) \cong_\alpha (n, y)$, then, taking $k$ as above,
  $(m, x) \equiv_\alpha (m+1, \alpha \cdot x) \equiv_\alpha \cdots
  \equiv_\alpha (k, \alpha^{k-m} \cdot x) = (k, \alpha^{k-n} \cdot y)
  \equiv_\alpha (k-1, \alpha^{k-n-1} \cdot y) \equiv_\alpha \cdots
  \equiv_\alpha (n, y)$.  Hence $\cong_\alpha$ is smaller than or
  equal to (included in) $\equiv_\alpha$.

  It is clear that $\cong_\alpha$ is reflexive and symmetric.  Let now
  $(m, x) \cong_\alpha (n, y)$ and $(n, y) \cong_\alpha (p, z)$.  Then
  $\alpha^{k-m} \cdot x = \alpha^{k-n} \cdot y$ for some
  $k \geq m, n$, and $\alpha^{k'-n} \cdot y = \alpha^{k'-p} \cdot z$
  for some $k' \geq n, p$.  Let us pick $k'' \geq k, k'$.  Then
  $k'' \geq m, p$ and
  $\alpha^{k''-m} \cdot x = \alpha^{k''-k} \cdot \alpha^{k-m} \cdot y
  = \alpha^{k''-k} \cdot \alpha^{k-n} \cdot y = \alpha^{k''-n} \cdot y
  = \alpha^{k''-k'} \cdot \alpha^{k'-n} \cdot y = \alpha^{k''-k'}
  \cdot \alpha^{k'-p} \cdot z = \alpha^{k''-p} \cdot z$, so
  $(m, x) \cong_\alpha (p, z)$.  Therefore $\cong_\alpha$ is
  transitive, and hence, an equivalence relation.

  For all $n \in \nat$ and $x \in \B$,
  $(n, x) \cong_\alpha (n+1, \alpha \cdot x)$.  Indeed, let
  $k \eqdef n+1$.  Then
  $\alpha^{k-n} \cdot x = \alpha^{k-(n+1)} \cdot \alpha \cdot x$,
  since both sides are equal to $\alpha \cdot x$, using the second
  equation of Lemma~\ref{lemma:bary:alg:pointed} to simplify the
  right-hand side.  Therefore $\cong_\alpha$ is an equivalence
  relation such that $(n, x) \cong_\alpha (n+1, \alpha \cdot n)$ for
  all $n \in \nat$ and $x \in \B$.  Since $\equiv_\alpha$ is the
  smallest such relation, it is smaller than or equal to
  $\cong_\alpha$.  Therefore $\equiv_\alpha$ and $\cong_\alpha$
  coincide.  \qed
\end{proof}

\begin{remark}
  \label{rem:equiv:alpha:classes}
  It is not in general the case that $(n, x) \equiv_\alpha (n, y)$
  (with the same $n$, chosen in advance) if and only if $x=y$.  A
  counterexample will be given in
  Example~\ref{exa:bary:alg:pointed:noembed:1} below.  The case of
  pointed $T_0$ semitopological barycentric algebras is nicer, as we
  will see in Remark~\ref{rem:equiv:alpha:semitop}.
\end{remark}

\begin{example}
  \label{exa:bary:alg:pointed:noembed:1}
  Keimel and Plotkin give an example of an ordered pointed barycentric
  algebra with two distinct points $x$ and $y$ such that
  $r \cdot x = r \cdot y$ for every $r \in {]0, 1[}$
  \cite[Example~2.19]{KP:mixed}.  We reformulate it as follows.  Let
  $\B^{KP}$ be the subset of $(\real \cup \{-\infty\}) \times \Rp$
  consisting of the points $(-\infty, s)$ with $s \in [0, 1]$, and the
  point $(0, 1)$.  We equip $\B^{KP}$ with the pointwise ordering and
  the pointwise operations $+_a$, so that
  $(-\infty, s) +_a (-\infty, t) = (-\infty, as+(1-a) t)$,
  $(-\infty, s) +_a (0, 1) = (-\infty, as + 1-a)$ if $a > 0$, $(0, 1)$
  if $a=0$, in particular.  This is a preordered pointed barycentric
  algebra, with distinguished (least) element
  $\bot \eqdef (-\infty, 0)$.  Letting $x \eqdef (0, 1)$ and
  $y \eqdef (-\infty, 1)$, we have $r \cdot x = r \cdot y$ for every
  $r \in {]0, 1[}$, but $x \not\leq y$: we recall that
  $r \cdot x = x +_r \bot$, and similarly for $r \cdot y$, from which
  we obtain that both are equal to $(-\infty, r)$.  In particular,
  $[(0, x)]_\alpha = [(0, y)]_\alpha$, but $x \neq y$.
\end{example}

\begin{proposition}
  \label{prop:tscope:bary:alg}
  For every pointed barycentric algebra $\B$, for every
  $\alpha \in {]0, 1[}$,
  \begin{enumerate}
  \item $\tscope_\alpha (\B)$ has the structure of a pointed
    barycentric algebra, with distinguished element
    $[(0, \bot)]_\alpha$, and with $+_a$ defined by:
    \ifta
    \begin{align*}
      [(m, x)]_\alpha +_a [(n, y)]_\alpha
      & \eqdef [(\max (m, n),
        (\alpha^{\max (m,n)-m} \cdot x) +_a (\alpha^{\max (m,n)-n} \cdot
        y))]_\alpha
    \end{align*}
    \else
    \begin{align*}
      & [(m, x)]_\alpha +_a [(n, y)]_\alpha \\
      & \eqdef [(\max (m, n), \\
      & \qquad
        (\alpha^{\max (m,n)-m} \cdot x) +_a (\alpha^{\max (m,n)-n} \cdot
        y))]_\alpha
    \end{align*}
    \fi
    for every $a \in [0, 1]$ and for all elements $[(m, x)]_\alpha$
    and $[(n, y)]_\alpha$ of $\tscope_\alpha (\B)$---in particular,
    $[(n, x)]_\alpha +_a [(n, y)]_\alpha \eqdef [(n, x +_a
    y)]_\alpha$.
  \item $a \cdot [(n, x)]_\alpha \eqdef [(n, a \cdot x)]_\alpha$ for
    all $a \in [0, 1]$, $n \in \nat$ and $x \in \B$.
  \item The latter formula extends to all $a \in \Rp$ by letting
    $a \cdot ([(n, x)]_\alpha) \eqdef [(n+k, (\alpha^k a) \cdot
    x)]_\alpha$ for all $a \in \Rp$, $n \in \nat$ and $x \in \B$,
    where $k$ is so large that $\alpha^k a \leq 1$.  The map
    $a \cdot \_$ is linear.
  \item In particular, $2 \cdot \_$ is a doubling map, and therefore
    $\tscope_\alpha (\B)$ has a cone structure, given by the
    construction of Proposition~\ref{prop:ba->cone}.
  \end{enumerate}

  If $\B$ is a preordered pointed barycentric algebra, then:
  \begin{enumerate}[resume]
  \item $\tscope_\alpha (\B)$ is itself a preordered pointed
    barycentric algebra, with the preordering $\leqca$ defined as
    follows: for all $u, v \in \tscope_\alpha (\B)$, $u \leqca v$ if
    and only if one can write $u$ as $[(n, x)]_\alpha$, $v$ as
    $[(n, y)]_\alpha$ with the same $n \in \nat$ and with $x \leq y$
    in $\B$; in particular, the distinguished element $[(0,
    \bot)]_\alpha$ is least;
  \item for every $a \in \Rp$, the map
    $a \cdot \_ \colon \tscope_\alpha (\B) \to \tscope_\alpha (\B)$ of
    item~3 is monotonic;
  \item in particular, $2 \cdot \_$ is a monotonic doubling map, and
    therefore $\tscope_\alpha (\B)$ has a preordered cone structure,
    given by the construction of Proposition~\ref{prop:ba->cone}.
  \end{enumerate}
\end{proposition}
\begin{proof}
  1.  We verify that the operation $+_a$ on $\tscope_\alpha (\B)$ is
  well-defined.  In order to do so, letting
  $f ((m, x), (n, y)) \eqdef (\max (m, n), (\alpha^{\max (m,n)-m}
  \cdot x) +_a (\alpha^{\max (m,n)-n} \cdot y))$, or equivalently
  $f ((m, x), (n, y)) \eqdef (\max (m, n), (\alpha^{\max (n-m,0)}
  \cdot x) +_a (\alpha^{\max (m-n,0)} \cdot y))$, we show that if
  $(m, x) \equiv_\alpha (m', x')$ then
  $f ((m, x), (n, y)) \equiv_\alpha f ((m', x'), (n, y))$.  It is
  enough to prove this when $m' \eqdef m+1$ and
  $x' \eqdef \alpha \cdot x$.  Then:
  \ifta
  \begin{align*}
    & f ((m+1, \alpha \cdot x), (n, y)) \\
    & = (\max (m+1, n),
      (\alpha^{\max (n-m-1, 0)} \cdot \alpha \cdot x)
      +_a (\alpha^{\max (m+1-n,0)} \cdot y)) \\
    & = (\max (m+1, n),
      (\alpha^{\max (n-m, 1)} \cdot x)
      +_a (\alpha^{\max (m+1-n,0)} \cdot y)),
  \end{align*}
  \else
  \begin{align*}
    & f ((m+1, \alpha \cdot x), (n, y)) \\
    & = (\max (m+1, n),
    \\ & \qquad
      (\alpha^{\max (n-m-1, 0)} \cdot \alpha \cdot x)
      +_a (\alpha^{\max (m+1-n,0)} \cdot y)) \\
    & = (\max (m+1, n),
    \\ & \qquad
      (\alpha^{\max (n-m, 1)} \cdot x)
      +_a (\alpha^{\max (m+1-n,0)} \cdot y)),
  \end{align*}
  \fi
  
  using the second equation of Lemma~\ref{lemma:bary:alg:pointed}.

  If $m \geq n$, then $\max (m+1, n) = m+1$, $\max (n-m, 1) = 1$ and
  $\max (m+1-n, 0) = m+1-n$, so
  $f ((m+1, \alpha \cdot x), (n, y)) = (m+1, (\alpha \cdot x) +_a
  (\alpha^{m+1-n} \cdot y)) = (m+1, (\alpha \cdot x) +_a (\alpha \cdot
  \alpha^{m-n} \cdot y))$ by the second equation of
  Lemma~\ref{lemma:bary:alg:pointed}.  By the
  \ifta
  second equation in the
  second row of equations of
  \else
  fifth equation of
  \fi
  Lemma~\ref{lemma:bary:alg:pointed}, that
  is equal to $(m+1, \alpha \cdot (x +_a (\alpha^{m-n} \cdot y))$,
  which is $\equiv_\alpha$-equivalent to
  $(m, x +_a (\alpha^{m-n} \cdot y)$; and
  $f ((m, x), (n, y)) = (m, x +_a \alpha^{m-n} \cdot y)$, since
  $m \geq n$.

  If $m < n$, then $\max (m+1, n) = n$, $\max (n-m, 1) = n-m$ and
  $\max (m+1-n, 0) = 0$, so
  $f ((m+1, \alpha \cdot x), (n, y)) = (n, (\alpha^{n-m} \cdot x) +_a
  y)$, while $\max (m, n) = n$, $\max (n-m, 0) = n-m$ and
  $\max (m-n, 0) = 0$, so
  $f ((m, x), (n, y)) = (n, (\alpha^{n-m} \cdot x) +_a y)$, which is
  the exact same value.

  Hence $+_a$ is well-defined on equivalence classes.  Since we can
  write any element of $\tscope_\alpha (\B)$ as $[(n, x)]_\alpha$ with
  $n$ large enough, in any expression involving finitely many points
  of $\tscope_\alpha (\B)$, we may assume that all these points are
  written as $[(n, x)]_\alpha$ with the same $n$ (and possibly
  different $x$).  With the formula
  $[(n, x)]_\alpha +_a [(n, y)]_\alpha \eqdef [(n, x +_a y)]$, it is
  then clear that $\tscope_\alpha (\B)$ is a barycentric algebra.

  2.  For every $a \in [0, 1]$, for every
  $[(n, x)]_\alpha \in \tscope_\alpha (\B)$, then,
  $a \cdot [(n, x)]_\alpha = [(n, x)]_\alpha +_a [(n, \bot)]_\alpha =
  [(n, x +_a \bot)]_\alpha = [(n, a \cdot x)]_\alpha$.

  3.  We extend the formula of item~2 and we define
  $a \cdot [(n, x)]_\alpha \eqdef [(n+k, (\alpha^k a) \cdot
  x)]_\alpha$ for every $a \in \Rp$, not just in $[0, 1]$, where $k$
  is so large that $\alpha^k a \leq 1$.  (If $a \in [0, 1]$, we can
  take $k \eqdef 0$, which gives back the previous definition.)  Such
  a $k$ always exists because $\alpha \in {]0, 1[}$.  It does not
  matter which $k$ we choose, since for all $k, k' \in \nat$ such that
  $k \leq k'$ and $\alpha^k a \leq 1$,
  $(n+k, (\alpha^k a) \cdot x) \equiv_\alpha (n+k', \alpha^{k'-k}
  \cdot (\alpha^k a) \cdot x) = (n+k', (\alpha^{k'} a) \cdot x)$,
  using the second equation of Lemma~\ref{lemma:bary:alg:pointed}.
  The definition of scalar multiplication is also independent of the
  representative $(n, x)$ in the equivalence class $[(n, x)]_\alpha$.
  Indeed, letting $f (n, x) \eqdef (n+k, (\alpha^k a) \cdot x)$, it
  suffices to show that
  $f (n+1, \alpha \cdot x) \equiv_\alpha f (n, x)$: indeed,
  $f (n+1, \alpha \cdot x) = (n+k+1, (\alpha^k a) \cdot \alpha \cdot
  x) = (n+k+1, (\alpha^{k+1} a) \cdot x) = (n+k+1, \alpha \cdot
  (\alpha^k a) \cdot x)$ (by the second equation of
  Lemma~\ref{lemma:bary:alg:pointed})
  $\equiv_\alpha (n+k, (\alpha^k a) \cdot x) = f (n, x)$.  Hence
  $a \cdot \_$ is well-defined on $\tscope_\alpha (\B)$ for every
  $a \in \Rp$.

  The map $a \cdot \_$ is strict:
  $a \cdot [(0, \bot)]_\alpha = [(k, (\alpha^k a) \cdot \bot)]_\alpha$
  (with $k$ so large that $\alpha^k a \leq 1$) $= [(k, \bot)]_\alpha$
  (by
  \ifta
  the first equation on the second row of equations of
  \else
  the fourth equation of
  \fi
  Lemma~\ref{lemma:bary:alg:pointed}) $= [(0, \bot)]_\alpha$.

  We claim that the map $a \cdot \_$ is affine.  We have
  $a \cdot ([(n, x)]_\alpha +_a [(n, y)]_\alpha) = (a \cdot [(n,
  x)]_\alpha) +_a (a \cdot [(n, y)]_\alpha)$.  Picking $k$ such that
  $\alpha^k a \leq 1$ as above, the left-hand side is equal to
  $[(n+k, (\alpha^k a) \cdot (x +_a y))]_\alpha$, while the right-hand
  side is equal to
  $[(n+k, ((\alpha^k a) \cdot x) +_a ((\alpha^k a) \cdot y))]$.  Those
  are equal by \ifta
  the middle equation in the second row of equations of
  \else
  the fifth equation of
  \fi
  Lemma~\ref{lemma:bary:alg:pointed}.

  Being strict and affine, $a \cdot \_$ is linear, thanks to
  Proposition~\ref{prop:strict:aff}.

  4.  In order to show that $2 \cdot \_$ is a doubling map, it remains
  to verify that $2 \cdot \frac 1 2 \cdot u = u$ for every
  $u \in \tscope_\alpha (\B)$.  We write $u$ as $[(n, x)]_\alpha$, and
  we pick $k$ so that $2\alpha^k \leq 1$.  Then
  $2 \cdot \frac 1 2 \cdot u = 2 \cdot [(n, \frac 1 2 \cdot x)]_\alpha
  = [(n+k, (2 \alpha^k) \cdot \frac 1 2 \cdot x)]_\alpha = [(n+k,
  \alpha^k \cdot x)]_\alpha$ (by the second equation of
  Lemma~\ref{lemma:bary:alg:pointed}) $= u$, since
  $(n+k, \alpha^k \cdot x) \equiv_\alpha (n, x)$.  Therefore
  $\tscope_\alpha (\B)$ is a cone by Proposition~\ref{prop:ba->cone}.
  
  5.  We now assume that $\B$ is a preordered pointed barycentric
  algebra.  The definition of the preordering $\leqca$ on
  $\tscope_\alpha (\B)$ is a bit subtle.  If $x \leq y$, then
  $[(n, x)]_\alpha \leq [(n, y)]_\alpha$ for some $n$, but not
  necessarily for the least $n$ such that we can write both sides of
  the inequality with the same $n$.

  At any rate, this defines a preordering.  Reflexivity is obvious,
  while transitivity is proved as follows.  If $u \leqca v$ and
  $v \leqca w$, we can write $u$ as $[(n, x)]_\alpha$, $v$ as
  $[(n, y)]_\alpha$ with $x \leq y$, and we can also write $v$ as
  $[(n', y')]_\alpha$ and $w$ as $[(n', z)]_\alpha$ with $y' \leq z$.
  Let $n'' \eqdef \max (n, n')$.  Now
  $(n, x) \equiv_\alpha (n+1, \alpha \cdot x) \equiv_\alpha \cdots
  \equiv_\alpha (n'', \alpha^{n''-n} \cdot x)$, similarly
  $(n, y) \equiv_\alpha (n'', \alpha^{n''-n} \cdot y)$,
  $(n', y') \equiv_\alpha (n'', \alpha^{n''-n'} \cdot y')$, and
  $(n', z) \equiv_\alpha (n'', \alpha^{n''-n'} \cdot z)$.  We recall
  that $v = [(n, y)]_\alpha = [(n', y')]_\alpha,$ so
  $(n, y) \equiv_\alpha (n', y')$, and therefore
  $(n'', \alpha^{n''-n} \cdot y) \equiv_\alpha (n'', \alpha^{n''-n'}
  \cdot y')$.  By Lemma~\ref{lemma:equiv:alpha},
  $\alpha^{k-n''} \cdot \alpha^{n''-n} \cdot y = \alpha^{k-n''} \cdot
  \alpha^{n''-n'} \cdot y'$ for some $k \geq n''$, therefore
  $\alpha^{k-n} \cdot y = \alpha^{k-n'} \cdot y'$.  Then,
  $(n, x) \equiv_\alpha (n'', \alpha^{n''-n} \cdot x) \equiv_\alpha
  (n''+1, \alpha^{n''+1-n} \cdot x) \equiv_\alpha \cdots \equiv_\alpha
  (k, \alpha^{k-n} \cdot x)$, similarly
  $(n', z) \equiv_\alpha (k, \alpha {k-n'} \cdot z)$, so we can write
  $u$ as $[(k, \alpha^{k-n} \cdot x)]_\alpha$, $w$ as
  $[(k, \alpha^{k-n'} \cdot z)]_\alpha$, and
  $\alpha^{k-n} \cdot x \leq \alpha^{k-n} \cdot y = \alpha^{k-n'}
  \cdot y' \leq \alpha^{k-n'} \cdot z$, showing that $u \leqca w$.

  We check that $\_ +_a w$ is monotonic from $\tscope_\alpha (\B)$ to
  itself, for every $a \in [0, 1]$ and for every
  $w \in \tscope_\alpha (\B)$.  Let $u \leqca v$, with
  $u \eqdef [(n, x)]_\alpha$, $v \eqdef [(n, y)]_\alpha$ and
  $x \leq y$.  We may write $w$ as $[(n, z)]_\alpha$, with the same
  $n$, possibly increasing the value of $n$ for $u$, $v$, and $w$.
  Then $x +_a y \leq z$, so $u + v = [(n, x +_a y)]_\alpha \leqca w$.
  Hence $\tscope_\alpha (\B)$ is a preordered barycentric algebra.

  The equivalence class of $(0, \bot)$ contains all the pairs
  $(n, \alpha^n \cdot \bot)$, namely all the pairs $(n, \bot)$,
  $n \in \nat$, because $\alpha^n \cdot \bot = \bot$, by
  \ifta
  the first equation of the second row of equations in
  \else
  the fourth equation of
  \fi
  Lemma~\ref{lemma:bary:alg:pointed}.  For every $x \in \B$,
  $\bot \leq x$, so
  $[(0, \bot)]_\alpha = [(n, \bot)]_\alpha \leqca [(n, x)]_\alpha$.
  Hence $[(0, \bot)]_\alpha$ is least in $\tscope_\alpha (\B)$, and
  therefore $\tscope_\alpha (\B)$ is a preordered pointed barycentric
  algebra.

  6.  We verify that $a \cdot \_$ is monotonic.  If $x \leq y$, so
  that $[(n, x)]_\alpha \leqca [(n, y)]_\alpha$, then
  $a \cdot [(n, x)]_\alpha \leqca a \cdot [(n, y)]_\alpha$ since
  $(\alpha^k a) \cdot x \leq (\alpha^k a) \cdot x$ (we recall that
  scalar multiplication on $\B$ is monotonic, by
  Lemma~\ref{lemma:bary:alg:pointed}), where $k$ is sufficiently
  large, so that $\alpha^k a \leq 1$.

  7.  By items~4 and~6, $2 \cdot \_$ is a monotonic doubling map, so
  $\tscope_\alpha (\B)$ is a preordered cone by
  Proposition~\ref{prop:ba->cone}.  \qed
\end{proof}

\begin{theorem}
  \label{thm:bary:alg:conify:ord:pointed}
  For every preordered pointed barycentric algebra $\B$ and every
  $\alpha \in {]0, 1[}$, $\tscope_\alpha (\B)$ is a preordered cone,
  and:
  \begin{enumerate}
  \item the function $\etaca_{\B} \colon x \mapsto [(0, x)]_\alpha$ is
    a monotonic linear map from $\B$ to $\tscope_\alpha (\B)$;
  \item for every map $f \colon \B \to \C$ to a cone $\C$ that
    commutes with $\alpha \cdot \_$ (namely such that
    $f (\alpha \cdot x) = \alpha \cdot f (x)$ for every $x \in \B$),
    there is a unique map
    $f^{\scope\alpha} \colon \tscope_\alpha (\B) \to \C$ that
    commutes with $\alpha \cdot \_$ and such that
    $f^{\scope\alpha} \circ \etaca_{\B} = f$, and it is given by the
    formula
    $f^{\scope\alpha} ([(n, x)]_\alpha) \eqdef (1/\alpha)^n \cdot f
    (x)$; if $f$ is affine, resp.\ strict, then so is
    $f^{\scope\alpha}$, in particular $f^{\scope\alpha}$ is linear
    if $f$ is; if $f$ is $\Iu$-homogeneous, then $f^{\scope\alpha}$
    is positively homogeneous;
  \item if $\C$ is a preordered cone, and if $f$ is monotonic (resp.\
    concave, convex, superlinear, sublinear in the sense of
    Definition~\ref{defn:bary:alg:pointed:linear}), then
    $f^{\scope\alpha}$ is monotonic (resp.\ concave, convex,
    superlinear, sublinear).
  \item In particular, $\tscope_\alpha (\B)$ is the free cone, and
    also the free preordered cone, over the preordered pointed
    barycentric algebra $\B$, and this is independent of $\alpha$ (up
    to natural isomorphism of preordered cones).
  \end{enumerate}
\end{theorem}
\begin{proof}
  We know that $\tscope_\alpha (\B)$ is a preordered cone from
  Proposition~\ref{prop:tscope:bary:alg}.

  1.  For all $x, y \in \B$ such that $x \leq y$, we have
  $[(0, x)]_\alpha \leqca [(0, y)]_\alpha$, so $\etaca_{\B}$ is
  monotonic.  It is strict by definition of the zero element of
  $\tscope_\alpha (\B)$.  In order to see that it is linear, it
  remains to show that it is affine, by Proposition~\ref{prop:strict:aff},
  and this follows from the equality
  $[(n, x)]_\alpha +_a [(n, y)]_\alpha = [(n, x +_a y)]_\alpha$ (see
  Proposition~\ref{prop:tscope:bary:alg}, item~1) with $n \eqdef 0$.

  2.  We first characterize the map $\alpha \cdot \_$ on
  $\tscope_\alpha (\B)$.  By Proposition~\ref{prop:tscope:bary:alg},
  item~4, for all $n \in \nat$ and $x \in \B$,
  $\alpha \cdot [(n, x)]_\alpha = [(n, \alpha \cdot x)]_\alpha$, so
  $\alpha \cdot [(n, x)]_\alpha = [(n-1, x)]_\alpha$ if $n \geq 1$.
  By induction on $n$, we obtain that
  $\alpha^n \cdot ([(n, x)]_\alpha) = [(0, x)]_\alpha$.  If
  $f^{\scope\alpha}$ exists and commutes with $\alpha \cdot$, then
  for every point $[(n, x)]_\alpha$ of $\tscope_\alpha (\B)$,
  $f^{\scope\alpha} (\alpha^n \cdot [(n, x)]_\alpha)$ is equal both
  to $\alpha^n \cdot (f^{\scope\alpha} ([(n, x)]_\alpha))$ and to
  $f^{\scope\alpha} ([(0, x)]_\alpha) = (f^{\scope\alpha} \circ
  \etaca_{\B}) (x) = f (x)$.  Therefore
  $f^{\scope\alpha} ([(n, x)]_\alpha)$ must be equal to
  $(1/\alpha)^n \cdot f (x)$, a notation that makes sense since scalar
  multiplication is taken in the cone $\C$.  It follows that
  $f^{\scope\alpha}$ is unique if it exists.

  As for existence, we define $f^{\scope\alpha}$ as mapping every
  $[(n, x)]_\alpha$ to $(1/\alpha)^n \cdot f (x)$.  This is
  independent of the chosen representative in the equivalence class of
  $(n, x)$.  Indeed, it suffices to show that
  $(1/\alpha)^{n+1} \cdot f (\alpha \cdot x) = (1/\alpha)^n \cdot f
  (x)$, which follows from the fact that $f$ commutes with
  $\alpha \cdot \_$.

  If $f$ is strict, then $f^{\scope\alpha} ([(0, \bot)]_\alpha) =
  (1/\alpha)^0 \cdot f (\bot) = f (\bot) = 0$, so $f^{\scope\alpha}$
  is strict.
  
  If $f$ is affine, then for all points $[(n, x)]_\alpha$ and
  $[(n, y)]_\alpha$ of $\tscope_\alpha (\B)$ (with the same $n$), for
  every $a \in [0, 1]$,
  $f^{\scope\alpha} ([(n, x)]_\alpha +_a [(n, y)]_\alpha) = f^{\scope,
    \alpha} ([(n, x +_a y)]_\alpha) = (1/\alpha)^n \cdot f (x +_a y) =
  (1/\alpha)^n \cdot (a \cdot f (x) + (1-a) \cdot f (y)) = a \cdot
  (1/\alpha)^n \cdot f (x) + (1-a) \cdot (1/\alpha)^n \cdot f (y) = a
  \cdot f^{\scope\alpha} ([(n, x)]_\alpha) + (1-a) \cdot f^{\scope,
    \alpha} ([(n, y)]_\alpha)$, so $f^{\scope\alpha}$ is affine.

  By Proposition~\ref{prop:strict:aff}, it follows that $f^{\scope\alpha}$
  is linear if $f$ is.

  If $f$ is $\Iu$-homogeneous, then we check that $f^{\scope\alpha}$
  is positively homogeneous.  Let $a \in \Rp$ and $[(n, x)]_\alpha$ be
  a point of $\tscope_\alpha (\B)$.  We pick $k$ so that
  $\alpha^k a \leq 1$.  Then
  $f^{\scope\alpha} (a \cdot [(n, x)]_\alpha) = f^{\scope\alpha}
  ([(n+k, (\alpha^k a) \cdot x)]_\alpha) = (1/\alpha)^{n+k} \cdot f
  ((\alpha^k a) \cdot x) = (1/\alpha)^{n+k} \cdot (\alpha^k a) \cdot f
  (x)$ (since $f$ is $\Iu$-homogeneous and $\alpha^k a \leq 1$)
  $= a \cdot (1/\alpha)^n \cdot f (x) = a \cdot f^{\scope\alpha}
  ([(n, x)]_\alpha)$.

  3.  Let us assume that $\C$ is preordered.  If $f$ is monotonic,
  then we check that for every $n \in \nat$, for all $x, y \in \B$
  such that $x \leq y$,
  $f^{\scope\alpha} ([(n, x)]_\alpha) \leq f^{\scope\alpha} ([(n,
  y)]_\alpha)$; and indeed,
  $f^{\scope\alpha} ([(n, x)]_\alpha) = (1/\alpha)^n \cdot f (x) \leq
  (1/\alpha)^n \cdot f (y) = f^{\scope\alpha} ([(n, y)]_\alpha)$,
  since $f$ is monotonic.

  If $f$ is concave, we reason as in the case where $f$ is affine (see
  item~2):
  $f^{\scope\alpha} ([(n, x)]_\alpha +_a [(n, y)]_\alpha) = f^{\scope,
    \alpha} ([(n, x +_a y)]_\alpha) = (1/\alpha)^n \cdot f (x +_a y)
  \geq (1/\alpha)^n \cdot (a \cdot f (x) + (1-a) \cdot f (y)) = a
  \cdot (1/\alpha)^n \cdot f (x) + (1-a) \cdot (1/\alpha)^n \cdot f
  (y) = a \cdot f^{\scope\alpha} ([(n, x)]_\alpha) + (1-a) \cdot
  f^{\scope\alpha} ([(n, y)]_\alpha)$, so $f^{\scope\alpha}$ is
  concave.  Similarly, if $f$ is convex then so is
  $f^{\scope\alpha}$.  Since superlinearity (resp.\ sublinearity) is
  equivalent to $\Iu$-homogeneity plus concavity (resp.\ convexity) by
  Remark~\ref{rem:bary:alg:pointed:linear}, $f^{\scope\alpha}$ is
  superlinear, resp.\ sublinear if $f$ is.

  4.  The fact that $\tscope_\alpha (\B)$ is the free preordered cone
  over $\B$ is an immediate consequence of items~2 and~3.  Hence
  $\tscope_\alpha$ define a functor that is left adjoint to the
  forgetful functor from preordered cones to preordered pointed
  barycentric algebras.  As such, it is determined in a unique way up
  to natural isomorphism in the category of preordered cones, hence is
  independent of $\alpha$, up to natural isomorphism.  \qed
\end{proof}

The map $\etaca_{\B}$ is almost a preorder embedding, as the following
Lemma~\ref{lemma:bary:alg:pointed:embed:almost} shows.  This rests on
the following clever observation; the first part is due to Neumann
\cite[Lemma~2]{neumann:bary}, the second part is due to Keimel and
Plotkin \cite[Lemma~2.8]{KP:mixed}, while the remaining is a trivial
consequence.
\begin{fact} 
  \label{fact:neumann}
  In a barycentric algebra $\B$, if $x +_a z = y +_a z$ for some
  $a \in {]0, 1[}$, then $x +_a z = y +_a z$ for every
  $a \in {]0, 1[}$.  In a preordered barycentric algebra $\B$, if
  $x +_a z \leq y +_a z$ for some $a \in {]0, 1[}$, then
  $x +_a z \leq y +_a z$ for every $a \in {]0, 1[}$.

  It follows that in a pointed barycentric algebra, if
  $r \cdot x = r \cdot y$ for some $r \in {]0, 1[}$, then
  $r \cdot x = r \cdot y$ for every $r \in {]0, 1[}$, and if
  $r \cdot x \leq r \cdot y$ for some $r \in {]0, 1[}$, then
  $r \cdot x \leq r \cdot y$ for every $r \in {]0, 1[}$.
\end{fact}

\begin{lemma}
  \label{lemma:bary:alg:pointed:embed:almost}
  Let $\alpha \in {]0, 1[}$.  For any two points $x$, $y$ in a
  preordered pointed barycentric algebra $\B$, if
  $\etaca_{\B} (x) \leqca \etaca_{\B} (y)$ in $\tscope_\alpha (\B)$,
  then $r \cdot x \leq r \cdot y$ for some $r \in {]0, 1[}$, and then
  for every $r \in {]0, 1[}$.
\end{lemma}
\begin{proof}
  If $\etaca_{\B} (x) \leqca \etaca_{\B} (y)$, then we can write
  $\etaca_{\B} (x)$ as $[(n, x')]_\alpha$, $\etaca_{\B} (y)$ as
  $[(n, y')]_\alpha$, where $x' \leq y'$.  Since
  $\etaca_{\B} (x) = [(0, x)]_\alpha = [(n, x')]_\alpha$, we have
  $(0, x) \equiv_\alpha (n, x')$, so
  $\alpha^k \cdot x = \alpha^{k-n} \cdot x'$ for some $k \geq n$, by
  Lemma~\ref{lemma:equiv:alpha}.  Similarly,
  $\alpha^{k'} \cdot y = \alpha^{k'-n} \cdot y'$ for some $k' \geq n$.
  Multiplying by powers of $\alpha$ if necessary (and using the second
  equation of Lemma~\ref{lemma:bary:alg:pointed}), we may assume that
  $k=k'$.  Since $x' \leq y'$, we obtain that
  $\alpha^{k-n} \cdot x' \leq \alpha^{k-n} \cdot x'$, hence that
  $\alpha^k \cdot x \leq \alpha^k \cdot y$.  We can therefore take
  $r \eqdef \alpha^k$---or $\alpha^{k+1}$ in order to ensure that
  $r < 1$.  Then $r \cdot x \leq r \cdot y$ for every
  $r \in {]0, 1[}$, by Fact~\ref{fact:neumann}.  \qed
\end{proof}
Frustratingly, but crucially, $r$ cannot be forced to be equal to $1$
in Lemma~\ref{lemma:bary:alg:pointed:embed:almost}.  We will see why
in Example~\ref{exa:bary:alg:pointed:noembed}.  In the meantime, let
us observe that the situation is nicer with pointed semitopological
barycentric algebras.
\begin{lemma}
  \label{lemma:bary:alg:pointed:embed:topo}
  Let $\B$ be a pointed semitopological barycentric algebra,
  considered with its specialization preordering $\leq$.  Then:
  \begin{enumerate}
  \item for all $x, y \in \B$, if $r \cdot x \leq r \cdot y$ for every
    $r \in {[0, 1[}$, then $x \leq y$;
  \item $\etaca_{\B}$ is a preorder embedding, where
    $\tscope_\alpha (\B)$ is given the preordering $\leqca$ of
    Proposition~\ref{prop:tscope:bary:alg}.
  \end{enumerate}
\end{lemma}
\begin{proof}
  1.  In order to show that $x \leq y$, we consider any open
  neighborhood $U$ of $x$, and we show that $y \in U$.  The map
  $\_ \cdot x$ is continuous from $\Rp$ to $\B$ by assumption, so
  $(\_ \cdot x)^{-1} (U)$ is a Scott-open subset of $\Rp$.  By
  assumption, it contains $1$, so it is of the form $]t, \infty[$ with
  $0 < t < 1$ or equal to the whole of $\Rp$.  In any case, it
  contains some $r \in {]0, 1[}$.  Then $r \cdot x \in U$, and since
  $U$ is upwards-closed and $r \cdot x \leq r \cdot y$, $r \cdot y$ is
  also in $U$.  In other words, $(\_ \cdot y)^{-1} (U)$ is a
  Scott-open subset of $\Rp$ that contains $r$. Since $r < 1$ and
  Scott-open sets are upwards-closed, it also contains $1$, so
  $y \in U$.    Since $U$ is arbitrary, it follows that $x \leq y$.

  2.  By Lemma~\ref{lemma:bary:alg:pointed:embed:almost}, for any two
  points $x, y \in \B$ such that
  $\etaca_{\B} (x) \leqca \etaca_{\B} (y)$, then the assumption of
  item~1 is satisfied; so $x \leq y$.  In other words, $\etaca_{\B}$
  is a preorder embedding.  \qed
\end{proof}

\begin{remark}
  \label{rem:equiv:alpha:semitop}
  Let $\alpha \in {]0, 1[}$, and $\B$ be a pointed $T_0$
  semitopological barycentric algebra.  Then the strange situations
  depicted in Remark~\ref{rem:equiv:alpha:classes} do not happen,
  namely, for every $n \in \nat$ and for all $x, y \in \B$,
  $(n, x) \equiv_\alpha (n, y)$ if and only if $x=y$.  Hence the
  equivalence classes with respect to $\alpha$ are all of the form
  $\{(n_0, x_0), (n_0+1, \alpha \cdot x_0), \cdots, (n_0+k, \alpha^k
  \cdot x_0), \cdots\}$ for a unique $n_0 \in \nat$ and a unique
  $x_0 \in \B$.  It also follows that $\etaca_{\B}$ is injective.
\end{remark}
\begin{proof}
  If $(n, x) \equiv_\alpha (n, y)$, then by
  Lemma~\ref{lemma:equiv:alpha},
  $\alpha^{k-n} \cdot x = \alpha^{k-n} \cdot y$ for some $k \geq n$.
  Then $r \cdot x = r \cdot y$, where $r \eqdef \alpha^{k+1-n} < 1$,
  so $r \cdot x = r \cdot y$ for every $r \in {]0, 1[}$, by
  Fact~\ref{fact:neumann}.  By
  Lemma~\ref{lemma:bary:alg:pointed:embed:topo}, $x \leq y$ and
  $y \leq x$, so $x=y$ since $\B$ is $T_0$.

  Given an element $u$ of $\tscope_\alpha (\B)$---an equivalence
  class---let $n_0$ be the least $n \in \nat$ such that some pair
  $(n, x)$ is in $u$.  We have just seen that there is a unique point
  $x_0 \in \B$ such that $(n_0, x_0)$ is in $u$.  Then all the
  $\equiv_\alpha$-equivalent points $(n_0+k, \alpha^k \cdot x_0)$ are
  in the equivalence class.  By the first part of the remark, any pair
  of the form $(n_0+k, y)$ must satisfy $y = \alpha^k \cdot x_0$, and
  there is no pair of the form $(n, x)$ with $n < n_0$ by minimality
  of $n_0$.

  In particular, for all points $x, y \in \B$, $\etaca_{\B} (x)$ is
  the equivalence class
  $\{(0, x), (1, \alpha \cdot x), \cdots, (n, \alpha^n \cdot x),
  \cdots\}$ and $\etaca_{\B} (y)$ is the equivalence class
  $\{(0, y), (1, \alpha \cdot y), \cdots, (n, \alpha^n \cdot y),
  \cdots\}$.  If they are equal, then in particular their (unique)
  elements with first component equal to $0$ are equal, namely $x=y$;
  so $\etaca_{\B}$ is injective.  \qed
\end{proof}

\begin{example}
  \label{exa:bary:alg:pointed:noembed}
  Let us consider $\B \eqdef \B^{KP}$, with the two points $x$ and $y$
  given in Example~\ref{exa:bary:alg:pointed:noembed:1}.  We have seen
  that $r \cdot x = r \cdot y$ for every $r \in {]0, 1[}$.  In
  particular, $\alpha \cdot x = \alpha \cdot y$, so
  $(0, x) \equiv_\alpha (1, \alpha \cdot x) = (1, \alpha \cdot y)
  \equiv_\alpha (0, y)$.  In other words,
  $\etaca_{\B} (x) = \etaca_{\B} (x)$, showing that $\etaca_{\B}$ is
  not injective.  It is not an order-embedding either, because $x
  \not\leq y$.
\end{example}

In general, we have the following equivalences.  The property
mentioned in item~3 below is called (OC2) by Keimel and Plotkin
\cite{KP:mixed}.
\begin{lemma}
  \label{lemma:bary:alg:pointed:embed}
  For a preordered pointed barycentric algebra $\B$, the following are
  equivalent:
  \begin{enumerate}
  \item $\etaca_{\B}$ is a preorder embedding for some $\alpha \in
    {]0, 1[}$;
  \item $\etaca_{\B}$ is a preorder embedding for every
    $\alpha \in {]0, 1[}$;
  \item for all $r \in {]0, 1[}$ and $x, y \in \B$, $r \cdot x \leq r
    \cdot y$ implies $x \leq y$;
  \item there is a linear preorder embedding of $\B$ in some
    preordered cone.
  \end{enumerate}
\end{lemma}
\begin{proof}
  $1 \limp 4$ is obvious, in light of
  Proposition~\ref{prop:tscope:bary:alg} and
  Theorem~\ref{thm:bary:alg:conify:ord:pointed}, item~1.
  
  $4 \limp 3$.  Let $i \colon \B \to \C$ be a linear preorder
  embedding of $\B$ in some preordered cone $\C$.  Let us also
  assume that $r \cdot x \leq r \cdot y$, where $0 < r < 1$.  Then
  $i (r \cdot x) = r \cdot i (x)$ is less than or equal to
  $i (r \cdot y) = r \cdot i (y)$.  Multiplying by $1/r$ on both
  sides, $i (x) \leq i (y)$.  Since $i$ is a preorder embedding,
  $x \leq y$.

  $3 \limp 2$.  Let $\alpha \in {]0, 1[}$ be arbitrary, and let
  $x, y \in \B$ be such that $\etaca_{\B} (x) \leqca \etaca_{\B} (y)$.
  By Lemma~\ref{lemma:bary:alg:pointed:embed:almost},
  $r \cdot x \leq r \cdot y$ for some $r \in {]0, 1[}$.  Then item~3
  implies $x \leq y$.
  
  Finally, $2 \limp 1$ is clear.  \qed
\end{proof}

\subsection{Semitopological telescopes}
\label{sec:semit-telesc}

The construction $\tscope_\alpha (\B)$ works in the semitopological
case, too: we simply take a colimit in $\Topcat$ instead of $\Setcat$.
\begin{definition}
  \label{defn:tscope:topo}
  For every pointed semitopological barycentric algebra $\B$, we give
  $\tscope_\alpha (\B) = \quot {(\coprod_{n \in \nat} \B)}
  {\equiv_\alpha}$ the quotient topology, where
  $\coprod_{n \in \nat} \B$ has the coproduct topology.
\end{definition}
We will need the following characterization of its open subsets, whose
proof we leave as an exercise.
\begin{lemma}
  \label{lemma:tscope:open}
  Let $\B$ be a pointed semitopological barycentric algebra and
  $\alpha \in {]0, 1[}$.  For short, let
  $q \eqdef [\_]_\alpha \colon \coprod_{n \in \nat} \B \to
  \tscope_\alpha (\B)$ be the quotient map.  The open subsets of
  $\tscope_\alpha (\B)$ are the sets $q [V_\infty]$ with
  $V_\infty \eqdef \coprod_{n \in \nat} V_n = \{(n, x) \mid n \in
  \nat, x \in V_n\}$, where each $V_n$ is open in $\B$ and
  $V_n = (\alpha \cdot \_)^{-1} (V_{n+1})$ for every $n \in \nat$.
  Additionally, $V_\infty = q^{-1} (q [V_\infty])$.
\end{lemma}

\begin{theorem}
  \label{thm:bary:alg:conify:ord:pointed:semitop}
  Let $\alpha \in {]0, 1[}$.  For every pointed semitopological
  barycentric algebra $\B$,
  \begin{enumerate}
  \item $\tscope_\alpha (\B)$, with the quotient topology given in
    Definition~\ref{defn:tscope:topo}, is a semitopological cone;
  \item the map $\etaca_{\B}$ is linear and continuous;
  \item for every continuous map $f \colon \B \to \C$ to a
    semitopological cone $\C$ that commutes with $\alpha \cdot \_$,
    there is a unique continuous map
    $f^{\scope\alpha} \colon \tscope_\alpha (\B) \to \C$ that
    commutes with $\alpha \cdot \_$ and such that
    $f^{\scope\alpha} \circ \etaca_{\B} = f$; if $f$ is affine
    (resp.\ concave, convex, strict, linear, superlinear, sublinear),
    then so is $f^{\scope\alpha}$; if $f$ is $\Iu$-homogeneous, then
    $f^{\scope\alpha}$ is positively homogeneous.
  \item In particular, $\tscope_\alpha (\B)$ is the free
    semitopological cone over the pointed semitopological barycentric
    algebra $\B$, and this is independent of $\alpha$ (up to natural
    isomorphisms of semitopological cones).
  \end{enumerate}
\end{theorem}
\begin{proof}
  1.  We verify that $\tscope_\alpha (\B)$ is a pointed semitopological
  barycentric algebra.  In light of Proposition~\ref{prop:tscope:bary:alg},
  only separate continuity of the barycentric algebra operations has
  to be checked.

  Let us fix $a \in [0, 1]$ and $v \eqdef [(n, y)]_\alpha$ of
  $\tscope_\alpha (\B)$; we verify that $\_ +_a v$ is continuous from
  $\tscope_\alpha (\B)$ to itself.  Let $q$ be the quotient map of
  $\coprod_{n \in \nat} \B$ onto $\tscope_\alpha (\B)$.  By definition
  of topological quotients and of topologies on coproducts, it
  suffices to verify that $f_m \colon x \mapsto [(m, x)]_\alpha +_a v$
  is continuous for every $m \in \nat$.  Let $V$ be an open subset of
  $\tscope_\alpha (\B)$.  For every point $x \in \B$,
  $f_m (x) = [(\max (m, n), \alpha^{\max (m, n) - m} \cdot x]_\alpha
  +_a [(\max (m, n), \alpha^{\max (m, n) - n} \cdot y]_\alpha = [(\max
  (m, n), (\alpha^{\max (m, n) - m} \cdot x) +_a (\alpha^{\max (m, n)
    - n} \cdot y))]_\alpha$, so $x \in f_m^{-1} (V)$ if and only if
  $x \in g^{-1} (q^{-1} (V))$, where $g$ is the map
  $x \mapsto (\max (m, n), (\alpha^{\max (m, n) - m} \cdot x) +_a
  (\alpha^{\max (m, n) - n} \cdot y))$, which is continuous from $\B$
  to $\coprod_{n \in \nat} \B$, since $+_a$ is separately continuous
  and since $\alpha^{\max (m, n) - m} \cdot$ is continuous, by
  Lemma~\ref{lemma:bary:alg:pointed}.

  We now fix $u \eqdef [(n, x)]_\alpha$ and $v \eqdef [(n, y)]_\alpha$
  (with the same $n$), and we verify that $a \mapsto u +_a v$ is
  continuous from $[0, 1]$ to $\tscope_\alpha (\B)$.  Since that is
  the composition of $q$ with the map $a \mapsto (n, x +_a y)$,
  continuity is immediate.

  2.  The map $\etaca_{\B}$ is linear by
  Theorem~\ref{thm:bary:alg:conify:ord:pointed}, item~1.  A
  subtle point here is that linearity includes $\Iu$-homogeneity,
  which is defined with respect to a least element of $\B$ (viz.,
  $\bot$) and a least element of $\tscope_\alpha (\B)$.  Now
  $\etaca_{\B} (\bot) = [(0, \bot)]_\alpha$ is a least element of
  $\tscope_\alpha (\B)$ with respect to $\leqca$, and this is what we
  used in Theorem~\ref{thm:bary:alg:conify:ord:pointed}.  In the
  current setting, we need to check that $[(0, \bot)]_\alpha$ is a
  least element of $\tscope_\alpha (\B)$ in its specialization
  preordering instead.  In other words, we need to check that the only
  open subset $V$ of $\tscope_\alpha (\B)$ that contains
  $[(0, \bot)]_\alpha$ is the whole of $\tscope_\alpha (\B)$.  Using
  the notations of Lemma~\ref{lemma:tscope:open}, we must have
  $V = q [V_\infty]$ and
  $V_\infty = \coprod_{n \in \nat} V_n = \{(n, x) \mid n \in \nat, x
  \in V_n\}$, where $V_n = (\alpha \cdot \_)^{-1} (V_{n+1})$ for every
  $n \in \nat$.  Then $(0, \bot)$ is in $V_\infty$, so $\bot \in V_0$,
  and this entails that $V_0 = \B$.  Since
  $\bot \in V_0 = (\alpha \cdot \_)^{-1} (V_1)$,
  $\alpha \cdot \bot = \bot$ is in $V_1$, so $V_1 = \B$ as well.  By
  an easy induction on $n \in \nat$, we obtain that $\bot \in V_n$,
  hence $V_n = \B$, for every $n$.  It follows that
  $V = \tscope_\alpha (\B)$.

  Next, $\etaca_{\B}$ is the composition of $x \mapsto (0, x)$ from
  $\B$ to $\coprod_{n \in \nat} \B$ and of the quotient map onto
  $\tscope_\alpha (\B)$, so it is continuous.

  3.  Given any continuous map $f \colon \B \to \C$ that commutes with
  $\alpha \cdot \_$, $f$ is monotonic, so there is a unique map
  $f^{\scope\alpha} \colon \tscope_\alpha (\B) \to \C$ that commutes
  with $\alpha \cdot \_$ and such that
  $f^{\scope\alpha} \circ \etaca_{\B} = f$, and it is monotonic by
  Theorem~\ref{thm:bary:alg:conify:ord:pointed}, item~2.
  Additionally, it is affine (resp.\ concave, convex, strict, linear,
  superlinear, sublinear) if $f$ is, by the same theorem, items~2
  and~3.  It remains to show that $f^{\scope\alpha}$ is continuous.
  By definition of topological quotients and topologies on coproducts,
  it suffices to show that the map
  $x \mapsto f^{\scope\alpha} ([(n, x)]_\alpha)$ is continuous from
  $\B$ to $\C$, for every $n \in \nat$.  For every $x \in \B$,
  $f^{\scope\alpha} ([(n, x)]_\alpha) = (1/\alpha)^n \cdot f (x)$, and
  the result then follows from the continuity of $f$ and of scalar
  multiplication.

  4.  It follows immediately that $\tscope_\alpha (\B)$, with the
  quotient topology, is the free semitopological cone over the pointed
  semitopological barycentric algebra $\B$.  Since free objects are
  unique up to natural isomorphism, the construction is independent of
  $\alpha$ up to natural isomorphism of semitopological cones.  \qed
\end{proof}

\begin{remark}
  \label{rem:tscope:top:spec}
  There is no reason to believe that the specialization preordering of
  $\tscope_\alpha (\B)$, with the topology of
  Definition~\ref{defn:tscope:topo} would be the preordering $\leqca$
  defined in Proposition~\ref{prop:tscope:bary:alg}, in general.  We
  will see in Lemma~\ref{lemma:bary:alg:spec:pointed} when the two
  preorderings coincide.
\end{remark}

\begin{problem}
  \label{pb:tscope:topo}
  Let $\B$ be a pointed topological barycentric algebra and
  $\alpha \in {]0, 1[}$.  Is $\tscope_\alpha (\B)$ necessarily a
  topological cone?
\end{problem}

\begin{definition}[Strictly embeddable]
  \label{defn:bary:alg:embed:strict}
  A \emph{strict embedding} of a pointed semitopological barycentric
  algebra $\B$ in a semitopological cone $\C$ is a linear topological
  embedding of $\B$ in $\C$.  $\B$ is \emph{strictly embeddable} if
  and only if there is a strict embedding of $\B$ in some
  semitopological cone $\C$.
\end{definition}

\begin{lemma}
  \label{lemma:bary:alg:pointed:embed:nec}
  Let $\B$ be a pointed $T_0$ semitopological barycentric algebra.  If
  $\B$ is strictly embeddable, then for every $a \in {]0, 1]}$, for
  every $\Iu$-homogeneous lower semicontinuous map
  $h \colon \B \to \creal$, there is an $\Iu$-homogeneous lower
  semicontinuous map $h' \colon \B \to \creal$ such that
  $h' (a \cdot x) = h (x)$ for every $x \in \B$.  If $h$ is linear
  (resp., superlinear, sublinear), then so is $h'$.
\end{lemma}
\begin{proof}
  We cannot just define $h'$ by $h' (x) \eqdef h (\frac 1 a \cdot x)$,
  since only scalar multiplication by numbers in $[0, 1]$ is permitted
  in pointed barycentric algebras. 

  Let us fix $\alpha \in {]0, 1[}$.  Given $h$ and $a$ as above, let
  $h' (x) \eqdef h^{\scope\alpha} (\frac 1 a \cdot \etaca_{\B} (x))$.
  Multiplying by $\frac 1 a$ is permitted, because we perform it in
  the cone $\tscope_\alpha (\B)$.  By
  Theorem~\ref{thm:bary:alg:conify:ord:pointed:semitop}, $h'$ is an
  $\Iu$-homogeneous map from $\B$ to $\creal$, which is linear (resp.\
  superlinear, sublinear) if $h$ is, owing notably to the fact that
  $\etaca_{\B}$ is linear.  For every $x \in \B$,
  $h' (a \cdot x) = h^{\scope\alpha} (\frac 1 a \cdot \etaca_{\B} (a
  \cdot x)) = h^{\scope\alpha} (\frac 1 a \cdot a \cdot \etaca_{\B}
  (x))$ (since $\etaca_{\B}$ is linear)
  $= h^{\scope\alpha} (\etaca_{\B} (x)) = h (x)$.
\end{proof}

\begin{proposition}
  \label{prop:bary:alg:embed:strict}
  The following are equivalent for a pointed $T_0$ semitopological
  barycentric algebra $\B$:
  \begin{enumerate}
  \item $\B$ is strictly embeddable;
  \item $\etaca_{\B}$ is a topological embedding for every $\alpha \in
    {]0, 1[}$;
  \item $\etaca_{\B}$ is a topological embedding for some $\alpha \in
    {]0, 1[}$;
  \item the proper open subsets of $\B$ are exactly the sets
    $h^{-1} (]1, \infty])$ where $h$ ranges over the $\Iu$-homogeneous
    lower semicontinuous maps from $\B$ to $\creal$;
  \item the sets $h^{-1} (]1, \infty])$ where $h$ ranges over the
    $\Iu$-homogeneous lower semicontinuous maps from $\B$ to $\creal$
    form a subbase of its topology;
  \item the map $\alpha \cdot \_ \colon \B \to \B$ is full, for every
    $\alpha \in {]0, 1[}$;
  \item the map $\alpha \cdot \_ \colon \B \to \B$ is full, for some
    $\alpha \in {]0, 1[}$.
  \end{enumerate}
\end{proposition}
\begin{proof}
  $1 \limp 2$.  Let $i \colon \B \to \C$ be a strict embedding of $\B$
  in a semitopological cone $\C$.  In particular, $i$ is linear.  We
  fix an arbitrary $\alpha \in {]0, 1[}$.  For every open subset $U$
  of $\B$, there is an open subset $V$ of $\C$ such that
  $U = i^{-1} (V)$.  By
  Theorem~\ref{thm:bary:alg:conify:ord:pointed:semitop}, item~3,
  $i^{\scope\alpha} \colon \tscope_\alpha (\B) \to \C$ is linear,
  continuous and $i^{\scope\alpha} \circ \etaca_{\B} = i$.  In
  particular, since $i$ is injective, so is $\etaca_{\B}$.  Also,
  $U = (i^{\scope\alpha} \circ \etaca_{\B})^{-1} (V) =
  (\etaca_{\B})^{-1} \allowbreak ((i^{\scope\alpha})^{-1} (V))$,
  showing that $U = (\etaca_{\B})^{-1} (V')$ where $V'$ is the open
  set $(i^{\scope\alpha})^{-1} (V)$.  Hence $\etaca_{\B}$ is full.
  Since $\etaca_{\B}$ is continuous by
  Theorem~\ref{thm:bary:alg:conify:ord:pointed:semitop}, item~2,
  $\etaca_{\B}$ is a topological embedding.

  $2 \limp 3$ is obvious.
  
  $3 \limp 4$.  We fix $\alpha \in {]0, 1[}$ so that $\etaca_{\B}$ is
  a topological embedding, thanks to item~3.  The open subsets of $\B$
  are exactly the sets of the form $U \eqdef {(\etaca_{\B})}^{-1} (V)$
  where $V$ is open in $\tscope_\alpha (\B)$.  If $\bot \in U$, then
  $[(0, \bot)]_\alpha = \etaca_{\B} (\bot)$ is in $V$.  Conversely, if
  $[(0, \bot)]_\alpha \in V$, and since $[(0, \bot)]_\alpha $ is the
  zero element, hence is least in $\tscope_\alpha (\B)$, we must have
  $V = \tscope_\alpha (\B)$, so $U = \B$, whence $\bot \in U$.  Hence
  $\bot \in U$ if and only if $U$ is non-proper if and only if
  $[(0, \bot)]_\alpha \in V$ if and only if $V$ is non-proper.  It
  follows that the proper open subsets $U$ of $\B$ are exactly the
  sets ${(\etaca_{\B})}^{-1} (V)$ where $V$ is a proper open subset of
  $\tscope_\alpha (\B)$.

  If so, then $V = (M^V)^{-1} (]1, \infty])$, where $M^V$ is the upper
  Minkowski function of $V$ (see Section~\ref{sec:preliminaries}), and
  $M^V$ is positively homogeneous and lower semicontinuous.  Since
  $\etaca_{\B}$ is itself linear and continuous by
  Theorem~\ref{thm:bary:alg:conify:ord:pointed:semitop}, item~2,
  $h \eqdef M^V \circ \etaca_{\B}$ is $\Iu$-homogeneous and lower
  semicontinuous.  It follows that $U = h^{-1} (]1, \infty])$.

  Conversely, every set of the form $h^{-1} (]1, \infty])$ with $h$
  $\Iu$-homogeneous and lower semicontinuous is proper (since
  $h (\bot) = 0 \not\in {]1, \infty]}$) and open.

  $4 \limp 5$ is clear.

  $5 \limp 2 \limp 1$.  We assume that the sets $h^{-1} (]1, \infty])$
  where $h$ ranges over the $\Iu$-homogeneous lower semicontinuous
  maps from $\B$ to $\creal$ form a subbase of the topology on $\B$.
  Let $\alpha \in {]0, 1[}$ be arbitrary.  Since
  $h^{\scope\alpha} \circ \etaca_{\B} = h$, every subbasic open set
  $h^{-1} (]1, \infty])$ is equal to
  ${(\etaca_{\B})}^{-1} ((h^{\scope\alpha})^{-1} (]1, \infty]))$.
  Hence $\etaca_{\B}$ is full.  Since $\etaca_{\B}$ is continuous and
  since $\B$ is $T_0$, $\etaca_{\B}$ is a topological embedding.  This
  shows item~2, and then item~1, since $\tscope_\alpha (\B)$ is a
  semitopological cone, by
  Theorem~\ref{thm:bary:alg:conify:ord:pointed:semitop}, item~1.

  $1 \limp 6$.  The map $\alpha \cdot \_$ is linear and continuous,
  and since $\B$ is $T_0$, we only have to show that it is full,
  namely that every open subset $U$ of $\B$ can be written as
  $(\alpha \cdot \_)^{-1} (V)$ for some open subset $V$ of $\B$.  If
  $U = \B$, then we just take $V \eqdef \B$.  Otherwise, by
  Proposition~\ref{prop:bary:alg:embed:strict},
  $U = h^{-1} (]1, \infty])$ for some $\Iu$-homogeneous lower
  semicontinuous map from $\B$ to $\creal$.  By
  Lemma~\ref{lemma:bary:alg:pointed:embed:nec}, there is an
  $\Iu$-homogeneous lower semicontinuous map $h' \colon \B \to \creal$
  such that $h' (\alpha \cdot x) = h (x)$ for every $x \in \B$.  Let
  $V \eqdef {h'}^{-1} (]1, \infty])$.  Then
  $(\alpha \cdot \_)^{-1} (V) = \{x \in \B \mid h' (\alpha \cdot x) >
  1\} = h^{-1} (]1, \infty]) = U$.

  $6 \limp 7$ is clear.
  
  $7 \limp 3$.  We consider an arbitrary open subset $U$ of $\B$, and
  we wish to find an open subset of $\tscope_\alpha (\B)$ such that
  $(\etaca_{\B})^{-1} (V) = U$.  We let $U_0 \eqdef U$, then, by
  induction on $n$, we pick an open subset $U_{n+1}$ of $\B$ such that
  $U_n = (\alpha \cdot \_)^{-1} (U_{n+1})$, using the fact that
  $\alpha \cdot \_$ is full.  Then we define $U_\infty$ as
  $\coprod_{n \in \nat} U_n = \{(n, x) \mid n \in \nat, x \in U_n\}$,
  an open subset of $\coprod_{n \in \nat} \B$.  By
  Lemma~\ref{lemma:tscope:open}, the image $q [U_\infty]$ of
  $U_\infty$ under the quotient map $q \eqdef [\_]_\alpha$ is open in
  $\tscope_\alpha (\B)$.

  

  
  Then
  $(\etaca_{\B})^{-1} (q [U_\infty]) = \{x \in \B \mid [(0, x)]_\alpha
  \in q [U_\infty]\} = \{x \in \B \mid (0, x) \in q^{-1} (q
  [U_\infty])\} = \{x \in \B \mid (0, x) \in U_\infty\} = U_0 = U$.
  Hence $\etaca_{\B}$ is full.  Since $\B$ is $T_0$ and $\etaca_{\B}$
  is continuous
  (Theorem~\ref{thm:bary:alg:conify:ord:pointed:semitop}, item~2),
  $\etaca_{\B}$ is a topological embedding.  \qed
\end{proof}

The following is an example of a non-strictly embeddable
semitopological barycentric algebra.  It is embeddable, by a result we
will see later on, Corollary~\ref{corl:bary:alg:pointed:embed}.
\begin{example}
  \label{ex:pBA:nonstrictembed}
  Let $I \eqdef [0, 1]$, with its usual metric topology.  We give
  $\nat$ the discrete topology, and we form the product cone
  $\C \eqdef \Lform I \times \Lform \nat$, with the algebraic
  operations defined componentwise.  For every $n \in \nat$, let
  $e_n \in \Lform \nat$ map $n$ to $1$ and all other numbers to $0$.
  Let also $\one$ be the constant function with value $1$, and $0$ be
  the zero function.  For each $a \in I$, let $V_a \eqdef {]1-a, 1]}$
  and $f_a \eqdef \chi_{V_a}$, the characteristic function of $V_a$.
  Let $\B^{\text{nse}}$ be the convex hull of
  $\{(0, 0), (\one, 0)\} \cup \{(f_a, e_n) \mid a \in {]0, 1]}, n \in
  \nat\}$.  Its elements are of the form
  $\alpha.(\one, 0) + \sum_{i=1}^m \alpha_i \cdot (f_{a_i}, e_{n_i})$
  where $\alpha, \alpha_i \in \Rp$ and
  $\alpha + \sum_{i=1}^n \alpha_1 \leq 1$.  $\B^{\text{nse}}$ is a
  pointed barycentric algebra.

  We define two topologies on $\B^{\text{nse}}$, $\tau_1$ and
  $\tau_2$.  $\Lform I$ and $\Lform \nat$ are semitopological cones in
  the Scott topology; in fact they are even topological since $I$ and
  $X$ are locally compact (see Example~\ref{exa:dcone:LX}).  We give
  $\C$ the product topology, and this turns $\C$ into a topological
  cone.  The topology $\tau_1$ is the subspace topology on
  $\B^{\text{nse}}$.  We call \emph{$\tau_1$-closed} its closed
  subsets.  We call a subset $C$ of $\B^{\text{nse}}$ \emph{closed} if
  and only if it is $\tau_1$-closed and for all $a, b, c$ in $[0, 1]$
  such that $b+c \leq 1$, for every $(f, e) \in \B^{\text{nse}}$, such
  that $ba.(f_a, e_n) + c.(f, e) \in C$ for infinitely many values of
  $n \in \nat$, $ba.(\one, 0) + c.(f, e)$ is in $C$.  This
  construction forces the sequence $a. (f_a, e_n)$ to converge to
  $a.(\one, 0)$, while being compatible with the algebraic operations.
  The closed sets defined above form the closed sets of a topology
  $\tau_2$, and $\B^{\text{nse}}$ is a $T_0$ pointed semitopological
  barycentric algebra with the topology $\tau_2$.  For every
  $\alpha \in {]0, 1[}$, for every $b \in {[\max (a, \alpha), 1[}$, it
  can be shown that the set
  $C^b \eqdef \{(f, e) \in \B^{\text{nse}} \mid b.f \leq a.f_a \text{
    and } b.e \leq a.e_n\} \cup \dc (\one, 0)$ is $\tau_2$-closed.  If
  $C^b$ were the inverse image of a $\tau_2$-closed set $C'$ under
  $\alpha \cdot \_$, then since $(a/b) . (f_a, e_n)$ is in $C^b$ for
  every $n \in \nat$, $(\alpha a/b) . (f_a, e_n)$ would be in $C'$,
  and since $C'$ is $\tau_2$-closed, $(\alpha a/b) . (\one, 0)$ would
  also have to be in $C'$, forcing $(a/b) . (\one, 0)$ to be in $C^b$;
  but this is not the case.  Hence $\alpha \cdot \_$ is not full, and
  therefore $\B^{\text{nse}}$ is not strictly embeddable by
  Proposition~\ref{prop:bary:alg:embed:strict}.  The details are given
  in Appendix~\ref{sec:non-strictly-embedd}.
\end{example}

\begin{lemma}
  \label{lemma:bary:alg:spec:pointed}
  Let $\B$ be a pointed semitopological barycentric algebra and
  $\alpha \in {]0, 1[}$.  We write $\preceq$ for the specialization
  preordering of $\tscope_\alpha (\B)$, to distinguish it from the
  preordering $\leqca$ of Proposition~\ref{prop:tscope:bary:alg}.
  Then, for all $u, v \in \tscope_\alpha (\B)$, if $u \leqca v$ then
  $u \preceq v$, and the converse implication holds if $\B$ is
  strictly embeddable and $T_0$.
\end{lemma}
\begin{proof}
  Let $u \leqca v$.  We can write $u$ as $[(n, x)]_\alpha$ and $v$
  as $[(n, y)]_\alpha$ so that $x \leq y$.  Every open neighborhood of
  $[(n, x)]_\alpha$ in $\tscope_\alpha (\B)$ is of the form
  $q [V_\infty]$ as given in Lemma~\ref{lemma:tscope:open}.  Using the
  same notations as there, $[(n, x)]_\alpha = q (n, x)$, so
  $(n, x) \in q^{-1} (q [V_\infty]) = V_\infty$, and hence
  $x \in V_n$.  Since $x \leq y$, $y$ is in $V_n$, so
  $(n, y) \in V_\infty$ and therefore
  $[(n, y)]_\alpha \in q [V_\infty]$.  It follows that
  $[(n, x)]_\alpha \preceq [(n, y)]_\alpha$.

  We now assume that $\B$ is strictly embeddable and $T_0$.  Let
  $u, v \in \tscope_\alpha (\B)$ and let us assume that
  $u \not\leqca v$.  Let us write $u$ as $[(n, x)]_\alpha$ and $v$ as
  $[(n, y)]_\alpha$ for some common $n \in \nat$.  Whichever $n$ we
  take, and whichever $x$ and $y$, we must have $x \not\leq y$, so
  there must be an open neighborhood $U$ of $x$ in $\B$ that does not
  contain $y$.  By Proposition~\ref{prop:bary:alg:embed:strict},
  $\etaca_{\B}$ is a topological embedding, so there is an open subset
  $V_1$ of $\tscope_\alpha (\B)$ such that
  $U = {(\etaca_{\B})}^{-1} (V_1)$.  Let
  $V \eqdef {(\alpha^n \cdot \_)}^{-1} (V_1)$.

  We verify that $u \in V$.  It suffices to show that
  $\alpha^n \cdot u \in V_1$.  But
  $\alpha^n \cdot u = \alpha^n \cdot [(n, x)]_\alpha = [(n, \alpha^n
  \cdot x)]_\alpha$ (by Proposition~\ref{prop:tscope:bary:alg}, since
  $a \eqdef \alpha^n$ is in $[0, 1]$)
  $= [(n-1, \alpha^{n-1} \cdot x)]_\alpha = \cdots = [(0, x)]_\alpha =
  \etaca_{\B} (x)$.  Hence we have to show that
  $\etaca_{\B} (x) \in V_1$, namely that
  $x \in (\etaca_{\B})^{-1} (V_1) = U$; and indeed $x \in U$.
  Similarly, $\alpha^n \cdot v = [(0, y)]_\alpha = \etaca_{\B} (y)$,
  and since $y \not\in U = (\etaca_{\B})^{-1} (V_1)$,
  $\alpha^n \cdot v \not\in V_1$, so $v \not\in V$.  \qed
\end{proof}

\subsection{The free semitopological cone over $\Val_{\leq 1} X$}
\label{sec:free-semit-cone-1}

\begin{theorem}
  \label{thm:Vleq1->Vb}
  For every topological space $X$, $\Val_b X$ is the free cone over
  the pointed barycentric algebra $\Val_{\leq 1} X$ (with
  distinguished element the zero valuation), and $\Val_b X$ is the
  free semitopological cone over the pointed semitopological
  barycentric algebra $\Val_{\leq 1} X$.

  As a consequence, $\Val_b X$ and $\tscope_\alpha (\Val_{\leq 1} X)$
  are naturally isomorphic semitopological cones for every
  $\alpha \in {]0, 1[}$.
\end{theorem}
\begin{proof}
  Let $f$ be a linear map from $\Val_{\leq 1} X$ to a cone $\C$.  We
  must show that it has a unique linear extension $\hat f$ from
  $\Val_b X$ to $\C$.  Let $f_1$ be the restriction of $f$ to
  $\Val_1 X$: this is an affine map from $\Val_1 X$ to $\C$, and any
  linear extension of $f$ to $\Val_b X$ must also be one of $f_1$.
  But such linear extensions are unique, by
  Theorem~\ref{thm:V1->Vb}.

  As for existence, we define $\hat f$ as the extension $\hat f_1$
  obtained from $f_1$ by using Theorem~\ref{thm:V1->Vb}.  By this
  theorem, $\hat f$ is linear.  Also, $\hat f$ maps the zero valuation
  to $0$ and every non-zero element $\nu \in \Val_b X$ to
  $\nu (X) \cdot f (\frac 1 {\nu (X)} \cdot \nu)$; when
  $\nu \in \Val_{\leq 1} X$, the latter is equal to $f (\nu)$ since
  $f$ is $\Iu$-homogeneous.  Therefore $\hat f$ extends $f$, not just
  $f_1$.

  When $\C$ is a semitopological cone, the various spaces of
  continuous valuations are given their weak topologies, and $f$ is
  continuous, then $f_1$ is continuous, too, since $\Val_1 X$ is a
  topological subspace of $\Val_{\leq 1} X$, and then
  $\hat f = \hat f_1$ is continuous by Theorem~\ref{thm:V1->Vb}.  \qed
\end{proof}

\subsection{The free pointed barycentric algebra over a barycentric
  algebra}
\label{sec:free-point-baryc}

There is a free \emph{pointed} barycentric algebra over any
barycentric algebra, and it is very easily constructed.
\begin{definition}
  \label{defn:bary:alg:pointed:free}
  For every barycentric algebra $\B$, let $\conify_{\leq 1} (\B)$ be
  the subset of $\conify (\B)$ consisting of $0$ and the pairs
  $(r, x)$ with $0 < r \leq 1$; we take $0$ as distinguished element.
  If $\B$ is a preordered barycentric algebra, $\conify_{\leq 1} (\B)$
  is preordered by the restriction of the ordering $\leqc$ on
  $\conify (\B)$ (see Definition~\ref{defn:bary:alg:conify:ord}),
  which makes $0$ least.  If $\B$ is a semitopological barycentric
  algebra, $\conify_{\leq 1} (\B)$ is given the subspace topology
  induced by the inclusion in $\conify (\B)$ with the cone topology
  (see Definition~\ref{defn:cone:top}); this makes $0$ least in the
  specialization preordering.
\end{definition}

We have already seen the following notion in
Example~\ref{exa:bary:alg:notembed}.
\begin{definition}
  \label{defn:level}
  Let $\B$ be a barycentric algebra.  The \emph{level} $\ell (u)$ of
  an element $u \in \conify (\B)$ is defined by $\ell (0) \eqdef 0$,
  $\ell (r, x) \eqdef r$.
\end{definition}
Hence $\conify_{\leq 1} (\B)$ is the subset of $\conify (\B)$
consisting of elements of level at most $1$.  The following is
immediate from the definition of the algebraic operations on
$\conify (\B)$ (see Section~\ref{sec:plain-baryc-algebr}).
\begin{fact}
  \label{fact:level:lin}
  The level map $\ell$ is linear: $\ell (u + v) = \ell (u) + \ell
  (v)$, $\ell (0) = 0$, $\ell (a \cdot u) = a \, \ell (u)$.
\end{fact}

We extend the notion of semi-concave map so that the codomain is a
pointed barycentric algebra, not necessarily a preordered cone, in the obvious
way.
\begin{definition}
  \label{defn:qaff:pointed}
  A \emph{semi-concave} map from a semitopological barycentric algebra
  $\B$ to a preordered pointed barycentric algebra $\Alg$ is a
  function $h \colon \B \to \Alg$ such that for all $x, y \in \B$ and
  $a \in [0, 1]$, $h (x +_a y) \geq a \cdot h (x)$.
\end{definition}

\begin{theorem}
  \label{thm:conify:leq1}
  For every barycentric algebra $\B$ (resp., preordered barycentric
  algebra),
  \begin{enumerate}
  \item $\conify_{\leq 1} (\B)$ is a pointed barycentric algebra
    (resp., a preordered pointed barycentric algebra);
  \item the map $\etac_{\B}$ is an injective affine map from $\B$
    (resp., an injective affine preorder embedding) into
    $\conify_{\leq 1} (\B)$;
  \item for every map (resp., monotonic map) $f \colon \B \to \Alg$ to
    a pointed barycentric algebra $\Alg$, there is a unique
    $\Iu$-homogeneous map
    $f^{\cext, \leq 1} \colon \conify_{\leq 1} (\B) \to \Alg$ such
    that $f^{\cext, \leq 1} \circ \etac_{\B} = f$; if $f$ is affine
    (resp.\ concave, convex) then $f^{\cext, \leq 1}$ is linear
    (resp.\ superlinear, sublinear).
  \item In particular, $\conify_{\leq 1} (\B)$ is the free pointed
    barycentric algebra (resp., the free preordered pointed
    barycentric algebra) over the barycentric algebra (resp., the
    preordered barycentric algebra) $\B$.
  \end{enumerate}
\end{theorem}
\begin{proof}
  1.  For every $a \in [0, 1]$, for all $u, v \in \conify_{\leq 1}
  (\B)$, the levels $\ell (u)$ and $\ell (v)$ are less than or equal
  to $1$.  Using the linearity of $\ell$, $\ell (u +_a v) = a \, \ell
  (u) + (1-a) \, \ell (v) \leq 1$, so $u +_ a \in \conify_{\leq 1}
  (\B)$.  Hence $\conify_{\leq 1} (\B)$ is convex.  By
  Lemma~\ref{lemma:bary:in:cone}, $\conify_{\leq 1} (\B)$ is a (resp.\
  preordered) barycentric algebra.

  When $\B$ is preordered, the distinguished element $0$ is least, so
  $\conify_{\leq 1} (\B)$ is a preordered pointed barycentric algebra.

  2.  For every $x \in \B$, $\etac_{\B} (x) = (1, x)$ is in
  $\conify_{\leq 1} (\B)$.  The map $\etac_{\B}$ is injective and
  affine.  If $\B$ is preordered, then $\etac_{\B}$ is also a preorder
  embedding by Proposition~\ref{prop:bary:alg:conify:ord}, item~2.

  3.  Let $f \colon \B \to \Alg$ be semi-concave (resp., and monotonic).  If
  $f^{\cext, \leq 1}$ exists, then it must map $0$ to $\bot$, since
  every $\Iu$-homogeneous map is strict (see
  Proposition~\ref{prop:strict:aff}).  For every $r \in {]0, 1]}$, for every
  $x \in \B$, we must also have
  $f^{\cext, \leq 1} (r, x) = f^{\cext, \leq 1} (r \cdot (1, x)) = r
  \cdot f (x)$ since $f^{\cext, \leq 1}$ is $\Iu$-homogeneous and
  $f^{\cext, \leq 1} \circ \etac_{\B} = \identity {\B}$.  This shows
  uniqueness.

  As for existence, we simply define $f^{\cext, \leq 1}$ just like
  $f^\cext$: it maps $0$ to the least element $\bot$ of $\Alg$ and
  $(r, x)$ to $r \cdot f (x)$ for all $r \in {]0, 1]}$ and $x \in \B$.
  Let us check that $f^{\cext, \leq 1}$ is $\Iu$-homogeneous.  For
  every $a \in [0, 1]$, either $a=0$ and
  $f^{\cext, \leq 1} (a \cdot (r, x)) = f^{\cext, \leq 1} (0) = \bot$,
  or $a > 0$ and
  $f^{\cext, \leq 1} (a \cdot (r, x)) = f^{\cext, \leq 1} (ar, x) = ar
  \cdot f (x) = a \cdot (r \cdot f (x))$ (by the second equality of
  Lemma~\ref{lemma:bary:alg:pointed})
  $= a \cdot f^{\cext, \leq 1} (r, x)$.

  If $f$ is concave, we verify that $f^{\cext, \leq 1}$ is concave,
  too, hence superlinear.  Let $a \in [0, 1]$.  The inequality
  $f^{\cext, \leq 1} (u +_a v) \geq f^{\cext, \leq 1} (u) +_a
  f^{\cext, \leq 1} (v)$ is obvious (and is an equality) if $a=0$ or
  if $a=1$, so we assume that $0 < a < 1$.  If $u=0$, the inequality
  is clear (and is also an equality), too, and similarly if $v=0$.
  Hence let us write $u$ as $(r, x)$ and $v$ as $(s, y)$, with
  $0 < r, s \leq 1$ and $x, y \in \B$.  Then:
  \begin{align*}
    f^{\cext, \leq 1} (u +_a v)
    & = f^{\cext, \leq 1} (ar+(1-a)s, x +_{\frac {ar} {ar+(1-a)s}} y)
    \\
    & = (ar+(1-a) s) \cdot f (x +_{\frac {ar} {ar+(1-a)s}} y) \\
    & \geq (ar+(1-a) s) \cdot (f (x) +_{\frac {ar} {ar+(1-a)s}} f (y))
  \end{align*}
  since $f$ is concave.  We claim that the latter is equal to
  $f^{\cext, \leq 1} (u) +_a f^{\cext, \leq 1} (v)$.  We have:
  \begin{align*}
    f^{\cext, \leq 1} (u) +_a f^{\cext, \leq 1} (v)
    & = (r \cdot f (x)) +_a (s \cdot f (y))
  \end{align*}
  We recall the formula $c \cdot z = z +_c \bot$ defining $\cdot$ in
  $\Alg$.  In order to reach our goal, we embed $\Alg$ in a cone
  $\C \eqdef \conify (\Alg)$ through $\etac_{\Alg}$, using
  Theorem~\ref{thm:conify:semitop}.  We write $cz$ for the scalar
  multiplication of $c$ by $z$ in $\conify (\Alg)$, in order not to
  confuse it with $c \cdot z$.  Letting
  $x' \eqdef \etac_{\Alg} (f (x))$ and
  $y' \eqdef \etac_{\Alg} (f (y))$, it remains to verify that
  $(x' +_{\frac {ar} {ar+(1-a)s}} y') +_{ar+(1-a)s} \bot = (x' +_r
  \bot) +_a (y' +_s \bot)$:
  \newcommand\conifyeqone{(x' +_{\frac {ar} {ar+(1-a)s}} y')
    +_{ar+(1-a)s} \bot}
  \newcommand\conifyeqtwo{(x' +_r \bot) +_a (y' +_s \bot)}
  \ifta
  \begin{align*}
    \conifyeqone
    & = (ar+(1-a)s) (\frac {ar} {ar+(1-a)s} x' + \frac {(1-a)s}
      {ar+(1-a)s} y') \\
    & \qquad + (1-(ar+(1-a)s)) \bot \\
    & = ar x' + (1-a)s y' + (1-(ar+(1-a)s)) \bot \\
  \end{align*}
  \else
  \begin{align*}
    & \conifyeqone \\
    & = (ar+(1-a)s) (\frac {ar} {ar+(1-a)s} x' + \frac {(1-a)s}
      {ar+(1-a)s} y') \\
    & \qquad + (1-(ar+(1-a)s)) \bot \\
    & = ar x' + (1-a)s y' + (1-(ar+(1-a)s)) \bot \\
  \end{align*}
  \fi
  while:
  \ifta
  \begin{align*}
    \conifyeqtwo
    & = a (rx' + (1-r)\bot) + (1-a) (s y' + (1-s) \bot) \\
    & = ar x' + (1-a) s y' + (a (1-r) + (1-a) (1-s)) \bot,
  \end{align*}
  \else
  \begin{align*}
    & \conifyeqtwo \\
    & = a (rx' + (1-r)\bot) + (1-a) (s y' + (1-s) \bot) \\
    & = ar x' + (1-a) s y' + (a (1-r) + (1-a) (1-s)) \bot,
  \end{align*}
  \fi
  and those are indeed equal since
  $a (1-r) + (1-a) (1-s) = 1 - (ar+(1-a)s)$, as one easily checks.
  The fact that $f^{\cext, \leq 1}$ is sublinear if $f$ is convex, and
  linear if $f$ is affine is proved similarly, replacing $\geq$ by
  $\leq$, or by $=$.

  In case $\B$ is preordered, we show that $f^{\cext, \leq 1}$ is
  monotonic, assuming that $f$ is.  Let $u \leq v$ in
  $\conify_{\leq 1} (\B)$.  If $u=0$, then
  $f^{\cext, \leq 1} (u) = \bot \leq f^{\cext, \leq 1} (v)$.  Hence we
  may assume that $u$ is of the form $(r, x)$ with $0 < r \leq 1$ and
  $x \in \B$, $v$ is of the form $(s, y)$ with $0 < s \leq 1$ and
  $y \in \B$, $u+u' = (s, y_1)$ for some $u' \in \conify (\B)$ (not
  just $\conify_{\leq 1} (\B)$, although a quick check shows that
  necessarily such a $u'$ must in fact lie in $\conify_{\leq 1} (\B)$)
  and $y_1 \leq y$ in $\B$.  Then
  $f^{\cext, \leq 1} (u) = r \cdot f (x)$.  If $u'=0$, then
  $(s, y_1) = (r, x)$, so $x = y_1 \leq y$ and
  $r \cdot f (x) \leq r \cdot f (y) = s \cdot f (y)$ (since $r=s$)
  $= f^{\cext, \leq 1} (v)$.  Otherwise, $u'$ is of the form
  $(r', x')$ with $r' > 0$ and $x' \in \B$, and
  $(s, y_1) = (r+r', x +_{\frac r {r+r'}} x')$.  Then
  $f^{\cext, \leq 1} (v) = s \cdot f (y) \geq s \cdot f (y_1)$ (since
  $f$ is monotonic)
  $= (r+r') \cdot f (x +_{\frac r {r+r'}} x') \geq (r+r') \cdot \frac
  r {r+r'} \cdot f (x)$ (since $f$ is semi-concave) $= r \cdot f (x)$
  (by the second equation of Lemma~\ref{lemma:bary:alg:pointed})
  $= f^{\cext, \leq 1} (u)$.

  4.  Trivial consequence of item~3.  \qed
\end{proof}

We have a similar result with pointed semitopological barycentric
algebras, except that $\etac_{\B}$ is not necessarily a topological
embedding; we will see that $\B^-$
(Example~\ref{exa:bary:alg:notembed}) is a counterexample in
Example~\ref{exa:bary:alg:pointed:notembed}.
\begin{theorem}
  \label{thm:conify:leq1:semitop}
  For every semitopological barycentric algebra $\B$,
  \begin{enumerate}
  \item $\conify_{\leq 1} (\B)$ is a pointed semitopological
    barycentric algebra;
  \item its proper open subsets are the upwards-closed subsets of
    ${]0, 1]}_\sigma \times \B$, the latter being given the product
    topology, and where upwards-closed is with respect to the
    preordering $\leqc$ on $\conify (\B)$, restricted to
    $\conify_{\leq 1} (\B)$;
  \item the map $\etac_{\B}$ is an injective affine continuous map
    from $\B$ into $\conify_{\leq 1} (\B)$;
  \item for every semi-concave (resp., concave, affine) continuous map
    $f \colon \B \to \Alg$ to a pointed semitopological barycentric
    algebra $\Alg$, there is a unique $\Iu$-homogeneous (resp.,
    superlinear, linear) continuous map
    $f^{\cext, \leq 1} \colon \conify_{\leq 1} (\B) \to \Alg$ such
    that $f^{\cext, \leq 1} \circ \etac_{\B} = f$.

    In particular, $\conify_{\leq 1} (\B)$ with the cone topology is
    the free pointed semitopological barycentric algebras $\B$.
  \item The following are equivalent:
    \begin{enumerate}[label=(\alph*)]
    \item $\B$ embeds in a pointed semitopological algebra $\Alg$
      through an affine continuous map;
    \item $\etac_{\B} \colon \B \to \conify_{\leq 1} (\B)$ is a
      topological embedding;
    \item $\B$ is embeddable (in the sense of
      Definition~\ref{defn:bary:alg:embed}, namely in a
      semitopological cone).
    \end{enumerate}
  \end{enumerate}
\end{theorem}
\begin{proof}
  1.  Just as in the proof of Theorem~\ref{thm:conify:leq1},
  $\conify_{\leq 1} (\B)$ is a convex subset of $\conify (\B)$ with a
  least element, and we conclude by Lemma~\ref{lemma:bary:in:cone}.

  2.  Let us agree to say ``upwards-closed'' for upwards-closed with
  respect to the preordering on $\conify (\B)$, or its restriction to
  $\conify_{\leq 1} (\B)$---not the componentwise preordering on $(\Rp
  \diff \{0\}) \times \B$.

  The proper open subsets of $\conify (\B)$ are the upwards-closed
  subsets of $(\Rp \diff \{0\}) \times \B$, where the latter is given
  the product topology.  Given such an open set $U$,
  $U \cap \conify_{\leq 1} (\B)$ is upwards-closed in
  $\conify_{\leq 1} (\B)$, and is a union of sets of the form
  $({]t, \infty[} \times U) \cap \B$ where $t > 0$ and $U$ is open in
  $\B$.  The latter are just sets of the form ${]t, 1]} \times U$ if
  $t < 1$ and empty otherwise, and they are all open in
  ${]0, 1]}_\sigma \times \B$.

  Conversely, let $W$ be an upwards-closed open subset of
  ${]0, 1]}_\sigma \times \B$, still in the preordering on
  $\conify_{\leq 1} (\B)$.  We write $W$ as
  $\bigcup_{i \in I} ({]t_i, 1]} \times U_i)$ where $0 < t_i < 1$ and
  $U_i$ is open in $\B$.  Let
  $W' \eqdef \bigcup_{i \in I} ({]t_i, \infty[} \times U_i) \cup ({]1,
    \infty[} \times \B)$.  $W'$ is open in
  $(\Rp \diff \{0\})_\sigma \times \B$ and
  $W' \cap \conify_{\leq 1} (\B) = W$.  Let us verify that $W'$ is
  upwards-closed in $\conify (\B)$.  Let $u$ be any point in $W'$ and
  $v$ be above $u$ in $\conify (\B)$.  Since $u \neq 0$, we write $u$
  as $(r, x)$ with $r > 0$ and $x \in \B$.  Since $u \leqc v$, $v$ must
  be of the form $(s, y)$, and there must be a point
  $u' \in \conify (\B)$ and an element $y_1 \in \B$ such that
  $u+u' = (s, y_1)$ and $y_1 \leq y$.  If $s > 1$, then
  $v \in {]1, \infty[} \times \B \subseteq W'$, so let us assume
  $s \leq 1$.  Then, since $u+u' = (r, x) + u' = (s, y_1)$, in
  particular $r \leq s \leq 1$.  The fact that $u \in W'$ and the
  definition of $W'$ implies that $u \in {]t_i, \infty[} \times U_i$
  for some $i \in I$; therefore $u$ is in $W$.  Since $W$ is
  upwards-closed in $\conify_{\leq 1} (\B)$, $u \leqc v$, and since $s
  \leq 1$ hence $v \in \conify_{\leq 1} (\B)$, $v$ is also in $W$.
  It is clear that $W \subseteq W'$, so $v$ is in $W'$.
    
  3.  For every $x \in \B$, $\etac_{\B} (x) = (1, x)$ is in
  $\conify_{\leq 1} (\B)$.  The remaining claims are from
  Theorem~\ref{thm:conify:semitop}, item~2.

  4.  Let $f \colon \B \to \Alg$ be semi-concave and continuous.
  Uniqueness is by Theorem~\ref{thm:conify:leq1}, item~3, considering
  $\B$ as a preordered cone with its specialization preordering.  We
  turn to existence.  We have to define $f^{\cext, \leq 1}$ as in
  Theorem~\ref{thm:conify:leq1} (whence the same notation): it maps
  $0$ to the least element $\bot$ of $\Alg$ and $(r, x)$ to
  $r \cdot f (x)$ for all $r \in {]0, 1]}$ and $x \in \B$.

  It remains to show that $f^{\cext, \leq 1}$ is continuous, and this
  is done as in the proof of Theorem~\ref{thm:conify:semitop}, item~3.
  As in item~2, ``upwards-closed'' will mean upwards-closed in the
  preordering on $\conify (\B)$ or its restriction to $\conify_{\leq
    1} (\B)$.

  Let $V$ be an open subset of $\Alg$.  If $\bot \in V$, then
  $V = \Alg$ since $\bot$ is least in $\Alg$, and then
  $(f^{\cext, \leq 1})^{-1} (V) = \conify_{\leq 1} (\B)$ is open.
  Hence we assume that $\bot \not\in V$.  Then
  $(f^{\cext, \leq 1})^{-1} (V)$ does not contain $0$, and we show
  that $(f^{\cext, \leq 1})^{-1} (V)$ is open in
  ${]0, 1]}_\sigma \times \B$ and upwards-closed in
  $\conify_{\leq 1} (\B)$, relying on item~2 above.

  For every $u \in (f^{\cext, \leq 1})^{-1} (V)$, for every
  $v \in \conify_{\leq 1} (\B)$ such that $u \leqc v$, we verify that
  $v \in (f^{\cext, \leq 1})^{-1} (V)$.  Since $u \neq 0$, $v$ is also
  different from $0$, and we write $v$ as $(s, y)$ with $0 < s \leq 1$
  and $y \in \B$.  There is a point $u' \in \conify (\B)$ and an
  element $y_1 \in \B$ such that $u+u' = (s, y_1)$ and $y_1 \leq y$.
  (Yes, we do take the sum $u+u'$ in $\conify (\B)$, and we expect it
  to yield $(s, y_1)$, which is in $\conify_{\leq 1} (\B)$.)  Then
  $f^{\cext, \leq 1} (v) = s \cdot f (y) \geq s \cdot f (y_1)$ (since
  $f$, being continuous, is monotonic)
  $= f^{\cext, \leq 1} (u+u') \geq f^{\cext, \leq 1} (u)$ (since $f$
  is monotonic and $u+u' \geq u$ by Remark~\ref{rem:bary:alg:ord}).
  Since $V$ is upwards-closed, $f^{\cext, \leq 1} (v) \in V$.
  Therefore $v \in (f^{\cext, \leq 1})^{-1} (V)$.

  Next, we check that $(f^{\cext, \leq 1})^{-1} (V)$ is open in
  ${]0, 1]}_\sigma \times \B$.  The restriction of $f^{\cext, \leq 1}$
  to ${]0, 1]}_\sigma \times \B$ is the function
  $(r, x) \mapsto r \cdot f (x)$.  This is Scott-continuous in $r$ and
  continuous in $x$, since $f$ is continuous.  We note that $]0, 1]$
  is a continuous dcpo, whose way-below relation is $<$.  By Ershov's
  observation, this restriction of $f^{\cext, \leq 1}$ is jointly
  continuous from ${]0, 1]}_\sigma \times \B$ to $\creal$, so
  $(f^{\cext, \leq 1})^{-1} (V)$ is open in
  ${]0, 1]}_\sigma \times \B$ with the product topology.

  5. $(a) \limp (b)$.  Let $i \colon \B \to \Alg$ be an affine
  topological embedding in a pointed semitopological algebra $\Alg$.
  By item~4, $i^{\cext, \leq 1}$ is a linear continuous map from
  $\conify_{\leq 1} (\B)$ to $\Alg$.  For every open subset $U$ of
  $\B$, there is an open subset $V$ of $\Alg$ such that
  $U = i^{-1} (V)$, since $i$ is a topological embedding.  Using the
  fact that $i = i^{\cext, \leq 1} \circ \etac_{\B}$, $U$ is the
  inverse image of the open set $(i^{\cext, \leq 1})^{-1} (V)$ under
  $\etac_{\B}$.  Hence $\etac_{\B}$ is full, and therefore a
  topological embedding.

  $(b) \limp (c)$.  The inclusion map from $\conify_{\leq 1} (\B)$
  into $\conify (\B)$ is an affine topological embedding by
  definition.  Composing it with
  $\etac_{\B} \colon \B \to \conify_{\leq 1} (\B)$, we obtain a
  topological embedding of $\B$ in $\conify (\B)$, which is affine
  by item~3.

  Finally, $(c) \limp (a)$ is obvious, since every semitopological
  cone is a pointed semitopological algebra, in particular.  \qed
\end{proof}

This leads to the following, perhaps surprising, result.
\begin{corollary}
  \label{corl:bary:alg:pointed:embed}
  Every pointed semitopological barycentric algebra is embeddable.
\end{corollary}
\begin{proof}
  Let $\B$ be a pointed semitopological barycentric algebra.  The
  identity map is an affine topological embedding.  By the
  $(a) \limp (c)$ implication of
  Theorem~\ref{thm:conify:leq1:semitop}, item~5, $\B$ is embeddable.  \qed
\end{proof}
Corollary~\ref{corl:bary:alg:pointed:embed} does \emph{not} say that
every pointed semitopological barycentric algebra $\B$ embeds through
an affine topological embedding that preserves bottom elements (a
strict affine, or linear, map).  And indeed, the embedding obtained in
the proof, namely $\etac_{\B}$, maps $\bot$ to $(1, \bot)$, not to $0$.

\begin{remark}
  \label{rem:bary:alg:topo}
  For every topological barycentric algebra $\B$,
  $\conify_{\leq 1} (\B)$ is a pointed topological, not just pointed
  semitopological barycentric algebra.  Indeed, $\conify (\B)$ is a
  topological cone in this case by Theorem~\ref{thm:bary:alg:topo},
  and $\conify_{\leq 1} (\B)$ embeds through the identity map, namely
  is a convex subspace of $\conify (\B)$; we conclude by
  Lemma~\ref{lemma:bary:in:cone}.
\end{remark}

\begin{example}
  \label{exa:bary:alg:pointed:notembed}
  There is no affine topological embedding of the topological
  barycentric algebra $\B^-$ described in
  Example~\ref{exa:bary:alg:notembed} in any pointed semitopological
  barycentric algebra.  Indeed, by
  Theorem~\ref{thm:conify:leq1:semitop}, item~5, $\B^-$ embeds in a
  pointed semitopological barycentric algebra if and only if it is
  embeddable.  But $\B^-$ is not embeddable, as we have seen in
  Example~\ref{exa:bary:alg:notembed}.
\end{example}

\subsection{The free pointed barycentric algebra over $\Val_1 X$}
\label{sec:free-point-baryc-1}

Applying the $\conify_{\leq 1}$ construction to $\Val_1 X$ gives
$\Val_{\leq 1}$, up to isomorphism, thanks to the following result.
\begin{theorem}
  \label{thm:V1->Vleq1}
  For every topological space $X$, $\Val_{\leq 1} X$, with
  distinguished element the zero valuation, is the free pointed
  barycentric algebra over the barycentric algebra $\Val_1 X$, and
  $\Val_{\leq 1} X$ is the free pointed semitopological barycentric
  algebra over the semitopological barycentric algebra
  $\Val_{1} X$.

  As a consequence, $\Val_{\leq 1} X$ and
  $\conify_{\leq 1} (\Val_1 X)$ are naturally isomorphic
  semitopological cones.
\end{theorem}
\begin{proof}
  Let $f$ be an affine map from $\Val_1 X$ to a pointed barycentric
  algebra $\B$.  We need to show that $f$ extends to a unique linear
  map $\hat f$ from $\Val_{\leq 1} X$ to $\B$, which is continuous if
  $f$ is (assuming $\B$ semitopological, and giving $\Val_1 X$ and
  $\Val_{\leq 1} X$ their weak topologies.)
  
  Uniqueness is proved as in Theorem~\ref{thm:V1->Vb}: if $\hat f$
  exists, then it must map the zero valuation to $\bot$ and every
  non-zero element $\nu$ of $\Val_{\leq 1} X$ to
  $\nu (X) \cdot f (\frac 1 {\nu (X)} \cdot \nu)$, using the fact that
  $f$ is $\Iu$-homogenous.

  We therefore define $\hat f$ as mapping the zero valuation to $\bot$
  and every non-zero element $\nu$ of $\Val_{\leq 1} X$ to
  $\nu (X) \cdot f (\frac 1 {\nu (X)} \cdot \nu)$.  By definition,
  $\hat f$ is strict.  In order to see that $\hat f$ is linear, it
  suffices to show that it is affine, thanks to
  Proposition~\ref{prop:strict:aff}.  We observe that
  $g \eqdef \etac_{\B} \circ f$ is an affine map from $\Val_1 X$ to
  $\conify (\B)$, because $\etac_{\B}$ is affine.  By
  Theorem~\ref{thm:V1->Vb}, $g$ extends to a unique linear map
  $\hat g$ from $\Val_b X$ to $\conify (\B)$, and then we note that
  $\etac_{\B} \circ \hat f = \hat g$.  Indeed, by the uniqueness of
  $\hat g$, this follows from the fact that $\etac_{\B} \circ \hat f$
  coincides with $g$ on $\Val_1 X$, a consequence of the fact that
  $\hat f$ coincides with $f$ on $\Val_1 X$.  Then, for all
  $\mu, \nu \in \Val_{\leq 1} X$ and for every $a \in [0, 1]$,
  $\etac_{\B} (\hat f (\mu +_a \nu)) = \hat g (\mu +_a \nu) = \hat g
  (\mu) +_a \hat g (\nu)$ (since $\hat g$ is linear, hence affine)
  $= \etac_{\B} (\hat f (\mu)) +_a \etac_{\B} (\hat f (\nu)) =
  \etac_{\B} (\hat f (\mu) +_a \hat f (\nu))$ (since $\etac_{\B}$ is
  affine).  Since $\etac_{\B}$ is injective,
  $\hat f (\mu +_a \nu) = \hat f (\mu) +_a \hat f (\nu)$, showing that
  $\hat f$ is affine.

  When $\B$ is a pointed semitopological barycentric algebra,
  $\Val_1 X$ and $\Val_b X$ are given their weak topologies, and $f$
  is continuous, $\etac_{\B}$ is injective, affine and continuous by
  Theorem~\ref{thm:conify:semitop}, item~2.  In particular, $g$ is
  continuous, so $\hat g$ is continuous by Theorem~\ref{thm:V1->Vb}.
  In other words, $\etac_{\B} \circ f$ is continuous.  By
  Corollary~\ref{corl:bary:alg:pointed:embed}, $\B$ is embeddable, and
  by the equivalence between (a) and (b) in
  Theorem~\ref{thm:conify:semitop}, item~4, $\etac_{\B}$ is a
  topological embedding.  The continuity of $f$ then follows from that
  of $\etac_{\B} \circ f$.  \qed
\end{proof}

\subsection{Barycenters, part 3}
\label{sec:barycenters-part-3}

In pointed barycentric algebras, there is an extended notion of
barycenters.  The following is an analogue of Definition and
Proposition~\ref{prop:bary:1}.  It would be tempting to define
$\sum_{i=1}^n a_i \cdot x_i$, where $\sum_{i=1}^n a_i$ is now less
than or equal to $1$, not equal to $1$, by saying something like: ``it
is the only point $x$ that is mapped by $\etaca_{\B}$ to
$\sum_{i=1}^n a_i \cdot \etaca_{\B} (x_i)$''.  That would be wrong,
since $\etaca_{\B}$ is not injective in general.  Hence we give an ad
hoc definition, and then we verify a property in the spirit of the
latter sentence, but relying on $\etac_{\B}$ instead of $\etaca_{\B}$
(item~3).
\begin{defprop}
  \label{prop:bary:leq1}
  For every $n \in \nat$, let
  $\widetilde\Delta_n \eqdef \{(a_1, \cdots, a_n) \in \Rp^n \mid
  \sum_{i=1}^n a_i \leq 1\}$.

  Let $\B$ be a pointed barycentric algebra.  For every $n \in \nat$,
  for all points $x_1$, \ldots, $x_n$ of $\B$ and for every
  $(a_1, \cdots, a_n) \in \widetilde\Delta_n$, there is a point of
  $\B$, which we write as $\sum_{i=1}^n a_i \cdot x_i$, and defined as
  $(\sum_{i=1}^n a_i) \cdot \sum_{i=1}^n \frac {a_i} {\sum_{i=1}^n
    a_i} \cdot x_i$ (where the latter sum is defined in
  Proposition~\ref{prop:bary:1}) when $\sum_{i=1}^n a_i \neq 0$ (in
  particular $n \geq 1$), and to $\bot$ otherwise.  We call it the
  \emph{barycenter} of the points $x_i$ with weights $a_i$. Then:
  \begin{enumerate}
  \item this construction is invariant under permutations: for every
    bijection $\sigma$ of $\{1, \cdots, n\}$ onto itself,
    $\sum_{i=1}^n a_i \cdot x_i = \sum_{j=1}^n a_{\sigma (i)} \cdot
    x_{\sigma (i)}$;
  \item for every linear function $f$ from $\B$ to a pointed
    barycentric algebra $\Alg$,
    $f (\sum_{i=1}^n a_i \cdot x_i) = \sum_{i=1}^n a_i \cdot f
    (x_i)$;
  \item $\sum_{i=1}^n a_i \cdot x_i$ is the only point $x$ of $\B$
    such that
    $\etac_{\B} (x) = \sum_{i=1}^{n} a_i \cdot \etac_{\B} (x_i) +
    (1-\sum_{i=1}^n a_i) \cdot \etac_{\B} (\bot)$.
  \end{enumerate}
\end{defprop}
\begin{proof}
  Items~1 and ~2 are immediate consequences of Definition and
  Proposition~\ref{prop:bary:1}.  As for item~3, let
  $a \eqdef \sum_{i=1}^n a_i$, so that $a_0 = 1-a$.  We recall that
  $\etac_{\B}$ is affine and injective.  Injectivity entails the
  uniqueness claim of item~3.  Let
  $x \eqdef \sum_{i=1}^n a_i \cdot x_i$, as defined in the
  proposition.  If $a \neq 0$, then:
  \ifta
  \begin{align*}
    \etac_{\B} (x)
    & = \etac_{\B} (a \cdot \sum_{i=1}^n \frac {a_i} a \cdot x_i) \\
    & = \etac_{\B} (\sum_{i=1}^n \frac {a_i} a \cdot x_i +_a \bot)
    & \text{by definition of $\cdot$ on $\B$}
    \\
    & = a \cdot \etac_{\B}  (\sum_{i=1}^n \frac {a_i} a \cdot x_i) +
      (1-a) \cdot \etac_{\B} (\bot)
    & \text{since $\etac_{\B}$ is affine}
    \\
    & = a \cdot \sum_{i=1}^n \frac {a_i} a \cdot \etac_{\B} (x_i) +
      (1-a) \cdot \etac_{\B} (\bot)
    & \text{by Def.\ and Prop.~\ref{prop:bary:1}}
    \\
    & = \sum_{i=1}^n a_i \cdot \etac_{\B} (x_i) + (1-a) \cdot
      \etac_{\B} (\bot).
  \end{align*}
  \else
  \begin{align*}
    \etac_{\B} (x)
    & = \etac_{\B} (a \cdot \sum_{i=1}^n \frac {a_i} a \cdot x_i) \\
    & = \etac_{\B} (\sum_{i=1}^n \frac {a_i} a \cdot x_i +_a \bot) \\
    & \qquad\text{by definition of $\cdot$ on $\B$}
    \\
    & = a \cdot \etac_{\B}  (\sum_{i=1}^n \frac {a_i} a \cdot x_i) +
      (1-a) \cdot \etac_{\B} (\bot) \\
    & \qquad \text{since $\etac_{\B}$ is affine}
    \\
    & = a \cdot \sum_{i=1}^n \frac {a_i} a \cdot \etac_{\B} (x_i) +
      (1-a) \cdot \etac_{\B} (\bot) \\
    & \qquad \text{by Def.\ and Prop.~\ref{prop:bary:1}}
    \\
    & = \sum_{i=1}^n a_i \cdot \etac_{\B} (x_i) + (1-a) \cdot
      \etac_{\B} (\bot).
  \end{align*}
  \fi
  If $a=0$, then $\etac_{\B} (x) = \etac_{\B} (\bot) =
  \sum_{i=1}^n a_i \cdot \etac_{\B} (x_i) + (1-a) \etac_{\B} (\bot)$,
  since $a_i=0$ for every $i \in \{1, \cdots, n\}$.  \qed
\end{proof}

\begin{lemma}
  \label{lemma:bary:leq1:mono}
  Let $\B$ be a preordered pointed barycentric algebra, and $x_1$,
  \ldots, $x_n$ be finitely many points of $\B$.  The function
  $(a_1, \cdots, a_n) \mapsto \sum_{i=1}^n a_i \cdot x_i$ is monotonic
  from $\widetilde\Delta_n$, with the pointwise ordering, to $\B$.
\end{lemma}
\begin{proof}
  This is clear if $n=0$, so let us assume $n \geq 1$.  We claim that
  this function is monotonic in
  $a_n \in [0, 1 - \sum_{i=1}^{n-1} a_i]$.  By invariance under
  permutations (Definition and Proposition~\ref{prop:bary:leq1},
  item~1), it will be monotonic in every $a_i$.  Let
  $f (a_n) \eqdef \sum_{i=1}^n a_i \cdot x_i$, understanding that
  $a_1$, \ldots, $a_{n-1}$, $x_1$, \ldots, $x_n$ are fixed.  We wish
  to show that $f$ is monotonic.  Since $\etac_{\B}$ is a preorder
  embedding in the preordered cone $\conify (\B)$
  (Proposition~\ref{prop:bary:alg:conify:ord}), it suffices to show
  that $\etac_{\B} \circ f$ is monotonic.  Now, and using $\cdot$ for
  scalar multiplication in $\B$ and $.$ for scalar multiplication in
  $\conify (\B)$,
  \ifta
  \begin{align*}
    & \etac_{\B} (f (a_n)) \\
    & = \etac_{\B} \left((\sum_{i=1}^n a_i) \cdot
      \sum_{i=1}^n \frac {a_i} {\sum_{i=1}^n
      a_i} \cdot x_i\right) \\
    & = \etac_{\B} \left(\sum_{i=1}^n \frac {a_i} {\sum_{i=1}^n
      a_i} \cdot x_i +_{\sum_{i=1}^n a_i} \bot\right)
    & \text{by definition of $\cdot$}
    \\
    & = (\sum_{i=1}^n a_i) .
      \etac_{\B} \left(\sum_{i=1}^n \frac {a_i} {\sum_{i=1}^n
      a_i} \cdot x_i\right)
      + (1-\sum_{i=1}^n a_i) .
      \etac_{\B} (\bot)
    & \text{since $\etac_{\B}$ is affine} \\
    & = (\sum_{i=1}^n a_i) . \sum_{i=1}^n \frac {a_i} {\sum_{i=1}^n
      a_i} . \etac_{\B} (x_i)
      + (1-\sum_{i=1}^n a_i) \etac_{\B} (\bot)
    & \text{by Def.\ and Prop.~\ref{prop:bary:1}} \\
    & = \sum_{i=1}^n a_i . \etac_{\B} (x_i)
      + (1-\sum_{i=1}^n a_i) \etac_{\B} (\bot) \\
    & = y + a_n . \etac_{\B} (x_n) + (A - a_n) . \etac_{\B} (\bot),
  \end{align*}
  \else
  \begin{align*}
    & \etac_{\B} (f (a_n)) \\
    & = \etac_{\B} \left((\sum_{i=1}^n a_i) \cdot
      \sum_{i=1}^n \frac {a_i} {\sum_{i=1}^n
      a_i} \cdot x_i\right) \\
    & = \etac_{\B} \left(\sum_{i=1}^n \frac {a_i} {\sum_{i=1}^n
      a_i} \cdot x_i +_{\sum_{i=1}^n a_i} \bot\right) \\
    & \qquad\text{by definition of $\cdot$}
    \\
    & = (\sum_{i=1}^n a_i) .
      \etac_{\B} \left(\sum_{i=1}^n \frac {a_i} {\sum_{i=1}^n
      a_i} \cdot x_i\right)
      + (1-\sum_{i=1}^n a_i) .
      \etac_{\B} (\bot) \\
    & \qquad\text{since $\etac_{\B}$ is affine} \\
    & = (\sum_{i=1}^n a_i) . \sum_{i=1}^n \frac {a_i} {\sum_{i=1}^n
      a_i} . \etac_{\B} (x_i)
      + (1-\sum_{i=1}^n a_i) \etac_{\B} (\bot) \\
    & \qquad\text{by Def.\ and Prop.~\ref{prop:bary:1}} \\
    & = \sum_{i=1}^n a_i . \etac_{\B} (x_i)
      + (1-\sum_{i=1}^n a_i) \etac_{\B} (\bot) \\
    & = y + a_n . \etac_{\B} (x_n) + (A - a_n) . \etac_{\B} (\bot),
  \end{align*}
  \fi
  where $y \eqdef \sum_{i=1}^{n-1} a_i . \etac_{\B} (x_i)$ and
  $A \eqdef 1 - \sum_{i=1}^{n-1} a_i$ are independent of $a_n$.
  Hence, for all $a, a' \in [0, 1 - \sum_{i=1}^{n-1} a_i]$ such that
  $a \leq a'$,
  \ifta
  \begin{align*}
    \etac_{\B} (f (a'))
    & = y + (a'-a) . \etac_{\B} (x_n) + a . \etac_{\B} (x_n) + (A-a')
      . \etac_{\B} (\bot) \\
    & \geq y + (a'-a) . \etac_{\B} (\bot) + a . \etac_{\B} (x_n) + (A-a')
      . \etac_{\B} (\bot),
  \end{align*}
  \else
  \begin{align*}
    \etac_{\B} (f (a'))
    & = y + (a'-a) . \etac_{\B} (x_n) + a . \etac_{\B} (x_n) \\
    & \qquad + (A-a')  . \etac_{\B} (\bot) \\
    & \geq y + (a'-a) . \etac_{\B} (\bot) + a . \etac_{\B} (x_n) \\
    & \qquad + (A-a')  . \etac_{\B} (\bot),
  \end{align*}
  \fi
  since $\bot \leq x_n$ and $\etac_{\B}$ and $(a'-a) . \_$ are
  monotonic; and that is equal to
  $y + a . \etac_{\B} (x_n) + (A-a) . \etac_{\B} (\bot) = \etac_{\B}
  (f (a))$.  \qed
\end{proof}

We topologize $\widetilde\Delta_n$ with the subspace topology induced
by the inclusion in $(\creal)^n$; we remember that $\creal$ has the
Scott topology of its usual ordering.  Since $\creal$ is a continuous
dcpo, $(\creal)^n$ is also a continuous dcpo, and the product topology
on $\creal^n$ coincides with the Scott topology of the componentwise
ordering \cite[Proposition 5.1.54]{JGL-topology}.  It is easy to see
that $\widetilde\Delta_n$ is Scott-closed in $\creal^n$.  It follows
that the topology on $\widetilde\Delta_n$ is also the Scott topology
of the componentwise ordering (the Scott topology on a Scott-closed
subset of a continuous dcpo coincides with the subspace topology
\cite[Exercise 5.1.52]{JGL-topology}).

\begin{proposition}
  \label{prop:bary:leq1:cont}
  Let $\B$ be a pointed semitopological (resp.\ topological)
  barycentric algebra.  For every $n \in \nat$, the function
  $(a_1, \cdots, a_n, x_1, \cdots, x_n) \mapsto \sum_{i=1}^n a_i \cdot
  x_i$ is separately (resp.\ jointly) continuous from
  $\widetilde\Delta_n \times \B \times \cdots \times \B$ to $\B$.
\end{proposition}
Let $f$ be this function.  The semitopological case needs some
clarification: in this case, Proposition~\ref{prop:bary:leq1:cont}
says that
$f (a_1, \cdots, \allowbreak a_n, \allowbreak x_1, \cdots, x_n)$ is
separately continuous in the tuple $(a_1, \cdots, a_n)$ (jointly in
$(a_1, \allowbreak \cdots, \allowbreak a_n) \in \widetilde\Delta_n$,
hence also separately in each $a_i$) and in the points $x_1$, \ldots,
$x_n$.

\begin{proof}
  This is obvious if $n=0$, so we assume $n \geq 1$ in the sequel.  By
  Corollary~\ref{corl:bary:alg:pointed:embed}, $\B$ is embeddable.  By
  Theorem~\ref{thm:conify:semitop}, item~4, we may take $\etac_{\B}$
  as an affine topological embedding in the semitopological cone
  $\C \eqdef \conify (\B)$.  Beware that $\etac_{\B}$ is not linear in
  general, equivalently not $\Iu$-homogeneous: we will keep the notation
  $\cdot$ for scalar multiplication in $\B$, and we will use $.$ for
  scalar multiplication in $\C$, and we will remember that
  $\etac_{\B} (a \cdot x) \neq a . \etac_{\B} (x)$ in general.  For
  all $(a_1, \cdots, a_n) \in \widetilde\Delta_n$ and
  $x_1, \cdots, x_n \in \B$,
  $\etac_{\B} (f (a_1, \cdots, a_n, x_1, \cdots, \allowbreak x_n))$ is
  equal to $\etac_{\B} (\bot)$ if $a_1=\cdots=a_n=0$.  Otherwise,
  \ifta
  \begin{align*}
    & \etac_{\B} (f (a_1, \cdots, a_n, x_1, \cdots, x_n)) \\
    & = \etac_{\B} \left((\sum_{i=1}^n a_i) \cdot \sum_{i=1}^n \frac {a_i}
      {\sum_{i=1}^n a_i} \cdot x_i \right)
    & \text{by Def.\ and Prop.~\ref{prop:bary:leq1}} \\
    & = \etac_{\B} \left(\sum_{i=1}^n \frac {a_i}
      {\sum_{i=1}^n a_i} \cdot x_i +_{\sum_{i=1}^n a_i} \bot \right)
    & \text{by definition of $\cdot$} \\
    & = (\sum_{i=1}^n a_i) . \etac_{\B} \left(\sum_{i=1}^n \frac {a_i}
      {\sum_{i=1}^n a_i} \cdot x_i\right)
      + (1-\sum_{i=1}^n a_i) . \etac_{\B} (\bot)
    & \text{since $\etac_{\B}$ is affine} \\
    & = (\sum_{i=1}^n a_i) . \sum_{i=1}^n \frac {a_i}
      {\sum_{i=1}^n a_i} . \etac_{\B} (x_i)
      + (1-\sum_{i=1}^n a_i) . \etac_{\B} (\bot)
    & \text{by Def.\ and Prop.~\ref{prop:bary:1}} \\
    & = \sum_{i=1}^n a_i . \etac_{\B} (x_i)
      + (1-\sum_{i=1}^n a_i) . \etac_{\B} (\bot).
  \end{align*}
  \else
  \begin{align*}
    & \etac_{\B} (f (a_1, \cdots, a_n, x_1, \cdots, x_n)) \\
    & = \etac_{\B} \left((\sum_{i=1}^n a_i) \cdot \sum_{i=1}^n \frac {a_i}
      {\sum_{i=1}^n a_i} \cdot x_i \right) \\
    & \qquad \text{by Def.\ and Prop.~\ref{prop:bary:leq1}} \\
    & = \etac_{\B} \left(\sum_{i=1}^n \frac {a_i}
      {\sum_{i=1}^n a_i} \cdot x_i +_{\sum_{i=1}^n a_i} \bot \right) \\
    & \qquad \text{by definition of $\cdot$} \\
  \end{align*}
  \begin{align*}
    & = (\sum_{i=1}^n a_i) . \etac_{\B} \left(\sum_{i=1}^n \frac {a_i}
      {\sum_{i=1}^n a_i} \cdot x_i\right)
      + (1-\sum_{i=1}^n a_i) . \etac_{\B} (\bot) \\
    & \qquad \text{since $\etac_{\B}$ is affine} \\
    & = (\sum_{i=1}^n a_i) . \sum_{i=1}^n \frac {a_i}
      {\sum_{i=1}^n a_i} . \etac_{\B} (x_i)
      + (1-\sum_{i=1}^n a_i) . \etac_{\B} (\bot) \\
    & \qquad \text{by Def.\ and Prop.~\ref{prop:bary:1}} \\
    & = \sum_{i=1}^n a_i . \etac_{\B} (x_i)
      + (1-\sum_{i=1}^n a_i) . \etac_{\B} (\bot).
  \end{align*}
  \fi
  The equality of the first and last terms continues to hold when
  $a_1=\cdots=a_n=0$.  This formula shows that $\etac_{\B} \circ f$ is
  separately continuous in $x_1$, \ldots, $x_n$, and jointly
  continuous in $x_1$, \ldots, $x_n$ if $\B$ is topological.

  Fixing $a_1$, \ldots, $a_{n-1}$, $x_1$, \ldots, $x_n$,
  $\etac_{\B} (f (a_1, \cdots, a_n, x_1, \cdots, x_n))$ is a function
  $g (a_n)$ of $a_n$, which we can write as
  $y + a_n . \etac_{\B} (x_n) + (A-a_n) . \etac_{\B} (\bot)$, where
  $y \eqdef \sum_{i=1}^{n-1} a_i . \etac_{\B} (x_i)$ and
  $A \eqdef 1-\sum_{i=1}^{n-1} a_i$; $a_n$ varies in $[0, A]$.  If
  $A=0$, $g$ is vacuously continuous.  Otherwise,
  $g (a_n) = y + A . (\etac_{\B} (x_n) +_{a_n/A} \etac_{\B} (\bot))$,
  and that is continuous in $a_n \in [0, A]$ (with its usual metric
  topology) since every cone is semitopological as a preordered
  barycentric algebra (see Lemma~\ref{lemma:bary:in:cone}).
  Additionally, $g$ is monotonic, since
  $f (a_1, \cdots, a_n, x_1, \cdots, x_n)$ is monotonic in
  $(a_1, \cdots, a_n)$, by Lemma~\ref{lemma:bary:leq1:mono}, and since
  $\etac_{\B}$ is continuous hence monotonic.  It follows that $g$ is
  continuous from $[0, A]_\sigma$ to $\C$: for every open subset $V$
  of $\C$, $g^{-1} (V)$ is an upwards-closed, open subset of $[0, A]$,
  hence a Scott-open subset of $[0, A]$.

  We now fix $x_1$, \ldots, $x_n$ only, and we consider the function
  $h \colon (a_1, \cdots, a_n) \mapsto \etac_{\B} (f (a_1, \cdots,
  a_n, x_1, \cdots, x_n))$ from $\widetilde\Delta_n$ to $\C$.  By
  invariance under permutations (Definition and
  Proposition~\ref{prop:bary:leq1}, item~1), $h$ is separately
  continuous in each $a_i$, varying in
  $[0, 1-\sum_{\substack{1 \leq j \leq n\\j\neq i}} a_j]$.  We claim
  that it is jointly continuous.  It is tempting to appeal to Ershov's
  observation, but the domain of $h$ is not a Cartesian product.
  Instead, we prove it by hand.  Let $V$ be an open subset of $\C$.
  Since $V$ is upwards-closed, by Lemma~\ref{lemma:bary:leq1:mono}
  $h^{-1} (V)$ is upwards-closed in $\widetilde\Delta_n$.  Let
  ${(\vec a_i)}_{i \in I}$ be a directed family with supremum $\vec a$
  in $\widetilde\Delta_n$, and let us assume that
  $\vec a \in h^{-1} (V)$.  We write $\vec a$ as $(a_1, \cdots, a_n)$
  and $\vec a_i$ as $(a_{i1}, \cdots, a_{in})$.  For each
  $j \in \{0, 1, \cdots, n\}$, let
  $\vec a_i [j] \eqdef (a_{i1}, \cdots, a_{ij}, a_{j+1}, \cdots,
  a_n)$.  Then ${(\vec a_i [j])}_{i \in I}$ is a directed family in
  $\widetilde\Delta_n$ whose supremum is $\vec a$, and we claim that
  $\vec a_i [j] \in h^{-1} (V)$ for some $i \in I$.  This is proved by
  induction on $j$.  This is clear if $j=0$, since then
  $\vec a_i [j] = \vec a$ for every $i \in I$, and
  $\vec a \in h^{-1} (V)$.  In the induction step, we assume that
  $\vec a_i [j] \in h^{-1} (V)$, where $0 < j < n$.  The element
  $a_{j+1}$ is the supremum of the directed family
  ${(a_{i'(j+1)})}_{i' \in I}$, and since $h$ is continuous in its
  $(j+1)$st argument, there is an $i' \in I$ such that
  $(a_{i1}, \cdots, a_{ij}, a_{i'(j+1)}, a_{j+2}, \cdots, a_n)$ is in
  $h^{-1} (V)$.  By directedness, there is an index $i'' \in I$ such
  that $(a_{i1}, \cdots, a_{in})$ and $(a_{i'1}, \cdots, a_{i'n})$ are
  both less than or equal to $(a_{i''1}, \cdots, a_{i''n})$.  Since
  $h^{-1} (V)$ is upwards-closed,
  $(a_{i''1}, \cdots, a_{i''j}, a_{i''(j+1)}, a_{j+2}, \cdots, a_n)
  \in h^{-1} (V)$, namely $\vec a_{i''} [j+1] \in h^{-1} (V)$.  The
  induction is complete.  In particular, for $j \eqdef n$, there is an
  $i \in I$ such that $\vec a_i [n] \in h^{-1} (V)$.  But
  $\vec a_i [n] = \vec a_i$, so $\vec a_i \in h^{-1} (V)$.  This shows
  that $h^{-1} (V)$ is (Scott-)open in $\widetilde\Delta_n$, and
  therefore $h$ is continuous.

  When $\B$ is semitopological, we have just shown that
  $\etac_{\B} \circ f$ is separately continuous in
  $(a_1, \cdots, a_n)$, $x_1$, \ldots, $x_n$.  Since $\etac_{\B}$ is a
  topological embedding, $f$ is itself separately continuous.
  
  When $\B$ is topological, we have shown more: $\etac_{\B} \circ f$
  is separately continuous as a function of the tuple
  $(a_1, \cdots, a_n) \in \widetilde\Delta_n$ and of the tuple
  $(x_1, \cdots, x_n) \in \B^n$ (jointly in $(x_1, \cdots, x_n)$).
  Since $\widetilde\Delta_n$ is a continuous dcpo,
  $\etac_{\B} \circ f$ is jointly continuous, thanks to Ershov's
  observation.  Therefore $f$ is jointly continuous, using the fact
  that $\etac_{\B}$ is a topological embedding.  \qed
\end{proof}

\section{Convexity}
\label{sec:convexity}

We have already defined convex and concave sets, but only in cones.
The definition in barycentric algebras is the same.
\begin{definition}[Convex, concave]
  \label{defn:convex}
  A subset $A$ of a barycentric algebra $\B$ is \emph{convex} if and
  only if for all $x, y \in A$ and for every $a \in [0, 1]$, $x +_a y$
  is in $A$.  It is \emph{concave} if and only if $\B \diff A$ is
  convex.
\end{definition}

The theory of convex subsets of barycentric algebras to come in this
section largely parallels the theory of convex subsets of cones
\cite[Section~4]{Keimel:topcones2}.

\begin{lemma}
  \label{lemma:convex:inter}
  All intersections of convex subsets of a barycentric algebra $\B$ are
  convex.  All directed unions of convex subsets of $\B$ are convex.
\end{lemma}
\begin{proof}
  The first claim is immediate.  For the second claim, let us consider
  a directed family ${(A_i)}_{i \in I}$ of convex subsets of $\B$, and
  $A \eqdef \bigcup_{i \in I} A_i$.  Let us take any two points
  $x, y \in A$.  Then $x$ is in some $A_i$ and $y$ is in some $A_j$.
  By directedness, there is a $k \in I$ such that $A_k$ contains both
  $A_i$ and $A_j$.  Therefore $x$ and $y$ are in $A_k$.  For every
  $a \in [0, 1]$, $x +_a y$ is in $A_k$ since $A_k$ is convex, hence
  in $A$.   \qed
\end{proof}

\subsection{Convex hulls}
\label{sec:convex-hulls}

We recall the notion of barycenters $\sum_{i=1}^n a_i \cdot x_i$ from
Definition and Proposition~\ref{prop:bary:1}.  We also recall that
$\Delta_n$ is the standard simplex
$\{\vec \alpha \in \Rp^n \mid \sum_{i=1}^n \alpha_i=1\}$.
\begin{deflem}[Convex hull]
  \label{deflem:convex:simpleval}
  Let $\B$ be a barycentric algebra.  For every subset $A$ of $\B$,
  there is a smallest convex subset of $\B$ containing $A$.  This is
  called the \emph{convex hull} $\conv A$ of $A$, and is exactly the
  set of barycenters $\sum_{i=1}^n a_i \cdot x_i$ where $n \geq 1$,
  $(a_1, \cdots, a_n) \in \Delta_n$ and each $x_i$ is in $A$.
\end{deflem}
\begin{proof}
  Let $B$ be the indicated set of barycenters.  Let us show that $B$
  is a convex set.  We consider two elements of $B$.  Without loss of
  generality, we can write them as
  $x \eqdef \sum_{i=1}^n a_i \cdot x_i$ and
  $y \eqdef \sum_{i=1}^n b_i \cdot x_i$ with $a_i, b_i \geq 0$,
  $\sum_{i=1}^n a_i = \sum_{i=1}^n b_i = 1$, using the same sequence
  of points $x_i \in A$, $1\leq i\leq n$, since we are allowed to use
  zero as weights $a_i$ or $b_i$.  (We implicitly reason in a cone in
  which $\B$ embeds through an injective affine map, typically
  $\etac_{\B}$.)  For every $a \in [0, 1]$, $x +_a y$ is equal to
  $\sum_{i=1}^n (aa_i + (1-a)b_i) \cdot x_i$.  Since
  $aa_i + (1-a)b_i \geq 0$ and
  $\sum_{i=1}^n (aa_i + (1-a)b_i) = a \sum_{i=1}^n a_i + (1-a)
  \sum_{i=1}^n b_i = a + (1-a) = 1$, $a \cdot x + (1-a) \cdot y$ is in
  $B$.

  Let us now show that $B$ is the smallest convex set containing $A$.
  To that end, we consider any convex set $E$ containing $A$, and a
  barycenter $x \eqdef \sum_{i=1}^n a_i \cdot x_i$ where every $x_i$
  is in $A$, $n \geq 1$ and $(a_1, \cdots, a_n) \in \Delta_n$.  We
  show that $x$ is in $E$ by induction on $n \in \nat$, using the
  inductive formula of Definition and Proposition~\ref{prop:bary:1},
  item~1.  If $a_n=1$ (a case that includes the base case $n=1$),
  $x=x_n$ is in $A$, hence in $E$.  Otherwise,
  $x = (\sum_{i=1}^{n-1} \frac {a_i} {1-a_n} x_i) +_{1-a_n} x_n$.  By
  induction hypothesis
  $x' \eqdef \sum_{i=1}^{n-1} \frac {a_i} {1-a_n} x_i$ is in $E$;
  $x_n$ is in $A$ hence in $E$, so $x \eqdef x' +_{1-a_n} x_n$ is in
  $E$, which is convex.  \qed
\end{proof}

\begin{lemma}
  \label{lemma:conv:img}
  For every affine map $f \colon \Alg \to \B$ between barycentric
  algebras, for every subset $A$ of $\Alg$,
  $f [\conv A] = \conv {f [A]}$.
\end{lemma}
\begin{proof}
  Every element of $\conv A$ is of the form
  $\sum_{i=1}^n a_i \cdot x_i$ with $n \geq 1$,
  $(a_1, \cdots, a_n) \in A$ and $x_1, \cdots, x_n \in A$, by
  Definition and Lemma~\ref{deflem:convex:simpleval}.  By Definition
  and Proposition~\ref{prop:bary:1}, item~4, and since $f$ is affine,
  $f (\sum_{i=1}^n a_i \cdot x_i) = \sum_{i=1}^n a_i \cdot f (x_i)$ is
  therefore in $\conv {f [A]}$.  Conversely, $f [\conv A]$ is convex
  since $\conv A$ is convex and $f$ is affine, it contains $f [A]$,
  hence it contains the smallest convex set that contains $f [A]$,
  which is $\conv {f [A]}$ by definition.  \qed
\end{proof}

\begin{lemma}
  \label{lemma:conv:union}
  Let $\B$ be a barycentric algebra, and $A_1$, \ldots, $A_n$ be
  finitely many non-empty convex subsets of $\B$.  The set of all
  barycenters $\sum_{i=1}^n a_i \cdot x_i$ where $x_i \in A_i$ for
  each $i$, $n \geq 1$ and $(a_1, \cdots, a_n) \in \Delta_n$, is equal
  to $\conv {(A_1 \cup \cdots \cup A_n)}$.
\end{lemma}
\begin{proof}
  Let $A$ be this set of barycenters.  We consider any two barycenters
  $\sum_{i=1}^n a_i \cdot x_i$ and $\sum_{i=1}^n b_i \cdot y_i$, where
  $x_i, y_i \in A_i$.  As usual, we implicitly reason in a cone in
  which $\B$ embeds through an injective affine map, typically
  $\etac_{\B}$.  For every $a \in {]0, 1[}$,
  $\sum_{i=1}^n a_i \cdot x_i +_a \sum_{i=1}^n b_i \cdot y_i =
  \sum_{i=1}^n aa_i \cdot x_i + (1-a)b_i \cdot y_i$.  For each $i$,
  let $c_i \eqdef {aa_i+(1-a)b_i}$.  If $c_i \neq 0$, we can write
  $aa_i \cdot x_i + (1-a)b_i \cdot y_i$ as $c_i$ times the point
  $z_i \eqdef \frac {aa_i} {c_i} x_i + \frac {(1-a) b_i} {c_i} y_i =
  x_i +_{\frac {aa_i} {c_i}} y_i$, which is in $A_i$ since $A_i$ is
  convex.  If $c_i = 0$, we define $z_i$ as $x_i$ or as $y_i$: again
  $z_i$ is in $A_i$, and since $c_i=0$, $aa_i=(1-a)b_i=0$, so
  $aa_i x_i + (1-a)b_i y_i = 0 = c_i \cdot z_i$ again.  Hence
  $\sum_{i=1}^n a_i \cdot x_i +_a \sum_{i=1}^n b_i \cdot y_i =
  \sum_{i=1}^n c_i \cdot z_i$.  Since $c_i \geq 0$ and
  $\sum_{i=1}^n c_i = a \sum_{i=1}^n a_i + (1-a) \sum_{i=1}^n b_i = a
  + (1-a) = 1$, $\sum_{i=1}^n c_i \cdot z_i$ is in $A$.

  This shows that $A$ is convex.  Since $A$ contains each $A_i$, it
  also contains $\conv {(A_1 \cup \cdots \cup A_n)}$.  Conversely,
  every point of $A$ is a barycenter of points from
  $A_1 \cup \cdots \cup A_n$, hence is in
  $\conv {(A_1 \cup \cdots \cup A_n)}$ by Definition and
  Lemma~\ref{deflem:convex:simpleval}; whence
  $A = \conv {(A_1 \cup \cdots \cup A_n)}$.  \qed
\end{proof}

\subsection{Closed convex hulls}
\label{sec:closed-convex-hulls}

\begin{defprop}[Closed convex hull]
  \label{defprop:clconv}
  In a semitopological barycentric algebra $\B$, the closure of every
  convex subset is convex.

  For every subset $A$ of $\B$, there is a smallest closed convex
  subset containing $A$.  This is called the \emph{closed convex hull}
  $\clconv A$ of $A$, and is equal to $cl (\conv A)$.
\end{defprop}
\begin{proof}
  Let $A$ be a convex subset of $\B$, and let us imagine that there
  were two points $x, y \in cl (A)$ and $a \in [0, 1]$ such that
  $x +_a y$ is not in $cl (A)$.  Let $f \colon \B \times \B \to \B$ be
  the map defined by $f (x', y') \eqdef x' +_a y'$.  This is a
  separately continuous map, and $f (x, y)$ is in the (open)
  complement $V$ of $cl (A)$.  Since $f$ is continuous in its second
  argument, $f (x, \_)^{-1} (V)$ is an open neighborhood of $y$, hence
  intersects $cl (A)$.  It must therefore also intersect $A$, say at
  $y'$.  Then $f (x, y')$ is in $V$.  Since $f$ is continuous in its
  first argument, $f (\_, y')^{-1} (V)$ is an open neighborhood of
  $x$, hence intersects $cl (A)$, hence also $A$, say at $x'$.  We
  have obtained that $f (x', y')$ is in $V$.  This is impossible,
  since $f (x', y')$ is in $A$, owing to the fact that $A$ is convex.

  It follows that, for every subset $A$ of $\B$, $cl (\conv A)$ is not
  only closed but also convex.  Every closed convex subset $B$ of $\B$
  containing $A$ must contain $\conv A$ (since $B$ is convex), and
  therefore also $cl (\conv A)$ (since $B$ is closed).  This shows
  that $cl (\conv A)$ is the smallest closed convex set containing
  $A$.  \qed
\end{proof}

\begin{lemma}
  \label{lemma:clconv:img}
  For every affine continuous map $f \colon \Alg \to \B$ between
  semitopological barycentric algebras, for every subset $A$ of
  $\Alg$, $f [\clconv A] \subseteq \clconv {f [A]}$.
\end{lemma}
\begin{proof}
  Since $f$ is continuous, $f [\clconv A] = f [cl (\conv A)] \subseteq
  cl (f [\conv A])$.  By Lemma~\ref{lemma:conv:img}, $f [\conv A] =
  \conv {f [A]}$, and then $cl (\conv {f [A]}) = \clconv {f [A]}$.  \qed
\end{proof}

\subsection{Saturated convex hulls}
\label{sec:satur-conv-hulls}

\begin{defprop}[Saturated convex hull]
  \label{defprop:upc:conv}
  Let $\B$ be a preordered (e.g., semitopological) barycentric
  algebra.  For every convex subset $A$ of $\C$, $\upc A$ is convex.

  For every subset $A$ of $\C$, there is a smallest upwards-closed
  convex subset containing $A$, called the \emph{saturated convex
    hull} of $A$, and this is $\upc \conv A$.
\end{defprop}
\begin{proof}
  If $A$ is convex and $x, y \in \upc A$, then $x \geq x'$ and
  $y \geq y'$ for some $x', y' \in A$.  For every $a \in {]0, 1[}$,
  $x +_a y$ is above $x' +_a y'$ since $+_a$ is monotonic;
  $x' +_a y' \in A$ since $A$ is convex, so $x +_a y$ is in $\upc A$.

  By the first part of the lemma, for every subset $A$ of $\C$,
  $\upc \conv A$ is upwards-closed and convex.  We consider any
  upwards-closed convex subset $B$ of $\C$ containing $A$.  Then $B$
  contains $\conv A$ since $B$ is convex, hence $\upc \conv A$ since
  $B$ is upwards-closed.  \qed
\end{proof}

\begin{lemma}
  \label{lemma:upconv:img}
  For every affine monotonic map $f \colon \Alg \to \B$ between
  preordered barycentric algebras (in particular for every affine
  continuous map between semitopological barycentric algebras), for
  every subset $A$ of $\Alg$,
  $f [\upc \conv A] \subseteq \upc \conv {f [A]}$.
\end{lemma}
\begin{proof}
  Since $f$ is monotonic,
  $f [\upc \conv A] \subseteq \upc f [\conv A]$.  We conclude since
  $f [\conv A] = \conv {f [A]}$ by Lemma~\ref{lemma:conv:img}.  \qed
\end{proof}

\begin{proposition}
  \label{prop:conv:union:compact}
  Let $\B$ be a topological barycentric algebra, and $K_1$, \ldots,
  $K_n$ be finitely many convex compact subsets of $\B$.  The convex
  hull $\conv (K_1 \cup \cdots \cup K_n)$ is convex and compact, and
  the upwards-closed convex hull
  $\upc \conv (K_1 \cup \cdots \cup K_n)$ is convex and compact
  saturated.
\end{proposition}
\begin{proof}
  Removing those sets $K_i$ that are empty if necessary, we may assume
  that every $K_i$ is non-empty.  By Lemma~\ref{lemma:conv:union},
  $\conv {(K_1 \cup \cdots \cup K_n)}$ is the image of the compact set
  $\Delta_n \times \prod_{i=1}^n K_i$ ($\Delta_n$ is compact by
  Lemma~\ref{lemma:Delta:scott}) under the map
  $(a_1, \cdots, \allowbreak a_n, x_1, \cdots, x_n) \mapsto
  \sum_{i=1}^n a_i \cdot x_i$ of Definition and
  Proposition~\ref{prop:bary:1}, which is continuous by
  Proposition~\ref{prop:bary:1:cont}; hence
  $\conv {(K_1 \cup \cdots \cup K_n)}$ is compact.  The remaining
  assertions are clear.  \qed
\end{proof}

\begin{corollary}
  \label{corl:conv:fin}
  For every finite subset $E$ of a topological barycentric algebra (in
  particular, of a topological cone), $\conv E$ is convex and compact,
  and $\upc \conv E$ is convex and compact saturated.
\end{corollary}
\begin{proof}
  Letting $E \eqdef \{x_1, \cdots, x_n\}$, we apply
  Proposition~\ref{prop:conv:union:compact} to the sets
  $K_1 \eqdef \{x_1\}$, \ldots, $K_n \eqdef \{x_n\}$.  \qed
\end{proof}

\begin{remark}
  \label{rem:upconv:compact:no}
  The sets $K_i$ in Proposition~\ref{prop:conv:union:compact} have to
  be compact \emph{and} convex already.  The following example shows
  that the convex hull of a compact set $K$ is \emph{not} compact in
  general, even in a topological cone.
\end{remark}

\begin{example}
  \label{exa:conv:compact:no}
  Let $X$ be a topological space, and ${(x_n)}_{n \in \nat}$ be a
  sequence converging to some point $x_\omega$ in $X$.  We assume
  that: $(a)$ for every $n \in \nat$, there is an open neighborhood
  $U_n$ of $x_n$ that does not contain any $x_m$ with $m \neq n$, in
  particular not $x_\omega$; $(b)$ for every $n \in \nat$, there is an
  open neighborhood $V_n$ of $x_\omega$ that does not contain $x_0$,
  \ldots, $x_n$.  (For example, $X \eqdef \real$ or $X$ equal to the
  compact Hausdorff space $[-1, 1]$ with its usual metric topology,
  $x_n \eqdef 1/2^n$, $x_\omega \eqdef 0$,
  $U_n \eqdef {]3/4 \cdot 1/2^n, 3/2\cdot 1/2^n[}$,
  $V_n \eqdef {]-1, 3/4 \cdot 1/2^n[}$.)  Let $\C$ be the topological
  cone $\Val X$, and
  $K \eqdef \{\delta_{x_n} \mid n \in \nat\} \cup
  \{\delta_{x_\omega}\}$.  Then $K$ is compact, but $\conv K$ is not.
\end{example}
\begin{proof}
  In order to show that $K$ is compact, we use Alexander's subbase
  lemma, which says that $K$ is compact if and only if for every open
  cover of $K$ by subbasic open sets has a finite subcover
  \cite[Theorem 4.4.29]{JGL-topology}.  Here, such an open cover is of
  the form ${([W_i > r_i])}_{i \in I}$, where each $W_i$ is open in
  $X$ and $r_i > 0$.  Since $\delta_{x_\omega} \in K$, there is an
  index $i \in I$ such that $\delta_{x_\omega} \in [W_i > r_i]$,
  namely such that $x_\omega \in W_i$ and $r_i < 1$.  Since
  ${(x_n)}_{n \in \nat}$ converges to $x_\omega$ in $X$, there is an
  $n_0 \in \nat$ such that $x_n \in U_i$ for every $n \geq n_0$, so
  that every $\delta_{x_n}$ is in $[W_i > r_i]$ for $n \geq n_0$.  The
  remaining finitely many continuous valuations $\delta_{x_n}$,
  $n < n_0$, can be covered by finitely many of the remaining open
  sets in the original cover.

  By Definition and Lemma~\ref{deflem:convex:simpleval}, the elements
  of $\conv K$ are the simple probability valuations of the form
  $\sum_{n \in N} a_n \delta_{x_n}$ with
  $N \subseteq \nat \cup \{\omega\}$ finite, $a_n > 0$ and
  $\sum_{n \in N} a_n=1$.

  Now let us assume that $\conv K$ is compact.  We consider the
  sequence of elements
  $\nu_k \eqdef \sum_{i=0}^k \frac 1 {2^{i+1}} \delta_{x_n} + \frac 1
  {2^{k+1}} \delta_{x_\omega}$, $k \in \nat$.  Each $\nu_k$ is in
  $\conv K$, because
  $\sum_{i=0}^k \frac 1 {2^{i+1}} + \frac 1 {2^{k+1}} = 1$.  Since
  $\conv K$ is compact, this sequence must have a cluster point
  $\nu \in \conv K$, namely, every open neighborhood of $\nu$ must
  contain $\nu_k$ for infinitely many indices $k \in \nat$
  \cite[Proposition 4.7.22]{JGL-topology}.  We have characterized the
  elements of $\conv K$ above: so $\nu$ must be of the form
  $\sum_{n \in N} a_n \delta_{x_n}$ with
  $N \subseteq \nat \cup \{\omega\}$ finite, $a_n > 0$ and
  $\sum_{n \in N} a_n=1$.

  If $\omega \in N$, then $a_\omega > 0$.  Let us fix
  $r \in {]0, a_\omega[}$.  For every $n \in \nat$, $\nu$ is in
  $[V_n > r]$.  By definition of a cluster point, there are infinitely
  many values of $k$ such that $\nu_k \in [V_n > r]$.  But, taking
  $k > n$,
  $\nu_k (V_n) = \sum_{n+1\leq i \leq k} \frac 1 {2^{i+1}} + \frac 1
  {2^{k+1}}$ since none of the points $x_0$, \ldots, $x_n$ are in
  $V_n$ by $(b)$.  Therefore $\nu_k (V_n) = \frac 1 {2^{n+1}}$.  This is
  impossible since $\nu_k (V_n) > r$ for every $n \in \nat$.
  Therefore $\omega$ is not in $N$.  In other words,
  $N \subseteq \nat$.

  For every $n \in N$, for every $r \in {]0, a_n[}$, $\nu$ is in
  $[U_n > r]$, so by definition of a cluster point, $\nu_k (U_n) > r$
  for infinitely many values of $k$.  Taking $k > n$,
  $\nu_k (U_n) = \frac 1 {2^{n+1}} $ since $U_n$ contains $x_n$ but
  not $x_m$ for any $m \neq n$ by $(a)$.  This implies that for every
  $n \in N$, for every $r \in {]0, a_n[}$, $\frac 1 {2^{n+1}} > r$, in
  other words, that $a_n \leq \frac 1 {2^{n+1}}$ for every $n \in N$.
  Since $\sum_{n \in N} a_n=1$, it follows that
  $\sum_{n \in N} \frac 1 {2^{n+1}} \geq 1$, and that is impossible
  since $N$ is finite.  We have reached a contradiction, and we must
  conclude that $\conv K$ is not compact.  \qed
\end{proof}

\section{Local convexity}
\label{sec:local-convexity}

In semitopological cones, there are at least three notions of local
convexity: local convexity in the sense of Keimel \cite[Definition
4.8]{Keimel:topcones2}, \emph{weak} local convexity or local convexity
in the sense of Heckmann \cite[Section~6.3]{heckmann96}, and local
linearity \cite[Definition 3.9]{GLJ:Valg}.  The author does not know
whether local convexity and weak local convexity are different
notions on semitopological cones, or whether there are any non-locally
convex semitopological cones at all.  An example of a locally convex,
non-locally linear topological cone is given in \cite[Example 3.15]{GLJ:Valg}.

There seems to be even more notions of local convexity in
semitopological barycentric algebras.  We will attempt to focus on the
closest analogues of the above three notions.  On top of that, we will
also consider the natural adaptation of Keimel's notion of local
convex-compactness \cite[Definition~4.8]{Keimel:topcones2} to
semitopological barycentric algebras.

\subsection{Locally convex semitopological barycentric algebras}
\label{sec:locally-conv-semit}

\begin{definition}[Locally convex]
  \label{defn:loc:convex}
  A semitopological cone, and more generally a semitopological
  barycentric algebra is \emph{locally convex} if and only if its
  topology has a base of convex open subsets.
\end{definition}
Equivalently, a semitopological barycentric algebra $\B$ is locally
convex if and only if for every $x \in \B$, every open neighborhood of
$x$ contains a convex open neighborhood of $x$.

Yet another equivalent characterization is: a semitopological
barycentric algebra is locally convex if and only if its topology has
a \emph{sub}base of convex open subsets.  This is because, given such
a subbase, the finite intersections of subbasic open sets form a base,
and are all convex by Lemma~\ref{lemma:convex:inter}.

\begin{example}
  \label{exa:VX:loc:convex}
  For every space $X$, $\Val X$ and $\Val_b X$ are locally convex
  topological cones and $\Val_1 X$ and $\Val_{\leq 1} X$ are locally
  convex topological barycentric algebras.  Local convexity stems from
  the fact that the subbasic open set $[h > r]$ is convex for every
  $h \in \Lform X$ and for every $r \in \Rp \diff \{0\}$.
  Additionally, $\Val X$ and $\Val_b X$ are topological by
  Example~\ref{exa:VX:top:cone}, and $\Val_1 X$ and $\Val_{\leq 1} X$
  are topological, being convex subspaces of $\Val X$, thanks to
  Lemma~\ref{lemma:bary:in:cone}.  Similarly, $\Val_\fin X$
  and $\Val_\pw X$ are locally convex topological cones.
\end{example}

\begin{example}
  \label{exa:LXp:loc:convex}
  For every space $X$, let $\Lform_\pw X$ be $\Lform X$ with the
  topology of pointwise convergence.  This topology has subbasic open
  sets $\{h \colon X \to \creal \mid h (x) > r\}$ where $x \in X$ and
  $r \in \Rp \diff \{0\}$, and they are all convex, so $\Lform_\pw X$
  is locally convex.
\end{example}

\begin{example}
  \label{exa:lattice:loc:convex}
  Every sup-semilattice $X$, seen as a semitopological barycentric
  algebra (Example~\ref{exa:lattice:top:bary:alg}), is locally convex.
  In fact, \emph{every} Scott-open subset $U$ of $X$ is convex, since
  the barycenters of any two points $x$ and $y$ in $U$ are $x$, $y$,
  and $x \vee y$, which are all in $U$, because $U$ is upwards-closed.
\end{example}

\begin{example}
  \label{exa:B-:loc:convex}
  The topological barycentric algebra $\B^-$ of
  Example~\ref{exa:bary:alg:notembed} is locally convex.  In fact, all
  its open subsets are convex.
\end{example}

The following result gives more examples of locally convex barycentric
algebras, and is a direct extension of a similar result on c-cones
\cite[Lemma~6.12]{Keimel:topcones2}.  The argument is due to
J. D. Lawson.  The same result appears in
\cite[Proposition~6.3]{heckmann96}.  We have introduced s-barycentric
algebras in Example~\ref{exa:cont:s-baryalg}.
\begin{proposition}
  \label{prop:dcone:loc:convex}
  Every semitopological barycentric algebra whose underlying
  topological space is a c-space is a locally convex topological
  barycentric algebra.  In particular:
  \begin{enumerate}
  \item every c-cone is a locally convex topological cone;
  \item every pointed continuous s-barycentric algebra
    is a locally convex pointed topological barycentric algebra in its
    Scott topology.
  \end{enumerate}
\end{proposition}
\begin{proof}
  Let $\B$ be a semitopological barycentric algebra and a c-space.
  By Proposition~\ref{prop:c:bary:alg}, $\B$ is topological.

  1.  Let $x$ be a point of $\B$ and $U$ be an open neighborhood of
  $x$ in $\B$.  Since $\B$ is a c-space, there is a point $x_1 \in U$
  such that $x \in \interior {\upc x_1}$.  Since $x_1 \in U$ and $\B$
  is a c-space, there is a point $x_2 \in U$ such that
  $x_1 \in \interior {\upc x_2}$.  We continue this way and we obtain
  a chain of points $x_n$ in $U$, $n \geq 1$, with
  $x_n \in \interior {\upc {x_{n+1}}}$.  Let
  $V \eqdef \bigcup_{n \in \nat} \interior {\upc {x_n}}$.  $V$ is open,
  contains $x$, and is included in $U$.  $V$ is also equal to
  $\bigcup_{n \in \nat} \upc x_n$: the inclusion
  $V \subseteq \bigcup_{n \in \nat} \upc x_n$ is clear, while conversely
  $\bigcup_{n \in \nat} \upc x_n \subseteq \bigcup_{n \in \nat} \interior
  {\upc {x_{n+1}}} \subseteq V$.  Since each of the sets $\upc x_n$ is
  convex, $V$ is convex as well by Lemma~\ref{lemma:convex:inter}.

  2.  Every continuous s-barycentric algebra is a topological
  barycentric algebra (see Example~\ref{exa:cont:s-baryalg}), and a
  c-space in its Scott topology.  \qed
\end{proof}

We conclude this section with an alternate characterization of local
convexity for cones, and of (pointed) semitopological barycentric
algebras whose free (pointed) semitopological cone is locally convex.
\begin{proposition}
  \label{prop:cone:locconvex:alt}
  A semitopological cone $\C$ is locally convex if and only if every
  positively homogeneous lower semicontinuous map
  $h \colon \C \to \creal$ is a pointwise supremum of superlinear
  lower semicontinuous maps.
\end{proposition}
\begin{proof}
  Let $\C$ be locally convex, and $h \colon \C \to \creal$ be
  positively homogeneous and lower semicontinuous.  Let $\mathcal C$
  be the collection of superlinear lower semicontinuous maps from $\C$
  to $\creal$ that are pointwise below $h$.  We claim that the
  pointwise supremum $\sup \mathcal C$ is equal to $h$.  Otherwise,
  there is a point $x \in \C$ such that
  $\sup_{g \in \mathcal C} g (x) < h (x)$.  Let $a \in \Rp$ be such
  that $\sup_{g \in \mathcal C} g (x) < a < h (x)$.  The set
  $h^{-1} (]a, \infty])$ is open in $\C$, and is proper since it does
  not contain $0$; indeed, $h (0) = 0$ is not in $]a, \infty]$.  Since
  $\C$ is locally convex, there is a convex open neighborhood $U$ of
  $x$ included in $h^{-1} (]a, \infty])$.  Since the latter is proper,
  so is $U$.  Therefore its upper Minkowski functional $M^U$ is a
  superlinear lower semicontinuous map from $\C$ to $\creal$.  Since
  $U \subseteq h^{-1} (]a, \infty]) = ((1/a) \cdot h)^{-1} (]1,
  \infty])$ and $(1/a) \cdot h$ is positively homogeneous and lower
  semicontinuous, the order isomorphism between positively homogeneous
  maps and proper open subsets tells us that $M^U \leq (1/a) \cdot h$.
  Now $g \eqdef a \cdot M^U$ is superlinear, lower semicontinuous, and
  below $h$, hence in $\mathcal C$.  But $M^U (x) > 1$ since
  $x \in U$, so $g (x) > a$.  This contradicts the fact that
  $\sup_{g \in \mathcal C} g (x) < a$.

  Conversely, let us assume that every positively homogeneous lower
  semicontinuous map $h \colon \C \to \creal$ is a pointwise supremum
  of superlinear lower semicontinuous maps.  For every $x \in \C$, for
  every open neighborhood $U$ of $x$, we look for a convex open
  neighborhood of $x$ included in $U$.  If $U = \C$, we simply take
  $\C$ itself.  Otherwise, $U$ is proper, hence $M^U$ is positively
  homogeneous and lower semicontinuous.  Additionally, $M^U (x) > 1$.
  By assumption, $M^U$ is a pointwise supremum of superlinear lower
  semicontinuous maps, so one of them, call it $g$, is such that
  $g (x) > 1$; we note that we must have $g \leq M^U$.  Then
  $g^{-1} (]1, \infty])$ is (proper) convex and open, contains $x$ and
  is included in $(M^U)^{-1} (]1, \infty]) = U$.  Hence $\C$ is
  locally convex.  \qed
\end{proof}

\begin{proposition}
  \label{prop:tscope:loc:convex}
  For every pointed semitopological barycentric algebra $\B$, for
  every $\alpha \in {]0, 1[}$, the following are equivalent:
  \begin{enumerate}
  \item the free semitopological cone $\tscope_\alpha (\B)$ is locally
    convex;
  \item every $\Iu$-homogeneous lower semicontinuous map
    $h \colon \B \to \creal$ is a pointwise supremum of superlinear
    lower semicontinuous maps.
  \end{enumerate}
\end{proposition}
\begin{proof}
  $1 \limp 2$.  Let $h \colon \B \to \creal$ be $\Iu$-homogeneous and
  lower semicontinuous.  By
  Theorem~\ref{thm:bary:alg:conify:ord:pointed:semitop}, item~3,
  $h^{\scope\alpha}$ is positively homogeneous.  By
  Proposition~\ref{prop:cone:locconvex:alt} and item~1,
  $h^{\scope\alpha}$ is a pointwise supremum of superlinear lower
  semicontinuous maps $g_i$, $i \in I$.  Then
  $h = h^{\scope\alpha} \circ \etaca_{\B} = \sup_{i \in I} g_i \circ
  \etaca_{\B}$, and each map $g_i \circ \etaca_{\B}$ is superlinear
  and lower semicontinuous, since $\etaca_{\B}$ is itself linear and
  continuous, by
  Theorem~\ref{thm:bary:alg:conify:ord:pointed:semitop}, item~2.

  $2 \limp 1$.  Let $g$ be a positively homogeneous lower
  semicontinuous map from $\tscope_\alpha (\B)$ to $\creal$.  Then
  $g \circ \etaca_{\B}$ is $\Iu$-homogeneous and lower semicontinuous.
  By item~2, it is a supremum of superlinear lower semicontinuous maps
  $h_i \colon \B \to \creal$, $i \in I$.  Hence
  $g \circ \etaca_{\B} = \sup_{i \in I} h_i = \sup_{i \in I}
  (h_i^{\scope\alpha} \circ \etaca_{\B}) = (\sup_{i \in I}
  h_i^{\scope\alpha}) \circ \etaca_{\B}$.  For every $\Iu$-homogeneous
  lower semicontinuous function $f$, $f^{\scope\alpha}$ is unique , so
  $g = \sup_{i \in I} h_i^{\scope\alpha}$.  For each $i \in I$, since
  $h_i$ is superlinear and lower semicontinuous, so is
  $h_i^{\scope\alpha}$, by
  Theorem~\ref{thm:bary:alg:conify:ord:pointed:semitop}, item~3.  \qed
\end{proof}

\begin{proposition}
  \label{prop:conify:loc:convex}
  For every semitopological barycentric algebra $\B$, the free
  semitopological cone $\conify (\B)$ is locally convex if and only if
  every semi-concave lower semicontinuous map $h \colon \B \to \creal$
  is a pointwise supremum of concave lower semicontinuous maps.
\end{proposition}
\begin{proof}
  We use Remark~\ref{rem:bary:alg:f*}.  The positively homogeneous
  lower semicontinuous maps from $\conify (\B)$ to $\creal$ are
  exactly the maps $h^\cext$ where $h$ is semi-concave and lower
  semicontinuous.  By Proposition~\ref{prop:cone:locconvex:alt},
  $\conify (\B)$ is locally convex if and only if every such map is a
  pointwise supremum of superlinear lower semicontinuous maps $q_i$
  from $\conify (\B)$ to $\creal$, where $i$ ranges over some index
  set $I$.  If so, then by Remark~\ref{rem:bary:alg:f*},
  $q_i = g_i^\cext$ for some unique concave lower semicontinuous map
  $g_i \colon \B \to \creal$, and
  $h = h^\cext \circ \etac_{\B} = \sup_{i \in I} g_i^\cext \circ
  \etac_{\B} = \sup_{i \in } g_i$.  Conversely, if $h$ is a pointwise
  supremum of concave lower semicontinuous maps $g_i$, $i \in I$, then
  $h^\cext = \sup_{i \in I} g_i^\cext$, so $h^\cext$ is a pointwise
  supremum of superlinear lower semicontinuous maps; as $h$ is
  arbitrary, $\conify (\B)$ is locally convex.  \qed
\end{proof}

\subsection{Locally affine semitopological barycentric algebras}
\label{sec:locally-affine-semit}

A semitopological cone $\C$ is \emph{locally linear} if and only if
its topology has a subbase of open half-spaces \cite[Definition
3.9]{GLJ:Valg}.  We may restrict to proper open half-spaces, and this
will remain a subbase.  Because of the order-isomorphism between
proper open half-spaces and linear lower semicontinuous maps, we
obtain the following, which justifies the name.
\begin{fact}
  \label{fact:loclin}
  A semitopological cone $\C$ is locally linear if and only if its
  topology has a subbase of open sets of the form
  $\Lambda^{-1} (]1, \infty])$, where $\Lambda$ ranges over the linear
  lower semicontinuous maps from $\C$ to $\creal$.
\end{fact}
This leads us to the following generalization of the notion to
barycentric algebras.
\begin{definition}[Locally affine, locally linear]
  \label{defn:loc:affine}
  A pointed semitopological barycentric algebra $\B$ is \emph{locally
    affine} if and only if its topology has a subbase of open sets of
  the form $\Lambda^{-1} (]1, \infty])$, where $\Lambda$ ranges over
  the affine lower semicontinuous maps from $\B$ to $\creal$.

  A pointed semitopological barycentric algebra $\B$ is \emph{locally
    linear} if and only if its topology has a subbase of open sets of
  the form $\Lambda^{-1} (]1, \infty])$, where $\Lambda$ ranges over
  the linear lower semicontinuous maps from $\B$ to $\creal$.
\end{definition}

\begin{remark}
  \label{rem:locaff=loclin}
  For a pointed semitopological barycentric algebra $\B$ (or a cone),
  ``locally affine'' and ``locally linear'' mean the same thing; we
  have kept the two denominations separate because ``locally linear''
  is the standard denomination on cones, but would not make sense on
  non-pointed barycentric algebras.  In order to see this, in one
  direction every linear map is affine.  In the other direction, if
  $\Lambda$ is affine and lower semicontinuous, then
  $\Lambda^{-1} (]1, \infty])$ is equal to the whole of $\B$ (if
  $\Lambda (\bot) > 1$) or else can be expressed as a union of sets of
  the form ${\Lambda'}^{-1} (]1, \infty])$ where $\Lambda'$ is linear
  (i.e., strict and affine) and lower semicontinuous.  If
  $\Lambda (\bot) < 1$,
  $\Lambda^{-1} (]1, \infty]) = {\Lambda'}^{-1} (]1, \infty])$ where
  $\Lambda'$ is the linear lower semicontinuous map
  $x \in \B \mapsto \frac {\Lambda (x) - \Lambda (\bot)} {1 - \Lambda
    (\bot)}$.  In the special case where $\Lambda (\bot)=1$,
  $\Lambda^{-1} (]1, \infty]) = \bigcup_{r > 1} (\frac 1 r \cdot
  \Lambda)^{-1} (]1, \infty])$ is therefore equal to
  $\bigcup_{r > 1} ((\frac 1 r \cdot \Lambda)')^{-1} (]1, \infty])$,
  with $((\frac 1 r \cdot \Lambda)')$ mapping every $x \in \B$ to
  $\frac {\Lambda (x) - 1} {r - 1}$.
\end{remark}

\begin{example}
  \label{exa:VX:loclin}
  For every topological space $X$, $\Val X$, $\Val_b X$, $\Val_\pw X$,
  $\Val_\fin X$ are locally linear.  Indeed, every subbasic open set
  $[h > r]$ is an open half-space.
\end{example}

\begin{example}
  \label{exa:Vleq1X:loclin}
  Given a pointed semitopological barycentric algebra $\B$ that embeds
  in another one $\Alg$ (e.g., in a semitopological cone) through a
  linear topological embedding, if $\Alg$ is locally linear, then so
  is $\B$.  Hence, for every topological space $X$, the $T_0$ pointed
  semitopological barycentric algebras $\Val_{\leq 1} X$,
  $\Val_{\leq 1, \pw} X$ and $\Val_{\leq 1, \fin} X$, whose inclusion
  maps into $\Val X$, $\Val_\pw X$ and $\Val_\fin X$ respectively are
  linear topological embeddings, are locally linear.
\end{example}

\begin{example}
  \label{exa:V1X:loclin}
  Given a semitopological barycentric algebra $\B$ that embeds in
  another one $\Alg$ (e.g., in a semitopological cone) through an
  affine topological embedding, if $\Alg$ is locally affine, then so
  is $\B$.  Hence, for every topological space $X$, the $T_0$
  topological barycentric algebras $\Val_1 X$, $\Val_{1, \pw} X$ and
  $\Val_{1, \fin} X$, whose inclusion maps into $\Val X$, $\Val_\pw X$
  and $\Val_\fin X$ respectively are affine topological embeddings,
  are locally affine.
\end{example}

\begin{example}
  \label{exa:LXp:loclin}
  For every space $X$, $\Lform_\pw X$ (see
  Example~\ref{exa:LXp:loc:convex}) is locally linear, since the
  subbasic open set $\{h \colon X \to \creal \mid h (x) > r\}$ is an
  open half-space.
\end{example}

By definition, every locally affine semitopological barycentric
algebra is linearly separated in the following sense.
\begin{definition}[Linearly separated]
  \label{defn:convexT0}
  A semitopological barycentric algebra $\B$ is \emph{linearly
    separated} if and only if for all $x, y \in \B$ such that
  $x \not\leq y$, there is an affine lower semicontinuous map
  $\Lambda$ such that $\Lambda (x) > \Lambda (y)$.
\end{definition}
This is called \emph{linearly} separated and not \emph{affinely}
separated because the notion was originally intended on cones, where
we can require $\Lambda$ to be linear, not affine, as the following
lemma shows.  Keimel \cite[Section~7]{Keimel:topcones2} calls
\emph{convex-$T_0$} the $T_0$ semitopological cones that are linearly
separated; and indeed, all locally convex semitopological cones, not
just the locally linear ones, are linearly separated: this is stated
for $T_0$ semitopological cones in
\cite[Corollary~9.3]{Keimel:topcones2}, but is valid even without the
$T_0$ requirement.
\begin{lemma}
  \label{lemma:convexT0}
  \begin{enumerate}
  \item A pointed semitopological barycentric algebra, in particular a
    semitopological cone, $\B$ is linearly separated if and only if
    for all $x, y \in \B$ such that $x \not\leq y$, there is a linear
    lower semicontinuous map $\Lambda$ such that
    $\Lambda (x) > \Lambda (y)$.
  \item A semitopological cone $\C$ is linearly separated if and only
    if for all $x, y \in \C$, $x \leq y$ if and only if every open
    half-space that contains $x$ also contains $y$.
  \end{enumerate}
\end{lemma}
\begin{proof}
  1.  Let $\B$ be a pointed semitopological barycentric algebra.  The
  if direction is clear, since every linear map is affine.
  Conversely, let us assume that $\B$ is linearly separated, and let
  $x, y \in \B$ be such that $x \not\leq y$.  There is a lower
  semi-continuous affine map $\Lambda$ such that
  $\Lambda (x) > \Lambda (y)$.  It cannot be the case that
  $\Lambda (\bot) = \infty$, otherwise $\Lambda$ would be the constant
  function with value $\infty$, since $\Lambda$ is lower
  semicontinuous hence monotonic, and then
  $\infty = \Lambda (y) < \Lambda (x)$ would be impossible.
  Therefore the function $\Lambda'$ defined by $\Lambda' (x) \eqdef
  \Lambda (x) - \Lambda (\bot)$ is well-defined.  $\Lambda'$ is
  lower semicontinuous, affine, and strict, hence it is linear by
  Proposition~\ref{prop:strict:aff}.  Additionally, $\Lambda' (x) >
  \Lambda' (y)$ (even when $\Lambda (x) = \infty$; note that $\Lambda
  (y) < \infty$ since $\Lambda (y) < \Lambda (x)$).

  2.  This follows from item~1 and the fact that proper half-spaces
  $H$ are in bijection with linear lower semicontinuous maps.  For the
  unique non-proper half-space $\C$, the implication
  $x \in \C \limp y \in \C$ is vacuously true.  \qed
\end{proof}

\begin{example}
  \label{exa:VX:convexT0}
  For every topological space $X$, and whether $\rast$ is nothing,
  ``$\leq 1$'', ``$1$'' or ``$b$'', $\Val_{\rast} X$ is linearly
  separated, and similarly with $\Val_{\rast, \fin} X$ and
  $\Val_{\rast, \pw} X$.  Indeed, if $\mu \not\leq \nu$, then there is
  an open subset $U$ of $X$ such that $\mu (U) > \nu (U)$.  We define
  $\Lambda \colon \Val_{\rast} X \to \creal$ by
  $\Lambda (\xi) \eqdef \xi (U)$: this is an affine lower
  semicontinuous map, and $\Lambda (\mu) > \Lambda (\nu)$.
\end{example}

\begin{example}
  \label{exa:subspace:convexT0}
  Since the restriction of an affine map to any convex subset is
  affine, every convex subspace of a linearly separated
  semitopological barycentric algebra is itself linearly separated.
  This is another way of deriving the fact that the spaces
  $\Val_{\rast} X$, $\Val_{\rast, \fin} X$ and $\Val_{\rast, \pw} X$
  are linearly separated, from just the knowledge that $\Val X$ is .
\end{example}

\begin{example}
  \label{exa:LX:convexT0}
  For every topological space $X$, $\Lform X$ is linearly separated,
  and that is not only true for the Scott topology, but also for the
  coarser compact-open topology and for the even coarser topology of
  pointwise convergence.  Indeed, if $h \not\leq h'$ in $\Lform X$,
  then $h (x) > h' (x)$ for some $x \in X$.  The map
  $\Lambda \colon \Lform X \to \creal$ defined by
  $\Lambda (k) \eqdef k (x)$ for every $k \in \Lform X$ is linear and
  lower semicontinuous, whatever the topology we take among the three
  listed above on $\Lform X$, and $\Lambda (h) > \Lambda (h')$.
\end{example}

Local linearity and linear separation play related roles through the
following construction, a direct generalization of the topological
duals of cones \cite[Example~5.5]{Keimel:topcones2}.
\begin{deflem}
  \label{deflem:B'*}
  For every semitopological barycentric algebra $\B$ (resp.\ pointed
  semitopological barycentric algebra $\B$), there is a topological
  cone $\B^\astar$ (resp.\ $\B^*$) of affine (resp.\ linear) lower
  semicontinuous maps from $\B$ to $\creal$, with topology generated
  by subbasic open sets
  $[x > r] \eqdef \{\Lambda \in \C \mid \Lambda (x) > r\}$ where $x$
  ranges over $\creal$ and $r$ over $\Rp \diff \{0\}$.  Then:
  \begin{enumerate}
  \item $\B^{\astar*}$ (resp.\ $\B^{**}$) are locally linear topological
    cones;
  \item the map $\etaps_{\B} \colon \B \to \B^{\astar*}$ (resp.\
    $\etass_{\B} \colon \B \to \B^{**}$) defined by
    $\etaps_{\B} (x) (\Lambda) \eqdef \Lambda (x)$ for every
    $\Lambda \in \B^\astar$ (resp.\
    $\etass_{\B} (x) (\Lambda) \eqdef \Lambda (x)$ for every
    $\Lambda \in \B^*$) is affine (resp.\ linear) and continuous.
  \end{enumerate}
  Additionally, $\etaps_{\B}$ (resp.\ $\etass_{\B}$) is:
  \begin{enumerate}[resume]
  \item a preorder embedding if and only if $\B$ is linearly
    separated;
  \item full if and only if $\B$ is locally affine (resp.\ linear);
  \item an affine topological embedding if and only if $\B$ is locally
    affine (resp.\ locally linear) and $T_0$.
  \end{enumerate}
\end{deflem}
\begin{proof}
  1.  Let $\C \eqdef \B^\astar$ (resp.\ $\B^*$).  With pointwise
  addition and scalar multiplication, $\C$ is a cone.  The inverse
  image of $[x > r]$ by addition is
  $\{(\Lambda, \Lambda') \in \C \times \C \mid \Lambda (x) + \Lambda'
  (x) > r\} = \{(\Lambda, \Lambda') \in \C \times \C \mid \text{for
    some } a, b \in \Rp\text{ such that }a+b>r, \Lambda (x) > a,
  \Lambda' (x) > b\} = \bigcup_{a, b \in \Rp, a+b > r} [x > a] \times
  [x > b]$, so addition is (jointly) continuous.  For every
  $a \in \Rp \diff \{0\}$, the inverse image of $[x > r]$ by
  $(a \cdot \_)$ is $[x>r/a]$, so $(a \cdot \_)$ is continuous; it is
  also continuous when $a=0$, since it is a constant map in that case.
  For every $\Lambda \in \C$, the inverse image of $[x > r]$ by
  $(\_ \cdot \Lambda)$ is $]r / \Lambda (x), \infty]$ if
  $0 < \Lambda (x) < \infty$, empty if $\Lambda (x)=0$, and
  $]0, \infty]$ if $\Lambda (x) = \infty$.  $\C$ is therefore a
  topological cone.

  2.  In the following, we will write $\eta$ for $\etaps_{\B}$ (resp.\
  for $\etass_{\B}$) and $\C$ for $\B^\astar$ (resp.\ $\B^*$).  The function
  $\eta$ is continuous because
  $\eta^{-1} ([\Lambda > r]) = \Lambda^{-1} (]r, \infty])$ for all
  $\Lambda \in \C$ and $r \in \Rp \diff \{0\}$.  It is affine: for all
  $x, y \in \B$ and $a \in [0, 1]$, $\eta (x +_a y)$ maps every
  $\Lambda \in \C$ to
  $\Lambda (x +_a y) = a \, \Lambda (x) + (1-a) \, \Lambda (y)$ (since
  $\Lambda$ is affine)
  $= a \, \eta (x) (\Lambda) + (1-a) \, \eta (y) (\Lambda)$.
  Additionally, when $\B$ is pointed (and $\C = \B^{**}$ consists of
  linear maps), $\eta$ is strict, since
  $\eta (\bot) (\Lambda) = \Lambda (\bot) = 0$ for every
  $\Lambda \in \C$; therefore $\eta = \etass_{\B}$ is linear, thanks
  to Proposition~\ref{prop:strict:aff}.

  3.  Let $\B$ be linearly separated.  Since $\eta$ is continuous, it
  is monotonic.  In order to see that it is a preorder embedding, we
  must show that for every two points $x, y \in \B$ such that
  $x \not\leq y$, $\eta (x) \not\leq \eta (y)$.  In the non-pointed
  case, since $\B$ is linearly separated, there is a $\Lambda \in \B^\astar$
  such that $\Lambda (x) > \Lambda (y)$.  In the pointed case, we can
  take $\Lambda \in \B^*$ by Lemma~\ref{lemma:convexT0}, item~1.  In
  any case, there is a $\Lambda \in \C$ such that
  $\Lambda (x) > \Lambda (y)$, namely such that
  $\eta (x) (\Lambda) > \eta (x) (\Lambda)$.  Therefore
  $\eta (x) \not\leq \eta (y)$.

  Conversely, if $\eta$ is a preorder embedding, then for every two
  points $x, y \in \B$ such that $x \not\leq y$,
  $\eta (x) \not\leq \eta (y)$, so there is a $\Lambda \in \B^\astar$
  (in $\B^*$ is $\B$ is pointed) such that
  $\eta (x) (\Lambda) \not\leq \eta (y) (\Lambda)$.  Therefore
  $\Lambda (x) > \Lambda (y)$, so $\B$ is linearly separated.

  4.  If $\B$ is locally affine (resp.\ pointed and locally linear),
  then $\B$ has a subbase of open sets of the form
  $\Lambda^{-1} (]1, \infty])$ with $\Lambda \in \C$, and every one of
  them is the inverse image of the open set $[\Lambda > 1]$ under
  $\eta$, so $\eta$ is full.

  Conversely, if $\eta$ is full, then a subbase of the topology on
  $\B$ is given by sets of the form $\eta^{-1} (U)$, where $U$ ranges
  over any given subbase of the topology of $\B^{\astar*}$ ($\B^{**}$
  if $\B$ is pointed).  By taking the standard subbase consisting of
  sets $[\Lambda > r]$ with $\Lambda \in \B^\astar$ (resp.\ $\B^*$)
  and $r \in \Rp$, we obtain a subbase of the topology on $\B$
  consisting of sets of the form $\eta^{-1} ([\Lambda > r])
  = \{x \in \B \mid \Lambda (\eta (x)) > r\} = (\Lambda \circ
  \eta)^{-1} (r)$.  Since $\eta$ and $\Lambda$ are both affine (resp.\
  linear), $\B$ is locally linear.

  5.  If $\B$ is $T_0$, then $\B$ is locally affine if and only if
  $\eta$ is full by item~4, hence a topological embedding.
  Conversely, if $\eta$ is a topological embedding, then $\B$ is
  locally affine by item~4, and we claim that $\B$ is $T_0$.  For all
  $x, y \in \B$ such that $x \leq y$ and $y \leq x$, we have
  $\eta (x) \leq \eta (y)$ and $\eta (y) \leq \eta (x)$, namely
  $\Lambda (x) \leq \Lambda (y)$ and $\Lambda (y) \leq \Lambda (x)$
  for every $\Lambda \in \B^\astar$ (resp.\ $\B^*$); so
  $\eta (x) = \eta (y)$.  Since $\eta$ is injective, $x=y$.  \qed
\end{proof}

\begin{lemma}
  \label{lemma:loclin:bary:alg}
  Let $\B$ be a semitopological barycentric algebra (resp.\ a pointed
  semitopological barycentric algebra).  Then:
  \begin{enumerate}
  \item $\B$ is linearly separated if and only if there is an affine
    (resp.\ linear) continuous map $\eta \colon \B \to \C$ of $\B$ to
    a locally linear semitopological cone $\C$ that is a preorder
    embedding with respect to the underlying specialization orderings.
  \item $\B$ is locally affine (resp.\ locally linear) if and only if
    there is an affine (resp.\ linear) full continuous map
    $\eta \colon \B \to \C$ of $\B$ into a locally linear
    semitopological cone $\C$.
  \item $\B$ is $T_0$ and locally affine (resp.\ $T_0$ and locally
    linear) if and only if it embeds in a locally linear
    semitopological cone $\C$ through an affine (resp.\ linear)
    topological embedding.
  \end{enumerate}
\end{lemma}
\begin{proof}
  1.  If $\eta \colon \B \to \C$ is affine (resp.\ linear), continuous
  and a preorder embedding, and if $\C$ is a locally linear
  semitopological cone, then given any two points $x, y \in \B$ such
  that $x \not\leq y$, we have $\eta (x) \not\leq \eta (y)$.  Since
  $\C$ is locally linear, there is an affine lower semicontinuous map
  $\Lambda \colon \C \to \creal$ such that
  $\Lambda (\eta (x)) > \Lambda (\eta (y))$ (resp.\ and we can take
  $\Lambda$ to be linear in the pointed case, thanks to
  Lemma~\ref{lemma:convexT0}, item~1).  Then $\Lambda \circ \eta$ is
  affine (resp.\ linear) and maps $x$ to a strictly larger value than
  $y$.  It follows that $\B$ is linearly separated.  The converse
  implication is by Definition and Lemma~\ref{deflem:B'*}, item~3.

  2.  If $\eta \colon \B \to \C$ is an affine (resp.\ linear) full
  continuous map, where $\C$ is a locally linear semitopological cone,
  then every open subset of $\B$ is a union of finite intersections of
  sets of the form $\eta^{-1} (V)$ with $V$ open in $\C$, and every
  such set $V$ is a union of finite intersections of sets of the form
  $\Lambda^{-1} (]1, \infty])$ with $\Lambda \colon \C \to \creal$
  affine (resp.\ linear, using Lemma~\ref{lemma:convexT0}, item~1) and
  lower semicontinuous.  The maps $\Lambda \circ \eta$ are then affine
  (resp.\ linear) and lower semicontinuous, and since the
  corresponding sets $(\Lambda \circ \eta)^{-1} (]1, \infty])$ form a
  subbase of the topology of $\B$, $\B$ is locally affine (resp.\
  locally linear).  The converse implication is by Definition and
  Lemma~\ref{deflem:B'*}, item~4.

  3.  Immediate consequence of item~2, since full maps with a $T_0$
  domain are topological embeddings.  \qed
\end{proof}

\begin{corollary}
  \label{corl:bary:alg:loclin:embed}
  Every locally affine $T_0$ semitopological barycentric algebra is
  embeddable.  Every locally linear $T_0$ pointed semitopological
  barycentric algebra is strictly embeddable.
\end{corollary}

Local linearity does not just imply local convexity, but also
topologicity.
\begin{proposition}
  \label{prop:loclin:topo}
  Every locally affine semitopological barycentric algebra $\B$, in
  particular every locally linear pointed semitopological barycentric
  algebra, and every locally linear semitopological cone is locally
  convex and topological.
\end{proposition}
\begin{proof}
  Local convexity is obvious, since $\Lambda^{-1} (]1, \infty])$ is
  convex for every affine lower semicontinuous map
  $\Lambda \colon \B \to \creal$.  Let $f$ be the function
  $(x, a, y) \mapsto x +_a y$ from $\B \times [0, 1] \times \B$ to
  $\creal$.  For every affine (resp.\ linear) lower semicontinuous map
  $\Lambda \colon \B \to \creal$, let
  $(x, a, y) \in f^{-1} (\Lambda^{-1} (]1, \infty]))$.  Hence
  $\Lambda (x +_a y) > 1$, namely
  $a \, \Lambda (x) + (1-a) \, \Lambda (y) > 1$.  Since $\creal$ is a
  topological cone (Example~\ref {exa:cone:R}), there is an open
  neighborhood $U$ of $\Lambda (x)$ in $\creal$, an open neighborhood
  $V$ of $\Lambda (y)$ in $\creal$ and an an open neighborhood $I$ of
  $a$ in $[0, 1]$ such that $a' \, s + (1-a') \, t > 1$ for all
  $(s, a', t') \in U \times I \times V$.  In particular, for all
  $(x', a', y') \in \Lambda^{-1} (U) \times I \times \Lambda^{-1}
  (V)$, $x' +_{a'} y' \in \Lambda^{-1} (]1, \infty])$, so
  $(x', a', y') \in f^{-1} (\Lambda^{-1} (]1, \infty]))$.

  It remains to see that $\Lambda^{-1} (U)$ is open in $\B$: if $U$ is
  of the form $]r, \infty]$ with $r > 0$, then
  $\Lambda^{-1} (U) = ((1/r) \cdot \Lambda)^{-1} (]1, \infty])$, which
  is open since $(1/r) \cdot \Lambda$ is affine (resp.\ linear) and
  lower semicontinuous; if $U = ]0, \infty]$, then
  $\Lambda^{-1} (U) = \bigcup_{r > 0} \Lambda^{-1} (]r, \infty])$, and
  if $U = [0, \infty]$ then $\Lambda^{-1} (U) = \B$.  Similarly with
  $\Lambda^{-1} (V)$.  \qed
\end{proof}

\begin{remark}
  \label{rem:lattice:not:loclin}
  Local linearity is strictly stronger than local convexity.  Let us
  consider any continuous complete lattice $L$, namely any complete
  lattice that is continuous as a dcpo.  Being a pointed
  sup-semilattice, $L$ is a cone and $L_\sigma$ is a semitopological
  cone (see Example~\ref{exa:semilatt:cone}).  It is shown in
  \cite[Remark~3.15]{GLJ:Valg} that letting $L \eqdef \Open \real$,
  where $\real$ has its usual metric topology, is locally convex but
  not locally linear.
\end{remark}
We will give another example of a non-locally linear cone in
Example~\ref{exa:LQ:not:strong:refl}.

\subsection{Dual cones}
\label{sec:dual-cones}

We have introduced the topological cones $\B^\astar$ and $\B^*$ in
Definition and Lemma~\ref{deflem:B'*}.  We call the topology defined
there the \emph{weak$^*$Scott topology}.  This is the coarsest
topology that makes the map
$\Lambda \in \B^* \text{ (resp.\ }\in \B^\astar\text{) }\mapsto
\Lambda (x)$ lower semicontinuous for every $x \in \B$, and extends
the notion of dual cone $\C^*$ of a semitopological cone
\cite[Example~5.5]{Keimel:topcones2}.

\begin{remark}
  \label{rem:dual:subbase}
  Given a pointed semitopological barycentric algebra $\B$, the sets
  $[x > 1]$, $x \in \B$, already form a subbase of the weak$^*$Scott
  topology on $\B^*$, since for every $r \in \Rp \diff \{0\}$,
  $[x > r]$ can be expressed as $[(1/r) \cdot x > 1]$.
\end{remark}
\begin{lemma}
  \label{lemma:weak*}
  Let $\B$ be a (resp.\ pointed) semitopological barycentric algebra .
  The specialization ordering of $\B^\astar$ (resp.\ $\B^*$) with its
  weak$^*$Scott topology is the pointwise ordering.
\end{lemma}
\begin{proof}
  Because $\B^\astar$ (resp.\ $\B^*$) is a subspace of
  ${(\creal)}^\B$, with the product topology, whose specialization
  ordering is pointwise.  \qed
\end{proof}

\begin{example}
  \label{exa:V=C*}
  For every topological space $X$, $\Val X$ is isomorphic to
  $(\Lform X)^*$ as a topological cone \cite[Satz~8.1]{kirch93}.  In
  one direction, every $\nu \in \Val X$ is mapped to the linear lower
  semicontinuous functional $h \in \Lform X \mapsto \int h \,d\nu$.
  In the other direction, every linear lower semicontinuous functional
  $\Lambda$ is mapped to $U \in \Open X \mapsto \Lambda (\chi_U)$,
  where $\chi_U$ is the characteristic function of $U$.
\end{example}

Several properties of $\Val X$ generalize to dual cones.
\begin{proposition}
  \label{prop:C*:topological}
  For every (resp.\ pointed) semitopological barycentric algebra $\B$,
  $\B^\astar$ (resp.\ $\B^*$) is a locally linear topological cone.
  In particular, for every semitopological cone $\C$, $\C^*$ is a
  locally linear topological cone.
\end{proposition}
\begin{proof}
  $\B^\astar$ (resp.\ $\B^*$) is a topological cone by Definition and
  Lemma~\ref{deflem:B'*}.  The sets $[x > r]$ with $x \in \B$ and
  $r \in \Rp \diff \{0\}$ are open half-spaces and form a subbase of
  the topology, so $\B^*$ (resp.\ $\B^*$) is locally linear.  \qed
\end{proof}

Another property is sobriety (see \cite[Chapter~8]{JGL-topology}).  A
closed set $C$ is \emph{irreducible} if and only if it is non-empty,
and for any two closed subsets $C_1$ and $C_2$ such that
$C \subseteq C_1 \cup C_2$, $C$ is included in $C_1$ or in $C_2$.  For
every point $x \in X$, $\dc x$ is irreducible closed.  A space in
which those are the only irreducible closed subsets is called
\emph{quasi-sober}.  A \emph{sober space} is a $T_0$ quasi-sober
space.  All $T_2$ spaces, all continuous dcpos are sober.  Every
retract of a sober space is sober, every equalizer
$[f=g] \eqdef \{x \in X \mid f (x)=g(x)\}$ of two continuous maps
$f, g \colon X \to Y$ where $X$ is sober and $Y$ is $T_0$ is itself
sober.

Now $\Val X$ is sober \cite[Proposition~5.1]{heckmann96}, and this
generalizes as follows.  The proof is the same as the one Heckmann
gives for $\Val X$.
\begin{theorem}
  \label{thm:C*:sober}
  For every (resp.\ pointed) semitopological barycentric algebra $\B$,
  $\B^\astar$ (resp.\ $\B^*$) is sober.  In particular, $\C^*$ is
  sober for every semitopological cone $\C$.
\end{theorem}
\begin{proof}
  $\B^\astar$ is a subspace of $\Lform_\pw \B$ (see
  Example~\ref{exa:LXp:loc:convex}), which is sober, see
  \cite[Lemma~5.8]{Tix:bewertung} for example.  This is the subspace
  of elements $h$ satisfying the equations
  $h (x +_a y) = a\, h (x) + (1-a) \, h (y)$ for all $x, y \in \B$ and
  $a \in [0, 1]$, hence it occurs as the equalizer $[f=g]$ of the
  continuous maps $f, g \colon \Lform_\pw \B \to \creal^I$, where
  $I \eqdef \B \times [0, 1] \times \B$,
  $f (h) \eqdef {(h (x +_a y))}_{(x, a, y) \in I}$,
  $g (h) \eqdef {(a\, h (x) + (1-a) \, h (y) )}_{(x, a, y) \in I}$ for
  every $h \in \Lform_\pw X$.

  We reason similarly with $\B^*$ when $\B$ is pointed, using the
  additional equation $\Lambda (\bot) = 0$.  We recall that a linear
  map is the same thing as a strict affine map, by
  Proposition~\ref{prop:strict:aff}.  \qed
\end{proof}

Dually to Example~\ref{exa:V=C*}, we must mention the
\emph{Schr\"oder-Simpson theorem}, announced in
\cite{SchSimp:obs,Simpson:2009}.  The first published proof is due to
Keimel \cite{Keimel:SchSimp}, and an elementary proof can be found in
\cite{GL-mscs13}.  Precisely, Theorem~2.3 (resp.\ Corollary~2.4) of
\cite{GL-mscs13} states that every linear lower semicontinuous map
$\Lambda \colon \Val_b X \to \creal$ (resp.\
$\Lambda \colon \Val X \to \creal$) is of the form
$\nu \mapsto \int h \, d\nu$ for a unique lower semicontinuous map
$h \colon X \to \creal$; explicitly, $h (x) \eqdef \Lambda (\delta_x)$
for every $x \in X$.  This yields a bijection between $(\Val_b X)^*$
(resp.\ $(\Val X)^*$) and $\Lform X$, which is an isomorphism of
cones.  Its inverse maps every $h \in \Lform X$ to
$\nu \mapsto \int h \, d\nu$.

This is also a homeomorphism, provided that we equip
$\Lambda X$ with the weakScott topology, defined as follows.
\begin{definition}[weakScott topology]
  \label{defn:weakScott}
  The \emph{weakScott topology} on a semitopological barycentric
  algebra $\B$ is the coarsest topology that makes every
  $\Lambda \in \B^\astar$ lower semicontinuous.  A subbase is given by
  the sets $\Lambda^{-1} (]r, \infty])$ with $\Lambda \in \B^\astar$
  and $r \in \Rp \diff \{0\}$.  We write $\B_\wS$ for $\B$ with the
  weakScott topology.
\end{definition}

\begin{remark}
  \label{rem:weakScott}
  If $\B$ is a pointed semitopological barycentric algebra (for
  example, if $\B$ is a semitopological cone $\C$), then the weakScott
  topology is also the coarsest topology that makes every
  $\Lambda \in \B^*$ (namely, where $\Lambda$ is linear, not just
  affine) lower semicontinuous.  A subbase is given by the sets
  $\Lambda^{-1} (]1, \infty])$ with $\Lambda \in \B^*$; equivalently,
  when $\B$ is a cone $\C$, by the open half-spaces of $\C$.  Indeed,
  every set $\Lambda^{-1} (]r, \infty])$, where $\Lambda$ is affine
  and lower semicontinuous and $r > 0$, can be rewritten as
  ${\Lambda_r}^{-1} (]1, \infty])$ where
  $\Lambda_r (x) \eqdef \frac 1 {r-\Lambda (\bot)} \cdot (\Lambda (x)
  - \Lambda (\bot))$ for every $x \in \B$ if $\Lambda (\bot) < r$; if
  $\Lambda (\bot) = r$,
  $\Lambda^{-1} (]r, \infty]) = \bigcup_{s > r} \Lambda^{-1} (]s,
  \infty]) = \bigcup_{s > r} \Lambda_s^{-1} (]1, \infty])$, and if
  $\Lambda (\bot) > r$, then $\Lambda^{-1} (]r, \infty])$ is the whole
  space $\B$.  We note that $\Lambda_r$ is linear and lower
  semicontinuous; linearity is because $\Lambda_r$ is affine and
  strict, and relying on Proposition~\ref{prop:strict:aff}.
\end{remark}

Explicitly, the weakScott topology on $(\Val_b X)^*$, or on $(\Val
X)^*$, equated with $\Lform X$, is given by subsets of the form $[\nu
> r] \eqdef \{h \in \Lform X \mid \int h \, d\nu > r\}$, with $\nu \in
\Val_b X$ (resp.\ $\nu \in \Val X$) and $r \in \Rp \diff \{0\}$.  We
sum the above mentioned theorem by Schr\"oder and Simpson as follows.
\begin{theorem}
  \label{thm:schsimp}
  For every topological space $X$, $(\Val_b X)^*$ (resp.\
  $(\Val X)^*$) and $(\Lform X)_\wS$ are isomorphic as topological
  cones through the maps
  $h \in \Lform X \mapsto (\nu \mapsto \int h \,d\nu)$ and
  $\psi \in (\Val_b X)^* \text{( resp.}(\Val X)^*\text{)} \mapsto (x
  \in X \mapsto \psi (\delta_x))$.
\end{theorem}

Schr\"oder and Simpson also show that, for every topological space
$X$, every linear lower semicontinuous map from $\Val_{\leq 1} X$ to
$\creal$ is of the form $\nu \mapsto \int h \, d\nu$ for a unique
lower semicontinuous map $h$ \cite[Corollary~2.5]{GL-mscs13}, leading
to the following.
\begin{theorem}
  \label{thm:schsimp:pba}
  For every topological space $X$, $(\Val_{\leq 1} X)^*$ and
  $(\Lform X)_\wS$ are isomorphic as topological cones through the
  maps $h \in \Lform X \mapsto (\nu \mapsto \int h \,d\nu)$ and
  $\psi \in (\Val_{\leq 1} X)^* \mapsto (x \in X \mapsto \psi
  (\delta_x))$.
\end{theorem}
This can also be deduced from the previous theorem and
Theorem~\ref{thm:Vleq1->Vb}; we leave this as an exercise (replay the
proof below, replacing Theorem~\ref{thm:V1->Vb} by
Theorem~\ref{thm:Vleq1->Vb}).  The third and final theorem that
Schr\"oder and Simpson show, and which is not mentioned in
\cite{GL-mscs13}, is the following.
\begin{proposition}
  \label{prop:schsimp:prob}
  For every topological space $X$, every affine lower semicontinuous
  map $\psi$ from $\Val_1 X$ to $\creal$ is of the form
  $\nu \mapsto \int h \, d\nu$ for a unique lower semicontinuous map
  $h \colon X \to \creal$.
\end{proposition}
\begin{proof}
  Uniqueness is trivial: we must have $h (x) = \psi (\delta_x)$ for
  every $x \in X$.  As for existence, since $\Val_b X$ is the free
  semitopological cone over the semitopological barycentric algebra
  $\Val_1 X$ by Theorem~\ref{thm:V1->Vb}, $\psi$ has a (unique)
  linear, lower semicontinuous extension
  $\hat\psi \colon \Val_b X \to \creal$.  By
  Theorem~\ref{thm:schsimp}, there is an $h \in \Lform X$ such that
  for every $\nu \in \Val_b X$, $\hat\psi (\nu) = \int h \,d\nu$.
  This holds in particular for every $\nu \in \Val_1 X$, for which
  $\hat\psi (\nu) = \psi (\nu)$.  \qed
\end{proof}

We obtain the following immediately.
\begin{theorem}
  \label{thm:schsimp:ba}
  For every topological space $X$, $(\Val_1 X)^\astar$ and
  $(\Lform X)_\wS$ are isomorphic as topological cones through the
  maps $h \in \Lform X \mapsto (\nu \mapsto \int h \,d\nu)$ and
  $\psi \in (\Val_1 X)^* \mapsto (x \in X \mapsto \psi (\delta_x))$.
\end{theorem}

\begin{remark}
  \label{rem:weakScott:facts}
  The following facts hold about the weakScott topology.
  \begin{enumerate}[label=(\roman*),leftmargin=*]
  \item The weakScott topology is always coarser than the original
    topology on $\B$, and sometimes strictly coarser, even on
    semitopological cones (we will illustrate this in
    Example~\ref{exa:LQ:not:strong:refl} below).
  \item The affine (resp.\ linear when $\B$ is pointed) lower
    semicontinuous maps from $\B$ or from $\B_{\wS}$ to $\creal$ are
    the same: notably, for every affine lower semicontinuous map
    $\Lambda \colon \B \to \creal$, $\Lambda$ is also lower
    semicontinuous from $\B_{\wS}$ to $\creal$, because for every
    $r \in \Rp \diff \{0\}$, $\Lambda^{-1} (]r, \infty])$ is open in
    $\B_{\wS}$ by definition.
  \item Hence ${(\B_{\wS})}^\astar = \B^\astar$, and
    ${(\B_{\wS})}^* = \B^*$ if $\B$ is pointed.  In particular,
    ${(\B_{\wS})}^{**} = \B^{**}$.
  \item $\B_{\wS}$ is always locally linear, because its topology has
    a base of open sets of the form $\Lambda^{-1} (]1, \infty])$ with
    $\Lambda \in \B^\astar = {(\B_{\wS})}^\astar$.
  \item In particular, the weakScott topology coincides with the
    original topology on $\B$ if and only if $\B$ is locally linear.
  \item On cones $\C$ of the form $\Lform X$, the weakScott topology
    is the coarsest one such that
    $h \mapsto \int_{x \in X} h (x) \,d\nu$ is lower semicontinuous
    from $\Lform X$ to $\creal$, for every $\nu \in \Val X$.  This
    follows from the isomorphism between $\Val X \cong (\Lform X)^*$
    of Example~\ref{exa:V=C*}, and Remark~\ref{rem:weakScott}.
  \end{enumerate}
\end{remark}


\begin{lemma}
  \label{lemma:weakScott}
  For every space $X$, the weakScott topology on $\Lform X$ is finer
  than the topology of pointwise convergence on $\Lform X$, and is
  coarser than the original, Scott topology on $\Lform X$.  These three
  topologies coincide if $X$ is a c-space, in particular if $X$ is a
  continuous dcpo in its Scott topology.
\end{lemma}
\begin{proof}
  For the first part, the subbasic open subset $[x > r] = \{h \in
  \Lform X \mid h (x) > r\}$ of the topology of pointwise convergence
  is equal to the subbasic open subset $[\delta_x > r]$ of the
  weakScott topology.

  For the second part, it suffices to show that $[\nu > r]$ is
  Scott-open for every $\nu \in \Val X$ and every
  $r \in \Rp \diff \{0\}$, and this is easy.

  Finally, if $X$ is a c-space, and since $\creal$ is a bc-domain
  (even a complete lattice), then the Scott topology coincides with
  the topology of pointwise convergence by \cite[Proposition
  11.2]{JGL:fullabstr:I}.  Since the weakScott topology lies
  inbetween, it is equal to both of them.  \qed
\end{proof}

\begin{example}
  \label{exa:LQ:not:strong:refl}
  The Scott topology on $\Lform \rat$ is strictly finer than the
  weakScott topology.  In particular, $\Lform \rat$ is not locally
  linear.
\end{example}
\begin{proof}
  We proceed in a series of steps.  $(A)$ For every space $X$, the
  topology of pointwise convergence is coarser than the weakScott
  topology.  Indeed, the subbasic open subset
  $[x > r] = \{h \in \Lform X \mid h (x) > r\}$ of the topology of
  pointwise convergence is equal to the subbasic open subset
  $[\delta_x > r]$ of the weakScott topology.
  
  Let a \emph{discrete valuation} be a possibly infinite sum $\sum_{i
    \in I} a_i \delta_{x_i}$, where $a_i \in \creal$ (not just $\Rp$
  as with simple valuations).
  
  We claim: $(B)$ given any topological space $X$ on which every
  continuous valuation is discrete, $(\Lform X)_\wS = \Lform_\pw X$,
  namely the weakScott topology on $\Lform X$ coincides with the
  topology of pointwise convergence.  In light of $(A)$, it suffices
  to show that $[\nu > r]$ is open in the topology of pointwise
  convergence, where $\nu$ is any continuous valuation and
  $r \in \Rp \diff \{0\}$.  By assumption, $\nu$ is of the form
  $\sum_{i \in I} a_i \delta_{x_i}$.  Without loss of generality, we
  can also assume $a_i \neq 0$ for every $i \in I$.  Let
  $h \in [\nu > r]$, hence $\sum_{i \in I} a_i h (x_i) > r$, or
  equivalently $r \ll \sum_{i \in I} a_i h (x_i)$.  There is a finite
  subset $J$ of $I$ such that $r < \sum_{i \in J} a_i h (x_i) $.
  Because multiplication and addition are continuous on $\creal$,
  there are numbers $r_i$ ($i \in J$) such that $r_i \ll a_i h (x_i)$
  for every $i \in J$ and $r \ll \sum_{i \in J} r_i$.  Then
  $h \in \bigcap_{i \in J} [(r_i/a_i) \ll x_i]$ (where $[s \ll x]$
  denotes $[x > s]$ if $s > 0$, $\bigcup_{t > 0} [x > t]$ if $s=0$,
  hence is open in the topology of pointwise convergence), and for
  every $k \in \bigcap_{i \in J} [(r_i/a_i) \ll x_i]$,
  $\sum_{i \in I} a_i k (x_i) \geq \sum_{i \in J} a_i k (x_i) \geq
  \sum_{i \in J} r_i > r$; therefore $k \in [\nu > r]$.

  $(B)$ applies to $X = \rat$.  Indeed, let $\nu$ be any continuous
  valuation on $\rat$.  For every $q \in \rat$, let
  $a_q \eqdef \inf_{\epsilon > 0} \nu (]q-\epsilon, q+\epsilon[)$.  We
  verify that $\nu = \sum_{q \in \rat} a_q \delta_q$, namely that
  $\nu (U) = \sum_{q \in U \cap \rat} a_q$ for every open subset
  $U$ of $\real$.
  \begin{itemize}
  \item Let us check that $\nu (U) \leq \sum_{q \in U \cap \rat} a_q$.
    If $a_q = \infty$ for some $q \in U \cap \rat$, this is
    obvious. Otherwise, we enumerate the elements of $\rat$ as
    ${(q_n)}_{n \in \nat}$.  For every $\epsilon > 0$, for every
    $n \in \nat$ such that $q_n \in U$, one can find $\epsilon_n > 0$
    such that
    $a_q \geq \nu (]q_n - \epsilon_n, q_n + \epsilon_n[) -
    \epsilon/2^{n+1}$.  Then $U$ is included in
    $\bigcup_{n \in \nat, q_n \in U} ]q_n - \epsilon_n, q_n +
    \epsilon_n[$, so
    $\nu (U) \leq \sum_{n \in \nat, q_n \in U} \nu (]q_n - \epsilon_n,
    q_n + \epsilon_n[)$
    $\leq \sum_{n \in \nat, q_n \in U} a_{q_n} + \sum_{n \in \nat}
    \epsilon/2^{n+1} = \sum_{q \in U \cap \rat} a_q + \epsilon$.
    Taking infima over $\epsilon > 0$,
    $\nu (U) \leq \sum_{q \in U \cap \rat} a_q$.
  \item Conversely, let us check that
    $\nu (U) \geq \sum_{q \in U \cap \rat} a_q$.  If $a_q = \infty$
    for some $q \in U \cap \rat$, then, picking $\epsilon > 0$ such
    that ${]q-\epsilon, q+\epsilon[} \subseteq U$, we have
    $\nu (]q-\epsilon, q+\epsilon[) = \infty$, hence
    $\nu (U) = \infty$, so that the inequality is obvious.  Otherwise,
    we show that $\sum_{q \in A} a_q \leq \nu (U)$ for every finite
    subset $A$ of $U \cap \rat$.  The result will follow by taking
    suprema over $A$.  We enumerate the finitely many elements of $A$
    as $q_1$, \ldots, $q_n$.  Let
    $\epsilon \eqdef \min_{1\leq i < j \leq n} |q_i-q_j|$.  This is
    strictly positive since $A$ is finite.  For each $i$,
    $1\leq i\leq n$, let $\epsilon_i > 0$ be such that
    ${]q_i-\epsilon_i, q_i+\epsilon_i[} \subseteq U$, and such that
    $\epsilon_i < \epsilon/2$.  Then all the intervals
    $]q_i-\epsilon_i, q_i+\epsilon_i[$ are pairwise disjoint, and
    therefore
    $\sum_{q \in A} a_q \leq \sum_{i=1}^n \nu (]q_i-\epsilon_i,
    q_i+\epsilon_i[) = \nu (\bigcup_{i=1}^n ]q_i-\epsilon_i,
    q_i+\epsilon_i[) \leq \nu (U)$.
  \end{itemize}

  Next, we claim: $(C)$ given any infinite compact subset $Q$ of a
  $T_1$ topological space $X$,
  $[Q > 1] \eqdef \{h \in \Lform \mid \forall x \in Q, h (x) > 1\}$ is
  not open in the topology of pointwise convergence.  Indeed, let us
  imagine that $[Q > 1]$ contains a basic open set
  $A \eqdef \{h \in \Lform X \mid h (x_1) > r_1, \cdots, h (x_n) >
  r_n\}$ in the topology of pointwise convergence, where the elements
  $x_i$ are pairwise distinct and $r_1, \cdots, r_n > 0$.  We pick any
  $x \in Q \diff \{x_1, \cdots, x_n\}$.  This is possible because $Q$
  is infinite.  Since $X$ is $T_1$, there are open neighborhoods $U_i$
  of $x_i$ that do not contain $x$, for each $i$, $1\leq i\leq n$.  We
  pick $s_1 > r_1$, \ldots, $s_n > r_n$, and we let
  $h \eqdef \sup_{i=1}^n s_i \chi_{U_i}$.  Then $h$ is lower
  semicontinuous, and $h (x_i) \geq s_i > r_i$ for every $i$, so $h$
  is in $A$.  However, $h (x) = 0$, by construction.  It follows that
  $h$ is not in $[Q > 1]$.

  The sets
  $[Q > r] \eqdef \{h \in \Lform \mid \forall x \in Q, h (x) > r\}$
  with $Q$ compact and $r \in \Rp \diff \{0\}$ form a subbase of the
  compact-open topology on $\Lform X$.  It is also finer than the
  topology of pointwise convergence, because the subbasic open subsets
  $[x > r]$ of the latter can be written as $[\upc x > r]$.  The
  compact-open topology is coarser than the Scott topology.  The core
  of the argument consists in showing that for every directed family
  ${(h_i)}_{i \in I}$ with (pointwise) supremum $h$ in $\Lform X$, if
  $h \in [Q > r]$ then some $h_i$ is in $[Q > r]$.  But
  $h \in [Q > r]$ means that
  $Q \subseteq h^{-1} (]r, \infty]) = \bigcup_{i \in I} h_i^{-1} (]r,
  \infty])$, and since $Q$ is compact, there is an $i \in I$ such that
  $Q \subseteq h_i^{-1} (]r, \infty])$, that is, $h_i \in [Q > r]$.

  We conclude.  The space $\rat$ is Hausdorff hence $T_1$.  Let
  $Q \eqdef \{1/2^n \mid n \in \nat\} \cup \{0\}$.  It is easy to see
  that $Q$ is compact and infinite in $\rat$.  By $(C)$, the
  compact-open topology on $\Lform \rat$ is strictly finer than the
  topology of pointwise convergence.  We have just argued that it is
  coarser than the Scott topology, so the Scott topology is also
  strictly finer than the topology of pointwise convergence.  The
  latter coincides with the weakScott topology by $(B)$, which applies
  to $X = \rat$, as we have seen.  We have proved that the weakScott
  and Scott topologies differ on $\Lform \rat$.  Hence $\Lform \rat$
  is not locally linear, thanks to Remark~\ref{rem:weakScott:facts},
  item~$(v)$.  \qed
\end{proof}

\subsection{Weak local convexity and barycenters, part 4}
\label{sec:weakly-locally-conv}

The following notion of local convexity was introduced and studied by
Heckmann \cite[Section~6.3]{heckmann96}, on cones.  The name he used
for that notion was ``locally convex'', but that should not be
confused with local convexity, as defined in
Section~\ref{sec:locally-conv-semit}.
\begin{definition}[Weakly locally convex]
  \label{defn:loc:convex:heckmann}
  A semitopological barycentric algebra $\B$ is \emph{weakly locally
    convex} if and only if every point $x \in \B$ has a neighborhood
  base of convex sets, namely if and only if every open neighborhood
  $U$ of $x$ contains a convex (not necessarily open) neighborhood of
  $x$.
\end{definition}

\begin{remark}
  \label{rem:locconv:weak}
  Every locally convex semitopological barycentric algebra is weakly
  locally convex.  In the realm of topological real vector spaces,
  there would be no difference between the two notions.  We will see
  that there is no difference either on consistent barycentric
  algebras and cones, which are the subject of
  Section~\ref{sec:cons-baryc-algebr} (see
  Corollary~\ref{corl:locconv:cons}).  See also
  Corollary~\ref{corl:locconvcomp:locconv} for another case where
  local convexity is implied by apparently weaker properties.
\end{remark}

\begin{problem}
  \label{pb:locconv:weak}
  Is every weakly locally convex semitopological barycentric algebra
  in fact locally convex?  If so, local convexity and weak local
  convexity would be the same notion.  The same question applies to
  semitopological cones.
\end{problem}

\begin{example}
  \label{exa:lp:heckmann}
  Not every semitopological barycentric algebra, in fact not every
  every topological barycentric algebra, is weakly locally convex, as
  the following counterexample, inspired by Tychonoff's example of a
  non-locally convex real vector space shows
  \cite[page~768]{Tychonoff:fixpunkt}.  Let $p \in {]0, 1[}$.  We
  write sequence of numbers as ${(x_n)}_{n \in \nat}$, abbreviated as
  $\vec x$; $-\vec x$ denotes ${(-x_n)}_{n \in \nat}$.  Let
  $\ell_-^p (\nat)$ be the set of all sequences $-\vec x$ such that
  $x_n \in \Rp$ for every $n \in \nat$ and
  $\sum_{n \in \nat} x_n^p < \infty$, with pointwise addition and
  scalar multiplication.  Its topology is the open ball topology of
  the quasi-metric $d^p$ defined by
  $d^p (-\vec x, -\vec y) \eqdef \max (y_n - x_n, 0)^p$; in other
  words, the open balls
  $B_{-\vec x, < r} \eqdef \{-\vec y \in \ell_-^p (\nat) \mid d^p
  (-\vec x, -\vec y) < r\}$ are a base of the topology.  The details
  are given in Appendix~\ref{sec:topol-baryc-algebra}.  Taking
  non-positive sequences $-\vec x$ is necessary: we let the reader
  check that the analogous topological cone $\ell_+^p (\nat)$ of
  sequences $\vec x$ such that $\sum_{n \in \nat} x_n^p < \infty$ is
  locally linear, using the fact that the projections
  $\vec x \mapsto x_i$ are linear and lower semicontinuous.
\end{example}

\begin{example}
  \label{exa:peck:nonlocconv}
  Not every semitopological cone is weakly locally convex either.  The
  following counterexample is based on similar construction principles
  as Example~\ref{ex:pBA:nonstrictembed}.  Let $\C$ be the cone of all
  functions from $[0, 1[$ to $\Rp$, with pointwise addition and scalar
  multiplication.  The zero is the constant function $0$.  We order
  $\C$ pointwise.  For every $n \in \nat$, let
  $U_n \eqdef \bigcup_{i=0}^{2^n-1} \left[\frac {2i} {2^{n+1}}, \frac
    {2i+1} {2^{n+1}}\right[$ and let $\one$ be the constant function
  equal to $1$.  We declare a subset $C$ of $\C$ \emph{closed} if and
  only if $C$ is Scott-closed (with respect to the pointwise
  ordering---we make this precise because there is no reason to
  believe that the specialization preordering of the forthcoming
  topology on $\C$ would be the pointwise ordering) and for all
  $a \in \Rp$ and $g \in \C$ such that $2a .  \chi_{U_n} + g \in C$
  for infinitely many values of $n \in \nat$, $a . \one + g$ is in
  $C$.  The closed sets in this sense form the closed sets of a
  topology, turning $\C$ into a semitopological cone in which the
  sequence ${(2.\chi_{U_n})}_{n \in \nat}$ converges to $\one$.  Let
  $\nat^* \eqdef \nat \diff \{0\}$.  For every $n \in \nat^*$, let
  $f_n \eqdef \frac 1 n \sum_{i=n}^{2n-1} 2.\chi_{U_i}$.  Let
  $f_\infty \eqdef \limsup_{n \in \nat^*} f_n$, namely
  $f_\infty (x) \eqdef \inf_{n \in \nat^*} \sup_{m \in \nat^*, m\geq
    n} f_m (x)$ for every $x \in {[0, 1[}$.  Let
  $C \eqdef \dc \{f_n \mid n \in \nat^*\} \cup \dc f_\infty$, where
  $\dc$ denotes downward-closure with respect to the pointwise
  ordering.  One can check that $C$ is closed in $\C$ and
  $\one \not\in C$, so ${(f_n)}_{n \in \nat^*}$ does not converge to
  $\one$.  However, for every open neighborhood $U$ of $\one$ in $\C$,
  $\conv U$ must intersect $C$, so $\C$ is not weakly locally convex.
  See Appendix~\ref{sec:semit-cone-that} for details.
\end{example}

\begin{problem}
  \label{pb:locconv:weak:1}
  Is every topological (not just semitopological) cone weakly locally
  convex?
\end{problem}

The following is due to Heckmann in the case of cones
\cite[Theorems~6.7, 6.8]{heckmann96}.  The map $\beta$ constructed
below is the \emph{barycenter map}, and $\beta (\nu)$ is the
\emph{barycenter} of $\nu$.
\begin{theorem}[Barycenters, part 4]
  \label{thm:locconv:heckmann:beta}
  Let $\C$ be an ordered cone (resp.\ an ordered pointed barycentric
  algebra $\B$, resp.\ an ordered barycentric algebra $\B$).

  There is a unique linear (resp.\ linear, resp.\ affine) map $\beta$
  from $\Val_\fin \C$ (resp.\ $\Val_{\leq 1, \fin} \B$, resp.\
  $\Val_{1, \fin} \B$) to $\C$ (resp., to $\B$) such that
  $\beta (\delta_x) = x$ for every $x \in \C$ (resp., in $\B$), and it
  is monotonic.  Explicitly,
  $\beta (\sum_{i=1}^n a_i \delta_{x_i}) = \sum_{i=1}^n a_i \cdot x_i$
  for every simple (resp.\ simple subnormalized, resp.\ simple
  normalized) valuation $\sum_{i=1}^n a_i \delta_{x_i}$.

  If $\C$ (resp.\ $\B$) is a weakly locally convex $T_0$
  semitopological cone (resp., a weakly locally convex $T_0$ pointed
  semitopological barycentric algebra, resp.\ a weakly locally convex
  $T_0$ semitopological barycentric algebra), then $\beta$ is
  continuous.
\end{theorem}
The notation $\sum_{i=1}^n a_i \cdot x_i$ is clear if $\C$ is a cone;
it is justified by Definition and Proposition~\ref{prop:bary:1} in the
case of barycentric algebras and by Definition and
Proposition~\ref{prop:bary:leq1} in the case of pointed barycentric
algebras.

\begin{proof}
  If $\beta$ exists and $\C$ is a cone, then linearity implies that
  $\beta (\sum_{i=1}^n a_i \delta_{x_i}) = \sum_{i=1}^n a_i \cdot x_i$
  for every simple valuation $\sum_{i=1}^n a_i \delta_{x_i}$, showing
  uniqueness.  When $\C$ is an barycentric algebra and
  $\sum_{i=1}^n a_i = 1$, we claim that we must have a similar
  formula, by induction on $n \geq 1$, and using Definition and
  Proposition~\ref{prop:bary:1}: if $a_n=1$, then $a_i=0$ for every
  $i \neq n$, so $\sum_{i=1}^n a_i \delta_{x_i} = \delta_{x_n}$, whose
  image under $\beta$ must be equal to
  $x_n = \sum_{i=1}^n a_i \cdot x_i$; otherwise,
  $\beta (\sum_{i=1}^n a_i \delta_{x_i}) = \beta ((1-a_n) \cdot
  \sum_{i=1}^n \frac {a_i} {1-a_n} \delta_{x_i} + a_n \delta_{x_n}) =
  \beta (\sum_{i=1}^n \frac {a_i} {1-a_n} \delta_{x_i}) +_{1-a_n}
  \beta (\delta_{x_n}) = (\sum_{i=1}^{n-1} \frac {a_i} {1-a_n} \cdot
  x_i) +_{1-a_n} x_n$ (by induction hypothesis and since
  $\beta (\delta_{x_n}) = x_n$) $= \sum_{i=1}^n a_i \cdot x_i$.  When
  $\C$ is a pointed barycentric algebra, we use Definition and
  Proposition~\ref{prop:bary:leq1}, then
  $\beta (\sum_{i=1}^n a_i \delta_{x_i})$ must be equal to
  $(\sum_{i=1}^n a_i) \cdot \beta (\sum_{i=1}^n \frac {a_i}
  {\sum_{i=1}^n a_i} \delta_{x_i}) = (\sum_{i=1}^n a_i) \cdot
  (\sum_{i=1}^n \frac {a_i} {\sum_{i=1}^n a_i} \cdot x_i) =
  \sum_{i=1}^n a_i \cdot x_i$ if $\sum_{i=1}^n a_i \neq 0$, since
  $\beta$ is linear; if $\sum_{i=1}^n a_i = 0$,
  $\beta (\sum_{i=1}^n a_i \delta_{x_i}) = \beta (0) = \beta (0 \cdot
  \delta_\bot) = 0 \cdot \bot = \bot$.

  In order to show that $\beta$ exists, we simply define
  $\beta (\sum_{i=1}^n a_i \cdot x_i)$ as
  $\sum_{i=1}^n a_i \cdot x_i$; if $\C$ is a barycentric algebra and
  $\sum_{i=1}^n a_i=1$, we rely on Definition and
  Proposition~\ref{prop:bary:1} to make sense of it, and if $\C$ is a
  pointed barycentric algebra and $\sum_{i=1}^n a_i \leq 1$, we rely
  on Definition and Proposition~\ref{prop:bary:leq1}.  We additionally
  need to show that if
  $\sum_{i=1}^n a_i \delta_{x_i} = \sum_{j=1}^m b_j \delta_{y_j}$,
  then $\sum_{i=1}^n a_i \cdot x_i = \sum_{j=1}^m b_j \cdot y_j$.  We
  will show that if
  $\sum_{i=1}^n a_i \delta_{x_i} \leq \sum_{j=1}^m b_j \delta_{y_j}$,
  then $\sum_{i=1}^n a_i \cdot x_i \leq \sum_{j=1}^m b_j \cdot y_j$,
  which will show this (since
  $\sum_{i=1}^n a_i \delta_{x_i} = \sum_{j=1}^m b_j \delta_{y_j}$ will
  imply that
  $\sum_{i=1}^n a_i \cdot x_i \leq \sum_{j=1}^m b_j \cdot y_j$ and
  $\sum_{i=1}^n a_i \cdot x_i \geq \sum_{j=1}^m b_j \cdot y_j$, and
  since $\C$ is ordered, namely since $\leq$ is antisymmetric), and
  also that $\beta$ is monotonic.

  We use Jones' splitting lemma \cite[Theorem 4.10]{Jones:proba}, see
  also \cite[Proposition IV-9.18]{GHKLMS:contlatt}.  We have already
  used it in the proof of Proposition~\ref{prop:schsimp:decomp}.  By
  this splitting lemma, there is a transport matrix
  ${(t_{ij})}_{\substack{1\leq i\leq n\\1\leq j\leq m}}$ of
  non-negative real numbers such that: $(a)$ the only non-zero numbers
  $t_{ij}$ are such that $x_i \leq y_j$, $(b)$
  $\sum_{j=1}^m t_{ij} = a_i$ for every $i \in \{1, \cdots, n\}$, and
  $(c)$ $\sum_{i=1}^n t_{ij} \leq b_j$ for every
  $j \in \{1, \cdots, m\}$.  In case $\C$ is a (pointed) ordered
  barycentric algebra $\B$, we use the injective affine preorder
  embedding $\etac_{\B}$ of $\C=\B$ into $\conify (\B)$ of
  Proposition~\ref{prop:bary:alg:conify:ord}, and we obtain:
  \ifta
  \begin{align*}
    \sum_{i=1}^n a_i \cdot \etac_{\B} (x_i)
    & = \sum_{\substack{1\leq i\leq n\\1\leq j\leq m}} t_{ij} \etac_{\B}
    (x_i)
    & \text{by $(b)$} \\
    & = \sum_{\substack{1\leq i\leq n\\1\leq j\leq m\\x_i \leq y_j}} t_{ij} \cdot \etac_{\B}
    (x_i)
    & \text{by $(a)$} \\
    & \leqc \sum_{\substack{1\leq i\leq n\\1\leq j\leq m\\x_i \leq
    y_j}} t_{ij} \cdot \etac_{\B} (y_j)
    & \text{since $\etac_{\B}$ is monotonic} \\
    & = \sum_{\substack{1\leq i\leq n\\1\leq j\leq m}} t_{ij} \cdot
    \etac_{\B} (y_j)
    & \text{by $(a)$} \\
    & \leqc \sum_{j=1}^m b_j \cdot \etac_{\B} (y_j)
    & \text{by $(c)$,}
  \end{align*}
  \else
  \begin{align*}
    & \sum_{i=1}^n a_i \cdot \etac_{\B} (x_i) \\
    & = \sum_{\substack{1\leq i\leq n\\1\leq j\leq m}} t_{ij} \etac_{\B}
    (x_i)
    & \text{by $(b)$} \\
    & = \sum_{\substack{1\leq i\leq n\\1\leq j\leq m\\x_i \leq y_j}} t_{ij} \cdot \etac_{\B}
    (x_i)
    & \text{by $(a)$} \\
    & \leqc \sum_{\substack{1\leq i\leq n\\1\leq j\leq m\\x_i \leq
    y_j}} t_{ij} \cdot \etac_{\B} (y_j)
    & \text{since $\etac_{\B}$ is monotonic} \\
    & = \sum_{\substack{1\leq i\leq n\\1\leq j\leq m}} t_{ij} \cdot
    \etac_{\B} (y_j)
    & \text{by $(a)$} \\
    & \leqc \sum_{j=1}^m b_j \cdot \etac_{\B} (y_j)
    & \text{by $(c)$,}
  \end{align*}
  \fi
  and since the cone operations on $\conify (\B)$ are monotonic.
  Since $\etac_{\B}$ is an affine preorder embedding, we conclude that
  $\sum_{i=1}^n a_i \cdot x_i \leq \sum_{j=1}^m b_j \cdot y_j$.  In
  case $\C$ is an ordered cone, the proof is the same, simply removing
  all mention of $\etac_{\B}$.

  We have shown that $\beta$ is well-defined and monotonic.  When $\C$
  is a cone, it is clearly linear.  When $\C$ is a barycentric
  algebra, we show that $\beta$ is affine as follows.  Let
  $a \in [0, 1]$ and $\mu \eqdef \sum_{i=1}^n a_i \delta_{x_i}$
  $\nu \eqdef \sum_{i=1}^n b_i \delta_{x_i}$ be in
  $\Val_{1, \fin} \B$, where $\mu$ and $\nu$ are built using the same
  list of points $x_i$, without loss of generality, and $n \geq 1$.
  Using $\etac_{\B}$ as above,
  $\etac_{\B} (\beta (\mu +_a \nu)) = \etac_{\B} (\sum_{i=1}^n (aa_i +
  (1-a)b_i) \delta_{x_i}) = \sum_{i=1}^n (aa_i + (1-a)b_i) \cdot
  \etac_{\B} (x_i)$, while
  $\etac_{\B} (\beta (\mu) +_a \beta (\nu)) = a \cdot \beta (\mu) +
  (1-a) \cdot \beta (\nu)$ (since $\etac_{\B}$ is affine)
  $= a \cdot \sum_{i=1}^n a_i \cdot \etac_{\B} (x_i) + (1-a) \cdot
  \sum_{i=1}^n b_i \cdot \etac_{\B} (x_i)$, and that is the same
  value.  Since $\etac_{\B}$ is injective,
  $\beta (\mu +_a \nu) = \beta (\mu) +_a \beta (\nu)$.  When $\C$ is
  an ordered pointed barycentric algebra $\B$, additionally, the least
  element of $\Val_{\leq 1, \fin} \B$ is the zero valuation $0$, and
  then $\etac_{\B} (\beta (0))$ is equal to $\etac_{\B} (\bot)$ by
  Definition and Proposition~\ref{prop:bary:leq1}, item~3; since
  $\etac_{\B}$ is injective, $\beta (0) = \bot$, showing that $\beta$
  is strict, hence linear, by Proposition~\ref{prop:strict:aff}.

  We now claim that $\beta$ is continuous, assuming additionally that
  $\C$ is weakly locally convex, $T_0$ and topological (as a cone,
  resp.\ a pointed barycentric algebra, resp.\ a barycentric algebra).
  Let $\nu \eqdef \sum_{i \in I} a_i \delta_{x_i}$, where $I$ is
  finite, $a_i > 0$ and $x_i \in \C$ for each $i \in I$, and let $V$
  be an open neighborhood of $\beta (\nu)$.  Our aim is to find an
  open neighborhood $\mathcal U$ of $\mu$ included in
  $\beta^{-1} (V)$.

  Let $m$ be the cardinality of $I$.  If $\C$ is a topological
  cone---not just a semitopological cone---, the map that sends
  $(\vec c, \vec z) \in {(\creal)}^m \times \C^m$ to
  $\sum_{i \in I} c_i \cdot z_i$ is jointly continuous, so there are
  open neighborhoods $]r_i, \infty]$ of $a_i$ in $\creal$ ($r_i > 0$)
  and $V_i$ of $x_i$ in $\C$, for each $i \in I$, such that for all
  $c_i > r_i$ and $z_i \in V_i$, $i \in I$,
  $\sum_{i \in I} c_i \cdot z_i$ is in $V$.  If $\C$ is a topological
  barycentric algebra, the similar map
  $(\vec c, \vec z) \mapsto \sum_{i \in I} c_i \cdot z_i$ is jointly
  continuous from $\Delta_m \times \C^m$ to $\C$ by
  Proposition~\ref{prop:bary:1:cont}.  Additionally,
  Lemma~\ref{lemma:Delta:scott} tells us that $\Delta_m$ has the
  subspace topology induced by the inclusion in ${(\creal)}^m$, so
  there are open neighborhoods $]r_i, \infty]$ of $a_i$ in $\creal$
  ($r_i > 0$) and $V_i$ of $x_i$ in $\C$, for each $i \in I$, such
  that for all $\vec c \in \Delta_m$ with $c_i > r_i$ for every
  $i \in I$, and for all $z_i \in V_i$ with $i \in I$,
  $\sum_{i \in I} c_i \cdot z_i$ is in $V$, as in the cone case.  If
  $\C$ is a pointed topological barycentric algebra, then we use
  Proposition~\ref{prop:bary:leq1:cont}:
  $(\vec c, \vec z) \mapsto \sum_{i \in I} c_i \cdot z_i$ is jointly
  continuous from $\widetilde\Delta_m \times \C^m$ to $\C$, and
  therefore there are open neighborhoods $]r_i, \infty]$ of $a_i$ in
  $\creal$ ($r_i > 0$) and $V_i$ of $x_i$ in $\C$, for each
  $i \in I$, such that for all $\vec c \in \widetilde\Delta_m$ with
  $c_i > r_i$ for every $i \in I$, and for all $z_i \in V_i$ with
  $i \in I$, $\sum_{i \in I} c_i \cdot z_i$ is in $V$, as in the cone
  and topological barycentric algebra cases.

  By weak local convexity, we can pick a further convex neighborhood
  $C_i$ of $z_i$ included in $V_i$ for each $i \in I$.  We also pick
  real numbers $c_i$ such that $r_i < c_i < a_i$ for every $i \in I$.

  For every non-empty subset $A$ of $I$, let
  $V_A \eqdef \bigcup_{i \in A} \interior {C_i}$ and
  $c_A \eqdef \sum_{i \in A} c_i$.  We define $\mathcal U$ as
  $\bigcap_{A \subseteq I, A \neq \emptyset} [V_A > c_A]$, seen as an
  open subset of $\Val_\fin \C$, or $\Val_{1,\fin} \C$, or $\Val_{\leq
    1, \fin} \C$, depending on the case.

  For every non-empty subset $A$ of $I$, $V_A$ contains every $z_i$
  with $i \in A$, so $\nu (V_A) \geq \sum_{i \in A} a_i > c_A$, and
  this shows that $\nu$ is in $\mathcal U$.

  We must now show that $\mathcal U$ is included in $\beta^{-1} (V)$.
  Let $\mu \eqdef \sum_{j \in J} b_j \delta_{y_j}$, where $J$ is a
  finite set, $b_j \in \Rp$ and $y_j \in \C$ for every $j \in J$, and
  let us assume that $\mu$ is in $\mathcal U$.  By definition,
  $\mu (V_A) > c_A$ for every non-empty subset $A$ of $I$, that is:
  $(*)$ $\sum_{j \in J, y_j \in V_A} b_j > \sum_{i \in A} c_i$.  We
  will now use the splitting lemma in a non-standard way---but in the
  same way that Heckmann used it \cite[Theorem~6.7]{heckmann96}.

  Let us consider the disjoint sum $I+J$, with the ordering
  $\sqsubseteq$ defined by: for all $i \in I$ and $j \in J$,
  $i \sqsubseteq j$ if and only if $y_j \in \interior {C_i}$; no
  $j \in J$ is below any $i \in I$, all elements of $I$ are pairwise
  incomparable, and all elements of $J$ are pairwise incomparable.
  For every non-empty $A \subseteq I$, let $A^+$ be the set of indices
  $j \in J$ such that $i \sqsubseteq j$ for some $i \in A$.  Then
  $A^+ = \{j \in J \mid \exists i \in A, y_j \in \interior {C_i}\} =
  \{j \in J \mid y_j \in V_A\}$.  Property $(*)$, proved above, states
  that $\sum_{i \in A} c_i < \sum_{j \in A^+} b_j$, for every
  non-empty subset $A$ of $I$.  By Jones' splitting lemma, there is a
  transport matrix ${(t_{ij})}_{i \in I, j \in J}$ whose only non-zero
  entries are such that $i \sqsubseteq j$, such that
  $\sum_{j \in J} t_{ij} = c_i$ for every $i \in I$, and
  $\sum_{i \in I} t_{ij} \leq b_j$ for every $j \in J$.

  Using the latter inequality,
  $\mu = \sum_{j \in J} b_j \delta_{y_j} \geq \sum_{i \in I, j \in J}
  t_{ij} \delta_{y_j} = \sum_{i \in I} \sum_{j \in J} t_{ij}
  \delta_{y_j} = \sum_{i \in I} \sum_{j \in J, y_j \in \interior
    {C_i}} t_{ij} \delta_{y_j}$---the latter equality rests on the
  fact that only the entries with $i \sqsubseteq j$, namely
  $y_j \in \interior {C_i}$, are non-zero.  We will show that
  $\beta (\mu) \in V$, distinguishing the cases where $\C$ is a cone
  or a (pointed) barycentric algebra.

  If $\C$ is a cone, since $\beta$ is monotonic,
  $\beta (\mu) \geq \sum_{i \in I} \sum_{j \in J, y_j \in \interior
    {C_i}} t_{ij} \cdot y_j$.  For each $i \in I$,
  $z_i \eqdef \sum_{j \in J, y_j \in \interior {C_i}} \frac {t_{ij}}
  {c_i} \cdot y_j$ is a linear combination of points $y_j$ that are
  all in the convex set $C_i$, with coefficients that sum up to
  $\sum_{j \in J, i \sqsubseteq j} \frac {t_{ij}} {c_i} = \sum_{j \in
    J} \frac {t_{ij}} {c_i} = 1$ (since $t_{ij}=0$ when
  $i \not\sqsubseteq j$).  Therefore $z_i$ is in $C_i$ for every
  $i \in I$.  Moreover,
  $\beta (\mu) \geq \sum_{i \in I} c_i \cdot z_i$.  Since
  $z_i \in C_i \subseteq V_i$ and $c_i > r_i$ for each $i \in I$, by
  definition of $V_i$ and $r_i$ we have
  $\sum_{i \in I} c_i \cdot z_i \in V$.  Since $V$ is upwards-closed,
  $\beta (\mu)$ is in $V$.

  If $\C$ is a pointed barycentric algebra $\B$,
  $\sum_{j \in J} b_j \leq 1$ and $\sum_{i \in I} c_i \leq 1$, then we
  embed $\B$ in $\conify (\B)$ through the injective affine preorder
  embedding $\etac_{\B}$ (see
  Proposition~\ref{prop:bary:alg:conify:ord}).  From
  $\mu \geq \sum_{i \in I, j \in J, y_j \in \interior {C_i}} t_{ij}
  \delta_{y_j}$ and since $\etac_{\B} \circ \beta$, we obtain that
  $\etac_{\B} (\beta (\mu)) \geq \sum_{i \in I, j \in J, y_j \in
    \interior {C_i}} t_{ij} \cdot \etac_{\B} (y_j)$.  (The computation
  of $\beta (\mu)$ is valid in this case, in other words, we apply
  $\beta$ to a subnormalized simple valuation, since
  $\sum_{i \in I, j \in J, y_j \in \interior {C_i}} t_{ij} \leq
  \sum_{j \in J} b_j \leq 1$.)  We define $z_i$ as above, so that
  $z_i \in C_i \subseteq V_i$, and then
  $\etac_{\B} (\beta (\mu)) \geq \sum_{i \in I} c_i \cdot \etac_{\B}
  (z_i) = \etac_{\B} (\sum_{i \in I} c_i \cdot z_i)$.  Since
  $\etac_{\B}$ is a preorder embedding,
  $\beta (\mu) \geq \sum_{i \in I} c_i \cdot z_i$.  Since each $c_i$
  is in $V_i$ and $c_i > r_i$, $\sum_{i \in I} c_i \cdot z_i$ is in
  $V$, and then $\beta (\mu) \in V$ since $V$ is upwards-closed, as in
  the cone case.

  If $\C$ is a (not necessarily pointed) barycentric algebra $\B$ and
  $\sum_{j \in J} b_j = \sum_{i \in I} c_i = 1$, the argument is the
  same.  We only have to pay attention to the validity of the
  computation of $\beta (\mu)$, which relies on the fact that
  $\sum_{i \in I, j \in J, y_j \in \interior {C_i}} t_{ij} = \sum_{i
    \in I, j \in J, i \sqsubseteq j} t_{ij} = \sum_{i \in I, j \in J}
  t_{ij}$ (since $t_{ij}=0$ when $i \not\sqsubseteq j$)
  $= \sum_{i \in I} c_i = 1$.  \qed
\end{proof}

\subsection{Linear and affine retractions}
\label{sec:line-affine-retr}

A \emph{retraction} $r \colon X \to Y$ is a continuous map between
topological spaces such that there is continuous map
$s \colon Y \to X$ (the \emph{section}) with $r (s (y)) = y$ for every
$y \in Y$.  $Y$ is then a \emph{retract} of $X$.  It is easy to see
that a retract of a $T_0$ space is $T_0$.  Retracts preserve many
other properties, and we will state those we need along the way.

The notion of linear retractions below is due to Heckmann
\cite[Proposition~6.6]{heckmann96}, and affine retractions are the
natural extension to barycentric algebras.
\begin{definition}[Linear retraction, affine retraction]
  \label{defn:lin:retr}
  An \emph{affine retraction} is any retraction $r \colon \B \to \Alg$
  between semitopological barycentric algebras that is also affine.
  Then $\Alg$ is a \emph{affine retract} of $\B$.

  A \emph{strict affine retraction}, also known as a \emph{linear
    retraction}, is an affine retraction between pointed
  semitopological algebras that is strict.  If $r \colon \B \to \Alg$
  is a linear retraction, then $\Alg$ is a \emph{linear retract} of
  $\B$.
\end{definition}
We remember that strict affine maps and linear maps between pointed
barycentric algebras are the same thing, by
Proposition~\ref{prop:strict:aff}.  Hence a linear retraction $r$ is
simply one that is linear as a function.

\begin{remark}
  \label{rem:lin:retr}
  In Definition~\ref{defn:lin:retr}, beware that we do not require the
  associated section to be affine, strict or linear in any way.
\end{remark}

\begin{remark}
  \label{rem:lin:retr:pointed}
  Given a pointed $T_0$ semitopological barycentric algebra $\Alg$,
  every affine retraction $r \colon \B \to \Alg$ from a pointed
  semitopological barycentric algebra is strict, hence linear.
  Indeed, letting $s$ be the associated section, for every
  $y \in \Alg$, we have $\bot \leq s (y)$, hence
  $r (\bot) \leq r (s (y)) = y$, since $r$, being continuous, is
  monotonic.  Hence $r (\bot)$ is least in $\Alg$, and is the unique
  least element in $\Alg$ because $\Alg$ is $T_0$.
\end{remark}

\begin{proposition}
  \label{prop:locconv:retract}
  The following hold.
  \begin{enumerate}
  \item Every affine retract of a weakly locally convex
    semitopological barycentric algebra is weakly locally convex; in
    particular, every linear retract of a weakly locally convex
    pointed semitopological barycentric algebra (resp.\ of a weakly
    locally convex semitopological cone) is weakly locally convex.
  \item Every affine retract of a topological barycentric algebra is
    topological; in particular, every linear retract of a pointed
    topological barycentric algebra (resp.\ of a topological cone) is
    topological.
  \item Every weakly locally convex $T_0$ topological barycentric
    algebra $\B$ is an affine retract of the $T_0$ locally affine
    topological barycentric algebra $\Val_{1, \fin} \B$, through the
    retraction $\beta \colon \Val_{1,\fin} \B \to \B$ and the section
    $\eta_{\B} \colon x \mapsto \delta_x$.

    If $\B$ is a weakly locally convex pointed $T_0$ topological
    barycentric algebra, then $\B$ is a linear retract of the $T_0$
    locally linear pointed topological barycentric algebra
    $\Val_{\leq 1, \fin} \B$ through
    $\beta \colon \Val_{\leq 1, \fin} \B \to \B$ and $\eta_{\B}$.

    If $\B$ is a weakly locally convex $T_0$ topological cone $\C$,
    then $\C$ is a linear retract of the $T_0$ locally linear
    topological cone $\Val_\fin \C$ through
    $\beta \colon \Val_\fin \C \to \C$ and $\eta_\C$.
  \end{enumerate}
  As consequences,
  \begin{enumerate}[resume]
  \item the weakly locally convex $T_0$ topological barycentric
    algebras are exactly the semitopological barycentric algebras $\B$
    that occur as affine retracts of locally linear $T_0$ topological
    barycentric algebras;
  \item the weakly locally convex $T_0$ pointed topological
    barycentric algebras are exactly the semitopological barycentric
    algebras $\B$ that occur as linear retracts of locally linear
    $T_0$ pointed topological barycentric algebras;
  \item the weakly locally convex pointed $T_0$ topological cones are
    exactly the semitopological cones $\C$ that occur as linear
    retracts of locally convex $T_0$ topological, or even of locally
    linear $T_0$ topological cones.
  \end{enumerate}
\end{proposition}
\begin{proof}
  Let $r \colon \B \to \Alg$ be an affine retraction, with associated
  section $s$, from a semitopological barycentric algebra $\B$ onto a
  semitopological barycentric algebra $\Alg$.

  1.  Let $\B$ be weakly locally convex, $y \in \Alg$ and $V$ be an
  open neighborhood of $y$ in $\Alg$.  Then $r^{-1} (V)$ is an open
  neighborhood of $s (y)$.  Since $\B$ is weakly locally convex, there
  is a convex subset $C$ of $r^{-1} (V)$ whose interior contains
  $s (y)$.  Then $s^{-1} (\interior C)$ is an open neighborhood of
  $y$.  The image $r [C]$ is convex because $r$ is affine, and is
  included in $V$.  Moreover, $s^{-1} (\interior C)$ is included in
  $r [C]$: for every $z \in s^{-1} (\interior C)$, $z = r (s (z))$ is
  the image of $s (z) \in \interior C \subseteq C$ under $r$.
  Therefore $\Cb$ is weakly locally convex.

  2. We now assume that $\B$ is topological.  Let
  $f \colon \Alg \times [0, 1] \times \Alg \to \Alg$ map $(x, a, y)$
  to $x +_a y$ in $\Alg$, and let $g$ be the similarly defined
  function on $\B$.  We know that $g$ is continuous, and we wish to
  show that $f$ is, too.  For all $x, y \in \Alg$ and $a \in [0, 1]$,
  $f (x, a, y) = x +_a y = r (s (x)) +_a r (s (y)) = r (s (x) +_a s
  (y)) = r (g (s (x), a, s (y)))$ since $r$ is affine.  Hence
  $f = r \circ g \circ (s \times \identity\relax \times s)$, which is
  therefore continuous.

  In the case of cones, we recall that a semitopological cone is
  topological as a cone if and only if it is topological as a
  barycentric algebra, by Lemma~\ref{lemma:bary:in:cone:iff}.

  Before we proceed with item~3, we recall that every retract of a
  $T_0$ space is $T_0$.  Hence every semitopological barycentric
  algebra (resp.\ pointed, resp.\ every semitopological cone) that
  occurs as an affine retract of a locally convex (hence weakly
  locally convex) $T_0$ topological barycentric algebra (resp.\
  pointed, resp.\ semitopological cone) is itself weakly locally
  convex, $T_0$ and topological.

  3.  Let $\B$ be a weakly locally convex $T_0$ topological
  barycentric algebra.  Using Theorem~\ref{thm:locconv:heckmann:beta},
  $\beta \colon \Val_{1,\fin} \B \to \B$ is a retraction with section
  the map $\eta_{\B} \colon x \mapsto \delta_x$, and $\beta$ is
  affine.  The map $\eta_{\B}$ is continuous since the inverse image
  of $[U > r]$ is $U$ if $r < 1$, empty otherwise.  We reason
  similarly in the case where $\B$ is pointed, or a cone, in which
  case $\beta$ is linear.

  Moreover, $\Val_\fin \B$ is $T_0$, so its subspaces
  $\Val_{1, \fin} \B$ and $\Val_{\leq 1, \fin} \B$ are, too.
  $\Val_\fin \B$ is topological (Example~\ref{exa:VX:top:cone}) and
  locally linear (Example~\ref{exa:VX:loclin}).  $\Val_1 X$ and
  $\Val_{\leq 1} X$ are topological, being convex subspaces of
  $\Val X$, thanks to Lemma~\ref{lemma:bary:in:cone}.
  $\Val_{\leq 1} X$ is locally linear (see
  Example~\ref{exa:Vleq1X:loclin}) and $\Val_1 X$ is locally affine
  (see Example~\ref{exa:V1X:loclin}).

  Items~4, 5, 6 follow immediately.  \qed
\end{proof}

\subsection{The free weakly locally convex (pointed) topological barycentric algebras}
\label{sec:free-point-topol}

Heckmann shows that $\Val_\fin X$ is the free weakly locally convex
$T_0$ topological cone \cite[Theorem~6.7]{heckmann96}.  We retrieve
this here, as well as natural extensions of this result to (pointed)
barycentric algebras.
\begin{theorem}
  \label{thm:locconv:heckmann:free}
  For every $T_0$ topological cone $\C$, $\Val_\fin \C$ is the free
  weakly locally convex $T_0$ topological cone over $\C$.

  For every pointed $T_0$ topological barycentric algebra $\B$,
  $\Val_{\leq 1, \fin} \B$ is the free weakly locally convex pointed
  $T_0$ topological barycentric algebra over $\B$.

  For every $T_0$ topological barycentric algebra $\B$,
  $\Val_{1, \fin} \B$ is the free weakly locally convex $T_0$
  topological barycentric algebra over $\B$.

  Namely, $\Val_\fin \C$ is a weakly locally convex (even locally
  linear) $T_0$ topological cone, and for every weakly locally convex
  $T_0$ topological cone $\Cb$ and every linear continuous map
  $f \colon \C \to \Cb$, there is a unique linear continuous map
  $\hat f \colon \Val_\fin \C \to \Cb$ such that
  $\hat f (\delta_x) = f (x)$ for every $x \in \C$, namely
  $\hat f \eqdef f \circ \beta$.  Similarly with pointed $T_0$
  topological barycentric algebras and linear continuous maps, or with
  $T_0$ topological barycentric algebras and affine continuous maps
  instead of $T_0$ topological cones and linear continuous maps.
\end{theorem}
\begin{proof}
  That $\Val_\fin \C$ is a locally linear $T_0$ topological cone is
  part of Proposition~\ref{prop:locconv:retract}.  Given $f$ and $\Cb$
  as above, if $\hat f$ exists then by linearity it must map every
  simple valuation $\nu \eqdef \sum_{x \in A} a_x \delta_x$ to
  $\sum_{x \in A} a_x f (x)$, so $\hat f$ is unique.  Conversely,
  $\hat f \eqdef f \circ \beta$ fits, where $\beta$ is given by
  Theorem~\ref{thm:locconv:heckmann:beta}.  The argument is the same
  with (pointed) $T_0$ topological barycentric algebras and linear (or
  affine) continuous maps.  \qed
\end{proof}

\subsection{Sober spaces, and barycenters, part 5}
\label{sec:sober-spac-baryc}

Heckmann also shows that $\Val_\pw X$ is the free weakly locally
convex \emph{sober} topological cone over $X$
\cite[Theorem~6.8]{heckmann96}.  We show a similar result for
(pointed) barycentric algebras, with a similar technique.

Every space $X$ has a \emph{sobrification}, and that is a sober space
$X^s$, together with a continuous map $\eta \colon X \to X^s$, such
that for every continuous map $f \colon X \to Y$ where $Y$ is sober,
there is a unique continuous map $\widehat f \colon X^s \to Y$ such
that $\widehat f \circ \eta = f$ (see \cite[Chapter~8]{JGL-topology}).
Sobrifications are unique up to isomorphism, and the \emph{canonical
  sobrification} $\Sober X$ is built as the space of all irreducible
closed subsets of $X$ with the \emph{hull-kernel topology}, whose open
subsets are the sets of the form
$\diamond U \eqdef \{C \in \Sober X \mid C \cap U \neq \emptyset\}$,
$U \in \Open X$.  The map $\eta$ is defined as
$\etaS \colon X \to \Sober X$, where $\etaS (x) \eqdef \dc x$ for
every $x \in X$.  The specialization ordering of $\Sober X$ is
inclusion.  The map $U \mapsto \diamond U$ defines an order
isomorphism between $\Open X$ and $\Open {\Sober X}$.

Heckmann shows that $\Val_\pw X$ (not $\Val X$ in general) is a
sobrification of $\Val_\fin X$, with inclusion map for $\eta$
\cite[Theorem~5.5]{heckmann96}.  We establish a similar statement for
$\Val_{\leq 1, \pw} X$, and for $\Val_{1, \pw} X$ when $X$ is pointed.
The latter requires a simple trick, originally due to Abbas Edalat
\cite[Section~3]{edalat95a}.
\begin{lemma}
  \label{lemma:VwX:pointed}
  For every topological space $X$, $\Val_1 X$ is homeomorphic to
  $\Val_{\leq 1} (X \diff \dc \bot)$.  In one direction, the
  homeomorphism maps every $\nu \in \Val_1 X$ to
  $\nu' \in \Val_{\leq 1} (X \diff \dc \bot)$, defined by
  $\nu' (U) \eqdef \nu (U)$ for every
  $U \in \Open (X \diff \dc \bot)$; in the other direction, every
  $\nu' \in \Val_{\leq 1} (X \diff \dc \bot)$ is mapped to the
  $\nu \in \Val_1 X$ defined by $\nu (U) \eqdef \nu' (U)$ if
  $\bot \not\in U$, $\nu (U) \eqdef 1$ otherwise.

  The homeomorphism restricts to one between $\Val_{1, \fin} X$ and
  $\Val_{\leq 1, \fin} (X \diff \dc \bot)$.
\end{lemma}
\begin{proof}
  Let $g$ be the map $\nu \mapsto \nu'$.  For every
  $U \in \Open (X \diff \dc \bot)$, $U$ is also open in $X$, because
  $\dc \bot$ is closed in $X$, so $X \diff \dc X$ is open in $X$.
  Hence the definition of $\nu' = g (\nu)$ makes sense.  It is also
  clear that $\nu'$ is a subprobability valuation.  For every open
  subset $U$ of $X \diff \dc \bot$, for every $r \in \Rp \diff \{0\}$,
  $g^{-1} ([U > r]) = [U > r]$, so $g$ is continuous.

  Let $f$ be the map $\nu' \to \nu$.  For every subprobability
  valuation $\nu'$ on $X$, $\nu \eqdef f (\nu')$ is strict
  ($\nu (\emptyset) = 0$); it is modular: for all open subsets $U$ and
  $V$ of $X$, $\nu (U) + \nu (V) = \nu (U \cup V) + \nu (U \cap V)$ if
  $\bot \not\in U, V$, while if $\bot \in U$ (hence $U=X$),
  $\nu (U) + \nu (V) = \nu (U \cup V) + \nu (U \cap V)$ (because
  $U = X = U \cup V$ and $U \cap V = V$), and symmetrically if
  $\bot \in V$.  Additionally, $\nu$ is Scott-continuous, because any
  directed family of open sets ${(U_i)}_{i \in I}$ has a union that
  contains $\bot$ if and only if some $U_i$ contain $\bot$.  The map
  $f$ is continuous, because $f^{-1} ([U > r]) = [U > r]$ if
  $\bot \not\in U$, otherwise $U=X$ and then $[U > r]$ is the whole of
  $\Val_1 X$ or the empty set depending on whether $r < 1$ or not, and
  its inverse image under $f$ is open, too.  It is clear that $f$ and
  $g$ are inverse of each other.

  Finally, $g$ maps every simple probability valuation
  $\sum_{i=1}^n a_i \delta_{x_i}$ on $X$ to
  $\sum_{\substack{i=1\\x_i \not\leq \bot}}^n a_i \delta_{x_i}$, and
  $g$ maps every simple subprobability valuation
  $\sum_{i=1}^n a_i \delta_{x_i}$ on $X \diff \dc \bot$ to
  $\sum_{i=1}^n a_i \delta_{x_i} + (1-\sum_{i=1}^n a_i) \delta_\bot$.
  \qed
\end{proof}

\begin{theorem}
  \label{thm:VpX:sobrif}
  Let $X$ be a topological space, and $\rast$ be nothing or
  ``$\leq 1$''.  Then $\Val_{\rast,\pw} X$ is a sobrification of
  $\Val_{\rast, \fin} X$.  When $\rast$ is ``1'', $\Val_{1,\fin} X$ is
  a sobrification of $\Val_{1,\fin X}$ if $X$ is pointed, namely if
  $X$ has a least element in its specialization preordering.
\end{theorem}
\begin{proof}
  The case where $\rast$ is nothing was dealt with by Heckmann.  When
  $\rast$ is ``$\leq 1$'', $\Val_{\leq 1, \pw} X$ is a closed subspace
  of $\Val_\pw X$, since it is equal to the complement of $[X > 1]$,
  which is open.  Every closed subspace of a sober space is sober.  In
  fact, the sober subspaces of a sober space are exactly those that
  are closed in the \emph{Skula topology}, which is the coarsest
  topology containing both the open and the closed subsets of $X$
  \cite[Corollary~3.5]{KL:dtopo}.

  We will now use the fact that given any space $Y$ included in a
  sober space $Z$, the \emph{Skula closure} $cl_s (Y)$ of $Y$ in $Z$,
  namely its closure in the Skula topology, is a sobrification of $Y$
  \cite[Corollary~3.6]{KL:dtopo}.

  We take $Y \eqdef \Val_{\leq 1, \fin} X$, $Z \eqdef \Val_\pw X$; $Z$
  is sober \cite[Proposition~5.1]{heckmann96}.  We have seen that
  $\Val_{\leq 1, \pw} X$ is closed, hence Skula-closed in $Z$.  Since
  it contains $Y$, $cl_s (Y) \subseteq \Val_{\leq 1, \pw} X$.  We show
  the reverse implication by showing that $\Val_{\leq 1, \pw} X$ is
  included in every Skula-closed subset of $Z$ that contains $Y$;
  equivalently, by showing that every Skula-open subset of $Z$ that
  intersects $\Val_{\leq 1, \pw} X$ intersects $Y$.  Given any
  Skula-open subset $\mathcal O$ of $Z$ that intersects
  $\Val_{\leq 1, \pw} X$, we pick $\nu$ in the intersection, and then
  $\nu$ is in $\mathcal U \cap \mathcal C \subseteq \mathcal O$ for
  some open subset $\mathcal U$ of $Z$ and some closed subset
  $\mathcal C$ of $Z$, by definition of the Skula topology.  By
  definition of the weak topology (on $Z = \Val_\pw X$), $\nu$ is in
  some intersection of subbasic open sets
  $\bigcap_{i=1}^n [U_i > r_i]$ included in $\mathcal U$.  Lemma~5.4
  of \cite{heckmann96} states that for every bounded point-continuous
  valuation $\nu$ on $X$, if $\nu$ is in some intersection of subbasic
  open sets $\bigcap_{i=1}^n [U_i > r_i]$, then there is a simple
  valuation $\mu \leq \nu$ in the same intersection.  In our case,
  $\mu$ is then in $\mathcal U$, and is also in $\mathcal C$ since
  $\mu \leq \nu$ and every closed set is downwards-closed.  Since
  $\mu \leq \nu$, $\mu$ is a simple valuation and
  $\nu \in \Val_{\leq 1, \pw} X$, $\mu$ is in $\Val_{\leq 1, \fin} X$.
  Hence $\mathcal O$ intersects $\Val_{\leq 1, \fin} X$ at $\mu$.

  When $\rast$ is ``$1$'' and $X$ is pointed, we use
  Lemma~\ref{lemma:VwX:pointed}: $\Val_1 X$ is homeomorphic to
  $\Val_{\leq 1} (X \diff \dc \bot)$, by a homeomorphism that
  restricts to one between $\Val_{1,\fin} X$ and
  $\Val_{\leq 1, \fin} (X \diff \dc \bot)$.  The result then follows
  from the first part of the theorem, applied to the space
  $X \diff \dc \bot$.  \qed
\end{proof}

The map $\beta$ below extends the map $\beta$ of
Theorem~\ref{thm:locconv:heckmann:beta}.  According, we also call it
the \emph{barycenter map}.
\begin{theorem}[Barycenters, part 5]
  \label{thm:locconv:heckmann:beta:pcont}
  Let $\C$ be a weakly locally convex, sober topological cone.  There
  is a unique linear continuous map $\beta \colon \Val_\pw \C \to \C$
  such that $\beta (\delta_x) = x$ for every $x \in
  \C$.  

  Given a weakly locally convex sober pointed topological barycentric
  algebra $\B$, there is a unique linear continuous map $\beta$ from
  $\Val_{\leq 1,\pw} \B$ to $\B$ such that $\beta (\delta_x) = x$ for
  every $x \in \B$, and a unique affine continuous map $\beta$ from
  $\Val_{1, \pw} \B$ to $\B$ such that $\beta (\delta_x) = x$ for
  every $x \in \B$.
\end{theorem}
One should beware that the existence and uniqueness of $\beta$ in the
last case, from $\Val_{1, \pw} \B$ to $\B$, still requires $\B$ to be
pointed.

\begin{proof}
  We deal with the first part of the Theorem first.  This is essentially
  Theorem~6.8 of \cite{heckmann96}.
  
  Let us consider the continuous map
  $\beta \colon \Val_\fin \C \to \C$ of
  Theorem~\ref{thm:locconv:heckmann:beta}.  Since $\Val_\pw \C$ is the
  sobrification of $\Val_\fin \C$, $\beta$ extends to unique
  continuous map from $\Val_\pw \C$ to $\C$, which we temporarily
  write as $\hat\beta$.

  For every $a \in \Rp$, the maps $\nu \mapsto a \hat\beta (\nu)$ and
  $\nu \mapsto \hat\beta (a\cdot \nu)$ are continuous from
  $\Val_\pw \C$ to $\C$, and coincide on simple valuations:
  $a \hat\beta (\sum_{x \in A} a_x \delta_x) = a \sum_{x \in A} a_x
  \cdot x = \sum_{x \in A} a a_x \cdot x = \hat\beta (\sum_{x \in A}
  aa_x \delta_x)$.  The continuous extension of a continuous map from
  $\Val_\fin \C$ to its sobrification $\Val_\pw \C$ is unique, so
  $a \hat\beta (\nu) = \hat\beta (a \cdot \nu)$ for all $a \in \Rp$
  and $\nu \in \Val_\pw \C$.

  Next, we show that $\hat\beta$ preserves sums.  Given a simple
  valuation $\mu \eqdef \sum_{x \in A} a_x \delta_x$, the maps
  $\nu \mapsto \beta (\mu) + \hat\beta (\nu)$ and
  $\nu \mapsto \hat\beta (\mu + \nu)$ are continuous from
  $\Val_\pw \C$ to $\C$, and coincide on simple valuations.  By the
  same argument on sobrifications as above, they define the same map.
  Explicitly,
  $\hat\beta (\mu+\nu) = \hat\beta (\mu) + \hat\beta (\nu)$ for every
  simple valuation $\mu$ and every point-continuous valuation $\nu$.
  It follows that, given any fixed point-continuous valuation $\nu$,
  the maps $\mu \mapsto \hat\beta (\mu) + \hat\beta (\nu)$ and
  $\mu \mapsto \hat\beta (\mu+\nu)$ are continuous from $\Val_\pw \C$
  to $\C$, and coincide on simple valuations.  Again, they coincide by
  the universal property of sobrifications, hence
  $\hat\beta (\mu+\nu) = \hat\beta (\mu) + \hat\beta (\nu)$ for all
  $\mu, \nu \in \Val_\pw \C$.

  This much shows the existence of a linear continuous map
  $\hat\beta \colon \Val_\pw \C \to \C$ such that
  $\hat\beta (\delta_x) = x$ for every $x \in \C$.  Uniqueness is
  clear: any such $\hat\beta$ must restrict to a linear continuous map
  $\beta \colon \Val_\fin \C \to \C$ such that $\beta (\delta_x) = x$
  for every $x \in \C$; such a $\beta$ is unique by
  Theorem~\ref{thm:locconv:heckmann:beta}, and then $\hat\beta$ is its
  unique continuous extension to its sobrification $\Val_\pw \C$.

  Second, we deal with the case of a weakly locally convex pointed
  sober topological barycentric algebra $\B$.  We consider the linear
  continuous map $\beta \colon \Val_{\leq 1, \fin} \B \to \B$ of
  Theorem~\ref{thm:locconv:heckmann:beta}, which we extend to a unique
  continuous map $\hat\beta \colon \Val_{\leq 1, \pw} \B \to \B$,
  using Theorem~\ref{thm:VpX:sobrif}.  Since the least element of
  $\Val_{\leq 1, \pw} \B$ is the zero valuation $0$, which is in
  $\Val_{\leq 1, \fin} \B$, $\hat\beta (0) = \beta (0) = \bot$, using
  the fact that $\beta$ is linear hence strict.  Hence $\hat\beta$ is
  itself strict.

  We claim that $\hat\beta$ is affine, which will show that it is
  linear, thanks to Proposition~\ref{prop:strict:aff}.  The argument
  is the same as for sums.  For every $a \in [0, 1]$, for every
  $\mu \in \Val_{\leq 1, \fin} \B$, the functions
  $\nu \mapsto \beta (\mu) +_a \hat\beta (\nu)$ and
  $\nu \mapsto \hat\beta (\mu +_a \nu)$ are continuous from
  $\Val_{\leq 1, \pw} \B$ to $\B$, and coincide on
  $\Val_{\leq 1, \fin} \B$, since $\beta$ is affine.  Hence they
  coincide on $\Val_{\leq 1, \pw} \B$, by the universal property of
  sobrifications.  Since $\hat\beta (\mu) = \beta (\mu)$, this shows
  that $\hat\beta (\mu +_a \nu) = \hat\beta (\mu) +_a \hat\beta (\nu)$
  for all $\mu \in \Val_{\leq 1, \fin} \B$ and
  $\nu \in \Val_{\leq 1, \pw} \B$.  It follows that, for every
  $\nu \in \Val_{\leq 1, \pw} \B$, the maps
  $\mu \mapsto \hat\beta (\mu +_a \nu)$ and
  $\mu \mapsto \hat\beta (\mu) +_a \hat\beta (\nu)$, which are
  continuous from $\Val_{\leq 1, \pw} \B$ to $\B$, coincide on
  $\Val_{\leq 1, \fin} \B$.  Hence they coincide on
  $\Val_{\leq 1, \pw} \B$, by the universal property of
  sobrifications.

  We have just shown the existence of a linear continuous map
  $\hat\beta \colon \Val_{\leq 1, \pw} \B \to \B$ such that
  $\hat\beta (\delta_x) = x$ for every $x \in \B$.  Uniqueness is
  proved as in the case of cones.

  Third, we still consider a weakly locally convex sober pointed
  topological barycentric algebra $\B$, but we focus on
  $\Val_{1, \pw}$ rather than $\Val_{\leq 1, \pw}$.  We have just
  built a linear continuous map
  $\beta \colon \Val_{\leq 1, \pw} \B \to \B$.  By restriction to
  $\Val_{1, \pw} \B$, we obtain an affine continuous map $\beta_1$
  from $\Val_{1, \pw} \B$ to $\B$ such that $\beta_1 (\delta_x) = x$
  for every $x \in \B$.  Uniqueness is, once again, proved as in the
  case of cones: given any affine continuous map $\hat\beta$ from
  $\Val_{1, \pw} \B$ to $\B$ such that $\hat\beta (\delta_x) = x$ for
  every $x \in \B$, its restriction to $\Val_{1, \fin} \B$ is uniquely
  determined by Theorem~\ref{thm:locconv:heckmann:beta}, and then
  $\hat\beta$ is its unique continuous extension to its sobrification
  $\Val_{1, \pw} \B$.  The fact that $\Val_{1, \pw} \B$ is a
  sobrification of $\Val_{1, \fin} \B$ is by
  Theorem~\ref{thm:VpX:sobrif}, which applies because $\B$ is
  pointed.  \qed
\end{proof}

\begin{remark}
  \label{rem:locconv:heckmann:beta:pcont}
  Theorem~\ref{thm:locconv:heckmann:beta:pcont} does not have a
  statement about the existence of $\beta \colon \Val_{1, \pw} \B \to
  \B$ for non-pointed weakly locally convex sober topological
  barycentric algebras $\B$.  This is because using
  Theorem~\ref{thm:VpX:sobrif} as we did in the proof requires $\B$ to
  be pointed anyway.
\end{remark}

\begin{corollary}
  \label{corl:locconv:retract:sober}
  The weakly locally convex sober topological cones are exactly the
  semitopological cones $\C$ that occur as linear retracts of locally
  convex (even locally linear) sober topological cones; explicitly,
  one such linear retraction is $\beta \colon \Val_\pw \C \to \C$.

  The weakly locally convex sober pointed topological barycentric
  algebras are exactly the pointed semitopological barycentric
  algebras $\B$ that occur as linear retracts of locally linear sober
  pointed topological barycentric algebras; one such linear retraction
  is $\beta \colon \Val_{\leq 1, \pw} \B \to \B$.
\end{corollary}
\begin{proof}
  Every linear retract of a weakly locally convex sober topological
  cone is weakly locally convex and topological by
  Proposition~\ref{prop:locconv:retract} and sober, because sobriety
  is preserved by retracts.  Conversely, if $\C$ is a weakly locally
  convex sober topological cone, then the map $\beta$ of
  Theorem~\ref{thm:locconv:heckmann:beta:pcont} is a linear
  retraction, with section $\eta_\C \colon \C \to \Val_\pw \C$.
  Moreover, $\Val_\pw \C$ is topological
  (Example~\ref{exa:VX:top:cone}), locally linear
  (Example~\ref{exa:VX:loclin}), and sober, being a sobrification
  (Theorem~\ref{thm:VpX:sobrif}).

  The argument is the same for pointed barycentric algebras $\B$
  instead of cones.  $\Val_{\leq 1, \pw} \B$ is topological since it
  is a subspace of the topological cone $\Val_\pw \B$, locally linear
  (Example~\ref{exa:Vleq1X:loclin}) and sober
  (Theorem~\ref{thm:VpX:sobrif}).  \qed
\end{proof}

The following is also due to Heckmann \cite{heckmann96}, in the case
of cones.
\begin{theorem}
  \label{thm:locconv:heckmann:sober:free}
  For every sober topological cone $\C$, $\Val_\pw \C$ is the free
  weakly locally convex sober topological cone over $\C$.

  For every sober pointed topological barycentric algebra $\B$,
  $\Val_{\leq 1, \pw} \B$ is the free weakly locally convex sober
  pointed topological barycentric algebras over $\B$.

  Namely, $\Val_\pw \C$ is a weakly locally convex, even locally
  linear, sober topological cone, and for every weakly locally convex
  sober topological cone $\Cb$ and every linear continuous map
  $f \colon \C \to \Cb$, there is a unique linear continuous map
  $\hat f \colon \Val_\pw \C \to \Cb$ such that
  $\hat f (\delta_x) = f (x)$ for every $x \in \C$, namely
  $\hat f \eqdef f \circ \beta$.  Similarly with $\B$ instead of $\C$,
  pointed topological barycentric algebras instead of topological
  cones, and $\Val_{\leq 1, \pw}$ instead of $\Val_\pw$.
\end{theorem}
\begin{proof}
  That $\Val_\pw \C$ is a locally convex $T_0$ topological cone is
  part of Corollary~\ref{corl:locconv:retract:sober}.  Given $f$ and
  $\Cb$ as above, if $\hat f$ exists then by linearity it must map
  every simple valuation $\nu \eqdef \sum_{x \in A} a_x \delta_x$ to
  $\sum_{x \in A} a_x \cdot f (x) = f (\beta (\nu))$, so $\hat f$ is
  determined uniquely on $\Val_\fin \C$.  Since $\Cb$ is sober, by the
  uniqueness part in the universal property of sobrifications and
  since $\Val_\pw \C$ is a sobrification of $\Val_\fin \C$
  (Theorem~\ref{thm:VpX:sobrif}), $\hat f$ is determined uniquely on
  the whole of $\Val_\pw \C$.  Conversely,
  $\hat f \eqdef f \circ \beta$ fits, where $\beta$ is given by
  Theorem~\ref{thm:locconv:heckmann:beta:pcont}.
  The argument is similar with pointed topological barycentric
  algebras instead of topological cones and adding ``$\leq 1$''
  indices to $\Val$ everywhere.  \qed
\end{proof}

\section{Consistent barycentric algebras}
\label{sec:cons-baryc-algebr}

There are a few theorems on semitopological cones $\C$ that require
addition to be \emph{almost open}, in the sense that $\upc (U + V)$ is
open for all open subsets $U$ and $V$ of $\C$
\cite[Definition~4.6]{Keimel:topcones2}.  There are two natural
extensions of the notion to semitopological barycentric algebras.
\begin{definition}[Consistent barycentric algebra]
  \label{defn:consistent:bary:alg}
  A semitopological barycentric algebra $\B$ is \emph{consistent} if
  and only if for all open subsets $U$ and $V$ of $\B$, for every open
  subset $I$ of $[0, 1]$, $\upc (U +_I V)$ is open in $\B$, where
  $U +_I V \eqdef \{x +_a y \mid x \in U, y \in V, a \in I\}$.

  It is \emph{strongly consistent} if and only if for all open subsets
  $U$ and $V$ of $\B$, for every $a \in {[0, 1]}$, $\upc (U +_a V)$ is
  open in $\B$, where
  $U +_a V \eqdef \{x +_a y \mid x \in U, y \in V\}$.
\end{definition}

\begin{fact}
  \label{fact:strong:consistent}
  A semitopological barycentric algebra $\B$ is strongly consistent if
  and only if for all open subsets $U$ and $V$ of $\B$, for every
  $a \in {]0, 1[}$, $\upc (U +_a V)$ is open in $\B$.  Indeed, for $a
  \eqdef 0$, $\upc (U +_a V) = \upc V = V$ and for $a \eqdef 1$, $\upc
  (U +_a V) = \upc U = U$, and both are open by assumption.
\end{fact}

\begin{fact}
  \label{fact:strong:consistent:consistent}
  Every strongly consistent semitopological barycentric algebra is
  consistent.
\end{fact}
Indeed, for all open subsets $U$ and $V$ of a strongly consistent
barycentric algebra $\B$, for every open subset $I$ of $[0, 1]$,
$\upc (U +_I V) = \bigcup_{a \in I} \upc (U +_a V)$, which is a union
of open sets if $\B$ is strongly consistent.

\begin{example}
  \label{exa:lattice:consistent:strong}
  Every sup-semilattice $L$ (see
  Example~\ref{exa:lattice:top:bary:alg}) is strongly consistent,
  hence consistent by Fact~\ref{fact:strong:consistent:consistent}.
  Indeed, let us rely on Fact~\ref{fact:strong:consistent}, and let us
  fix $a \in {]0, 1[}$.  For all Scott-open subsets $U$ and $V$,
  $U +_a V = \{x \vee y \mid x \in U, y \in V\} = U \cap V$: for all
  $x \in U$ and $y \in V$, $x \vee y$ is both in $U$ and in $V$, since
  both are upwards-closed, and conversely, every element
  $z \in U \cap V$ can be written as $z \vee z$, where $z \in U$ and
  $z \in V$.  Hence $\upc (U +_a V) = U \cap V$ is open.
\end{example}

\begin{lemma}
  \label{lemma:consistent:bary:alg}
  A semitopological barycentric algebra $\B$ is consistent if and only
  if $U +_I V$ is open for all open subsets $U$ and $V$ of $\B$ and
  every non-empty open interval $I$ included in $]0, 1[$.
\end{lemma}
\begin{proof}
  The only if direction is clear.  In the if direction, let $U$ and
  $V$ be open in $\B$ and $I$ be an open subset of $[0, 1]$.  Then $I$
  is a union of a family ${(I_k)}_{k \in K}$ of non-empty intervals of
  the form $]a, b[$ with $0\leq a < b\leq 1$, or $[0, b[$ or $]a, 1]$.
  We rewrite the latter two as the union of $]0, b[$ with $\{0\}$, and
  as the union of $]a, 1[$ with $\{1\}$, if necessary; so we may
  assume that every $I_k$ is a non-empty interval included in
  $]0, 1[$, or $\{0\}$ or $\{1\}$.  By assumption, if $I_k$ is of the
  first kind, $\upc (U +_{I_k} V)$ is open in $\B$, while
  $\upc (U +_{\{0\}} V) = \upc V = V$ and
  $\upc (U +_{\{1\}} V) = \upc U = U$.  Therefore
  $\upc (U +_I V) = \bigcup_{k \in K} \upc (U +_{I_k} V)$ is open in
  $\B$.  \qed
\end{proof}
\begin{proposition}
  \label{prop:cons:equiv}
  The following are equivalent for a semitopological cone $\C$:
  \begin{enumerate}
  \item addition is almost open on $\C$;
  \item $\C$ is strongly consistent, seen as a semitopological
    barycentric algebra;
  \item $\C$ is consistent, seen as a semitopological barycentric
    algebra.
  \end{enumerate}
\end{proposition}
\begin{proof}
  $1 \limp 2$.  If addition is almost open on $\C$, then let $U$ and
  $V$ be two open subsets of $\C$ and $a \in {]0, 1[}$.  We have
  $U +_a V = (a \cdot U) + ((1-a) \cdot V) = (\frac 1 a \cdot \_)^{-1}
  (U) + (\frac 1 {1-a} \cdot \_)^{-1} (V)$.  Since
  $(\frac 1 a \cdot \_)^{-1} (U)$ and
  $(\frac 1 {1-a} \cdot \_)^{-1} (V)$ are open in $\C$,
  $\upc (U +_a V)$ is open.  Hence $\C$ is strongly consistent as a
  semitopological barycentric algebra, thanks to
  Fact~\ref{fact:strong:consistent}.

  $2 \limp 3$ is by Fact~\ref{fact:strong:consistent:consistent}.

  $3 \limp 1$.  Let us assume that $\C$ is consistent as a
  semitopological barycentric algebra.  Let $U$ and $V$ be two open
  subsets of $\C$.  For every $z \in \upc (U + V)$, there are points
  $x \in U$ and $y \in V$ such that $x+y \leq z$.  Since $\_ \cdot x$
  is continuous from $\Rp$ to $\C$, there is a number $b \in {]0, 1[}$
  such that $b \cdot x \in U$, and similarly there is a number
  $c \in {]0, 1[}$ such that $c \cdot y \in V$.  Let
  $U' \eqdef (b \cdot \_)^{-1} (U)$ and
  $V' \eqdef (c \cdot \_)^{-1} (V)$.  Since $b \cdot x \in U$, $x$ is
  in $U'$, and similarly $y \in V'$.  Let $\epsilon > 0$ be such that
  $1-2\epsilon > \max (b, c)$.  For all $x' \in U'$, $y' \in U'$ and
  $a \in {]\frac 1 2 - \epsilon, \frac 1 2 + \epsilon[}$,
  $2a > 1 - 2 \epsilon > b$ and $2 (1-a) > 1 - 2 \epsilon > c$, so
  $2a \cdot x' \geq b \cdot x'$ is in $U$,
  $2(1-a) \cdot y' \geq c \cdot y'$ is in $V$, and therefore
  $2 \cdot (x' +_a y') = 2a \cdot x' + 2 (1-a) \cdot y'$ is in
  $U + V$.  We have just shown that
  $2 \cdot (U' +_{]\frac 1 2 - \epsilon, \frac 1 2 + \epsilon[} V')$
  is included in $U+V$.  Since $2 \cdot \_$ is monotonic, it follows
  that $W \subseteq \upc (U+V)$, where
  $W \eqdef 2 \cdot \upc (U' +_{]\frac 1 2 - \epsilon, \frac 1 2 +
    \epsilon[} V')$.  By assumption,
  $\upc (U' +_{]\frac 1 2 - \epsilon, \frac 1 2 + \epsilon[} V')$ is
  open in $\C$.  Using the equality
  $2 \cdot A = (\frac 1 2 \cdot \_)^{-1} (A)$, $2 \cdot A$ is open for
  every open set $A$; so $W$ is open in $\C$.  Since $x \in U'$ and
  $y \in V'$, $2 \cdot (x +_{\frac 1 2} y)$ is in $W$, equivalently
  $x + y \in W$, and since $x + y \leq z$, $z$ is in $W$.  We have
  therefore found an open neighborhood $W$ of $z$ included in
  $\upc (U + V)$, for every $z \in \upc (U+V)$, and this shows that
  $\upc (U+V)$ is open in $\C$.  \qed
\end{proof}

There is another case where consistency and strong consistency
coincide.
\begin{lemma}
  \label{lemma:topo:consistent}
  For every topological barycentric algebra $\B$, $\B$ is strongly
  consistent if and only if $\B$ is consistent.
\end{lemma}
\begin{proof}
  Let $U$ and $V$ be two open subsets of $\B$ and $a \in {]0, 1[}$.
  In order to show that $\upc (U +_a V)$ is open, we consider any
  point $z \in \upc (U +_a V)$, and we will find an open neighborhood
  $W$ of $z$ contained in $\upc (U +_a V)$.  Since
  $z \in \upc (U +_a V)$, there are two points $x \in U$ and $y \in V$
  such that $x +_a y \leq z$.  The function
  $f \colon (x', a', y') \mapsto x' +_{a'} y'$ is continuous since
  $\B$ is topological, and $f (x, 1, y) = x \in U$, so there is a
  basic open set $U' \times I \times V'$ (with $U'$, $I$, $V'$ open in
  $\B$, $[0, 1]$ and $\B$ respectively) containing $(x, 1, y)$ and
  mapped by $f$ to points of $U$.  Replacing $U'$ by $U \cap U'$ and
  $V'$ by $V \cap V'$ if necessary, we may assume that
  $U' \subseteq U$ and $V' \subseteq V$.  We may also assume that
  $I = {]1-\epsilon, 1]}$ for some $\epsilon \in {]0, 1[}$.  For all
  $x' \in U'$, $r \in I$ and $y' \in V'$, $x' +_r y' \in U$, and since
  $y' \in V' \subseteq V$, $(x' +_r y') +_a y' \in \upc (U +_a V)$.
  Now $(x' +_r y') +_a y' = x' +_{ra} y'$ (a fact that is easier to
  see by embedding $\B$ in a cone first), and $ra$ ranges over
  ${](1-\epsilon)a, a]}$ when $r$ ranges over $I$.  Hence
  $U' +_{](1-\epsilon)a, a]} V' \subseteq \upc (U +_a V)$.

  Symmetrically, we can find an open neighborhood $U'' \subseteq U$ of
  $x$, an open neighborhood $I' \eqdef {[0, \epsilon'[}$ of $0$ and an
  open neighborhood $V'' \subseteq V$ of $y$ such that
  $x' +_r y' \in V$ for all $x' \in U''$, $r \in I'$ and $y' \in V''$.
  We now use the equality
  $x' +_a (x' +_r y') = x' +_{1-(1-a)(1-r)} y'$, and we realize that
  $1-(1-a)(1-r) = a + (1-a) r$ ranges over $[a, a+(1-a)\epsilon'[$
  when $r$ ranges over $I'$.  Hence
  $U'' +_{[a, a+(1-a)\epsilon'[} V'' \subseteq \upc (U +_a V)$.
  Together with
  $U' +_{](1-\epsilon)a, a]} V' \subseteq \upc (U +_a V)$, this allows
  us to infer that
  $(U' \cap U'') +_{](1-\epsilon)a, a+(1-a)\epsilon'[} (V' \cap V'')$
  is included in $\upc (U +_a V)$.  Let
  $W \eqdef \upc ((U' \cap U'') +_{](1-\epsilon)a, a+(1-a)\epsilon[}
  (V' \cap V''))$; this is open since $\B$ is consistent, and
  $W \subseteq \upc (U +_a V)$.  Since $x \in U' \cap U''$,
  $a \in {](1-\epsilon)a, a+(1-a)\epsilon[}$ and $y \in V' \cap V''$,
  $z$ is in $W$, and this completes the proof.  \qed
\end{proof}

\begin{problem}
  \label{pb:consistent:strong}
  Is every consistent semitopological barycentric algebra strongly
  consistent?  A counterexample, if it exists, cannot be topological
  (because of Lemma~\ref{lemma:topo:consistent}), and cannot be a cone
  (by Proposition~\ref{prop:cons:equiv}).
\end{problem}

\subsection{Convex open sets in consistent semitopological barycentric
  algebras}
\label{sec:convex-open-sets}

\begin{proposition}
  \label{prop:int:conv}
  Let $\B$ be a semitopological barycentric algebra.
  \begin{enumerate}
  \item If $\B$ is consistent, then the interior of every convex
    saturated subset is convex.
  \item If $\B$ is consistent, then for every $n \geq 1$, for every
    open subset $V$ of $\Delta_n$, for all open subsets $U_1$, \ldots,
    $U_n$ of $\B$, the set
    $\upc \{\sum_{i=1}^n a_i \cdot x_i \mid (a_1, \cdots, a_n) \in V,
    x_1 \in U_1, \cdots, x_n \in U_n\}$ is open in $\B$.
  \item If $\B$ is consistent and weakly locally convex, then for
    every open subset $U$ of $\B$, $\upc \conv U$ is open and convex,
    and is the smallest convex open set containing $U$.
  \end{enumerate}
\end{proposition}
\begin{proof}
  1.  Let $A$ be a convex saturated subset of $\C$.  Let
  $x, y \in \interior A$, $a \in [0, 1]$ and $I$ be any open
  neighborhood of $a$ in $[0, 1]$ with its standard topology.  Since
  $\B$ is consistent, $U \eqdef \upc (\interior A +_I \interior A)$ is
  open, and $x +_a y \in U$.  Additionally, for every $z \in U$, there
  are two points $x', y' \in \interior A$ (hence in $A$) and a number
  $a' \in I$ such that $x' +_{a'} y' \leq z$.  Since $A$ is convex and
  saturated, $z \in A$.  This shows that $U \subseteq A$, hence
  $U \subseteq \interior A$, and therefore $x +_a y \in \interior A$.

  2.  Let
  $\mathrm{Bary} (V; U_1, \cdots, U_n) \eqdef \upc \{\sum_{i=1}^n a_i
  \cdot x_i \mid (a_1, \cdots, a_n) \in V, x_1 \in U_1, \cdots, x_n
  \in U_n\}$, for each $n \geq 1$.  We show that
  $\upc\mathrm{Bary} (V; U_1, \cdots, U_n)$ is open by induction on
  $n \geq 1$, using the formula of Definition and
  Proposition~\ref{prop:bary:1}, item~1, for the sum
  $\sum_{i=1}^n a_i \cdot x_i$.  When $n=1$, either $V$ is empty and
  $\upc\mathrm{Bary} (V; U_1) = \emptyset$, or $V = \{(1)\}$ and
  $\upc\mathrm{Bary} (V; U_1) = \upc U_1 = U_1$, which is open.  In
  the induction case, let $V' \eqdef V \diff \{(0, \cdots, 0, 1)\}$.
  This is also open in $\Delta_n$, and
  $\upc\mathrm{Bary} (V; U_1, \cdots, U_n)$ is equal to
  $\upc\mathrm{Bary} (V'; U_1, \cdots, U_n)$ if
  $(0, \cdots, 0, 1) \not\in V$, to
  $\upc\mathrm{Bary} (V'; U_1, \cdots, U_n) \cup U_n$ otherwise.
  Hence it suffices to show that
  $\upc\mathrm{Bary} (V'; U_1, \cdots, U_n)$ is open in $\B$.  Let $Z$
  be the subspace of $\Delta_{n-1} \times {]0, 1]}$ with the product
  of the usual metric topologies.  Let
  $f \colon \Delta_n \diff \{(0, \cdots, 0, 1)\} \to Z$ map every
  $(a_1, \cdots, a_n)$ to
  $((\frac {a_1} {1-a_n}, \cdots, \frac {a_{n-1}} {1-a_n}), 1-a_n)$,
  and
  $g \colon Z \to \Delta_n \diff \{(0, \cdots, 0, \allowbreak 1)\}$
  map every $((b_1, \cdots, b_{n-1}), b_n)$ to
  $(b_n b_1, \allowbreak \cdots, \allowbreak b_n b_{n-1}, \allowbreak
  1-b_n)$.  The functions $f$ and $g$ are inverse of each other, and
  continuous.  In particular, $f [V'] = g^{-1} (V')$ is open in $Z$,
  so we can write it as a union $\bigcup_{j \in J} V_j \times I_j$
  where each $V_j$ is open in $\Delta_{n-1}$ and $I_j$ is open in
  $]0, 1]$.  Then:
  \ifta
  \begin{align*}
    & \upc\mathrm{Bary} (V'; U_1, \cdots, U_n) \\
    & = \upc \left\{\left(\sum_{i=1}^{n-1} \frac {a_i} {1-a_n} \cdot x_i\right)
      +_{1-a_n} x_n
      \mid (a_1, \cdots, a_n) \in V', x_1 \in U_1, \cdots, x_n \in
      U_n\right\} \\
    & = \upc \left\{\left(\sum_{i=1}^{n-1} b_i \cdot x_i\right) +_{b_n} x_n \mid ((b_1,
      \cdots, b_{n-1}), b_n) \in f [V'], x_1 \in U_1, \cdots, x_n \in
      U_n\right\} \\
    & = \bigcup_{j \in J} \upc \left\{\left(\sum_{i=1}^{n-1} b_i \cdot
      x_i\right) +_{b_n} x_n \mid (b_1, \cdots, b_{n-1}) \in V_j, b_n
      \in I_j, x_1 \in U_1, \cdots, x_n \in U_n\right\} \\
    & = \bigcup_{j \in J} \upc (\mathrm{Bary} (V_j; U_1, \cdots,
      U_{n-1}) +_{I_j} U_n).
  \end{align*}
  \else
    \begin{align*}
    & \upc\mathrm{Bary} (V'; U_1, \cdots, U_n) \\
    & = \upc \biggl\{\left(\sum_{i=1}^{n-1} \frac {a_i} {1-a_n} \cdot x_i\right)
      +_{1-a_n} x_n \\
      & \qquad
      \mid (a_1, \cdots, a_n) \in V', x_1 \in U_1, \cdots, x_n \in
      U_n\biggr\} \\
    & = \upc \biggl\{\left(\sum_{i=1}^{n-1} b_i \cdot x_i\right)
      +_{b_n} x_n \\
      & \qquad \mid ((b_1,
        \cdots, b_{n-1}), b_n) \in f [V'], \\
      & \qquad\qquad x_1 \in U_1, \cdots, x_n \in
        U_n\biggr\} \\
    \end{align*}
    \begin{align*}
    & = \bigcup_{j \in J} \upc \biggl\{\left(\sum_{i=1}^{n-1} b_i \cdot
      x_i\right) +_{b_n} x_n \\
      & \qquad \mid (b_1, \cdots, b_{n-1}) \in V_j, b_n
        \in I_j, \\
      & \qquad\qquad x_1 \in U_1, \cdots, x_n \in U_n\biggr\} \\
    & = \bigcup_{j \in J} \upc (\mathrm{Bary} (V_j; U_1, \cdots,
      U_{n-1}) +_{I_j} U_n).
  \end{align*}
  \fi
  For every $j \in J$,
  $\upc (\mathrm{Bary} (V_j; U_1, \cdots, U_{n-1}) +_{I_j} U_n)
  \subseteq \upc (\upc \mathrm{Bary} (V_j; U_1, \allowbreak \cdots,
  \allowbreak U_{n-1}) +_{I_j} U_n)$ since
  $\mathrm{Bary} (V_j; U_1, \cdots, U_{n-1}) \subseteq \upc
  \mathrm{Bary} (V_j; U_1, \cdots, U_{n-1})$.  Conversely, every
  element $z$ of
  $\upc (\upc \mathrm{Bary} (V_j; U_1, \cdots, U_{n-1}) +_{I_j} U_n)$
  is larger than or equal to $x +_a y$ for some
  $x \in \upc \mathrm{Bary} (V_j; U_1, \cdots, U_{n-1})$, some
  $a \in I_j$ and some $y \in U_n$.  Then $x \geq x'$ for some
  $x' \in \mathrm{Bary} (V_j; U_1, \cdots, U_{n-1})$, so
  $z \geq x' +_a y$, since $+_a$ is continuous, hence monotonic.  It
  follows that
  $\upc (\upc \mathrm{Bary} (V_j; U_1, \cdots, U_{n-1}) +_{I_j} U_n)$
  is included in, and therefore equal to, the set
  $\upc (\mathrm{Bary} (V_j; U_1, \cdots, U_{n-1}) +_{I_j} U_n)$.  We
  have shown that
  $\upc\mathrm{Bary} (V'; U_1, \cdots, \allowbreak U_n) = \bigcup_{j
    \in J} \upc (\upc \mathrm{Bary} (V_j; U_1, \cdots, U_{n-1})
  +_{I_j} U_n)$.  By induction hypothesis,
  $\upc \mathrm{Bary} (V_j; U_1, \cdots, U_{n-1})$ is open.  Since
  $\B$ is consistent,
  $\upc (\upc \mathrm{Bary} (V_j; U_1, \cdots, \allowbreak U_{n-1})
  +_{I_j} U_n)$ is open for every $j \in J$, so
  $\upc\mathrm{Bary} (V'; \allowbreak U_1, \allowbreak \cdots,
  \allowbreak U_n)$ is open in $\B$.

  3.  $\upc \conv U$ is convex by Definition and
  Proposition~\ref{defprop:upc:conv}.  In order to show that it is
  open, we take any $x \in \upc \conv U$, and we find an open
  neighborhood of $x$ included in $\upc \conv U$.  By Definition and
  Lemma~\ref{deflem:convex:simpleval},
  $x \geq \sum_{i=1}^n a_i \cdot x_i$, for some points $x_i$ is in
  $U$, some coefficients $a_i \in \Rp$, and where
  $\sum_{i=1}^n a_i = 1$.  Let $V$ be some open neighborhood of
  $(a_1, \cdots, a_n)$ in $\Delta_n$.  Using weak local convexity, for
  each $i$, we can find a convex subset $C_i$ of $U$ whose interior
  $U_i \eqdef \interior {C_i}$ contains $x_i$.

  Let $W \eqdef \upc \mathrm{Bary} (V; U_1, \cdots, U_n)$, as defined
  and studied in item~2.  Notably, $W$ is open in $\B$.  By
  construction, $x \in W$.  $W$ is the upward closure of a collection
  of barycenters $\sum_{i=1}^n a'_i \cdot x'_i$ where
  $(a'_1, \cdots, a'_n) \in V \subseteq \Delta_n$ and
  $x'_1 \in U_1 \subseteq C_1$, \ldots, $x'_n \in U_n \subseteq C_n$.
  Therefore $W$ is included in
  $\upc \{\sum_{i=1}^n a'_i \cdot x'_i \mid (a'_1, \cdots, a'_n) \in
  \Delta_n, x'_1 \in C_1, \cdots, x'_n \in C_n\}$, which is equal to
  $\upc \conv {(C_1 \cup \cdots \cup C_n)}$ by
  Lemma~\ref{lemma:conv:union}.  By definition of saturated convex
  hulls (Definition and Proposition~\ref{defprop:upc:conv}), and since
  $C_1 \cup \cdots \cup C_n \subseteq U \subseteq \upc \conv U$,
  $\upc \conv {(C_1 \cup \cdots \cup C_n)}$ is included in
  $\upc \conv U$, hence $W \subseteq \upc \conv U$.  This shows that
  $W$ is an open neighborhood of $x$ included in $\upc \conv U$, and
  since $x$ is arbitrary in $\upc \conv U$, $\upc \conv U$ is open in
  $\B$.

  In order to show that $\upc \conv U$ is the smallest convex open set
  containing $U$, we consider any convex open set $V$ containing
  $U$.  $V$ contains $\conv U$ by definition, and therefore also
  $\upc \conv U$, since $V$ is upwards-closed.  \qed
\end{proof}

\begin{corollary}
  \label{corl:locconv:cons}
  Every consistent, weakly locally convex semitopological barycentric
  algebra is locally convex.
\end{corollary}
\begin{proof}
  Let $\B$ be a consistent, weakly locally convex semitopological
  barycentric algebra.  Let also $x \in \B$, and $U$ be an open
  neighborhood of $x$ in $\B$.  Since $\B$ is weakly locally convex,
  there is a convex subset $C$ of $\B$ such that
  $x \in \interior C \subseteq C \subseteq U$.  By
  Proposition~\ref{prop:int:conv}, item~1, $\interior C$ is convex.  \qed
\end{proof}

Henceforth we will therefore say ``consistent locally convex'' instead
of ``consistent weakly locally convex''.

\subsection{Lower semicontinuous envelopes}
\label{sec:lower-semic-envel}

Finally, we will use the notion of lower semicontinuous envelope of a
monotonic map $f \colon X \to \creal$.  This is the largest lower
semicontinuous map $\check f$ less than or equal to $f$, and is
obtained as the pointwise supremum of all lower semicontinuous maps
less than or equal to $f$.  This section parallels
\cite[Lemma~5.7]{Keimel:topcones2} for barycentric algebras instead of
cones.

\begin{lemma}
  \label{lemma:lsc:env}
  For every semitopological barycentric algebra $\B$, for every map
  $f \colon \B \to \creal$, if $f$ is convex, then so is $\check f$.
\end{lemma}
\begin{proof}
  Let $f \colon \B \to \creal$ be convex.  We fix $a \in {]0, 1[}$ and
  $x \in \B$, and we show that
  $\check f (x +_a y) \leq a f (x) + (1-a) \check f (y)$ for every
  $y \in \B$.  The claim is clear if $f (x) = \infty$, since in that
  case $a f (x) + (1-a) \check f (y) = \infty$.  Hence, let us assume
  that $f (x) < \infty$.  We define $g \colon \B \to \creal$ by
  $g (y) \eqdef \frac 1 {1-a} \max (\check f (x +_a y) - a f (x), 0)$
  for every $y \in \B$.  This defines a lower semicontinuous function
  $g$.  For every $y \in \B$,
  $g (y) \leq \frac 1 {1-a} \max (f (x +_a y) - a f (x), 0) \leq \frac
  1 {1-a} \max ((1-a) f (y), 0) = f (y)$ since $f$ is convex.  Since
  $\check f$ is the largest lower semicontinuous function below $f$,
  $g \leq \check f$.  This entails that
  $\check f (x +_a y) - a f (x) \leq (1-a) \check f (y)$, namely
  $\check f (x +_a y) \leq a f (x) + (1-a) \check f (y)$, for every
  $y \in \B$.

  Knowing that, we fix $a \in {]0, 1[}$ and $y \in \B$, and we show
  that $\check f (x +_a y) \leq a \check f (x) + (1-a) \check f (y)$
  for every $x \in \B$.  The argument is similar.  This is obvious if
  $\check f (y) = \infty$.  Otherwise, we let
  $g (x) \eqdef \frac 1 a \max (\check f (x +_a y) - (1-a) \check f
  (y), 0)$, defining a lower semicontinuous function
  $g \colon \B \to \creal$ such that $g \leq f$ by the result of the
  previous paragraph; hence $g \leq \check f$, so that
  $\check f (x +_a y) \leq a \check f (x) + (1-a) \check f (y)$ for
  every $x \in \B$.  This inequality remains true for $a \eqdef 0$ or
  for $a \eqdef 1$, trivially, so $\check f$ is convex.  \qed
\end{proof}

\begin{lemma}
  \label{lemma:lsc:env:superlinear}
  For every concave monotonic map $f$ from a consistent
  semitopological barycentric algebra $\B$ to $\creal$, $\check f$ is
  concave, too.
\end{lemma}
\begin{proof}
  We deal with the first part first.  Let us assume that $\check f$ is
  not concave.  There are two points $x, y \in \B$ and a real number
  $a \in [0, 1]$ such that
  $\check f (x +_a y) < a \check f (x) + (1-a)\check f (y)$.
  Additionally, $a$ must be different from $0$ and $1$, since
  $\check f (x +_0 y) = \check f (y) = 0 \check f (x) + 1 \check f
  (y)$, and similarly when $a=1$.

  Let $r \in \Rp$ be such that
  $\check f (x +_a y) < r < a \check f (x) + (1-a)\check f (y)$.  We
  recall that: $(*)$ on $\creal$, $s \ll t$ if and only if $s=0$ or
  $s < t$.  In particular, $r \ll a \check f (x) + (1-a)\check f (y)$.
  Since scalar multiplication and addition are Scott-continuous on
  $\creal$, and $a$ is the supremum of the strictly positive numbers
  strictly below $a$, and similarly for $1-a$, there are two numbers
  $\beta, \gamma > 0$ such that $\beta < a$, $\gamma < 1-a$, and
  $r \ll \beta \check f (x) + \gamma \check f (y)$.  Then
  $a \in {]\beta, 1-\gamma[} \subseteq {]0, 1[}$.  By There are
  $b, c \in \Rp$ such that $r \leq b+c$, $b \ll \beta \check f (x)$,
  and $c \ll \gamma \check f (y)$.  Using $(*)$,
  $\frac b \beta \ll \check f (x)$ and
  $\frac c {\gamma} \ll \check f (y)$.  Hence $x$ is in the open set
  $U \eqdef {\check f}^{-1} (\uuarrow \frac b \beta)$ and $y$ is in
  the open set $V \eqdef {\check f}^{-1} (\uuarrow \frac c {\gamma})$.
  Since $\B$ is consistent,
  $W \eqdef \upc (U +_{]\beta, 1-\gamma[} V)$ is open in $\B$.

  Let $g \colon \B \to \creal$ be defined by
  $g (z) \eqdef \max (\check f (z), (b+c) \chi_W (z))$ for every
  $z \in \B$. This is lower semicontinuous, and above $\check f$. We
  claim that $g \leq f$. Since $\check f \leq f$, it suffices to show
  that $(b+c) \chi_W \leq f$, namely that for every $z \in W$,
  $b+c \leq f (z)$.  Given that $z \in W$, we find $x' \in U$,
  $y' \in V$ and $a' \in {]\beta, 1-\gamma[}$ such that
  $x' +_{a'} y' \leq z$. Then
  $\frac b \beta \ll \check f (x') \leq f (x')$ and
  $\frac c {\gamma} \ll \check f (y') \leq f (y')$, in particular
  $b \leq \beta f (x')$ and $c \leq \gamma f (y')$, so
  $b + c \leq \beta f (x') + \gamma f (y') \leq a' f (x') + (1-a') f
  (y')$, where the second inequality comes from
  $a' \in {]\beta, 1-\gamma[}$.  Using the fact that $f$ is concave
  and monotonic,
  $a' f (x') + (1-a') f (y') \leq f (x' +_{a'} y') \leq f (z)$, so
  $b+c \leq f (z)$, as desired.

  Hence $g \leq f$ is lower semicontinuous, and $\check f \leq g$. By
  the maximality of $\check f$, $g = \check f$. Therefore
  $\check f (x +_a y) = g (x +_a y) \geq (b+c) \chi_W (x +_a y) = b+c
  \geq r > \check f (x +_a y)$, which is impossible.  \qed
\end{proof}

\begin{lemma}
  \label{lemma:checkf:cspace}
  For every c-space $X$, for every monotonic map
  $f \colon X \to \creal$, for every $x \in X$, $\check f (x)$ is the
  directed supremum of the values $f (y)$ where $y$ ranges over the
  points of $X$ such that $x \in \interior {\upc y}$.

  Given any subspace $L$ of $\creal^X$ consisting of monotonic maps,
  The map $\mathfrak r \colon f \mapsto \check f$ is a retraction of
  $L$ onto its image $L \cap \Lform X$, with the subspace topology;
  the section is the inclusion map.
\end{lemma}
\begin{proof}
  The first part is sometimes referred to as \emph{Scott's formula}
  when $X$ is a continuous dcpo.  For every $x \in X$, the set of
  points $y$ such that $x \in \interior {\upc y}$ is directed, since
  $X$ is a c-space.  Indeed, given any finite collection of such
  points $y_1$, \ldots, $y_n$, $\bigcap_{i=1}^n \interior {\upc y_i}$
  is an open neighborhood of $x$, which therefore contains a point $y$
  such that $x \in \interior {\upc y}$.  Let $g (x)$ be the directed
  supremum $\sup_{y \in X, x \in \interior {\upc y}} f (y)$, for
  every $x \in X$.

  The function $g$ is lower semicontinuous, since
  $g^{-1} (]t, \infty]) = \bigcup_{y \in X, f (y) > r} \interior {\upc
    y}$.  Since $f$ is monotonic and $x \in \interior {\upc y}$
  implies $y \leq x$, $f \leq g$.  Hence $g \leq \check f$.  Let us
  imagine that the inequality is strict.  For some $x \in X$,
  $g (x) < \check f (x)$, and therefore there is an $r \in \Rp$ such
  that $g (x) < r < \check f (x)$.  Then $x$ is in the open set
  $U \eqdef {(\check f)}^{-1} (]r, \infty])$, and since $X$ is a
  c-space, there is a point $y \in U$ such that
  $x \in \interior {\upc y}$.  In particular, $f (y) > r$, and the
  definition of $g$ entails that $g (x) > r$, which is impossible.
  Therefore $g = \check f$.

  For the second part, we let $\mathfrak s$ be the inclusion map of
  $L \cap \Lform X$ into $L$.  For every $f \in L \cap \Lform X$,
  $\mathfrak r (\mathfrak s (f))$ is the largest lower semicontinuous
  map less than or equal to $f$, and that must be $f$ since $f$ is
  already lower semicontinuous.  The inclusion map $\mathfrak s$ is
  clearly continuous.

  The sets of the form $[x > r]_L \eqdef \{f \in L \mid f (x) > r\}$
  with $x \in X$ and $r \in \Rp \diff \{0\}$ form a subbase of the
  topology on $L$.  In order to show that $\mathfrak r$ is continuous,
  it is enough to show that
  $\mathfrak r^{-1} ([x > r]_L \cap \Lform X)$ is open in $L$.  Its
  elements are the (monotonic) functions $f$ such that
  $\sup_{y \in X, x \in \interior {\upc y}} f (y) > r$, namely such
  that $f \in [y > r]$ for some $y \in X$ such that
  $x \in \interior {\upc y}$.  In other words,
  $\mathfrak r^{-1} ([x > r]_L \cap \Lform X) = \bigcup_{y \in X, x
    \in \interior {\upc y}} [y > r]_L$, so $\mathfrak r$ is
  continuous.  \qed
\end{proof}

\subsection{The Plotkin-Banach-Alaoglu theorem}
\label{sec:plotk-banach-alaoglu}

The Banach-Alaoglu theorem states that the closed unit ball in a
topological dual of a topological vector space is compact.  Plotkin
gives a variant of this result on continuous d-cones whose addition
preserves the way-below relation \cite[Theorem~2]{Plotkin:alaoglu}; we
will see how this related to consistency in
Section~\ref{sec:consistency-c-spaces}.  We show a similar
generalization to duals of consistent semitopological barycentric
algebras.  The proof is basically the same.

In a non-Hausdorff setting, compactness has to be replaced by stable
compactness (see \cite[Chapter~9]{JGL-topology}).  A \emph{stably
  compact space} is a sober, locally compact, coherent, and compact
topological space.  We explain each adjective.
\begin{itemize}
\item A space is \emph{locally compact} if and only if every point has
  a base of compact neighborhoods, or equivalently of compact
  saturated neighborhoods.  One should beware that, in a non-Hausdorff
  setting, this property is not equivalent with the more usual
  definition (every point has a compact neighborhood), which is
  strictly weaker.  Every locally compact space is core-compact
  \cite[Theorem 5.2.9]{JGL-topology}, and the two notions are
  equivalent for a sober space \cite[Theorem 8.3.10]{JGL-topology}.
\item A space is \emph{coherent} if and only if the intersection of
  any two compact saturated subsets is compact (and saturated).
\end{itemize}
$\creal$ and $[0, 1]_\sigma$ are stably compact spaces.  In general,
every \emph{compact pospace} $(X, \preceq)$, namely every compact
space $X$ with a partial ordering $\preceq$ whose graph is closed in
$X \times X$, gives rise to a stably compact space $X^\uparrow$, whose
elements are those of $X$, and whose open subsets are the open subsets
of $X$ that are upwards-closed with respect to $\preceq$.  The
specialization ordering of $X^\uparrow$ is $\preceq$, and this
construction exhausts all possible stably compact spaces: given any
stably compact space $X$, $(X^\patch, \leq)$ is a compact pospace and
$(X^\patch)^\uparrow = X$, where $\leq$ is the specialization ordering
of $X$ and $X^\patch$ is $X$ with the \emph{patch topology}, the
coarser topology that contains the open subsets of $X$ and the
complements of compact saturated subsets of $X$.  The patch topology
makes sense for any space, not just stably compact spaces, but the
properties just mentioned only hold for stably compact spaces.

Every product of stably compact spaces is stably compact, and the
product of the patch topologies is the patch topology of the product;
and retracts preserve sobriety, local compactness, coherence, and
compactness, hence also stable compactness.  A subset that is closed
in $X^\patch$ is called \emph{patch-closed} (in $X$), and similarly
with patch-compact.  A \emph{patch-continuous} map $f \colon X \to Y$
between stably compact spaces is a continuous map from $X^\patch$ to
$Y^\patch$.  Every compact pospace is compact Hausdorff.

We will use the following gimmick \cite[Definition~5.4]{GL:duality}.
Let $T$ be a set, and $A$ be a topological space.  A
\emph{patch-continuous inequality} on $T, A$ is any formula $E$ of the
form:
\[
  f (\_ (t_1), \ldots, \_ (t_m)) \dotleq g (\_ (t'_1), \ldots, \_ (t'_n))
\]
where $f$ and $g$ are patch-continuous maps from $A^n$ to $A$, and
$t_1$, \ldots, $t_m$, $t'_1$, \ldots, $t'_n$ are $m+n$ fixed elements
of $T$.  $E$ \emph{holds} at $\alpha \colon T \to A$ iff
$f (\alpha (t_1), \ldots, \alpha (t_m)) \leq g (\alpha (t'_1), \ldots,
\alpha (t'_n))$, where $\leq$ is the specialization quasi-ordering of
$A$.  A \emph{patch-continuous system} $\Sigma$ on $T, A$ is a
(possibly infinite) set $\Sigma$ of patch-continuous inequalities on
$T,A$.  $\Sigma$ \emph{holds} at $\alpha \colon T \to A$ iff every
element of $\Sigma$ holds at $\alpha$.  We also allow equations, since
$a \dot= b$ is equivalent to $a \dotleq b$ and $b \dotleq a$.  In this
setting, the subset $[\Sigma]$ of $A^T$ of all maps
$\alpha \colon T \to A$ such that $\Sigma$ holds at $\alpha$ is
patch-closed in $A^T$ \cite[Proposition~5.5]{GL:duality}.

Here is the variant of the Plotkin-Banach-Alaoglu theorem on
barycentric algebras.
\begin{theorem} 
  \label{thm:alaoglu}
  For every consistent semitopological barycentric algebra $\B$ whose
  underlying topological space is a c-space, $\B^\astar$ is stably
  compact.  If $\B$ is also pointed (for example, a consistent
  semitopological cone), then $\B^*$ is stably compact.
\end{theorem}
\begin{proof}
  We consider the space $\creal^\B$ with the product topology.  Since
  $\creal$ is stably compact, $\creal^\B$ is stably compact, too.  The
  subset $L$ of all affine (resp.\ linear) monotonic maps from $\B$ to
  $\creal$ is the collection defined by the following inequalities:
  $\_ (x) \leq \_ (y)$ for all pairs $x \leq y$ in $\B$,
  $\_ (x +_a y) = a \, \_ (x) + (1-a) \, \_ (y)$ for all $x, y \in \B$
  and $a \in [0, 1]$ (resp.\ and $\_ (\bot) = 0$, since a map is
  linear if and only if it is strict and affine, by
  Proposition~\ref{prop:strict:aff}).  We note that $+$ is
  patch-continuous on $\creal$, equivalently it is continuous with
  respect to the usual Hausdorff topology on $\creal$, and
  multiplication by $a \in \Rp$ is, too.  This is therefore a
  patch-continuous system, from which it follows that $L$ is
  patch-closed, hence stably compact when given the subspace topology
  from $\creal^\B$.

  For every $f \in L$, $\check f$ is convex by
  Lemma~\ref{lemma:lsc:env}, and concave by
  Lemma~\ref{lemma:lsc:env:superlinear}, which applies since $\B$ is
  consistent.  Therefore $\check f$ is in $\B^\astar$ (resp.\ $\B^*$).
  Since $\B$ is a c-space, the map
  $\mathfrak r \colon f \mapsto \check f$ of
  Lemma~\ref{lemma:checkf:cspace} is a retraction of $\creal^\B$ onto
  $L \cap \Lform_\pw \B$, namely onto $\B^\astar$ (resp.\ $\B^*$).  As
  a retract of a stably compact space, $\B^\astar$ (resp.\ $\B^*$) is
  then stably compact.  \qed
\end{proof}

Given any topological space $X$, the complements of compact saturated
sets form subbasic open sets for a topology, the \emph{cocompact}
topology.  If $X$ is stably compact, the cocompact topology consists
exactly of the complements of compact saturated sets.  The set of
points of $X$ with the cocompact topology is the \emph{de Groot dual}
$X^\dG$ of $X$, and $X^{\dG\dG} = X$ when $X$ is stably compact.  The
compact saturated subsets of $X$ are the closed subsets of $X^\dG$,
and also the upwards-closed patch-closed subsets of $X$.  (Once again,
see \cite[Chapter~9]{JGL-topology}.)  We characterize the cocompact
topology on $\B^\astar$ and on $\B^*$.  Once again, the proof strategy
is due to Plotkin \cite{Plotkin:alaoglu}.
\begin{definition}[Lower-open topology]
  \label{defn:U>=r}
  Given a (resp.\ pointed) semitopological barycentric algebra $\B$,
  an open subset $U$ of $\B$, and $r \in \Rp$, the set $[U \geq r]$ is
  defined as the collection of elements $\Lambda$ of $\B^\astar$
  (resp.\ $\B^*$) such that $U \subseteq \Lambda^{-1} ([r, \infty])$.

  The \emph{lower-open}{topology} on $\B^\astar$ (resp.\ $\B^*$)
  is generated by the complements $[U < r]$ of the sets $[U \geq r]$
  where $U \in \Open \B$ and $r \in \Rp \diff \{0\}$.
\end{definition}
Alternatively, $[U \geq r]$ is the set of all $\Lambda \in \B^\astar$
(resp.\ $\B^*$) such that $\Lambda (x) \geq r$ for every $x \in U$.

Given a topological space $X$ with specialization preordering $\leq$,
a \emph{dual topology} $\Open'$ on $X$ is a topology whose
specialization preordering is the opposite preordering $\geq$.  A
\emph{separating dual topology} is a dual topology $\Open'$ on
$X$ such that, for all points $x, y \in X$ such that $x \not\leq y$,
there is an open neighborhood $U \in \Open X$ of $x$ and an element
$U' \in \Open'$ containing $y$ such that $U \cap U' = \emptyset$.

Given a topological space $X$, the cocompact topology on $X$ is always
dual, because saturated subsets are upwards-closed, hence their
complements are downwards-closed.  If $X$ is locally compact, then the
cocompact topology is also separating.  Indeed, if $x \not\leq y$,
then $x$ is in the open set $X \diff \dc y$.  Since $X$ is locally
compact, there is a compact saturated subset $Q$ of $X$ and an open
subset $U$ of $Q$ such that
$x \in U \subseteq Q \subseteq X \diff \dc y$.  Hence $x$ is in $U$,
$y$ is in $X \diff Q$, and $U \cap (X \diff Q) = \emptyset$.  The
cocompact topology is also coarser than the patch topology, by
definition.

Conversely, if $X$ is stably compact, then the only separating dual
topology coarser than the patch topology on $X$ is the cocompact
topology \cite[Lemma~1]{Plotkin:alaoglu}.

\begin{proposition}
  \label{prop:C*:cocompact=loweropen}
  Let $\B$ be a consistent (resp.\ pointed) semitopological
  bary\-centric algebra whose underlying space is a c-space. For every
  $U \in \Open \B$ and every $r \in \Rp$, $[U \geq r]$ is compact
  saturated in $\B^\astar$ (resp.\ $\B^*$).

  The cocompact topology coincides with the lower-open topology on
  $\B^\astar$ (resp.\ $\B^*$). A subbase is given by the sets
  $[\interior {\upc x} < r]$ with $x \in \B$ and
  $r \in \Rp \diff \{0\}$.
\end{proposition}
\begin{proof}
  Let $U$ be open in $\B$ and $r \in \Rp$.  As in the proof of
  Theorem~\ref{thm:alaoglu}, the subset $L$ of all affine (resp.\
  linear) monotonic maps from $\B$ to $\creal$ is patch-compact, and
  $\mathfrak r \colon \Lambda \mapsto \check \Lambda$ defines a
  retraction of $L$ onto $\B^\astar$ (resp.\ $\B^*$). The set $A$ of
  all maps $\Lambda \in L$ such that $\Lambda (x) \geq r$ for every
  $x \in U$ is defined by patch-continuous inequalities, so $A$ is
  patch-closed in $L$. Since $A$ is upwards-closed in $L$, it is
  compact saturated. It follows that its image $B$ under
  $\Lambda \mapsto \check \Lambda$ is compact in $\B^*$.

  We claim that $B = [U \geq r]$. For every $\Lambda \in A$, for every
  $x \in U$,
  $\check \Lambda (x) = \sup_{y \in \B, x \in \interior {\upc y}}
  \Lambda (y)$ by Lemma~\ref{lemma:checkf:cspace}.  Since $U$ is
  Scott-open, some $y \in \B$ such that $x \in \interior {\upc x}$ is
  in $U$, and then $\Lambda (y) \geq r$, whence
  $\check \Lambda (x) \geq r$; so $\check \Lambda$ is in
  $B$. Conversely, every map $\Lambda$ in $[U \geq r]$ is in $A$, and
  equal to $\check \Lambda$, hence in $B$.

  Since $B$ is compact and $B = [U \geq r]$, $[U \geq r]$ is compact.
  It is clear that $[U \geq r]$ is upwards-closed in $\B^\astar$
  (resp.\ $\B^*$), hence it is compact saturated.

  Let $\Open'$ be the topology generated by subbasic open sets
  $[\interior {\upc x} < r]$ with $x \in \B$ and
  $r \in \Rp \diff \{0\}$.

  We claim that $\Open'$ is a dual topology on $\B^\astar$ (resp.\
  $\B^*$). It suffices to show that, for all
  $\Lambda, \Lambda' \in \B^\astar$ (resp.\ $\B^*$),
  $\Lambda \geq \Lambda'$ if and only if for every $x \in \B$ and
  every $r \in \Rp \diff \{0\}$,
  $\Lambda \in [\interior {\upc x} < r]$ implies
  $\Lambda' \in [\interior {\upc x} < r]$.  If
  $\Lambda \in [\interior {\upc x} < r]$ and $\Lambda \geq \Lambda'$,
  clearly $\Lambda'$ is in $[\interior {\upc x} < r]$. Conversely, if
  for every $x \in \B$ and every $r \in \Rp \diff \{0\}$,
  $\Lambda \in [\interior {\upc x} < r]$ implies
  $\Lambda' \in [\interior {\upc x} < r]$, then for every $y \in \B$,
  we claim that $\Lambda (y) \leq \Lambda' (y)$.  Otherwise, there
  would be an $r \in \Rp \diff \{0\}$ such that
  $\Lambda (y) > r > \Lambda' (y)$.  Since
  $y \in \Lambda^{-1} (]r, \infty])$ and $\B$ is a c-space, there is
  an $x \in \Lambda^{-1} (]r, \infty])$ such that
  $y \in \interior {\upc x}$.  In particular, $x \leq y$, so
  $\Lambda (y) \geq \Lambda (x) > r$, so
  $\Lambda \in [\interior {\upc x} < r]$, and therefore
  $\Lambda' \in [\interior {\upc x} < r]$.  But the latter implies
  $\Lambda' (y) < r$, which is impossible.  Hence
  $\Lambda \geq \Lambda'$.

  We claim that $\Open'$ is separating. Let
  $\Lambda, \Lambda' \in \B^\astar$ (resp.\ $\B^*$) be such that
  $\Lambda \not\leq \Lambda'$. There is a point $x \in \B$ such that
  $\Lambda (x) > \Lambda' (x)$. Let $r \in \Rp \diff \{0\}$ be such
  that $\Lambda (x) > r > \Lambda' (x)$, and $y \in B$ be such that
  $x \in \interior {\upc y}$ and $\Lambda (y) > r$. Then $\Lambda$ is
  in $[y > r] \in \Open \B$; $\Lambda'$ is in
  $[\interior {\upc y} < r] \in \Open'$, otherwise $\Lambda'$ would be
  in $[\interior {\upc y} \geq r]$, so $\interior {\upc y}$ would be
  included in ${\Lambda'}^{-1} ([r, \infty])$, which is impossible
  since $x$ belongs to the former and not to the latter.  Finally,
  $[y > r]$ and $ [\interior {\upc y} < r]$ are disjoint.  Indeed,
  otherwise there would be an affine (resp.\ linear) continuous map
  $h$ such that $h (y) > r$ and $h (z) < r$ for every
  $z \in \interior {\upc y}$, in particular $h (x) < r$. Since
  continuous maps are monotonic and $y \leq x$, this is impossible.

  Since $\Open'$ is dual and separating, $\Open'$ coincides with the
  cocompact topology.  It is coarser than the lower-open topology by
  definition, and since every set $[U \geq r]$ with $U \in \Open \B$
  and $r \in \Rp \diff \{0\}$ is compact saturated, the lower-open
  topology is coarser than the cocompact topology.  Hence all these
  topologies coincide.  \qed
\end{proof}

\subsection{Consistency and c-spaces}
\label{sec:consistency-c-spaces}

Extending the notion described earlier, a map $h \colon X \to Y$ is
\emph{almost open} if and only if $\upc h [U]$ is open for every open
subset $U$ of $X$.  A monotonic map $h \colon X \to Y$ \emph{preserves
  the way-below relation} if and only if for all $x, x' \in X$ such
that $x \ll x'$, $h (x) \ll h (x')$.
\begin{lemma}
  \label{lemma:quasi-open:poset}
    For every c-space $X$, for every topological space $Y$, a
    monotonic map $h \colon X \to Y$ is almost open if and only if for
    all $x, x' \in X$ such that $x' \in \interior {\upc x}$,
    $h (x') \in \interior {\upc h (x)}$.
    For every continuous poset $X$ and every poset $Y$, considered
    with their Scott topologies, the monotonic map $h$ is quasi-open
    if and only if it preserves the way-below relation.
\end{lemma}
\begin{proof}
  Let us assume that $h$ is quasi-open.  If
  $x' \in U \eqdef \interior {\upc x}$, then $h (x')$ is in the
  (Scott-)open set $\upc h [U]$.  In turn, $\upc h [U]$ is included in
  $\upc h (x)$, hence in its interior.

  Conversely, let us assume that $x' \in \interior {\upc x}$ implies
  $h (x') \in \interior {\upc h (x)}$.  For every open subset $U$ of
  $X$, let $y$ be any point of $\upc h [U]$.  There is a point
  $x' \in U$ such that $h (x') \leq y$.  Since $X$ is a c-space, there
  is a point $x \in U$ such that $x' \in \interior {\upc x}$.  By
  assumption, $h (x') \in \interior {\upc h (x)}$, so $y$ is in the
  open set $\interior {\upc h (x)}$, which is included in
  $\upc h [U]$.  This shows that $\upc h [U]$ is an open neighborhood
  of each of its points, hence is itself open.

  The second claim is a direct consequence of the first one, and the
  fact that in a continuous poset, $y \in \interior {\upc x}$ if and
  only if $x \ll y$.  \qed
%
%
\end{proof}

\begin{lemma}
  \label{lemma:consistent:ccone}
  Let $\B$ be a semitopological barycentric algebra and a c-space.
  $\B$ is strongly consistent if and only if for all $x, x', y', y'
  \in \B$ such that $x' \in \interior {\upc x}$ and $y' \in \interior
  {\upc y}$, for every $a \in {]0, 1[}$, $x'+y'$ is in $\interior
  {\upc (x +_a y)}$.

  Let $\C$ be a c-cone. $\C$ is consistent if and only if for all
  $x, x', y, y' \in \C$, if $x' \in \interior {\upc x}$ and
  $y' \in \interior {\upc y}$, then
  $x'+y' \in \interior {\upc (x+y)}$.
\end{lemma}
\begin{proof}
  This is Lemma~\ref{lemma:quasi-open:poset}, 
  applied to the operator $+_a$ in the first case, to $+$ in the
  second case, and realizing that in a product of two c-spaces,
  $\interior {\upc (x, y)} = \interior {\upc x} \times \interior {\upc
    y}$.  \qed
\end{proof}

The following is due to \cite[Proposition~2.2]{TKP:nondet:prob} in the
case of continuous d-cones, namely continuous s-cones (see
Example~\ref{exa:cont:s-baryalg}) that are directed-complete.
\begin{lemma}
  \label{lemma:consistent:scone}
  An continuous s-barycentric algebra $\B$ is strongly consistent in
  its Scott topology if and only if $+_a$ preserves the way-below
  relation on $\B$ for every $a \in {]0, 1[}$, namely: for all
  $x, x', y, y' \in \B$, if $x \ll x'$ and $y \ll y'$ then
  $x +_a y \ll x' +_a y'$.
  
  A continuous s-cone $\C$ is consistent in its Scott topology if and
  only if addition preserves the way-below relation on $\C$, namely:
  for all $x, x', y, y' \in \C$,$x \ll x'$ and $y \ll y'$ imply
  $x+y \ll x'+y'$.
\end{lemma}
\begin{proof}
  Since every continuous poset is a c-space in its Scott topology, a
  continuous s-cone is a c-cone, and a continuous s-barycentric
  algebra is a c-space in its Scott topology.  We can therefore apply
  Lemma~\ref{lemma:consistent:ccone}, and we conclude since
  $\uuarrow z = \interior {\upc z}$ for every point $z$.  \qed
\end{proof}

\begin{remark}
  \label{lemma:smult:<<}
  The situation is simpler with scalar multiplication: on any s-cone
  $\C$, for every $a \in \Rp \diff \{0\}$, for all $x, y \in \C$, $x
  \ll y$ if and only if $a \cdot x \ll a \cdot y$.  This is an easy
  exercise, since the map $x \mapsto a \cdot x$ is an order isomorphism.
\end{remark}

\begin{remark}
  \label{rem:+:reflect:<<}
  It is sometimes desirable to consider a dual property: addition
  \emph{reflects} $\ll$ (instead of preserving it) if and only if for
  all $z \ll x+y$, there are points $x' \ll x$ and $y' \ll y$ such
  that $z \ll x'+y'$.  That property holds on every continuous s-cone,
  consistent or not.  Indeed, since $+$ is jointly continuous,
  $+^{-1} (\uuarrow z)$ is an open neighborhood of $(x, y)$, hence
  contains a product of basic open sets
  $\interior {\upc x}' \times \uuarrow y'$ such that $x \in \interior {\upc x}'$ and
  $y \in \uuarrow y'$.  Similarly, $+_a$ reflects $\ll$ for every
  $a \in {]0, 1[}$ on every ordered barycentric algebra that is also a
  continuous poset.
\end{remark}


\subsection{Strongly consistent spaces of valuations}
\label{sec:cons-spac-valu}

We will see below that all our familiar barycentric algebras of
continuous valuations are strongly consistent.  We need the following
second decomposition lemma.  Its proof is somewhat similar to the
first one (Proposition~\ref{prop:schsimp:decomp}).
\begin{proposition}
  \label{prop:decomp:2}
  Let $X$ be a set with a lattice of subsets $\Latt$, and $\mu$,
  $\nu$, $\varpi$ be three bounded valuations on $(X, \Latt)$.  Let
  $L \eqdef \{U_1, \cdots, U_n\}$ be a finite lattice of subsets of
  $X$ included in $\Latt$, and let us assume that
  $\mu (U_i) + \nu (U_i) \leq \varpi (U_i)$ for every $i$,
  $1\leq i\leq n$.  Then there are two bounded valuations $\mu'$ and
  $\nu'$ on $(X, \Latt)$ such that:
  \begin{enumerate}
  \item $\mu' + \nu' \leq \varpi$;
  \item $\mu (U_i) \leq \mu' (U_i)$ and $\nu (U_i) \leq \nu' (U_i)$
    for every $i$, $1\leq i\leq n$.
  \end{enumerate}
  Additionally, $\mu'$ and $\nu'$ are finite linear combinations of
  constrictions $\varpi_{|C}$ of $\varpi$ to pairwise disjoint
  crescents $C$.
  In particular, if $\Latt$ is a topology and $\varpi$ is continuous,
  then $\mu'$ and $\nu'$ are continuous.
\end{proposition}
\begin{proof}
  We extend $\mu$, $\nu$, and $\varpi$ uniquely to bounded valuations
  on $\mathcal A (\Latt)$, using the Smiley-Horn-Tarski theorem.  As
  in the proof of Proposition~\ref{prop:schsimp:decomp}, let
  $M \eqdef \pow (\{1, \cdots, n\})$, with the Alexandroff topology of
  inclusion, let
  $C_I \eqdef \bigcap_{i \in I} U_i \diff \bigcup_{i \not\in I} U_i$,
  and $f \colon X \to M$ be defined by
  $f (x) \eqdef \{i \in \{1, \cdots, n\} \mid x \in U_i\}$.  The
  crescents $C_I$ are pairwise disjoint, and, as in the proof of
  Proposition~\ref{prop:schsimp:decomp}, $f^{-1} (\upc I)$ is in $L$
  for every $I \in M$.

  We claim that $f [\mu] + f [\nu] \leq f [\varpi]$.  For every open
  subset $V$ of $M$, $V$ is a finite union of sets $\upc I$,
  $I \in M$, and $f^{-1} (\upc I)$ is in $L$; so $f^{-1} (V)$, which
  is a finite union of elements of $L$, is also in $L$, hence equal to
  $U_i$ for some $i$.  Then
  $f [\mu] (V) + f [\nu] (V) = \mu (f^{-1} (V)) + \nu (f^{-1} (V)) =
  \mu (U_i) + \nu (U_i) \leq \varpi (U_i) = \varpi (f^{-1} (V)) = f
  [\varpi] (V)$.

  Since $M$ is finite, $f [\mu]$, $f [\nu]$, and $f [\varpi]$ are
  simple valuations.  Explicitly,
  $f [\mu] = \sum_{I \in M} a_I \delta_I$ where
  $a_I \eqdef \mu (C_I)$, $f [\nu] = \sum_{I \in M} b_I \delta_I$
  where $b_I \eqdef \nu (C_I)$, and
  $f [\varpi] = \sum_{I \in M} a_I \delta_I$ where
  $c_I \eqdef \varpi (C_I)$,


  By Jones' splitting lemma (see \cite[Theorem 4.10]{Jones:proba} or
  also \cite[Proposition IV-9.18]{GHKLMS:contlatt}), there is a
  transport matrix ${(t_{IJ})}_{I \in M, J \in M}$ of non-negative
  real numbers such that:
  \begin{enumerate*}[series=decomp:2,label=(\roman*)]
  \item\label{q:decomp:2:a} for all $I, J \in M$, if $t_{IJ} \neq 0$
    then $I \subseteq J$;
  \item\label{q:decomp:2:b} $\sum_{J \in M} t_{IJ} = a_I + b_I$ for
    every $I \in M$, and
  \item\label{q:decomp:2:c} $\sum_{I \in M} t_{IJ} \leq c_J$ for every
    $J \in M$.
  \end{enumerate*}
  
  We define two new transport matrices ${(u_{IJ})}_{I \in M, J \in M}$
  and ${(v_{IJ})}_{I \in M, J \in M}$ by:
  $u_{IJ} \eqdef \frac {a_I} {a_I+b_I} t_{IJ}$ and
  $v_{IJ} \eqdef \frac {b_I} {a_I+b_I} t_{IJ}$ if $a_I+b_I \neq 0$,
  otherwise $u_{IJ} \eqdef v_{IJ} \eqdef 0$.  We note that:
  \begin{enumerate*}[resume*=decomp:2]
  \item\label{q:decomp2:tuv} if $t_{IJ}=0$ then $u_{IJ}=0$ and $v_{IJ}=0$;
  \item\label{q:decomp2:uIJ} if $u_{IJ} \neq 0$ then $I \subseteq J$
    (if $I \not\subseteq J$, then $t_{IJ}=0$ by~\ref{q:decomp:2:a}, so
    $u_{IJ}=0$ by~\ref{q:decomp2:tuv});
  \item\label{q:decomp:2:vIJ} if $v_{IJ} \neq 0$ then $I \subseteq J$
    (similar);
  \item\label{q:decomp:2:f} $u_{IJ} + v_{IJ} = t_{IJ}$ (even when
    $a_I+b_I = 0$, in which case, by~\ref{q:decomp:2:b}, $t_{IJ}=0$
    for every $J \in M$);
  \item\label{q:decomp:2:sumJu} for every $I \in M$,
    $\sum_{J \in M} u_{IJ} \geq a_I$ (if $a_I + b_I \neq 0$, then
    $\sum_{J \in M} u_{IJ} = \frac {a_I} {a_I+b_I} (a_I + b_I)$
    by~\ref{q:decomp:2:b}, which is equal to $a_I$; otherwise,
    $a_I=0$);
  \item\label{q:decomp:2:sumJv} for every $I \in M$,
    $\sum_{J \in M} v_{IJ} \geq b_I$ (similar).
  \end{enumerate*}

  For every $J \in M$, we let
  $a'_J \eqdef \frac {\sum_{I \in M} u_{IJ}} {c_J}$ if $c_J \neq 0$,
  $a'_J \eqdef 0$ otherwise; similarly,
  $b'_J \eqdef \frac {\sum_{I \in M} v_{IJ}} {c_J}$ if $c_J \neq 0$,
  $b'_J \eqdef 0$ otherwise.  We note that:
  \begin{enumerate*}[resume*=decomp:2]
  \item\label{q:decomp:2:g} $a'_J + b'_J \leq 1$ (since
    $a'_J + b'_J = \frac {\sum_{I \in M} t_{IJ}} {c_J} \leq 1$
    by~\ref{q:decomp:2:f} and~\ref{q:decomp:2:c}, when $c_J \neq 0$);
  \item\label{q:decomp:2:h} $a'_J c_J = \sum_{I \in M} u_{IJ}$ (even
    if $c_J=0$, since in that case $t_{IJ}=0$ by~\ref{q:decomp:2:c},
    so $u_{IJ}=0$ by definition); and
  \item\label{q:decomp:2:i} $b'_J c_J = \sum_{I \in M} v_{IJ}$
    (similar).
  \end{enumerate*}

  Finally, we let $\mu' \eqdef \sum_{J \in M} a'_J \varpi_{|C_J}$ and
  $\nu' \eqdef \sum_{J \in M} b'_J \varpi_{|C_J}$.  If $\Latt$ is a
  topology and $\varphi$ is continuous, we note that $\mu'$ and $\nu'$
  are continuous.

  Let us verify that $\mu' + \nu' \leq \varpi$.  We have
  $\mu' + \nu' = \sum_{J \in M} (a'_J + b'_J) \varpi_{|C_J} \leq
  \sum_{J \in M} \varpi_{|C_J} = \varpi$, using~\ref{q:decomp:2:g}.

  We now check that $\mu (U_i) \leq \mu' (U_i)$ for every $i \in I$.
  We note that $\varpi_{|C_J} (U_i) = \varpi (C_J \cap U_i)$, which is
  equal to $\varpi (C_J) = c_J$ if $i \in J$, and to $0$ otherwise.
  Therefore $\mu' (U_i) = \sum_{J \in M, i \in J} a'_J c_J$.  This is
  equal to $\sum_{I, J \in M, i \in J} u_{IJ}$ by~\ref{q:decomp:2:h},
  hence to the sum of the terms $u_{IJ}$ over the pairs $I, J \in M$
  such that $i \in J$ and $I \subseteq J$, since those such that
  $I \not\subseteq J$ are equal to $0$, by~\ref{q:decomp2:uIJ}.  Hence
  that sum is larger than or equal to the sum of the terms $u_{IJ}$
  over the pairs $I, J \in M$ such that $i$ is in $I$ (instead of $J$)
  and $I \subseteq J$; indeed, $i \in I \subseteq J$ implies
  $i \in J$.  Using~\ref{q:decomp2:uIJ} again, that sum is equal to
  $\sum_{I, J \in M, i \in I} u_{IJ} = \sum_{I \in M, i \in I}
  (\sum_{J \in M} u_{IJ})$, which is larger than or equal to
  $\sum_{I \in M, i \in I} a_I$ by~\ref{q:decomp:2:sumJu}, and the
  latter is equal to
  $\sum_{I \in M, i \in I} \mu (C_I) = \mu (U_i)$.

  The verification that $\nu (U_i) \leq \nu' (U_i)$ for every
  $i \in I$ is similar.  \qed
\end{proof}

\begin{theorem}
  \label{thm:VX:consistent:strong}
  Let $X$ be a topological space, and $\B$ be a convex subspace of
  $\Val X$, seen as a topological barycentric algebra.  Let $\B_b$ be
  the collection of bounded continuous valuations in $\B$.  Under the
  following assumptions:
  \begin{enumerate}[label=(\roman*),leftmargin=*]
  \item every element of $\B$ is the supremum in $\Val X$ of a
    directed family of elements of $\B_b$,
  \item for every $\mu \in \B_b$, for every $c \in {]0, 1[}$,
    $c \cdot \mu \in \B$,
  \item for all $\mu, \varpi \in \B_b$, for every $r > 1$, every
    finite linear combination $\mu'$ of constrictions of
    $r \cdot \varpi$ to pairwise disjoint crescents such that
    $\mu (X) = \mu' (X)$ is in $\B$,
  \end{enumerate}
  $\B$ is a strongly consistent topological barycentric algebra.

  In particular,
  \begin{enumerate}
  \item $\Val X$, $\Val_b X$, $\Val_\pw X$ and $\Val_\fin X$ are
    consistent topological cones;
  \item $\Val_{\leq 1} X$, $\Val_{\leq 1, \pw} X$ and
    $\Val_{\leq 1, \fin} X$ are strongly consistent topological
    barycentric algebras;
  \item if $X$ is pointed, then $\Val_1 X$, $\Val_{1, \pw} X$ and
    $\Val_{1, \fin} X$ are strongly consistent topological
    barycentric algebras.
  \end{enumerate}
\end{theorem}
\begin{proof}
  We first deal with the second part of the proposition, assuming the
  first part.  By Example~\ref{exa:V1X:top:bary:alg}, $\Val X$ is a
  topological, not just semitopological cone, and therefore it is a
  topological barycentric algebras by Lemma~\ref{lemma:bary:in:cone}.
  Assumption~(ii) is clear for all ten spaces considered in items~1, 2
  and~3.

  1.  When $\B = \Val X$, assumption~(i) is one half of Heckmann's Theorem~4.2 in
  \cite{heckmann96}.  As for assumption~(iii), constrictions of
  bounded continuous valuations are continuous valuations, and
  $\Val X$ is closed under finite linear combinations.

  When $\B = \Val_\pw X$, assumption~(i) is the other half of
  Heckmann's Theorem~4.2 in \cite{heckmann96}.  Constrictions of
  bounded point-continuous valuations are point-continuous by
  \cite[Section~3.3]{heckmann96}, and bounded point-continuous
  valuations are closed under finite linear combinations
  \cite[Section~3.2]{heckmann96}, proving assumption~(iii).

  When $\C = \Val_\fin X$, assumption~(i) is vacuously true since
  every simple valuation is bounded.  Given any simple valuation
  $\nu \eqdef \sum_{i=1}^n a_i \delta_{x_i}$, it is easy to see that
  its constriction to a crescent $C$ is equal to
  $\sum_{\substack{i=1\\x_i \in C}}^n a_i \delta_{x_i}$, and is
  therefore simple.  Any finite linear combination of simple
  valuations is in $\Val_\fin X$, and this establishes
  assumption~(iii).

  When $\C = \Val_b X$, assumption~(i) is vacuously true, and the
  other two assumptions are trivial.

  2.  Here assumption~(i) is vacuous, since the three given spaces
  consist of bounded valuations already.  We deal with
  assumption~(iii).  Whether $\B$ is $\Val_{\leq 1} X$,
  $\Val_{\leq 1, \pw} X$ or $\Val_{\leq 1, \fin}$, for all
  $\mu, \varpi \in \B_b$, for every $r > 1$, let $\mu'$ be a finite
  linear combination of constrictions of $r \cdot \varpi$ to pairwise
  disjoint crescents such that $\mu (X)=\mu' (X)$.  By item~1, $\mu'$
  is in $\Val X$, or $\Val_\pw X$, or in $\Val_\fin X$, depending
  on $\B$.  Since $\mu \in \B_b$, $\mu$ is subnormalized, hence so is
  $\mu'$, and therefore $\mu' \in \B$.

  3.  This follows from item~2 and Edalat's trick
  (Lemma~\ref{lemma:VwX:pointed}).
  
  Let us prove the first part of the proposition.
  
  We consider any two open subsets of $\B$.  By definition, they are
  of the form $\mathcal U \cap \B$ and $\mathcal V \cap \B$, where
  $\mathcal U$ and $\mathcal V$ are two open subsets of $\Val X$.
  Let also $a \in {]0, 1[}$: we will show that the upward closure of
  $(\mathcal U \cap \B) +_a (\mathcal V \cap \B)$ in $\B$ is open, and
  this will suffice to show that $\B$ is strongly consistent, thanks
  to Lemma~\ref{fact:strong:consistent}.  Let $\xi$ be any element of
  the upward closure of
  $(\mathcal U \cap \B) +_a (\mathcal V \cap \B)$ in $\B$; that is,
  there are continuous valuations $\mu \in \mathcal U \cap \B$ and
  $\nu \in \mathcal V \cap \B$ such that $\mu +_a \nu \leq \xi$.  Our
  task is to find an open set $\mathcal W \cap \B$, with $\mathcal W$
  open in $\Val X$, which contains $\xi$ and is included in the
  upward closure of $(\mathcal U \cap \B) +_a (\mathcal V \cap \B)$ in
  $\B$.

  Every open subset of $\Val X$ is Scott-open (with respect to the
  stochastic ordering).  This follows from the fact that every
  subbasic open set $[U > r]$ is Scott-open.  Hence $\mathcal U$ is
  Scott-open, and since $\mu$ is a supremum of a directed family of
  elements of $\B_b$ by assumption~$(i)$, one of them is in
  $\mathcal U$, hence in $\mathcal U \cap \B_b$.  We reason similarly
  with $\nu$ and $\mathcal V$, to the effect that we can assume
  without loss of generality that $\mu$ and $\nu$ are bounded
  valuations.

  By definition of the weak topology, $\mu$ is in some finite
  intersection of subbasic open sets $[r_i \ll U_i]$, $1\leq i\leq n$,
  with $r_i \in \Rp$ and $U_i \in \Open X$, and that intersection is
  included in $\mathcal U$; we obtain a similar intersection of
  subbasic open sets $[s_j \ll V_j]$ for $\nu$.  By adding subbasic
  open sets of the form $[0 \ll V_j]$ or $[0 \ll U_i]$ to each
  intersection if necessary, we may assume that the family of open
  sets $U_i$ and the family of open sets $V_j$ are the same.  Hence we
  will assume that $\nu$ is in a finite intersection
  $\bigcap_{i=1}^n [s_i \ll U_i]$ included in $\mathcal V$.

  We may also assume that the family ${(U_i)}_{1\leq i\leq n}$ is a
  lattice $\Latt$ of open subsets of $X$: it suffices to add all sets
  of the form $[0 \ll U]$, where $U$ ranges over the (finite) unions
  of (finite) intersections of sets of the form $U_i$,
  $1\leq i\leq n$, if necessary.
  
  We make a final adjustment to the shape of our subbasic open sets.
  By construction, $r_i \ll \mu (U_i)$ and $s_i \ll \nu (U_i)$ for
  every $i$, $1 \leq i \leq n$.  Since
  $\mu (U_i) = \sup_{r \in {]0, 1[}} r\, \mu (U_i)$ (and this is a
  directed supremum) and similarly with $\nu$, there is a
  $c \in {]0, 1[}$ such that $r_i \ll a_i \ll \mu (U_i)$ and
  $s_i \ll b_i \ll \nu (U_i)$ for every $i$, $1\leq i\leq n$, where
  $a_i \eqdef c\, \mu (U_i)$ and $b_i \eqdef c\, \nu (U_i)$.  Now we
  have
  $\mu \in \bigcap_{i=1}^n [a_i \ll U_i] \subseteq \bigcap_{i=1}^n
  [r_i \ll U_i] \subseteq\mathcal U$,
  $\nu \in \bigcap_{i=1}^n [b_i \ll U_i] \subseteq \bigcap_{i=1}^n
  [s_i \ll U_i] \subseteq\mathcal V$, $\mu$ and $\nu$ are bounded, and
  $\mu +_a \nu \leq \xi$.

  Let $\mathcal W \eqdef \bigcap_{i=1}^n [c_i \ll U_i]$, where
  $c_i \eqdef aa_i+(1-a)b_i = c\, (a\, \mu (U_i) + (1-a)\,\nu (U_i))$.
  For each $i \in \{1, \cdots, n\}$, $a_i \ll \mu (U_i)$ hence $a_i=0$
  or $a_i < \mu (U_i)$, and similarly $b_i=0$ or $b_i < \nu (U_i)$;
  then $c_i=0$ (if $a_i=b_i=0$), or $c_i \neq 0$, so that $a_i$ or
  $b_i$ is non-zero, and since $a \neq 0, 1$,
  $c_i < a\, \mu (U_i) + (1-a)\,\nu (U_i) \leq \xi (U_i)$.  Therefore
  $\mathcal W$ is an open neighborhood of $\xi$ in $\Val X$.
  Hence $\mathcal W \cap \B$ is an open neighborhood of $\xi$ in $\B$.
  We claim that $\mathcal W \cap \B$ is included in the upward closure
  in $\B$ of $(\mathcal U \cap \B) +_a (\mathcal V \cap \B)$.

  In order to see this, let $\xi'$ be any element of
  $\mathcal W \cap \B$.  By assumption~$(i)$, as above, there is a
  $\varpi \in \mathcal W \cap \B_b$ such that $\varpi \leq \xi'$.
  Since $\varpi$ is in $\mathcal W$, in particular
  $ca\, \mu (U_i) + c(1-a)\, \nu (U_i) \leq \varpi (U_i)$ for every
  $i$, $1\leq i\leq n$.  By the second decomposition lemma
  (Proposition~\ref{prop:decomp:2}), we can find two bounded
  continuous valuations $\mu'$ and $\nu'$ such that
  $\mu' + \nu' \leq \varpi$, $ca\, \mu (U_i) \leq \mu' (U_i)$ and
  $c(1-a)\, \nu (U_i) \leq \nu' (U_i)$ for every $i$, $1\leq i\leq n$.
  Additionally $\mu'$ and $\nu'$ are finite linear combinations of
  constrictions of $\varpi$ to pairwise disjoint crescents.

  We claim that we can require that $ca\, \mu (X) = \mu' (X)$ and
  $c(1-a)\nu (X) = \nu' (X)$.  By the (first) decomposition lemma
  (Proposition~\ref{prop:schsimp:decomp}), we can find two bounded
  continuous valuations $\mu'_1$ and $\mu'_2$ on $X$ such that: 1.\
  $ca\,\mu = \mu'_1 + \mu'_2$; 2.\ $ca\, \mu (U_i) \leq \mu'_1 (U_i)$
  for every $i$, $1 \leq i \leq n$; and 3.\
  $ca\,\mu (X) = \mu'_1 (X)$.  Additionally, $\mu'_1$ is a finite
  linear combination of constrictions of $\mu'$ to pairwise distinct
  crescents, and since $\mu'$ is itself a finite linear combination of
  $\varpi$ to pairwise distinct crescents, so is $\mu'_1$.  We can
  also find two bounded continuous valuations $\nu'_1$ and $\nu'_2$ on
  $X$ such that: 1'.\ $c(1-a)\,\nu = \nu'_1 + \nu'_2$; 2'.\
  $c(1-a)\, \nu (U_i) \leq \nu'_1 (U_i)$ for every $i$,
  $1 \leq i \leq n$; and 3'.\ $c(1-a)\,\nu (X) = \nu'_1 (X)$.  As for
  $\mu'_1$, $\nu'_1$ is a finite linear combination of constrictions
  of $\varpi$ to pairwise distinct crescents.  Then replacing $\mu'$
  by $\mu'_1$ and $\nu'$ by $\nu'_1$, which are smaller valuations,
  the property $\mu' + \nu' \leq \varpi$ will still hold, while we
  will retrieve $ca\, \mu (U_i) \leq \mu' (U_i)$ and
  $c(1-a)\, \nu (U_i) \leq \nu' (U_i)$ thanks to items~2 and~2', and
  additionally items~3 and~3' will ensure that
  $ca\,\mu (X) = \mu' (X)$ and $c(1-a)\, \nu (X) = \nu' (X)$, as
  announced.

  This serves to show that $\frac 1 a \cdot \mu'$ and
  $\frac 1 {1-a} \cdot \nu'$ are in $\B$.  Indeed, by assumption~(ii)
  $c\,\mu$ and $c\,\nu$ are in $\B$, hence in $\B_b$ since $\mu$ and
  $\nu$ are bounded.  The valuation $\frac 1 a \cdot \mu'$ is a finite
  linear combination of constrictions of $r \cdot \varpi$ (with
  $r \eqdef \frac 1 a$) to pairwise disjoint crescents and
  $\frac 1 a \cdot \mu' (X) = c\,\mu (X)$, so $\mu'$ is in $\B$ by
  assumption~(iii).  And similarly, $\nu'$ is in $\B$.
  
  Since $r_i \ll a_i = c\, \mu (U_i)$, we deduce from
  $ca\, \mu (U_i) \leq \mu' (U_i)$ that
  $r_i \ll \frac 1 a \cdot \mu' (U_i)$ for every index $i$; so
  $\frac 1 a \cdot \mu'$ is in $\bigcap_{i=1}^n [r_i \ll U_i]$, hence
  in $\mathcal U$, hence in $\mathcal U \cap \B$.  Similarly,
  $\frac 1 {1-a} \cdot \nu'$ is in $\mathcal V \cap \B$.  Since
  $\mu' + \nu' \leq \varpi \leq \xi'$ and
  $\mu' + \nu' = (\frac 1 a \cdot \mu') +_a (\frac 1 {1-a} \cdot
  \nu')$, $\xi'$ is in the upward closure of
  $(\mathcal U \cap \B) + (\mathcal V \cap \B)$.
  Since $\xi'$ was chosen arbitrarily in $\mathcal W \cap \B$, this
  completes our proof.  \qed
\end{proof}

\section{Separation theorems}
\label{sec:separation-theorems}

\subsection{Sandwich theorems}
\label{sec:sandwich-theorems}

A fundamental theorem in the theory of (semipre)ordered cones is
\emph{Roth's sandwich theorem} \cite{Roth:sandwich}: given a
semipreordered cone $\C$, a sublinear map $p \colon \C \to \creal$, a
superlinear map $q \colon \C \to \creal$, under the assumption that
for all $x, y \in \C$ such that $x \leq y$, $q (x) \leq p (y)$, there
is a linear monotonic map $\Lambda \colon \C \to \creal$ such that
$q \leq \Lambda \leq p$.  The assumption holds if $q \leq p$ and $q$
or $p$ is monotonic.  Furthermore, we can require $\Lambda$ to be a
minimal sublinear monotonic map between $q$ and $p$.  We immediately
deduce the following.
\begin{proposition}[Sandwich theorem, preordered]
  \label{prop:Roth}
  Let $\B$ be a preordered barycentric algebra,
  $p \colon \B \to \creal$ be convex, $q \colon \B \to \creal$ be
  concave, such that for all $x, y \in \B$ such that $x \leq y$,
  $q (x) \leq p (y)$; for example if $q \leq p$, and $p$ or $q$ is
  monotonic.  There is an affine monotonic map
  $\Lambda \colon \B \to \creal$ such that $q \leq \Lambda \leq p$.

  If $\B$ is a preordered pointed barycentric algebra and $p$ is
  strict (for example, sublinear), then $\Lambda$ is linear.
\end{proposition}
\begin{proof}
  We embed $\B$ in the preordered cone $\conify (\B)$ through the
  injective affine preorder embedding $\etac_{\B}$, see
  Proposition~\ref{prop:bary:alg:conify:ord}.  We will not need that
  $\etac_{\B}$ is injective or a preorder embedding, just that it is
  affine and monotonic.  By Lemma~\ref{lemma:f*:conv:conc}, $q^\cext$
  is superlinear and $p^\cext$ is sublinear.

  We claim that for all $u, v \in \conify (\B)$ such that $u \leqc v$,
  $q^\cext (u) \leq p^\cext (v)$.  If $u=v$, this is clear.
  Otherwise, $v$ can be written as $(s, y)$ and there are
  $u' \in \conify (\B)$ and $y_1 \in \B$ such that $u+u' = (s,y_1)$
  and $y_1 \leq y$.  Then
  $q^\cext (u + u') \geq q^\cext (u) + q^\cext (v)$ (since $q^\cext$
  is superlinear) $\geq q^\cext (u)$.  Since $u+u' = (s, y_1)$,
  $q^\cext (u) \leq q^\cext (u+u') = s q (y_1)$, which is less than or
  equal to $s p (y)$ by the assumption on $p$ and $q$.  But
  $s p (y) = p^\cext (v)$.

  By Roth's sandwich theorem, there is a linear monotonic map
  $\Lambda_1 \colon \conify (\B) \to \creal$ and
  $q^\cext \leq \Lambda_1 \leq p^\cext$.  Now
  $q = q^\cext \circ \etac_{\B} \leq \Lambda_1 \circ \etac_{\B} \leq
  p^\cext \circ \etac_{\B} = p$.  Therefore, letting
  $\Lambda \eqdef \Lambda_1 \circ \etac_{\B}$,
  $q \leq \Lambda \leq p$.  Additionally, since $\etac_{\B}$ is affine
  and monotonic, so is $\Lambda$.

  In the case of preordered pointed barycentric algebras, it would be
  tempting to use the monotonic linear map $\etaca_{\B}$ of $\B$ to
  $\tscope_\alpha (\B)$ instead, but this would require additional
  assumptions on $q$ and $p$, and is useless: instead, we use the
  first part of the corollary, and we find a monotonic affine map
  $\Lambda$ such that $q \leq \Lambda \leq p$; since $p$ is strict,
  $p (\bot) = 0$, so $\Lambda (\bot)=0$, and therefore $\Lambda$ is
  strict, too.  But strict affine maps and linear maps are the same
  thing by Proposition~\ref{prop:strict:aff}.  \qed
\end{proof}

The topological analogue is \emph{Keimel's sandwich theorem}
\cite[Theorem~8.2]{Keimel:topcones2}: given a semitopological cone
$\C$, a sublinear map $p \colon \C \to \creal$ and a lower
semicontinuous superlinear map $q \colon \C \to \creal$ such that
$q \leq p$, there is a lower semicontinuous linear map $\Lambda$ such
that $q \leq \Lambda \leq p$.  Keimel deduces it from Roth's sandwich
theorem.  A remarkable fact about Keimel's sandwich theorem is that we
do not need to assume any continuity or even any monotonicity property
of $p$.  Only $q$ has to be lower semicontinuous.

\begin{theorem}[Sandwich theorem, semitopological]
  \label{thm:corl:keimel}
  Let $\B$ be a semitopological barycentric algebra,
  $p \colon \B \to \creal$ be convex, $q \colon \B \to \creal$ be
  concave and lower semicontinuous, and let us assume that $q \leq p$.
  There is an affine lower semicontinuous map
  $\Lambda \colon \B \to \creal$ such that $q \leq \Lambda \leq p$.

  If $\B$ is pointed and $p$ is strict (e.g., sublinear), then
  $\Lambda$ is linear.
\end{theorem}
\begin{proof}
  We use Theorem~\ref{thm:conify:semitop} and we obtain that
  $\etac_{\B}$ is an injective, affine continuous map from $\B$ to the
  semitopological cone $\conify (\B)$.  Since $q$ is concave,
  $q^\cext$ is superlinear and lower semicontinuous from
  $\conify (\B)$ to $\creal$.  We cannot use
  Theorem~\ref{thm:conify:semitop} on $p^\cext$, since it says nothing
  about convex maps, but $p^\cext$ is sublinear by
  Lemma~\ref{lemma:f*:conv:conc}.  Now $q^\cext \leq p^\cext$:
  $q^\cext (0) = 0 = p^\cext (0)$, and for all points $(r, x)$ with
  $r > 0$ and $x \in \B$,
  $q^\cext (r, x) = r\, q (x) \leq r \, p (x) = p^\cext (r, x)$.
  Then, by using Keimel's sandwich theorem (and remembering that
  $p^\cext$ does not have to be lower semicontinuous, or even
  monotonic), there is a linear lower semicontinuous map
  $\Lambda_1 \colon \conify (\B) \to \creal$ such that
  $q^\cext \leq \Lambda_1 \leq p^\cext$.  Then
  $\Lambda \eqdef \Lambda_1 \circ \etac_{\B}$ is affine and lower
  semicontinuous, and $q \leq \Lambda \leq p$.

  Finally, if $\B$ is pointed and $p$ is strict, then $\Lambda$ is
  strict, hence linear by Proposition~\ref{prop:strict:aff}.  \qed
\end{proof}

The following is proved by Keimel \cite[Corollary 9.3, Corollary
9.4]{Keimel:topcones2} under the additional assumption that $\C$ is
$T_0$.  That assumption is unneeded.  We include the proof for
completeness.  One may wonder why we do not state this at the level of
barycentric algebras, and this is because it would be wrong, see
Remark~\ref{rem:loc:conv:halfspace:bary}.
\begin{proposition}
  \label{prop:loc:conv:halfspace}
  Let $\C$ be a locally convex semitopological cone.  Then:
  \begin{enumerate}
  \item $\C$ is linearly separated;
  \item for every non-empty closed convex subset $C$ of $\C$, every
    point $x$ outside $C$ is contained in some open half-space that
    does not intersect $C$.
  \end{enumerate}
\end{proposition}
\begin{proof}
  We prove item~2.  Item~1 follows: given $x, y \in \C$ such that $x
  \not\leq y$, we apply item~2 with $C \eqdef \dc y$.

  Using the fact that $\C$ is locally convex, there is a convex open
  subset $U$ of $\C$ that contains $x$ and is included in
  $\C \diff C$, namely is disjoint from $C$.  Let $q$ be the upper
  Minkowski functional $M^U$; since $U$ is proper, convex and open,
  $q$ is superlinear and lower semicontinuous.  Let
  $p \eqdef M^{\C \diff \dc y}$.  Since $C$ is convex,
  $\C \diff \dc y$ is concave; it is open and proper, so $p$ is
  sublinear (and lower semicontinuous).  Additionally, $q \leq p$,
  because $U \subseteq \C \diff C$.  By Keimel's sandwich theorem,
  there is a linear lower semicontinuous map $\Lambda$ such that
  $q \leq \Lambda \leq p$.  Then
  $q^{-1} (]1, \infty]) = U \subseteq \Lambda^{-1} (]1, \infty])
  \subseteq p^{-1} (]1, \infty]) = \C \diff C$, so
  $H \eqdef \Lambda^{-1} (]1, \infty])$ is an open half-space that
  contains $U$ (hence $x$) but is disjoint from $C$.  \qed
\end{proof}

\begin{remark}
  \label{rem:loc:conv:halfspace:bary}
  A locally convex semitopological barycentric algebra may fail to be
  linearly separated: by Example~\ref{exa:B-:loc:convex}, the
  semitopological barycentric algebra $\B^-$ of
  Example~\ref{exa:bary:alg:notembed} is locally convex; it is not
  linearly separated because its only semi-concave lower
  semicontinuous maps, hence also its only linear lower semicontinuous
  maps, are constant.
\end{remark}

Cones also enjoy the following \emph{strict separation theorem}
\cite[Theorem~10.5]{Keimel:topcones2}: given a locally convex
semitopological cone $\C$, a convex compact subset $Q$ of $\C$, and a
non-empty closed convex $C$ disjoint from $Q$, there is a linear lower
semicontinuous map $\Lambda \colon \C \to \creal$ and a number $a > 1$
such that $\Lambda$ takes values in $[a, \infty]$ on $Q$ and values in
$[0, 1]$ on $C$.

\subsection{Convex compact saturated sets}
\label{sec:conv-comp-satur}

The strict separation theorem has the following consequence.
\begin{lemma}
  \label{lemma:compact:convex}
  In a locally convex semitopological cone $\C$, every convex compact
  saturated subset $Q$ is the intersection of the open half-spaces
  that contain it.
\end{lemma}
\begin{proof}
  That intersection, call it $A$, contains the intersection of all the
  open neighborhoods of $Q$, which coincides with $Q$ since $Q$ is
  saturated.  Conversely, we show that $A$ is included in $Q$ by
  showing that every point $x \in \C \diff Q$ is outside $A$.  For
  each such $x$, $C \eqdef \dc x$ is a non-empty closed convex subset
  of $\C$ that is disjoint from $Q$, since $Q$ is upwards-closed.  By
  the strict separation theorem, there is an open half-space $H$
  containing $Q$ and not containing $x$.  Hence $x$ is not in $A$.  \qed
\end{proof}

We apply all this to barycentric algebras.  Item~2 below was proved
for cones as \cite[Corollary 10.4]{Keimel:topcones2}.
\begin{proposition}
  \label{prop:compact:convex}
  In a linearly separated semitopological barycentric algebra $\B$,
  every convex compact saturated subset $Q$ of $\B$ is equal to:
  \begin{enumerate}
  \item the intersection of its open neighborhoods of the form
    $\Lambda^{-1} (]1, \infty])$ with $\Lambda \colon \B \to \creal$
    affine (even linear if $\B$ is pointed) and lower semicontinuous;
  \item the filtered intersection of its convex open neighborhoods.
  \end{enumerate}
\end{proposition}
\begin{proof}
  1.  It suffices to show that for every point $x \in \B$ that is not
  in $Q$, there is a convex open neighborhood of $Q$, of the required
  form $\Lambda^{-1} (]1, \infty])$, that does not contain $x$.

  By Lemma~\ref{lemma:loclin:bary:alg}, item~1, there is an affine
  (even linear when $\B$ is pointed) continuous map
  $\eta \colon \B \to \C$ of $\B$ to a locally linear semitopological
  cone $\C$ that is also a preorder embedding.  In particular, $\C$ is
  locally convex.

  Since $\eta$ is continuous, $\upc \eta [Q]$ is compact saturated,
  and $\upc \eta [Q]$ is convex because $\eta [Q]$ is convex, owing
  to the fact that $\eta$ is affine, and by Definition and
  Proposition~\ref{defprop:upc:conv} for upward closures.

  We claim that $\eta (x) \not\in \upc \eta [Q]$.  Otherwise there would be
  a point $y \in Q$ such that $\eta (y) \leq \eta (x)$.  Since $\eta$
  is a preorder embedding, we would have $y \leq x$, and since $Q$ is
  upwards-closed, $x$ would be in $Q$, which is impossible.

  Hence we can apply Lemma~\ref{lemma:compact:convex}: there is a
  (necessarily proper) open half-space $H$ of $\C$ that contains
  $\upc \eta [Q]$ and does not contain $\eta (x)$.  Then
  $U \eqdef \eta^{-1} (H)$ is open and convex.  We can write $H$ as
  $(M^H)^{-1} (]1, \infty])$, where $M^H$ is the upper Minkowski
  functional of $H$, and then $U = \Lambda^{-1} (]1, \infty])$, where
  $\Lambda \eqdef M^H \circ \eta$ is affine (even linear when $\B$ is
  pointed) and lower semicontinuous.  Every point of $Q$ is in $U$
  since $\eta [Q] \subseteq H$, and $x \not\in U$ since otherwise
  $\eta (x)$ would be in $H$.


  2.  By item~1, $Q$ is also the intersection of the larger family of
  its convex open neighborhoods.  Since $\B$ is convex and open, and
  since the intersection of any two convex open sets is convex and
  open, we conclude.  \qed
\end{proof}


\section{Local convex-compactness}
\label{sec:local-conv-comp}

The notion of a \emph{locally convex-compact} cone was introduced by
Keimel \cite[Definition~4.9]{Keimel:topcones2}.  It easily generalizes
to semitopological barycentric algebras, as follows.

\subsection{Locally convex-compact semitopological barycentric
  algebras}
\label{sec:locally-conv-comp}

\begin{definition}
  \label{defn:locconvcomp}
  A \emph{locally convex-compact barycentric algebra} is a
  semitopological barycentric algebra in which every point has a
  neighborhood base consisting of convex compact sets.
\end{definition}

It is clear that local convex-compactness implies both \emph{weak}
local convexity, and local compactness.

\begin{remark}
  \label{rem:locconvcomp}
  A semitopological barycentric algebra is locally convex-comp\-act if
  and only if every point has a neighborhood base of convex compact
  \emph{saturated} sets.  Indeed, if $Q$ is convex and compact, then
  $\upc Q$ is convex, compact saturated, and contains $Q$, while still
  being included in any open neighborhood of $Q$.
\end{remark}

We already know that every c-cone, every semitopological barycentric
algebra whose underlying topological space is a c-space, is locally
convex and topological (Proposition~\ref{prop:dcone:loc:convex}).
Additionally, we have the following.
\begin{proposition}
  \label{prop:dcone:locconvcomp}
  Every semitopological barycentric algebra whose underlying
  topological space is a c-space is locally convex-compact.
%
\end{proposition}
\begin{proof}
  Let $\B$ be such a semitopological barycentric algebra.  For every
  point $x \in \B$, for every open neighborhood $U$ of $x$, there is a
  point $y \in U$ such that $x \in \interior {\upc y}$.  We conclude
  since $\upc y$ is both convex and compact.  \qed
\end{proof}


Another source of such barycentric algebras and cones is the
following.
\begin{theorem}
  \label{thm:VX:locconvcomp}
  For every core-compact space $X$, $\Val X$ is a locally linear,
  locally convex-compact, sober topological cone; $\Val_{\leq 1} X$ is
  a locally linear, locally convex-compact, sober pointed topological
  barycentric algebra; if $X$ is compact in addition, then $\Val_1 X$
  is a locally affine, locally convex-compact, sober topological
  barycentric algebra.
\end{theorem}
\begin{proof}
  Theorem~4.1 of \cite{JGL:V:lcomp} states that $\Val X$ and
  $\Val_{\leq 1} X$ are locally compact if $X$ is core-compact, and
  that $\Val_1 X$ is locally compact if $X$ is core-compact and
  compact.  The proof reveals that a base of compact saturated
  neighborhoods of each element $\nu$ of the relevant space
  $\Val_\rast X$ of valuations is given by sets $s_\DN (F)$, where:
  \begin{align*}
    s_\DN (F)
    & \eqdef \Big\{\nu' \in \Val_\rast X \mid \forall h \in \Lform X,
      \int h \,d\nu' \geq F (h)\Big\},
  \end{align*}
  and $F$ ranges over a certain set of superlinear previsions
  (subnormalized if $\rast$ is ``$\leq 1$'', normalized if $\rast$ is
  ``$1$'').  We will not describe what superlinear previsions are: we
  only need to observe that $s_\DN (F)$ is convex.  Hence
  $\Val_\rast X$ is locally convex-compact.  It is sober by Lemma~4.3
  of \cite{JGL:V:lcomp}; the argument is due to Regina Tix
  \cite[Satz~5.4]{Tix:bewertung}, following ideas by Heckmann (see
  \cite[Section~2.3]{heckmann96}).  It is locally linear (locally
  affine when $\rast$ is ``$1$'') by Examples~\ref{exa:VX:loclin},
  \ref{exa:Vleq1X:loclin} and~\ref{exa:V1X:loclin}.  It is topological
  by Examples~\ref{exa:VX:top:cone} and~\ref{exa:V1X:top:bary:alg}.
  \qed
\end{proof}

\begin{proposition}
  \label{prop:linretr:locconvcomp}
  Every affine retract of a locally convex-compact semitopological
  barycentric algebra is locally convex-compact.

  In particular, every linear retract of a locally convex-compact
  semitopological cone is locally convex-compact.
\end{proposition}
\begin{proof}
  The proof is as for Proposition~\ref{prop:locconv:retract}, item~1.
  Let $r \colon \B \to \Alg$ be an affine retraction between
  semitopological barycentric algebras, with associated section $s$,
  where $\B$ is locally convex-compact.  Let $y \in \Alg$ and $V$ be
  an open neighborhood of $y$ in $\Alg$.  Then $r^{-1} (V)$ is an open
  neighborhood of $s (y)$, and therefore contains a convex compact
  neighborhood $Q$ of $s (y)$.  Then $s^{-1} (\interior Q)$ is an open
  neighborhood of $y$, which is included in $r [Q]$, and $r [Q]$ is
  convex because $Q$ is and $r$ is affine, and is compact because $Q$
  is compact and $r$ is continuous.  Finally, $r [Q] \subseteq V$,
  since $Q \subseteq r^{-1} (V)$.  \qed
\end{proof}

\subsection{Convenient barycentric algebras}
\label{sec:conv-baryc-algebr}

\begin{definition}[Convenient]
  \label{defn:convenient:bary:alg}
  A \emph{convenient barycentric algebra} is a linearly separated,
  locally convex-compact, sober topological barycentric algebra.
  Similarly with pointed barycentric algebras and cones.
\end{definition}

\begin{example}
  \label{exa:VX:convenient}
  Since every locally affine semitopological barycentric algebra is
  linearly separated, Theorem~\ref{thm:VX:locconvcomp} implies that
  $\Val X$, $\Val_{\leq 1} X$ and $\Val_1 X$ are convenient for every
  core-compact space $X$.
\end{example}

\begin{example}
  \label{exa:dbary:convenient}
  Let us call \emph{d-barycentric algebra} any s-barycentric algebra
  that is a dcpo in its ordering.  (A \emph{d-cone} is an ordered cone
  that is a d-barycentric algebra \cite[Definition
  6.1]{Keimel:topcones2}.)  The continuous d-barycentric algebras,
  with their Scott topology, are the same thing as the semitopological
  barycentric algebras whose underlying space is a sober c-space;
  indeed, a sober c-space is the same thing as a continuous dcpo with
  its Scott topology \cite[Proposition 8.3.36]{JGL-topology}.

  Then every continuous d-barycentric algebra $\B$ that is linearly
  separated is a convenient barycentric algebra, and every continuous
  d-cone is a convenient cone.
\end{example}
\begin{proof}
  By Proposition~\ref{prop:dcone:loc:convex}, $\B$ is locally convex
  and topological.  Hence, by
  Proposition~\ref{prop:dcone:locconvcomp}, $\B$ is locally convex,
  locally convex-compact.  It is sober, and topological (by Ershov's
  observation) in its Scott topology; only linear separation may be
  missing.  When $\B$ is a cone, nothing is missing: as we will see in
  Proposition~\ref{prop:loc:conv:halfspace}, every locally convex
  semitopological cone is linearly separated.
  \qed
\end{proof}

\begin{lemma}
  \label{lemma:locconvcomp:interp}
  Let $\B$ be a locally convex-compact topological barycentric
  algebra.  For every convex compact saturated subset $Q$ of $\B$ and
  every convex upwards-closed neighborhood $C$ of $Q$, there is a
  convex compact saturated subset $Q'$ of $C$ whose interior contains
  $Q$.
\end{lemma}
\begin{proof}
  Let $Q$ and $C$ be as given.  For each $x \in Q$, there is a convex
  compact neighborhood $Q_x$ of $x$ included in $\interior C$, and we
  can take it saturated by Remark~\ref{rem:locconvcomp}.  Since $Q$ is
  compact, there is a finite subset $E$ of $U$ such that
  $Q \subseteq \bigcup_{x \in E} \interior {Q_x}$.  Let
  $Q' \eqdef \upc\conv {(\bigcup_{x \in E} Q_x)}$.  By
  Proposition~\ref{prop:conv:union:compact}, which we can apply since
  $\B$ is topological, $Q'$ is convex and compact saturated.  In order
  to see that $Q' \subseteq C$, we observe that $Q'$ is the saturated
  convex hull of $\bigcup_{x \in E} Q_x$, by Definition and
  Proposition~\ref{defprop:upc:conv}.  Since $C$ is upwards-closed and
  convex, the inclusion then follows from the fact that
  $Q_x \subseteq C$ for every $x \in E$.  Finally, $Q'$ contains the
  open set $\bigcup_{x \in E} \interior {Q_x}$, which contains $Q$.  \qed
\end{proof}

The following is similar to \cite[Proposition~10.6]{Keimel:topcones2},
which states the same for locally convex, locally convex-compact
topological cones, and part of the argument is due to Achim Jung.
Note that, instead of local convexity, we only require linear
separation, as part of the definition of ``convenient''.

\begin{proposition}
  \label{prop:locconvcomp:base}
  Let $\B$ be a convenient barycentric algebra.  Then every convex
  compact saturated subset $Q$ of $\B$ has:
  \begin{enumerate}
  \item a base of convex compact saturated neighborhoods;
  \item and a base of convex open neighborhoods.
  \end{enumerate}
\end{proposition}
In other words, for every open neighborhood $U$ of $Q$, there is a
convex compact saturated set $Q'$ such that
$Q \subseteq \interior {Q'} \subseteq Q' \subseteq U$, and there is a
convex open set $V$ such that $Q \subseteq V \subseteq U$.

\begin{proof}
  For every convex upwards-closed neighborhood $C$ of $Q$, there is a
  convex compact saturated neighborhood $Q'$ of $Q$ included in $C$,
  by Lemma~\ref{lemma:locconvcomp:interp}.  Taking $C \eqdef \B$, this
  entails that the family $\mathcal F$ of all convex compact saturated
  neighborhoods of $Q$ is non-empty.  Additionally, for any two
  elements $Q_1$, $Q_2$ of $\mathcal F$, $C \eqdef Q_1 \cap Q_2$ is a
  convex upwards-closed neighborhood of $Q$, so there is an element
  $Q'$ of $\mathcal F$ included in $C$, showing that $\mathcal F$ is a
  filtered family.  Since $\B$ is sober hence well-filtered, the
  intersection $\bigcap \mathcal F$ is compact saturated and contains
  $Q$.
  
  By Proposition~\ref{prop:compact:convex}, item~2, which we can apply
  since $\B$ is linearly separated, $Q$ is the intersection of its
  convex open neighborhoods.  Each one contains an element of
  $\mathcal F$, by Lemma~\ref{lemma:locconvcomp:interp}.  It follows
  that $Q = \bigcap \mathcal F$.
  
  1. Let $U$ be any open neighborhood of $Q$.  From
  $\bigcap \mathcal F \subseteq U$ and the fact that $\B$ is
  well-filtered, there is an element $Q' \in \mathcal F$ such that
  $Q' \subseteq U$.

  2. Let $U$ be, once again, an open neighborhood of $Q$.  By item~1,
  there is a convex compact saturated neighborhood $Q_0$ of $Q$ such
  that $Q_0 \subseteq U$.  By item~1 again, there is a convex compact
  saturated neighborhood $Q_1$ of $Q_0$ such that $Q_1 \subseteq U$.
  Proceeding in this way, we obtain an ascending sequence
  $Q \subseteq \interior {Q_0} \subseteq Q_0 \subseteq \interior {Q_1}
  \subseteq Q_1 \subseteq \cdots \subseteq \interior {Q_n} \subseteq
  Q_n \subseteq \cdots$, where each $Q_n$ is convex, compact
  saturated, and included in $U$.  Let
  $V \eqdef \bigcup_{n \in \nat} Q_n$.  By
  Lemma~\ref{lemma:convex:inter}, $V$ is convex.  $V$ is also equal to
  $\bigcup_{n \in \nat} \interior {Q_n}$, hence is open.  Finally, it is
  clear that $V$ contains $Q$ and is included in $U$.  \qed
\end{proof}

\begin{corollary}
  \label{corl:locconvcomp:locconv}
  Every convenient barycentric algebra is locally convex.
\end{corollary}
\begin{proof}
  For each point $x$, we apply
  Proposition~\ref{prop:locconvcomp:base}, item~2, to the convex
  compact saturated $\upc x$.  \qed
\end{proof}

\subsection{The Smyth poweralgebra}
\label{sec:smyth-poweralgebra}

On top of any topological cone $\C$, one can build its \emph{Smyth
  powercone} $\Smyth^{cvx} \C$ \cite[Section~11]{Keimel:topcones2}.
The construction extends to topological barycentric algebras.

\begin{definition}
  \label{defn:Qcvx:bary:alg}
  For every topological barycentric algebra $\B$, its \emph{Smyth
    poweralgebra} $\Smythc \B$ over $\B$ is the collection of all
  non-empty, convex, compact saturated subsets of $\B$.  For all
  $Q_1, Q_2 \in \Smythc \B$, for every $a \in [0, 1]$, we define
  $Q_1 +^\sharp_a Q_2$ as
  $\upc \{x_1 +_a x_2 \mid x_1 \in Q_1, x_2 \in Q_2\}$.
\end{definition}

We require $\B$ to be topological in
Definition~\ref{defn:Qcvx:bary:alg}: this is needed for $+^\sharp_a$
to be well-defined, as the proof of the following lemma shows.

\begin{lemma}
  \label{lemma:Qcvx:bary:alg}
  For every topological barycentric algebra $\B$, $\Smythc \B$ with
  the operations $+^\sharp_a$, $a \in [0, 1]$, is a barycentric
  algebra.
\end{lemma}
\begin{proof}
  We first verify that $+^\sharp_a$ is well-defined for every
  $a \in [0, 1]$.  For all $Q_1, Q_2 \in \Smythc \B$,
  $\upc \{x_1 +_a x_2 \mid x_1 \in Q_1, x_2 \in Q_2\}$ is compact
  saturated because it is the upward closure of the image of the
  compact set $Q_1 \times Q_2$ by the (jointly) continuous map
  $(x_1, x_2) \mapsto x_1 +_a x_2$.  In order to see that it is
  convex, it suffices to show that
  $A \eqdef \{x_1 +_a x_2 \mid x_1 \in Q_1, x_2 \in Q_2\}$ is convex,
  since this property is preserved by upward closures, as we have seen
  in Definition and Proposition~\ref{defprop:upc:conv}.  For every
  $b \in [0, 1]$, for every two elements $x \eqdef x_1 +_a x_2$ and
  $y \eqdef y_1 +_a y_2$ of $A$ (with $x_1, y_1 \in Q_1$ and
  $x_2, y_2 \in Q_2$), $x +_b y = (x_1 +_b y_1) +_a (x_2 +_b y_2)$, as
  one can check easily by embedding $\B$ in a cone and checking the
  equation there (this is the entropic law (\ref{eq:entropic})).
  Since $Q_1$ is convex, $x_1 +_b y_1$ is in $Q_1$ and since $Q_2$ is
  convex, $x_2 +_b y_2$ is in $Q_2$, so $x +_b y \in A$.  Hence
  $\upc A$ is convex.  We have seen that it is compact saturated, and
  it is clearly non-empty.  Hence it is a well-defined element
  $Q_1 +^\sharp_a Q_2$ of $\Smythc \B$.

  We now verify the laws of barycentric algebras.
  When $a = 1$, $Q_1 +^\sharp_a Q_2 = \upc \{x_1 \mid x_1 \in Q_1, x_2
  \in Q_2\} = \upc Q_1 = Q_1$.
  
  When $Q_1 = Q_2$, every point of $Q_1 +^\sharp_a Q_1$ is of the form
  $x_1 +_a x_2$ with $x_1, x_2 \in Q_1$; since $Q_1$ is convex, every
  such point is in $Q_1$, so $Q_1 +^\sharp_a Q_1 \subseteq Q_1$.
  Conversely, for every $x \in Q_1$, we can write $x$ as $x +_a x$,
  showing that $Q_1 \subseteq Q_1 +^\sharp Q_1$.  Hence $Q_1 +^\sharp
  Q_1 = Q_1$.

  It is clear that $Q_1 +^\sharp_a Q_2 = Q_2 +^\sharp_{1-a} Q_1$ for
  all $Q_1, Q_2 \in \Smythc \B$.

  Finally, we must show that
  $(Q_1 +^\sharp_a Q_2) +^\sharp_b Q_3 = Q_1 +^\sharp_{ab} (Q_2
  +^\sharp_{\frac {(1-a)b} {1-ab}} Q_3)$, for all
  $Q_1, Q_2, Q_3 \in \Smythc \B$ and $a, b \in {[0, 1[}$.  We will
  only show that the left-hand side is included in the right-hand
  side, as the reverse inclusion is proved similarly.  For every
  $z \in (Q_1 +^\sharp_a Q_2) +^\sharp_b Q_3$, there are points
  $y \in Q_1 +^\sharp_a Q_2$ and $x_3 \in Q_3$ such that
  $y +_b x_3 \leq z$; then, there are points $x_1 \in Q_1$ and
  $x_2 \in Q_2$ such that $x_1 +_a x_2 \leq y$.  Since continuous maps
  are monotonic, $(x_1 +_a x_2) +_b x_3 \leq z$.  But
  $(x_1 +_a x_2) +_b x_3 = x_1 +_{ab} (x_2 +_{\frac {(1-a)b} {1-ab}}
  x_3)$, and
  $x_2 +_{\frac {(1-a)b} {1-ab}} x_3 \in Q_2 +^\sharp_{\frac {(1-a)b}
    {1-ab}} Q_3$, so
  $z \in Q_1 +^\sharp_{ab} (Q_2 +^\sharp_{\frac {(1-a)b} {1-ab}}
  Q_3)$.  \qed
\end{proof}

The weighted sums used in the following lemma have to be understood as
in Definition and Proposition~\ref{prop:bary:1}.
\begin{lemma}
  \label{lemma:Qcvx:sums:bary:alg}
  Let $\B$ be a topological barycentric algebra.
  \begin{enumerate}
  \item for all points $x_1, \cdots, x_m, y_1, \cdots, y_n \in \B$
    ($m, n \geq 1$), for every $a \in [0, 1]$,
    ${\upc \conv} {\{x_1, \cdots, x_m\}} +^\sharp_a \allowbreak {\upc
      \conv} {\{y_1, \cdots, y_n\}} = \upc \conv \{x_i +_a y_j \mid
    1\leq i\leq m, \allowbreak 1\leq j \leq n\}$;
  \item for all finite non-empty subsets $E_1$, \ldots, $E_n$ of $\B$
    and for every $n$-tuple $(a_1, \cdots, a_n)$ in the standard
    simplex $\Delta_n$ (with $n \geq 1$),
    $\sum_{i=1}^n a_i \cdot {\upc \conv {E_i}} = \upc \conv
    {\{\sum_{i=1}^n a_i \cdot x_i \mid x_1 \in E_1, \cdots, x_n \in
      E_n\}}$;
  \item for all finite non-empty subsets $E_1$, \ldots, $E_n$ of $\B$,
    the convex closure of
    $\{\upc \conv {E_1}, \cdots, \upc \conv {E_n}\}$ in $\Smythc \B$
    consists of all the sets of the form
    $\upc \conv {\{\sum_{i=1}^n a_i \cdot x_i \mid x_1 \in E_1,
      \cdots, x_n \in E_n\}}$, where
    $(a_1, \cdots, a_n) \in \Delta_n$.
  \end{enumerate}
\end{lemma}
\begin{proof}
  1.  Every point $z$ in the left-hand side is larger than or equal to
  some point of the form $x +_a y$ where $x$ is larger than or equal
  to some convex combination $\sum_{i=1}^m a_i \cdot x_i$ and $y$ is
  larger than or equal to some convex combination
  $\sum_{j=1}^n b_j \cdot y_j$, with $(a_1, \cdots, a_m) \in \Delta_m$
  and $(b_1, \cdots, b_n) \in \Delta_n$, by Definition and
  Lemma~\ref{deflem:convex:simpleval}.  In order to reason on those
  sums, we use Definition and Proposition~\ref{prop:bary:1}, and
  therefore we reason in the cone $\conify (\B)$, using the injective
  affine map $\etac_{\B}$.

  Using the fact that $\sum_{j=1}^n b_j = 1$, we have
  $\etac_{\B} (\sum_{i=1}^m a_i \cdot x_i) = \sum_{i=1}^m a_i \cdot
  \etac_{\B} (x_i) = \sum_{i=1}^m a_i (\sum_{j=1}^n b_j) \cdot
  \etac_{\B} (x_i) = \sum_{\substack{1\leq i\leq m\\1\leq j\leq n}}
  a_i b_j \cdot \etac_{\B} (x_i)$, and similarly
  $\etac_{\B} (\sum_{j=1}^n b_j \cdot y_j) = \sum_{\substack{1\leq
      i\leq m\\1\leq j\leq n}} a_i b_j \cdot \etac_{\B} (y_j)$, so
  $\etac_{\B} (\sum_{i=1}^m a_i \cdot x_i +_a \sum_{j=1}^n b_j \cdot
  y_j) = a \cdot \sum_{\substack{1\leq i\leq m\\1\leq j\leq n}} a_i
  b_j \cdot \etac_{\B} (x_i) + (1-a) \cdot \sum_{\substack{1\leq i\leq
      m\\1\leq j\leq n}} a_i b_j \cdot \etac_{\B} (y_j) =
  \sum_{\substack{1\leq i\leq m\\1\leq j\leq n}} a_i b_j \cdot (a
  \cdot \etac_{\B} (x_i) + (1-a) \cdot \etac_{\B} (y_j)) =
  \sum_{\substack{1\leq i\leq m\\1\leq j\leq n}} a_i b_j \cdot
  \etac_{\B} (x_i +_a y_j)$ (since $\etac_{\B}$ is affine)
  $= \etac_{\B} (\sum_{\substack{1\leq i\leq m\\1\leq j\leq n}} a_i
  b_j \cdot (x_i +_a y_j))$.  Since $\etac_{\B}$ is injective,
  $\sum_{i=1}^m a_i \cdot x_i +_a \sum_{j=1}^n b_j \cdot y_j =
  \sum_{\substack{1\leq i\leq m\\1\leq j\leq n}} a_i b_j \cdot (x_i
  +_a y_j)$.  Now the left-hand side is less than or equal to
  $x +_a y$ (since $+_a$ is monotonic), which is less than or equal to
  $z$, so
  $z \geq \sum_{\substack{1\leq i\leq m\\1\leq j\leq n}} a_i b_j \cdot
  (x_i +_a y_j)$.  Hence
  $z \in \upc \conv {\{x_i +_a y_j \mid 1\leq i\leq m, 1\leq j \leq
    n\}}$.

  Conversely, for every
  $z \in \upc \conv {\{x_i +_a y_j \mid 1\leq i\leq m, 1\leq j \leq
    n\}}$, there is a tuple of non-negative real numbers $c_{ij}$
  whose sum equals $1$ and such that
  $z \geq \sum_{\substack{1\leq i\leq m\\1\leq j\leq n}} c_{ij} \cdot
  (x_i +_a y_j)$.  Hence $z \geq x +_a y$ where
  $x \eqdef \sum_{\substack{1\leq i\leq m\\1\leq j\leq n}} c_{ij}
  \cdot x_i$ and
  $y \eqdef \sum_{\substack{1\leq i\leq m\\1\leq j\leq n}} c_{ij}
  \cdot y_j$, showing that
  $z \in \upc \conv {\{x_1, \cdots, x_m\}} +^\sharp_a \allowbreak \upc
  \conv {\{y_1, \cdots, y_n\}}$.

  2. By induction on $n$, using Definition and
  Proposition~\ref{prop:bary:1}, item~1.  If $a_n=1$, then
  $\sum_{i=1}^n a_i \cdot \upc \conv E_i = \upc \conv E_n$, while
  $\upc \conv \{\sum_{i=1}^n a_i \cdot x_i \mid x_1 \in E_1,
  \allowbreak \cdots, \allowbreak x_n \in E_n\} = \upc \conv \{x_n
  \mid x_1 \in E_1, \cdots, x_n \in E_n\} = \upc \conv \{x_n \mid x_n
  \in E_n\}$ (since $E_1$, \ldots, $E_{n-1}$ are non-empty)
  $= \upc \conv E_n$.

  If $a_n \neq 1$, then:
  \ifta
  \begin{align*}
    \sum_{i=1}^n a_i \cdot \upc \conv E_i
    & = \left(\sum_{i=1}^{n-1} \frac {a_i} {1-a_n} \cdot \upc \conv E_i\right)
      +_{1-a_n} \upc \conv E_n \\
    & = \upc \conv \left\{\sum_{i=1}^{n-1} \frac {a_i} {1-a_n} \cdot x_i
      \mid x_1 \in E_1, \cdots, x_{n-1} \in E_{n-1}\right\} \\
    & \quad +_{1-a_n} \upc
      \conv E_n
    \qquad\qquad \text{by induction hypothesis} \\
    & = \upc \conv \left\{\Big(\sum_{i=1}^{n-1} \frac {a_i} {1-a_n} \cdot x_i\Big)
      +_{1-a_n} x_n \right.\\
    & \qquad\qquad \mid x_1 \in E_1, \cdots, 
      x_n \in E_n\Big\}
    \qquad \text{by item~1} \\
    & = \upc \conv
      {\left\{\sum_{i=1}^n a_i \cdot x_i \mid x_1 \in E_1, \cdots, x_n \in
      E_n\right\}}
  \end{align*}
  \else
  \begin{align*}
    & \sum_{i=1}^n a_i \cdot \upc \conv E_i \\
    & = \left(\sum_{i=1}^{n-1} \frac {a_i} {1-a_n} \cdot \upc \conv E_i\right)
      +_{1-a_n} \upc \conv E_n \\
    & = \upc \conv \left\{\sum_{i=1}^{n-1} \frac {a_i} {1-a_n} \cdot x_i
      \mid x_1 \in E_1, \cdots, x_{n-1} \in E_{n-1}\right\} \\
    & \quad +_{1-a_n} \upc
      \conv E_n
    \qquad\qquad \text{by induction hypothesis} \\
    & = \upc \conv \left\{\Big(\sum_{i=1}^{n-1} \frac {a_i} {1-a_n} \cdot x_i\Big)
      +_{1-a_n} x_n \right.\\
    & \qquad\qquad \mid x_1 \in E_1, \cdots, 
      x_n \in E_n\Big\}
    \qquad \text{by item~1} \\
    & = \upc \conv
      {\left\{\sum_{i=1}^n a_i \cdot x_i \mid x_1 \in E_1, \cdots, x_n \in
      E_n\right\}}
  \end{align*}
  \fi

  3. By Definition and Lemma~\ref{deflem:convex:simpleval}, the
  desired convex closure is the collection of convex compact saturated
  sets of the form $\sum_{i=1}^n a_i \cdot \upc \conv {E_i}$, and we
  conclude by item~2.  \qed
\end{proof}

There is a classical \emph{Smyth hyperspace} construction $\Smyth X$,
consisting of all non-empty compact saturated subsets of a subspace
$X$.  One of the standard topologies we give it is the \emph{upper
  Vietoris} topology, with a base of open sets
$\Box U \eqdef \{Q \in \Smyth X \mid Q \subseteq U\}$,
$U \in \Open X$.  We write $\SV X$ for $\Smyth X$ with the Vietoris
topology.  Its specialization ordering is \emph{reverse} inclusion.
\begin{definition}[$\SVc \B$]
  \label{defn:QcvxC:Vietoris}
  For every topological barycentric algebra $\B$, let $\SVc \B$ be
  $\Smythc \B$ with its \emph{upper Vietoris topology}, generated by
  basic open sets
  $\Box^\cvx U \eqdef \{Q \in \Smythc \B \mid Q \subseteq U\}$ where
  $U$ ranges over the open subsets of $\B$.
\end{definition}
Hence $\SVc \B$ is a subspace of $\SV \B$.

There is a continuous map $\eta^\Smyth_X \colon X \to \SV X$, defined
by $\eta^\Smyth_X (x) \eqdef \upc x$ on every space $X$.  It is
continuous, and in fact full because
$(\eta^\Smyth_X)^{-1} (\Box U) = U$.
\begin{lemma}
  \label{lemma:Qcvx:eta}
  For every topological barycentric algebra $\B$, the map
  $\eta^\Smyth_\B$, defined as sending every $x \in \B$ to $\upc x$,
  is affine, continuous and full.  It is a topological embedding if
  $\B$ is $T_0$.
\end{lemma}
\begin{proof}
  First, $\upc x$ is indeed convex, non-empty, and compact saturated.
  Then, for every open subset $U$ of $\B$,
  ${(\eta^\Smyth_\B)}^{-1} (\Box^\cvx U) = \{x \in \B \mid \upc x
  \subseteq U\} = U$, showing that $\eta^\Smyth_\B$ is continuous and
  full; when $\B$ is $T_0$, $\eta^\Smyth_\B$ is therefore a
  topological embedding.  The fact that $\eta^\Smyth_\B$ is affine is
  the special case $m=1$, $n=1$ of
  Lemma~\ref{lemma:Qcvx:sums:bary:alg}, item~1.  \qed
\end{proof}

\begin{lemma}
  \label{lemma:Qcvx:spec}
  Given a topological barycentric algebra $\B$,
  \begin{enumerate}
  \item the sets $\Box^\cvx U$ with $U$ open in $\B$ form a base of
    the topology on $\SVc \B$;
  \item the specialization ordering on $\SVc \B$ is reverse inclusion
    $\supseteq$, and in particular $\SVc \B$ is $T_0$;
  \item if $\B$ is compact and non-empty, then $\SVc \B$ is pointed,
    and $\B$ is its least element (its zero);
  \item if $\B$ is pointed, then $\SVc \B$ is pointed and
    $\eta^\Smyth_{\B}$ is linear; if $\B$ is pointed and $T_0$, then
    $\eta^\Smyth_{\B}$ is a linear topological embedding of $\B$ in
    $\SVc \B$.
  \end{enumerate}
\end{lemma}
\begin{proof}
  1.  For all $U, V \in \Open \B$, $\Box^\cvx U \cap \Box^\cvx V$,
  while $\Box^\cvx \emptyset = \emptyset$.  Hence every finite
  intersection of subbasic open sets of the form $\Box^\cvx U$ is
  again of this form.

  2.  For all $Q_1, Q_2 \in \SVc \B$, if $Q_1 \supseteq Q_2$ then
  for every basic open neighborhood $\Box^\cvx U$ of $Q_1$, $Q_1$ is
  included in $U$, hence also $Q_2$, so $Q_2 \in \Box^\cvx U$.
  Conversely, if $Q_1$ is below $Q_2$ in the specialization
  preordering of $\SVc \B$, every open neighborhood of the form
  $\Box^\cvx U$ of $Q_1$ contains $Q_2$; in other words,
  for every open subset $U$ of $\B$, if $Q_1$ is included in $U$, then
  $Q_2 \subseteq U$.  Since $Q_1$ is saturated, it is equal to the
  intersection of its open neighborhoods $U$, hence $Q_2 \subseteq
  Q_1$.

  3.  $\B$ is convex and upwards-closed.  If $\B$ is compact and
  non-empty, then it is an element of $\SVc \B$.  By item~2, it is the
  least element of $\SVc \B$.

  4.  If $\B$ is pointed, then it is compact and non-empty, so
  $\SVc \B$ is pointed by item~3.  Then
  $\eta^\Smyth_{\B} (\bot) = \dc \bot = \B$, so $\eta^\Smyth_{\B}$ is
  strict.  Since it is affine by Lemma~\ref{lemma:Qcvx:eta}, it is
  linear by Proposition~\ref{prop:strict:aff}.  Finally, when $\B$ is
  also $T_0$, $\eta^\Smyth_{\B}$ is a topological embedding by
  Lemma~\ref{lemma:Qcvx:eta}.  \qed
\end{proof}

We will see that $\SVc \B$ is topological, not just semitopological.
For this, we need \emph{Wallace's theorem} \cite[Chapter~II,
Theorem~12]{Kelley:topology}: given $n$ topological spaces $X_1$,
\ldots, $X_n$, and compact subsets $K_i$ of $X_i$ for each
$i \in \{1, \cdots, n\}$, every open neighborhood $W$ of
$\prod_{i=1}^n K_i$ contains a product $\prod_{i=1}^n U_i$ where each
$U_i$ is an open neighborhood of $K_i$ in $X_i$.

\begin{proposition}
  \label{prop:QcvxB:topo}
  For every topological barycentric algebra $\B$, $\SVc \B$ is a
  topological barycentric algebra.
\end{proposition}
\begin{proof}
  We wish to show that the map
  $(Q_1, a, Q_2) \mapsto Q_1 +^\sharp_a Q_2$ is continuous from
  $\SVc \B \times [0, 1] \times \SVc \B$ to $\SVc \B$.  It suffices to
  show that the inverse image of any set $\Box^\cvx W$, with $W$ open
  in $\B$, under this function is open in
  $\SVc \B \times [0, 1] \times \SVc \B$.  Let us consider any element
  $(Q_1, a, Q_2)$ of that inverse image.  In other words,
  $Q_1 +^\sharp_a Q_2 \in \Box^\cvx W$.  Since $\B$ is topological,
  the set
  $O \eqdef \{(x', a', y') \in \B \times [0, 1] \times \B \mid x'
  +_{a'} y' \in W\}$ is open, and $Q_1 +^\sharp_a Q_2 \in \Box^\cvx W$
  can be rephrased as $Q_1 \times \{a\} \times Q_2 \subseteq O$.  By
  Wallace's theorem, there are open neighborhoods $U_1$ of $Q_1$, $I$
  of $a$ and $U_2$ of $Q_2$ such that
  $U_1 \times I \times U_2 \subseteq O$.  Hence
  $\Box^\cvx {U_1} \times I \times \Box^\cvx {U_2}$ is an open
  neighborhood of $(Q_1, a, Q_2)$ in
  $\SVc \B \times [0, 1] \times \SVc \B$.  It remains to see that it
  is included in $\Box^\cvx W$.  Every triple $(Q'_1, a', Q'_2)$ in
  $\Box^\cvx {U_1} \times I \times \Box^\cvx {U_2}$ is such that
  $Q'_1 \subseteq U_1$, $a' \in I$ and $Q'_2 \subseteq U_2$, so for
  all $x'_1 \in Q'_1$, $a' \in I$ and $x'_2 \in Q'_2$,
  $(x'_1, a', x'_2) \in U_1 \times I \times U_2 \subseteq O$, namely
  $x'_1 +_{a'} x'_2 \in W$.  This entails that
  $\{x'_1 +_{a'} x'_2 \mid x'_1 \in Q'_1, x'_2 \in Q'_2\}$ is included
  in $W$, and since $W$ is upwards-closed, that
  $Q'_1 +^\sharp_{a'} Q'_2$ is included in $W$, equivalently is an
  element of $\Box^\cvx W$.  \qed
\end{proof}
For a topological cone $\C$, $\SVc \C$ is a topological cone
\cite[Theorem~11]{Keimel:topcones2}.

We turn to linear separation.
\begin{lemma}
  \label{lemma:Lambda:bar:Qcvx}
  Let $\B$ be a topological barycentric algebra.  For every affine
  lower semicontinuous map $\Lambda \colon \B \to \creal$, the
  function $\min\Lambda \colon \SVc \B \to \creal$ defined by
  $\min\Lambda (Q) \eqdef \min_{x \in Q} \Lambda (x)$ is affine and
  lower semicontinuous.  If $\B$ is pointed and $\Lambda$ is linear,
  then $\min\Lambda$ is linear.
\end{lemma}
\begin{proof}
  The min is attained because $\Lambda$ is lower semicontinuous and
  $Q$ is compact and non-empty.  Hence for every $r \in \Rp$,
  $\min_{x \in Q} \Lambda (x) > r$ if and only if for every $x \in Q$,
  $\Lambda (x) > r$.  It follows that
  ${(\min\Lambda)}^{-1} (]r, \infty]) = \Box^\cvx {\Lambda^{-1} (]r,
    \infty])}$; so $\min\Lambda$ is continuous.

  For all $Q_1, Q_2 \in \Smythc \B$ and for every $a \in [0, 1]$,
  $\min\Lambda (Q_1 +^\sharp_a Q_2) = \min \{\Lambda (x_1 +_a x_2)
  \mid x_1 \in Q_1, x_2 \in Q_2\}$ (since taking the minimum of a
  monotonic map over the upward-closure of a non-empty set is the same
  as taking the minimum on that set)
  $= \min \{a \, \Lambda (x_1) + (1-a) \, \Lambda (x_2) \mid x_1 \in
  Q_1, x_2 \in Q_2\}$ (since $\Lambda$ is affine)
  $= a \, \min\Lambda (Q_1) + (1-a) \, \min\Lambda (Q_2)$.  Hence
  $\min\Lambda$ is affine.

  If $\B$ is pointed then $\B$ itself is the least element of
  $\SVc \B$.  Then $\min\Lambda (\B) = \Lambda (\bot)$, which is equal
  to $0$ if $\Lambda$ is linear, hence strict.  Therefore
  $\min\Lambda$ is strict and affine, hence linear by
  Proposition~\ref{prop:strict:aff}.  \qed
\end{proof}

\begin{proposition}
  \label{prop:Qcvx:linsep}
  For every topological barycentric algebra $\B$, if $\B$ is linearly
  separated then so is $\SVc \B$.
\end{proposition}
\begin{proof}
  Let $Q_1, Q_2 \in \SVc \B$ be such that $Q_1 \not\supseteq Q_2$.
  Therefore, there is a point $x \in Q_2$ that is not in $Q_1$.
  By Proposition~\ref{prop:compact:convex}, item~1, $Q_1$ is the
  intersection of its open neighborhoods of the form
  $\Lambda^{-1} (]1, \infty])$ with $\Lambda \colon \B \to \creal$
  affine and lower semicontinuous.  Hence one of them does not contain
  $x$.  With this $\Lambda$, the function $\min\Lambda$ is affine and
  lower semicontinuous by Lemma~\ref{lemma:Lambda:bar:Qcvx}.  Since
  $Q_1 \subseteq \Lambda^{-1} (]1, \infty])$, it follows that
  $\min\Lambda (Q_1) > 1$.  Since
  $x \not\in \Lambda^{-1} (]1, \infty])$ and since $x \in Q_2$,
  $\min\Lambda (Q_2) \leq \Lambda (x) \leq 1$.  Hence
  $\min\Lambda (Q_1) > \min\Lambda (Q_2)$.  \qed
\end{proof}

In what follows, we will need the following technical lemma.  A
\emph{monotone convergence space} is a $T_0$ space that is a dcpo in
its specialization ordering, and in which every open set is Scott-open
\cite[Definition 8.2.33]{JGL-topology}.  Every sober space is a
monotone convergence space.
\begin{lemma}
  \label{lemma:qcont}
  Given any locally finitary compact monotone convergence space $X$,
  the topology of $X$ is the Scott topology of its specialization
  ordering $\leq$, and $X$ is sober.
\end{lemma}
\begin{proof}
  We define $\leq^\sharp$ on the collection of non-empty finite
  subsets of $X$ by $A \leq^\sharp B$ if and only if
  $\upc A \supseteq \upc B$.  \emph{Rudin's Lemma}
  \cite[Proposition~5.2.25]{JGL-topology} states that in a poset $P$,
  for every directed family ${(A_i)}_{i \in I}$ of non-empty finite
  subsets of $P$ with respect to $\leq^\sharp$, there is a directed
  subset $D$ of $\bigcup_{i \in I} A_i$ such that $D$ intersects every
  $A_i$.  As a consequence, if $P$ is a dcpo, then for any such
  directed family, for every Scott-open subset $U$ of $P$ that
  contains $\bigcap_{i \in I} \upc A_i \subseteq U$, $U$ contains some
  $A_i$ already \cite[Proposition~5.8.28]{JGL-topology}.
  
  Every open subset is Scott-open, by definition.  Conversely, let $U$
  be a Scott-open subset of $X$.  We show that $U$ is an open
  neighborhood (in the original topology) of any of its points $x$.
  Let $D$ be the collection of all the non-empty finite sets $A$ such
  that $x \in \interior {\upc A}$.  By definition of locally finitary
  compactness: $(*)$ for every open neighborhood $V$ (in the original
  topology) of $x$, there is an $A \in D$ such that $A \subseteq V$.
  (Note that $A$ cannot be empty, since $x \in \upc A$.)  Taking
  $V \eqdef X$ in $(*)$ shows that $D$ is non-empty.  Given
  $A, B \in D$, taking
  $V \eqdef \interior {\upc A} \cap \interior {\upc B}$ in $(*)$ shows
  that there is a $C \in D$ such that $C \subseteq \interior {\upc A}$
  and $C \subseteq \interior {\upc B}$, in particular
  $A, B \leq^\sharp C$.  Hence $D$ is directed in $\Pfin^* (X)$ with
  respect to $\leq^\sharp$.

  We claim that $\bigcap_{A \in D} \upc A \subseteq \upc x$.  It
  suffices to note that every open neighborhood $V$ of $x$ contains
  $\bigcap_{A \in D} \upc A$: $(*)$ even tells us that $V$ contains
  $\upc A$ for some $A \in D$.  Since $U$ is upwards-closed, it
  follows that $\bigcap_{A \in D} \upc A \subseteq U$.  Now $U$ is
  Scott-closed, and $X$ is a dcpo in its specialization ordering, so,
  using the consequence of Rudin's Lemma stated at the beginning of
  this proof, some $A \in D$ is included in $U$.  This says that
  $x \in \interior {\upc A} \subseteq \upc A \subseteq U$, so $U$ is
  an open neighborhood of $x$ in the original topology.

  We now know that the topology of $X$ is the Scott topology.  By
  \cite[Exercise 5.2.31]{JGL-topology}, the dcpos that are locally
  finitary compact are exactly the quasi-continuous dcpos, and every
  quasi-continuous dcpo is sober in its Scott topology \cite[Exercise
  8.2.15]{JGL-topology}.  \qed
\end{proof}

\begin{proposition}
  \label{prop:QcvxC:dcone}
  For every convenient barycentric algebra $\B$, $\Smythc \B$, with
  the Scott topology of reversed inclusion $\supseteq$, is a linearly
  separated continuous d-barycentric algebra; in particular, it is
  convenient.  Directed suprema are filtered intersections.  Moreover,
  $\Smythc \B = \SVc \B$, namely the topology of $\SVc \B$ coincides
  with the Scott topology.
\end{proposition}
\begin{proof}
  Given any directed family ${(Q_i)}_{i \in I}$ in $\Smythc \B$,
  namely any filtered family of non-empty convex compact saturated
  sets, its intersection $Q \eqdef \bigcap_{i \in I} Q_i$ is compact
  saturated and non-empty, because $\B$ is sober hence well-filtered
  \cite[Proposition~8.3.5]{JGL-topology}.  A \emph{well-filtered}
  space is one in which if a filtered intersection
  $\bigcap_{i \in I} Q_i$ of compact saturated subsets $Q_i$ is
  included in an open set $U$, then some $Q_i$ is already included in
  $U$.  In a well-filtered space, every such filtered intersection is
  compact saturated \cite[Proposition 8.3.6]{JGL-topology}.  If $Q$
  were empty, then it would be covered by the empty family, hence some
  $Q_i$ would, too, by well-filteredness; but that is impossible.
  Finally, any intersection of convex sets is convex, so $Q$ is in
  $\Smythc \B$.  It is then clear that $Q$ is the least upper bound of
  ${(Q_i)}_{i \in I}$.  Since least upper bound always exists,
  $\Smythc \B$ is a dcpo, and directed suprema are filtered
  intersections.

  Every basic open subset $\Box^\cvx U$ of $\SVc \B$ is then
  Scott-open, since $\bigcap_{i \in Q_i} \subseteq U$ implies
  $Q_i \subseteq U$ for some $i \in I$, by well-filteredness.  This
  shows that $\SVc \B$ is a monotone convergence space.

  Next, we claim that $\SVc \B$ is a c-space.  Let us consider an
  element $Q \in \Smythc \B$, and an open neighborhood $\mathcal U$ of
  $Q$.  The sets $\Box^\cvx U$, $U \in \Open \B$, form a base by
  Lemma~\ref{lemma:Qcvx:spec}, item~1.  Hence, $\mathcal U$ contains a
  basic open subset $\Box^\cvx U$ that contains $Q$; namely,
  $Q \subseteq U$.  By Proposition~\ref{prop:locconvcomp:base}, $Q$
  has a base of convex compact saturated neighborhoods, so one of
  them, call it $Q'$, is included in $U$.  Since
  $Q \in \Box^\cvx {(\interior {Q'})}$, and since
  $\Box^\cvx {(\interior {Q'})}$ is included in the upward closure
  $\upc_{\SVc \B} Q' \eqdef \{Q'' \in \SVc \B \mid Q'' \subseteq Q'
  \}$ of $Q'$ in $\SVc \B$, $Q$ is in the interior of
  $\upc_{\SVc \B} Q'$.  Moreover,
  $Q' \in \Box^\cvx U \subseteq \mathcal U$.  Hence $\SVc \B$ is a
  c-space.

  Every c-space that is a monotone convergence space is sober, and its
  topology coincides with the Scott topology, by
  Lemma~\ref{lemma:qcont}; hence $\Smythc \B = \SVc \B$.  Since $\B$
  is linearly separated, so is $\SVc \B$ by
  Proposition~\ref{prop:Qcvx:linsep}, hence so is $\Smythc \B$, with
  its Scott topology.  Finally, a continuous d-barycentric algebra
  that is linearly separated is convenient, by
  Example~\ref{exa:dbary:convenient}.  \qed
\end{proof}

\section{Barycenters, part 6}
\label{sec:barycenters-part-6}

\subsection{Barycenters, after Choquet}
\label{sec:baryc-after-choq}

There is a general notion of barycenter of a continuous valuation,
imitating a definition due to Gustave Choquet \cite[Chapter~6,
26.2]{Choquet:analysis:2}, and due to Ben Cohen, Mart\'in Escard\'o
and Klaus Keimel \cite{CEK:prob:scomp} and to the author and Xiaodong
Jia \cite{GLJ:bary} on semitopological cones $\C$: a \emph{barycenter}
of $\nu \in \Val \C$ is any point $x_0 \in C$ such that, for every
lower semicontinuous linear map $\Lambda \colon \C \to \creal$,
$\Lambda (x_0) = \int \Lambda \,d\nu$.

We adapt the definition to barycentric algebras and pointed
barycentric algebras.
\begin{definition}[Barycenter]
  \label{defn:bary:choquet:bary:alg}
  Let $\B$ be a semitopological barycentric algebra, and
  $\nu \in \Val_1 \B$.  A \emph{barycenter} of $\nu$ is any point
  $x_0$ of $\B$ such that for every lower semicontinuous affine map
  $\Lambda \colon \B \to \creal$,
  $\Lambda (x_0) = \int \Lambda \,d\nu$.

  If $\B$ is pointed, and $\nu \in \Val_{\leq 1} \B$, a
  \emph{barycenter} of $\nu$ is any point $x_0$ of $\B$ such that for
  every lower semicontinuous linear map
  $\Lambda \colon \B \to \creal$,
  $\Lambda (x_0) = \int \Lambda \,d\nu$.
\end{definition}
This makes quite a few competing definitions of barycenters.  We argue
that any two of them coincide on the intersections of their domains of
definition.  We start with the previous definition on cones
\cite{CEK:prob:scomp,GLJ:bary} and
Definition~\ref{defn:bary:choquet:bary:alg}.
\begin{itemize}
\item A semitopological cone $\C$ is also a pointed semitopological
  barycentric algebra, hence the definition of
  \cite{CEK:prob:scomp,GLJ:bary} and the second part of
  Definition~\ref{defn:bary:choquet:bary:alg} run a risk of being in
  conflict.  They are not because linear means the same on $\C$,
  whether we consider it as a cone or as a pointed barycentric
  algebra, see Remark~\ref{rem:bary:alg:pointed:linear}.
\item A pointed barycentric algebra $\B$ is also a barycentric
  algebra, so the two definitions of
  Definition~\ref{defn:bary:choquet:bary:alg} appear to be in
  conflict, at least for probability valuations $\nu \in \Val_1 \B$.
  If $x_0$ is a barycenter in the sense of the first part of
  Definition~\ref{defn:bary:choquet:bary:alg}, then it is one in the
  sense of the second part, because every linear map is affine.  If
  $x_0$ is a barycenter in the sense of the second part, then let
  $\Lambda \colon \B \to \creal$ be any lower semicontinuous affine
  map.  If $\Lambda (\bot) = \infty$, then $\Lambda$ is the constant
  map with value $\infty$ by monotonicity, hence
  $\Lambda (x_0) = \infty$, while
  $\int_{x \in \B} \Lambda (x) \,d\nu = \infty . \nu (\B) = \infty$.
  If $\Lambda (\bot) < \infty$, then the function $\Lambda'$ defined
  by $\Lambda' (x) \eqdef \Lambda x) - \Lambda (\bot)$ is affine and
  strict, hence linear by Proposition~\ref{prop:strict:aff}.
  Therefore $\Lambda' (x_0) = \int_{x \in \B} \Lambda' (x) \,d\nu$, so
  $\Lambda (x_0) = \Lambda' (x_0) + \Lambda (\bot) = \int_{x \in \B}
  \Lambda' (x) \,d\nu + \int_{x \in \B} \Lambda (\bot)\,d\nu = \int_{x
    \in \B} \Lambda (x) \,d\nu$.
\end{itemize}
We also have to compare the Choquet-like definitions of barycenters of
Definition~\ref{defn:bary:choquet:bary:alg} with our previous notions
of barycenters.  We will return to this shortly.  We first make a
trivial but important observation.

\begin{lemma}
  \label{lemma:bary:unique}
  On a linearly separated $T_0$ semitopological barycentric algebra,
  every probability valuation has at most one barycenter.  Similarly
  for subprobability valuations on linearly separated $T_0$ pointed
  semitopological barycentric algebras, for continuous valuations on
  linearly separated $T_0$ semitopological cones.
\end{lemma}
\begin{proof}
  Let $x_0$ and $x_1$ be two distinct barycenters of the same
  probability valuation $\nu$ (resp.\ subprobability valuation, resp.\
  continuous valuation).  The $T_0$ property implies that
  $x_0 \not\leq x_1$ or $x_1 \not\leq x_0$.  By symmetry, let us
  assume $x_0 \not\leq x_1$.  By linear separation, there is an affine
  (resp.\ linear, by Lemma~\ref{lemma:convexT0}, item~1) lower
  semicontinuous map $\Lambda$ with values in $\creal$ such that
  $\Lambda (x_0) > \Lambda (x_1)$.  But by definition of barycenters,
  $\Lambda (x_0)$ are $\Lambda (x_1)$ are both equal to the integral
  of $\Lambda$ with respect to $\nu$, leading to a contradiction.
  \qed
\end{proof}

Things can go pretty badly without linear separation, as the following
example shows.
\begin{example}
  \label{exa:B-:bary:notunique}
  Every point of the locally convex $T_0$ topological barycentric
  algebra $\B^-$ (see Example~\ref{exa:bary:alg:notembed} and
  Example~\ref{exa:B-:loc:convex}) is a barycenter of every
  probability valuation on $\B^-$.  Indeed, by
  Example~\ref{exa:bary:alg:notembed}, item~2, every affine lower
  semicontinuous map $\Lambda$, which must be semi-concave by
  Fact~\ref{fact:qaff}, is constant, say with value $c$.  For every
  point $x_0 \in \B$, $\Lambda (x_0) = c$, and for every probability
  valuation $\nu$ on $\B$,
  $\int_{x \in \B} \Lambda (x) \,d\nu = \int_{x \in \B} c \, d\nu = c
  \nu (\B) = c$.
\end{example}

We recall the barycenter map $\beta$ of
Theorem~\ref{thm:locconv:heckmann:beta}, which maps simple (possibly
subnormalized or normalized) valuations
$\sum_{i=1}^n a_i \delta_{x_i}$ to their barycenters
$\sum_{i=1}^n a_i \cdot x_i$, as defined in Definition and
Proposition~\ref{prop:bary:1} or Definition and
Proposition~\ref{prop:bary:leq1} in the case of (possibly pointed)
barycentric algebras.  We also recall the map, also named $\beta$,
which sends (possibly subnormalized or normalized) point-continuous
valuations to what we claimed to be an extension of the notion of
barycenter in Theorem~\ref{thm:locconv:heckmann:beta:pcont}.

\begin{lemma}
  \label{lemma:bary:simple}
  \begin{enumerate}
  \item For every simple normalized valuation $\nu$ on a
    semitopological barycentric algebra $\B$, $\beta (\nu)$ as defined
    in Theorem~\ref{thm:locconv:heckmann:beta} is a barycenter of
    $\nu$.
  \item For every simple subnormalized valuation $\nu$ on a pointed
    semitopological barycentric algebra $\B$, $\beta (\nu)$ as defined
    in Theorem~\ref{thm:locconv:heckmann:beta} is a barycenter of
    $\nu$.
  \item For every simple valuation $\nu$ on a semitopological cone
    $\C$, $\beta (\nu)$ as defined in
    Theorem~\ref{thm:locconv:heckmann:beta} is a barycenter of $\nu$.
    \index{point-continuous?valuation, barycenter}%
  \item For every weakly locally convex sober pointed topological
    barycentric algebra $\B$, for every
    $\nu \in \Val_{\leq 1, \pw} \B$, $\beta (\nu)$ as defined in
    Theorem~\ref{thm:locconv:heckmann:beta:pcont} is a barycenter of
    $\nu$.
  \item For every weakly locally convex, sober topological cone $\C$,
    for every $\nu \in \Val_\pw \C$, $\beta (\nu)$ as defined in
    Theorem~\ref{thm:locconv:heckmann:beta:pcont} is a
    barycenter of $\nu$.
  \end{enumerate}
\end{lemma}
In all cases, this will be the unique barycenter under an additional
assumption of linear separability and $T_0$-ness, by
Lemma~\ref{lemma:bary:unique}.

\begin{proof}
  $1, 2, 3$.  Let $\nu \eqdef \sum_{i=1}^n a_i \delta_{x_i}$ be a
  simple valuation, which we assume normalized in case~1,
  subnormalized in case~2, and arbitrary in case~3.  Let $\Lambda$ be
  an affine (in case~1) or linear (in cases~2 and~3) lower
  semicontinuous map with values in $\creal$.  Then
  $\int_{x \in \C} \Lambda (x) \,d\nu = \sum_{i=1}^n a_i \Lambda
  (x_i)$ since integration is linear, and that is equal to
  $\Lambda (\sum_{i=1}^n a_i \cdot x_i)$, by Definition and
  Proposition~\ref{prop:bary:1}, item~4 in case~1, by Definition and
  Proposition~\ref{prop:bary:leq1}, item~2 in case~2, and by
  definition of linearity in the case of cones (case~3).  Therefore
  $\int_{x \in \C} \Lambda (x) \,d\nu = \Lambda (\beta (\nu))$, so
  $\beta (\nu)$ is a barycenter of $\nu$.
  
  $4, 5$.  Let $\Lambda$ be a lower semicontinuous linear map from
  $\B$ (in case~4) or $\C$ (in case~5) to $\creal$.  By the first part
  of the lemma, the functions that map every $\nu$ in
  $\Val_{\leq 1} \B$ (resp.\ $\Val \C$) to
  $\int_{x \in \B \text{ (resp.\ $\C$)}} \Lambda (x) \,d\nu$ and
  $\Lambda \circ \beta$ coincide on simple valuations.  They are both
  lower semicontinuous, by definition of the weak topology for the
  first one, and because $\beta$ is continuous
  (Theorem~\ref{thm:locconv:heckmann:beta:pcont}) for the second one.
  By the universal property of sobrifications, and recalling that
  $\Val_{\rast, \pw} X$ is a sobrification of $\Val_{\rast, \fin} X$
  for every space $X$ and where $\rast$ is nothing or ``$\leq 1$''
  (Theorem~\ref{thm:VpX:sobrif}), they must coincide.
\end{proof}
In items~4 and~5 of Lemma~\ref{lemma:bary:simple}, sobriety is
required: we will see in Example~\ref{exa:step:nobary} an example of
point-continuous valuation on a weakly locally convex (in fact, even a
locally linear) cone that has no barycenter.

We recall the map $\etass_\B \colon \B \to \B^{**}$, defined by
$\etass_\B (x) (\Lambda) \eqdef \Lambda (x)$.  We have the following
restatement of Definition~\ref{defn:bary:choquet:bary:alg}.
\begin{fact}
  \label{fact:bary:alt}
  For every semitopological cone $\C$ (resp.\ pointed semitopological
  algebra $\B$, resp.\ semitopological algebra $\B$), for every
  $\nu \in \Val \C$ (resp.\ $\Val_{\leq 1} \B$, resp.\ $\Val_1 \B$),
  let
  $\beta^* (\nu) (\Lambda) \eqdef \int_{x \in \B} \Lambda (x) \,d\nu$
  for every $\Lambda \in \C^*$ (resp.\ $\B^*$, resp.\ $\B^\astar$).

  For every $\nu \in \Val \C$ (resp.\ $\Val_{\leq 1} \B$, resp.\
  $\Val_1 \B$), a point $x_0$ is a barycenter of $\nu$ if and only if
  $\etass_\B (x_0) = \beta^* (\nu)$.
\end{fact}

\begin{lemma}
  \label{lemma:bary:exist:cont}
  Let $\C$ be a semitopological cone (resp.\ a pointed semitopological
  algebra $\B$, resp.\ a semitopological algebra $\B$), and
  $\nu \in \Val \C$ (resp.\ $\Val_{\leq 1} \B$, resp.\ $\Val_1 \B$).
  If $\nu$ has a barycenter, then $\beta^* (\nu)$ is lower
  semicontinuous from $\C^*$ (resp.\ $\B^*$, resp.\ $\B^\astar$) to
  $\creal$.
\end{lemma}
\begin{proof}
  We deal with cones.  The argument is the same for (pointed)
  barycentric algebras.  For every point $x_0$ of $\C$,
  $\etass_\C (x_0)$ is an element of $\C^*$, hence is in particular a
  lower semicontinuous map.  If $x_0$ is a barycenter of $\nu$, then
  $\etass_\C (x_0) = \beta^* (\nu)$ by Fact~\ref{fact:bary:alt}, so
  $\beta^* (\nu)$ must be lower semicontinuous.  \qed
\end{proof}
Hence, if $\beta^* (\nu)$ is not lower semicontinuous, then $\nu$
cannot have any barycenter.

Conversely, $\beta^* (\nu)$ is always a linear map from $\C^*$ (resp.\
$\B^*$, resp.\ $\B^\astar$) to $\creal$, hence if $\beta^* (\nu)$ is
lower semicontinuous, then it is an element of $\C^{**}$ (resp.\
$\B^{**}$, resp.\ $\B^{\astar*}$).  Then, by Fact~\ref{fact:bary:alt},
$\nu$ has a barycenter if and only if that element is in the image of
$\etass_\C$ (resp.\ $\etass_{\B}$, resp.\ $\etass_{\B}$).


Another easy observation is that, since
Definition~\ref{defn:bary:choquet:bary:alg} only mentions
\emph{affine} (or linear) lower semicontinuous maps $\Lambda$, the
notion of barycenter only depends on the weakScott topology on $\B$.
Indeed, $\B$ and $\B_{\wS}$ have exactly the same linear lower
semicontinuous maps to $\creal$, by Remark~\ref{rem:weakScott:facts},
item~(ii).
\begin{fact}
  \label{fact:bary:weakscott}
  Let $\B$ be a semitopological barycentric algebra (resp.\ a pointed
  semitopological barycentric algebra, resp.\ a semitopological cone).
  For every probability (resp.\ subprobability, resp.\ continuous)
  valuation $\nu$ on $\B$, the barycenters of $\nu$ on $\B$ are
  exactly the barycenters of the restriction of $\nu$ to
  $\Open (\B_{\wS})$ on $\B_{\wS}$.  \qed
\end{fact}  

Given a continuous map $f \colon X \to Y$, the \emph{image valuation}
$f [\nu]$ of $\nu \in \Val X$ is the continuous valuation
$V \in \Open Y \mapsto \nu (f^{-1} (V))$.  For every $h \in \Lform Y$,
we have the \emph{change of variable} formula
$\int h \,df[\nu] = \int (h \circ f) \,d\nu$, see for example
\cite[Lemma~2.10$(vi)$]{GLJ:Valg}; this is easy to prove using the
Choquet formula for the integral.
\begin{lemma}
  \label{lemma:bary:map}
  Let $f \colon \C \to \Cb$ be a linear continuous map between two
  semitopological cones.  For every $\nu \in \Val \C$ with barycenter
  $x_0$ in $\C$, $f (x_0)$ is a barycenter of $f [\nu]$ in $\Cb$.
  
  Similarly if $f \colon \B \to \Alg$ is a linear continuous map
  between pointed semitopological barycentric algebras and
  $\nu \in \Val_{\leq 1} \B$, or if $f \colon \B \to \Alg$ is an
  affine continuous map between not necessarily pointed
  semitopological barycentric algebras and $\nu \in \Val_1 \B$.
\end{lemma}
\begin{proof}
  For every linear lower semicontinuous map
  $\Lambda \colon \Cb \to \creal$, $\Lambda \circ f$ is linear and
  lower semicontinuous.  Since $x_0$ is a barycenter of $\nu$,
  $\Lambda (f (x_0)) = \int_{x \in \C} \Lambda (f (x)) \,d\nu$, and
  that is equal to $\int_{y \in \Cb} \Lambda (y) \,df[\nu]$ by the
  change of variable formula.  The argument is the same on (possibly
  pointed) semitopological barycentric algebras, only taking $\Lambda$
  to be affine for not necessarily pointed barycentric algebras.  \qed
\end{proof}

\begin{corollary}
  \label{corl:bary:retract}
  Let $r \colon \C \to \Cb$ be a linear retraction between
  semitopological cones (resp.\ a linear retraction between pointed
  semitopological barycentric algebras, resp.\ an affine retraction
  between semitopological barycentric algebras), with associated section $s$.
  \begin{enumerate}
  \item For every continuous (resp.\ subprobability, resp.\
    probability) valuation $\nu$ on $\Cb$, if $x_0$ is a
    barycenter of $s [\nu]$ in $\C$, then $r (x_0)$ is a barycenter of
    $\nu$ in $\Cb$.
  \item Letting $\rast$ be nothing, resp.\ ``$\leq 1$'', resp.\
    ``$1$'', for every function $\beta \colon \Val_\rast \C \to \C$
    that maps every element of $\Val_\rast \C$ to one of its
    barycenters,
    $r \circ \beta \circ \Val s \colon \Val_\rast \Cb \to \Cb$ maps
    every element of $\Val_\rast \Cb$ to one of its barycenters.
  \end{enumerate}
\end{corollary}
\begin{proof}
  1.  By Lemma~\ref{lemma:bary:map}, $r (x_0)$ is a barycenter of $r [s
  [\nu]] = \nu$.

  2.  For every $\nu \in \Val_\rast \Cb$, $\beta (\Val s (\nu)) =
  \beta (s [\nu])$ is a barycenter of $s [\nu]$ by assumption.
  By item~1, $r (\beta (\Val s (\nu)))$ is a barycenter of $\nu$.  \qed
\end{proof}

\begin{example}
  \label{exa:step:nobary}
  Let $X$ be a continuous dcpo in its Scott topology, and let us
  consider the subcone $\C$ of $\Lform X$ consisting of the step
  functions from a topological space $X$ to $\creal$, namely those
  that take only finitely many values.  Then, discrete valuations
  $\sum_{i \in I} a_i \delta_{g_i}$ where each $g_i$ is a step
  function and $\sum_{i \in I} a_i g_i$ takes infinitely many values
  have no barycenter in $\C$.  In particular, there are topological,
  locally linear cones on which barycenters may fail to exist.
\end{example}
\begin{proof}
  Let $\nu \in \Val \C$.  If $f_0$ is a barycenter of $\nu$, then by
  Lemma~\ref{lemma:bary:map}, $\etass_\C (f_0)$ is a barycenter of
  $\etass_\C [\nu]$.  Now we can equate $\C^*$ with $\Val X$, by
  Example~\ref{exa:V=C*}, and then $\C^{**}$ with $\Lform X$, with the
  weakScott topology, by Theorem~\ref{thm:schsimp}.  By
  Lemma~\ref{lemma:weakScott}, the weakScott topology coincides with
  the Scott topology.  Since a continuous dcpo is locally compact,
  hence core-compact, in the Scott topology, we can apply
  Proposition~4.15 of \cite{GLJ:Valg}, which says that the function
  $x \mapsto \int_{f \in \Lform X} f (x) \,d\nu$ is the unique
  barycenter of $\nu$ on $\Lform X$.  Up to the isomorphism between
  $\C$ and $\C^{**}$, $\etass_\C (f_0)$ must therefore be equal to the
  map $x \in X \mapsto \int_{f \in \Lform X} f (x) \,d\nu$.  That
  isomorphism is simply $\etass_\C$, hence $f_0$ must be that map.

  When $\nu = \sum_{i \in I} a_i \delta_{g_i}$, $f_0$ must then be the
  map $x \in X \mapsto \sum_{i \in I} a_i g_i (x)$.  If $f_0$ takes
  infinitely many values, then $f_0$ cannot be in $\C$, so $\nu$ has
  no barycenter.

  Finally, $\C$ is locally linear, because, by
  Lemma~\ref{lemma:weakScott}, the topology of $\C$ is also the
  topology of pointwise convergence, which is always locally linear
  and topological.  \qed
\end{proof}

\begin{example}
  \label{exa:Vp:nobary}
  Let $X$ be a topological space, $\rast$ be nothing, ``$\leq 1$'' or
  ``$1$'', and $\C \eqdef \Val_{\rast, \pw} X$.  $\C$ is a cone if
  $\rast$ is nothing, otherwise a (possibly pointed) barycentric
  algebra.  $\C$ is locally affine (or locally linear), linearly
  separated, topological and sober.  Despite this, continuous
  valuations in general have no barycenter on $\C$.  Indeed, given any
  any continuous, non-point-continuous valuation $\mu$ on $X$
  (subnormalized if $\rast$ is ``$\leq 1$'', normalized if $\rast$ is
  ``$1$), that $i [\mu] \in \Val_{\rast} \C$ has no barycenter, where
  $i \colon X \mapsto \C$ is the map $x \mapsto \delta_x$; an example
  of such a $\mu$ is given in \cite{GLJ:minval}.
\end{example}
\begin{proof}
  We first show that for every $\varpi \in \Val_\rast \C$, if $\varpi$
  has a barycenter $\nu_0$ on $\C$, then for every open subset $U$ of
  $X$, $\nu_0 (U) = \int_{\nu \in \C} \nu (U) \,d\varpi$.  By
  definition, we must have
  $\Lambda (\nu_0) = \int_{\nu \in \C} \Lambda (\nu) \,d\varpi$ for
  every continuous linear (affine if $\rast$ is ``$1$'') map $\Lambda$
  from $\C$ to ${(\creal)}_\sigma$.  For each given open subset $U$ of
  $X$, let us define $\Lambda \colon \Val_{\rast, \pw} X \to \creal$
  by $\Lambda (\nu) \eqdef \nu (U)$.  This is linear if $\rast$ is
  nothing or ``$\leq 1$'', and affine if $\rast$ is ``$1$''.  This is
  continuous since $\Lambda^{-1} (]r, \infty]) = [U > r]$.  The result
  follows.

  The map $i$ is continuous, since $i^{-1} ([U > r])$ is equal to $U$
  if $0\leq r <1$, and to $\emptyset$ otherwise.  We have just argued
  that if $i [\mu]$ has a barycenter $\nu_0$ on $\C$, then for every
  open subset $U$ of $X$,
  $\nu_0 (U) = \int_{\nu \in \C} \nu (U) \,di[\mu] = \int_{x \in X} i
  (x) (U) \,d\mu$, by the change of variable formula.  Now
  $i (x) (U) = \delta_x (U) = \chi_U (x)$, so $\nu_0 (U) = \mu (U)$.
  Since $U$ is arbitrary, $\nu_0 = \mu$.  However a barycenter on $\C$
  must be in $\C$, so $\nu_0$ must be in $\C = \Val_{\rast, \pw} X$,
  but $\mu$ is not point-continuous, and this is impossible.

  We conclude by recalling that $\Val_\pw X$ is sober by
  \cite[Proposition~5.1]{heckmann96}.  $\Val_{\leq 1, \pw} X$ and
  $\Val_{1, \pw} X$ are sober, too, because they are Skula-closed in
  $\Val_\pw X$.  When $\rast$ is nothing, $\Val_{\rast, \pw} X$ is a
  topological cone; this was mentioned in
  Example~\ref{exa:VX:top:cone}.  When $\rast$ is ``$\leq 1$'' or
  ``$1$'', it is a topological barycentric algebra by
  Example~\ref{exa:V1X:top:bary:alg}.  It was also mentioned that
  $\Val_{\rast, \pw} X$ is locally convex in
  Example~\ref{exa:VX:loc:convex}, in fact locally linear (when
  $\rast$ is nothing, see Example~\ref{exa:VX:loclin}, or when $\rast$
  is ``$\leq 1$'', see Example~\ref{exa:Vleq1X:loclin}) or locally
  affine (when $\rast$ is ``$1$'', see Example~\ref{exa:V1X:loclin}).
  It is linearly separated, as mentioned in
  Example~\ref{exa:VX:convexT0}.  \qed
\end{proof}



\subsection{Barycenters and convex sets}
\label{sec:baryc-conv-sets}

We will see that barycenters of probability valuations supported on
various kinds of convex sets $A$ are in $A$.  We first need the
following easy remarks.  A continuous valuation $\nu$ on a space $X$
is \emph{supported} on a subset $A$ of $X$ if and only if for all
$U, V \in \Open X$ such that $U \cap A = V \cap A$,
$\nu (U) = \nu (V)$.
\begin{fact}
  \label{fact:f[nu]:supp}
  Let $f \colon X \to Y$ be a continuous map, and $A$ be a subset of
  $X$.  For every continuous valuation $\nu$ on $X$, if $\nu$ is
  supported on $A$ then $f [\nu]$ is supported on $f [A]$, on
  $cl (f [A])$, and on $\upc f [A]$.
\end{fact}
Indeed, if $U \cap f [A] = V \cap f [A]$, then
$f^{-1} (U) \cap A = f^{-1} (V) \cap A$, so
$f [\nu] (U) = \nu (f^{-1} (U)) = \nu (f^{-1} (V)) = f [\nu] (V)$.
Then, any continuous valuation supported on a set is also supported on
any larger set, such as $cl (f [A])$ or $\upc f [A]$.
  
\begin{fact}
  \label{fact:bary:img}
  Let $\B$ be a (resp.\ pointed) semitopological barycentric algebra,
  $\C$ be a semitopological cone, and $f \colon \B \to \C$ be a
  continuous map.  For every $\nu \in \Val_1 \B$ (resp.\
  $\nu \in \Val_{\leq 1} \B$), for every point $x \in \B$, if $x$ is a
  barycenter of $\nu$ then $f (x)$ is a barycenter of $f [\nu]$.
\end{fact}
Indeed, for every linear lower semicontinuous map we have
$\Lambda \colon \C \to \creal$,
$\int_{y \in \C} \Lambda (y) \,df [\nu] = \int_{x \in \B} \Lambda (f
(x)) \,d\nu$ by the change of variable formula, and the latter is
equal to $\Lambda (f (x_0))$ since $\Lambda \circ f$ is affine (resp.\
linear) and lower semicontinuous, and by definition of barycenters.

\begin{fact}
  \label{fact:supp:f<=g}
  Let $\nu$ be a continuous valuation on a space $X$, supported on a
  subset $A$.  For any two maps $f, g \in \Lform X$ such that $f (x)
  \leq g (x)$ for every $x \in A$, $\int f \,d\nu \leq \int g \,d\nu$.
\end{fact}
Indeed, we first observe that for all $U, V \in \Open X$ such that
$U \cap A \subseteq V \cap A$, $\nu (U) \leq \nu (V)$: if
$U \cap A \subseteq V \cap A$, then
$U \cap A = (U \cap A) \cap (V \cap A) = (U \cap V) \cap A$, so
$\nu (U) = \nu (U \cap V) \leq \nu (V)$.  Using the Choquet formula,
$\int f \,d\nu = \int_0^\infty \nu (f^{-1} (]t, \infty])) \,dt$.
Since $f \leq g$ on $A$,
$A \cap f^{-1} (]t, \infty]) \subseteq A \cap g^{-1} (]t, \infty])$
for every $t \in \Rp$.  Therefore
$\int f \,d\nu \leq \int_0^\infty \nu (g^{-1} (]t, \infty])) \,dt =
\int g \,d\nu$.

\begin{definition}[Lc-embeddable]
  \label{defn:lc:embed}
  An \emph{lc-embeddable barycentric algebra} is a semitopological
  barycentric algebra that embeds through an affine topological
  embedding into a locally convex semitopological cone.
\end{definition}
A \emph{pointed} lc-embeddable barycentric algebra is an lc-embeddable
barycentric algebra that is also pointed.  We do not require the
embedding to be linear in this case, just affine.

\begin{remark}
  \label{rem:lc:embed:loclin}
  By Definition and Lemma~\ref{deflem:B'*}, every locally affine $T_0$
  semitopological barycentric algebra is lc-embeddable.
\end{remark}

\begin{lemma}
  \label{lemma:bary:closed}
  \begin{enumerate}
  \item Given any non-empty closed convex subset $C$ of a locally
    convex semitopological cone $\C$, the barycenters of all
    subprobability valuations supported on $C$ are in $C$.
  \item The same result holds, more generally, if we replace $\C$ by a
    pointed lc-embeddable barycentric algebra.
  \item Given any closed convex subset $C$ of an lc-embeddable
    barycentric algebra (not necessarily pointed) $\B$, the
    barycenters of all probability valuations supported on $C$ are in
    $C$.
  \end{enumerate}
\end{lemma}
\begin{proof}
  1.  Let $\nu$ be a subprobability valuation supported on $C$, and
  $x_0$ be a barycenter of $\nu$.  We assume that $x_0 \not\in C$, and
  we aim for a contradiction.  By
  Proposition~\ref{prop:loc:conv:halfspace}, item~2, there is an open
  half-space $H$ containing $x_0$ and disjoint from $C$.  The upper
  Minkowski functional $\Lambda \eqdef M^H$ is linear and lower
  semicontinuous.  By definition of barycenters,
  $\int_{x \in \C} \Lambda (x) \,d\nu = \Lambda (x_0)$.  However,
  $\Lambda (x) \leq 1$ for every $x \in C$, and $\nu$ is supported on
  $C$.  By Fact~\ref{fact:supp:f<=g},
  $\int_{x \in \C} \Lambda (x) \,d\nu \leq \int_{x \in \C} 1 \,d\nu
  \leq 1$ (the last inequality being because $\nu$ is a subprobability
  valuation), and that contradicts $\Lambda (x_0) > 1$.

  2.  Let $C$ be a non-empty closed convex subset of a pointed
  semitopological barycentric algebra $\B$ with an affine topological
  embedding $i$ into a locally convex semitopological cone $\C$
  instead.  Then $i [C]$ is convex, and $cl (i [C])$ is closed and
  convex by Definition and Proposition~\ref{defprop:clconv}.  As
  above, we let $x_0$ be a barycenter of a subprobability valuation
  $\nu$ supported on $C$, and we assume that $x_0 \not\in C$.  (We
  need $\B$ to be pointed so that this barycenter makes sense.)

  We note that $i [\nu]$ is supported on $cl (i [C])$ by
  Fact~\ref{fact:f[nu]:supp}.  Also, $i (x_0)$ is a barycenter of
  $i [\nu]$, by Fact~\ref{fact:bary:img}.

  By item~1, $i (x_0)$ must be in $cl (i [C])$.  However, since $i$ is
  a topological embedding, there is an open subset $V$ of $\C$ such
  that $\B \diff C = i^{-1} (V)$.  Hence $x_0 \in i^{-1} (V)$ since
  $x_0 \not\in C$, in other words $i (x_0) \in V$.  Since
  $i (x_0) \in cl (i [C])$, $V$ intersects $cl (i [C])$, and therefore
  also $i [C]$; hence $i^{-1} (V)$ intersects $C$, which is impossible
  since $i^{-1} (V) = \B \diff C$.

  3.  The argument is the same as in item~2 when $C$ is non-empty.
  When $C$ is empty, the statement is vacuously true: if $\nu$ is
  supported on the empty set, then $\nu$ must be the zero valuation,
  and cannot be a probability valuation.  \qed
\end{proof}

In non-Hausdorff spaces, compact saturated subsets may fail to be
closed, hence we need a separate result for compact saturated convex
subsets.
\begin{lemma}
  \label{lemma:bary:compact}
  \begin{enumerate}
  \item Given a compact saturated convex subset $Q$ of a locally
    convex semitopological cone $\C$, the barycenters of all
    continuous valuations $\nu$ with $\nu (X) \geq 1$ and supported on
    $Q$ are in $Q$.
  \item Given a compact saturated convex subset $Q$ of an
    lc-embeddable barycentric algebra $\B$, the barycenters of all
    probability valuations supported on $Q$ are in $Q$.
  \end{enumerate}
\end{lemma}
\begin{proof}
  1.  Let $\nu$ be a continuous valuation supported on $Q$ such that
  $\nu (\C) \geq 1$, and $x_0$ be some barycenter of $\nu$.  Let us
  assume that $x_0$ is not in $Q$.  By the strict separation theorem
  \cite[Theorem~10.5]{Keimel:topcones2} with $C \eqdef \dc x_0$, there
  is a lower semicontinuous linear map $\Lambda \colon \C \to \creal$
  and a constant $a > 1$ such that $\Lambda (x) \leq 1$ for every
  $x \in C$ and $\Lambda (x) \geq a$ for every $x \in Q$.  By
  definition of barycenters,
  $\int_{x \in \C} \Lambda (x) \,d\nu = \Lambda (x_0)$.  However,
  $\Lambda (x) \geq a$ for every $x \in Q$, and $\nu$ is supported on
  $Q$.  By Fact~\ref{fact:supp:f<=g}
  $\int_{x \in \C} \Lambda (x) \,d\nu \geq \int_{x \in \C} a \,d\nu
  \geq a$ (since $\nu (\C) \geq 1$), and that contradicts
  $\Lambda (x_0) \leq 1$.

  2.  If $Q$ is a compact saturated convex subset of a semitopological
  barycentric algebra $\B$ with an affine topological embedding $i$
  into a locally convex semitopological cone $\C$ instead, then
  $i [Q]$ is convex, and $ \upc i [Q]$ is compact saturated and
  convex.

  Let $x_0$ be a barycenter of a probability valuation $\nu$ supported
  on $Q$.  Using Fact~\ref{fact:f[nu]:supp}, $i [\nu]$ is supported on
  $\upc i [Q]$.  Also, $i (x_0)$ is a barycenter of $i [\nu]$, by
  Fact~\ref{fact:bary:img}.  By item~1, $i (x_0)$ must be in
  $\upc i [Q]$.  Hence there is a point $x \in Q$ such that
  $i (x) \leq i (x_0)$, whence $x \leq x_0$ since every topological
  embedding is a preorder embedding, and therefore $x_0 \in Q$ since
  $Q$ is upwards-closed.  \qed
\end{proof}

\subsection{A general barycenter existence theorem}
\label{sec:gener-baryc-exist}

We will imitate the main theorem of \cite{GLJ:bary}, which shows the
existence of barycenters of continuous valuations on convenient cones,
and we will generalize it to convenient (possibly pointed) barycentric
algebras.

Let $\Min Q$ denote the set of minimal elements of a subset $Q$ of a
topological space.  For every compact saturated subset $Q$ of a
topological space $X$, $Q = \upc \Min Q$; see, for example, the
comment after Proposition~3.2 in \cite{GLJ:bary}.  The following lemma
is a variant of \cite[Lemma~3.3]{GLJ:bary}.
\begin{lemma}
  \label{lemma:jia}
  Let $\B$ be a linearly separated semitopological barycentric algebra
  and $Q$ be a non-empty compact saturated subset of $\B$.  If the map
  $\varphi \colon \Lambda \mapsto \min_{x \in Q} \Lambda (x)$ is
  additive from $\B^\astar$ to $\creal$, then $Q = \upc x$ for some
  $x \in \B$.
\end{lemma}
\begin{proof}
  The fact that we have not assumed $\B$ to be $T_0$ will incur a bit
  of overhead.  Let us write $x \equiv y$ if and only if $x \leq y$
  and $y \leq x$.  We have $Q = \upc \Min Q$.  $\Min Q$ is a union of
  equivalence classes with respect to $\equiv$; using the Axiom of
  Choice, we pick one element from each, and we collect all those
  elements into one set which we call $A$.  (If $\B$ is $T_0$, then
  $A$ is simply $\Min Q$.)  Then $A$ is a collection of pairwise
  incomparable points, and $Q = \upc A$.

  Let us assume that $Q$ cannot be written as $\upc x$ for any
  $x \in \B$, or equivalently, that $A$ is not reduced to just one
  point.  $A$ is non-empty, since $Q$ is non-empty, so for every
  $x \in A$, $A$ contains a distinct, hence incomparable, point
  $\overline x$.

  Since $\B$ is linearly separated, there is an affine lower
  semicontinuous map $\Lambda_x \colon \B \to \creal$ such that
  $\Lambda_x (x) > 1$ and $\Lambda_x (\overline x) \leq 1$, for each
  $x \in A$.  The family ${(\Lambda_x^{-1} (]1, \infty]))}_{x \in A}$
  is an open cover of $A$, hence of $Q = \upc A$.  Since $Q$ is
  compact, there is a finite subset $E$ of $A$ such that
  $Q \subseteq \bigcup_{x \in E} \Lambda_x^{-1} (]1, \infty])$.  We
  note that $E$ is non-empty, since $Q$ is non-empty.

  Let $\Lambda \eqdef \sum_{x \in E} \Lambda_x$.  This is also an
  affine lower semicontinuous map from $\B$ to $\creal$.  Since
  $\varphi$ is additive by assumption, and the sum defining $\Lambda$
  is non-empty,
  $\varphi (\Lambda) = \sum_{x \in E} \varphi (\Lambda_x)$.

  Relying on the compactness of $Q$,
  $\varphi (\Lambda) = \min_{z \in Q} \Lambda (z)$ is equal to
  $\Lambda (x_0)$ for some $x_0 \in Q$.  Since
  $Q \subseteq \bigcup_{x \in E} \Lambda_x^{-1} (]1, \infty])$, there
  is an $x \in E$ such that $x_0 \in \Lambda_x^{-1} (]1, \infty])$,
  namely such that $\Lambda_x (x_0) > 1$.  Since $\overline x$ is not
  in $\Lambda_x^{-1} (]1, \infty])$, $\Lambda_x (\overline x) \leq 1$,
  and therefore
  $\varphi (\Lambda_x) = \min_{z \in Q} \Lambda_x (z) \leq 1 <
  \Lambda_x (x_0)$.  For all the other points $y$ of $E$,
  $\varphi (\Lambda_y) = \min_{z \in Q} \Lambda_y (z) \leq \Lambda_y
  (x_0)$, so, summing up, we obtain that
  $\varphi (\Lambda) = \varphi (\Lambda_x) + \sum_{y \in E \diff
    \{x\}} \varphi (\Lambda_y) < \Lambda_x (x_0) + \sum_{y \in E \diff
    \{x\}} \Lambda_y (x_0) = \Lambda (x_0) = \varphi (\Lambda)$, a
  contradiction.  \qed
\end{proof}

\begin{theorem}[Barycenter existence theorem]
  \label{thm:bary}
  Let $\rast$ be nothing, ``$\leq 1$'' or ``$1$''.  Let $\B$ be a
  convenient barycentric algebra $\B$, and let us assume that $\B$ is
  a cone if $\rast$ is nothing, resp.\ is pointed if $\rast$ is
  ``$\leq 1$'', resp.\ is compact and non-empty if $\rast$ is ``$1$''.
  
  Every continuous valuation $\nu$ (resp.\ every subprobability
  valuation if $\rast$ is ``$\leq 1$'', resp.\ every probability
  valuation if $\rast$ is ``$1$'') on $\B$ has a unique barycenter
  $\beta (\nu)$, and the barycenter map $\beta$ is continuous.
\end{theorem}
\begin{proof}
  $\Smythc \B$ is a linearly separated continuous d-barycentric
  algebra by Proposition~\ref{prop:QcvxC:dcone}, and coincides with
  $\SVc \B$.  When $\rast$ is ``$1$'', we have assumed $\B$ compact
  and non-empty, and $\B$ is also compact and non-empty since it is
  pointed in the other cases; so $\SVc \B$ is pointed in all cases, by
  Lemma~\ref{lemma:Qcvx:spec}, item~3.  Additionally, if $\rast$ is
  nothing, then $\B$ is a topological cone, since every continuous
  d-cone is topological by Proposition~\ref{prop:dcone:loc:convex},
  and therefore $\Smythc \B = \SVc \B$ is a topological cone
  \cite[Theorem~11]{Keimel:topcones2}.

  By Theorem~\ref{thm:locconv:heckmann:beta:pcont}, there is a unique
  linear (affine if $\rast$ is ``$1$'') lower semicontinuous map
  $\alpha \colon \Val_{\rast, \pw} {(\Smythc \B)} \to \Smythc \B$ such
  that $\alpha (\delta_Q) = Q$ for every $Q \in \Smythc \B$.
  Theorem~4.1 of \cite{heckmann96} states that continuous valuations
  and point-continuous valuations coincide on any locally finitary
  compact space; this applies to any c-space, hence to any continuous
  dcpo with its Scott topology in particular, and therefore to
  $\Smythc \B$.  Hence $\alpha$ is really a continuous linear (resp.\
  affine if $\rast$ is ``$1$'') map from $\Val_\rast {(\Smythc \B)}$
  to $\Smythc \B$.  Additionally, $\alpha$ maps every continuous
  valuation to a barycenter by Lemma~\ref{lemma:bary:simple}, item~4.

  Let us consider any element $\nu$ of $\Val_{\rast} \B$.  Since
  $\eta^\Smyth$ is continuous (Lemma~\ref{lemma:Qcvx:eta}), we can
  form the image continuous valuation $\eta^\Smyth [\nu]$ on
  $\SVc \B$.  We let $Q \eqdef \alpha (\eta^\Smyth [\nu])$.

  For every affine (if $\rast$ is ``$1$'', linear otherwise) lower
  semicontinuous map $\Lambda \colon \B \to \creal$, by
  Lemma~\ref{lemma:Lambda:bar:Qcvx}, $\min\Lambda$ is an affine (if
  $\rast$ is ``$1$'', linear otherwise) lower semicontinuous map from
  $\Smythc \B$ to $\creal$.  Since $\alpha$ maps every element of
  $\Val_{\rast} {(\Smythc \B)}$ to its barycenter, and since the
  function $\min\Lambda$ of Lemma~\ref{lemma:Lambda:bar:Qcvx} is
  affine (if $\rast$ is ``$1$'', linear otherwise) and lower
  semicontinuous, we must have
  $\min\Lambda (Q) = \int_{Q' \in \Smythc \B} \min\Lambda (Q')
  \,d\eta^\Smyth [\nu]$.  By the change of variable formula, the
  latter integral is equal to
  $\int_{x \in \B} \min\Lambda (\upc x) \,d\nu = \int_{x \in \B}
  \Lambda (x) \,d\nu$.  Also,
  $\min\Lambda (Q) = \min_{x \in Q} \Lambda (x) = \varphi (\Lambda)$,
  taking the notation $\varphi$ from Lemma~\ref{lemma:jia}.  In other
  words, $\varphi$ is the function
  $\Lambda \mapsto \int_{x \in \B} \Lambda (x) \,d\nu$, and that is an
  additive function since integration with respect to continuous
  valuations is additive in its integrated function.  By
  Lemma~\ref{lemma:jia} (which applies since $\B$ is linearly
  separated), $Q = \upc x_\nu$ for some $x_\nu \in \B$.

  For every affine (if $\rast$ is ``$1$'', linear otherwise) lower
  semicontinuous map $\Lambda \colon \B \to \creal$, the equation
  $\min\Lambda (Q) = \int_{x \in \B} \Lambda (x) \,d\nu$ then
  simplifies to
  $\Lambda (x_\nu) = \int_{x \in \B} \Lambda (x) \,d\nu$, showing that
  $x_\nu$ is a barycenter of $\nu$.  It is unique by
  Lemma~\ref{lemma:bary:unique}.

  Let us define $\beta$ by $\beta (\nu) \eqdef x_\nu$, for every
  $\nu \in \Val_{\rast} \B$.  We recall that $\upc x_\nu$ is what
  we have written as $Q$, which is defined as
  $\alpha (\eta^\Smyth [\nu])$.  We claim that $\beta$ is continuous.
  For every open subset $U$ of $\B$, for every
  $\nu \in \Val_{\rast} \B$, $\beta (\nu) \in U$ if and only if
  $\alpha (\eta^\Smyth [\nu]) \in \Box^\cvx U$.  In other words,
  $\beta^{-1} (U) = (\alpha \circ \Val_{\rast}
  (\eta^\Smyth))^{-1} (\Box^\cvx U)$, which is open since $\alpha$ and
  $\eta^\Smyth$ are continuous.  \qed
\end{proof}

\section*{Acknowledgments}
I would like to thank Mikhaïl Volkov, who pointed me to the survey
\cite{Skornyakov:stoch:alg}, which led me to the work of Skornyakov
and Ignatov, who studied barycentric algebras under the name of
convexors.

\ifta
\bibliographystyle{elsarticle-harv}
\else
\bibliographystyle{abbrv}
\fi
\bibliography{baryalg}






\appendix

\section{The non-strictly embeddable pointed semitopological
  barycentric algebra $\B^{\text{nse}}$}
\label{sec:non-strictly-embedd}

We recall that $\B^{\text{nse}}$ is the subset of
$\C \eqdef \Lform I \times \Lform \nat$ consisting of elements
$\alpha.(\one, 0) + \sum_{i=1}^m \alpha_i \cdot (f_{a_i}, e_{n_i})$
where $\alpha, \alpha_i \in \Rp$ and
$\alpha + \sum_{i=1}^n \alpha_1 \leq 1$; $f_a$ is the characteristic
function of $V_a \eqdef {]1-a, 1]}$, $I \eqdef [0, 1]$, $e_n$ maps $n$
to $1$ and all other natural numbers to $0$.

Since $I$ and $\nat$ are locally compact, the Scott topology on
$\Lform I$ and $\Lform \nat$ coincides with the \emph{compact-open}
topology, as a special case of \cite[Corollary~13.5]{dBGLJL:LCS} for
example.  Let $\tau_1$ be the subspace topology on $\B^{\text{nse}}$
induced by the inclusion in $\C$ with the product topology.  This
topology therefore has subbasic open subsets of the form
$[Q > r] \eqdef \{(f, e) \in \B^{\text{nse}} \mid \min_{x \in Q} f (x)
> r\}$ and
$[n > r] \eqdef \{(f, e) \in \B^{\text{nse}} \mid e (n) > r\}$, where
$Q$ is compact in $I$, $n \in \nat$ and $r \in \real$.

We call \emph{$\tau_1$-open} the open subsets of $\B^{\text{nse}}$ in
the topology $\tau_1$, and \emph{$\tau_1$-closed} its closed subsets.
We say that a subset $C$ of $\B^{\text{nse}}$ is \emph{closed} if and
only if it is $\tau_1$-closed and for all $a, b, c$ in $[0, 1]$ such
that $b+c \leq 1$, for every $(f, e) \in \B^{\text{nse}}$, such that
$ba.(f_a, e_n) + c.(f, e) \in C$ for infinitely many values of
$n \in \nat$, $ba.(\one, 0) + c.(f, e)$ is in $C$.

\begin{lemma}
  \label{lemma:Bnse:topo}
  The family of all closed sets, as defined above, is closed under
  arbitrary intersections and finite unions, and therefore form the
  closed sets of a topology.
\end{lemma}
We will call $\tau_2$ that topology.

\begin{proof}
  The $\tau_1$-closed sets are closed under arbitrary intersections
  and finite unions, since they are the closed sets for the topology
  $\tau_1$.  Let ${(C_i)}_{i \in I}$ be a family of closed sets, and
  $C$ be their intersection.  If $ba.(f_a, e_n) + c.(f, e) \in C$ for
  infinitely many values of $n$, then $ba.(f_a, e_n) + c.(f, e)$ is in
  every $C_i$ for the same values of $n$, so $ba.(\one, 0) + c.(f, e)$
  is in every $C_i$, hence in $C$; therefore $C$ is closed.  If $I$ is
  finite, let now $C$ be $\bigcup_{i \in I} C_i$, and let us assume
  that $ba.(f_a, e_n) + c.(f, e) \in C$ for infinitely many values of
  $n$.  By the pigeonhole principle, there is an $i \in I$ such that
  $ba.(f_a, e_n) + c.(f, e) \in C_i$; so
  $ba.(\one, 0) + c.(f, e) \in C_i$, hence
  $ba.(\one, 0) + c.(f, e) \in C$, showing that $C$ is closed.  \qed
\end{proof}

\begin{lemma}
  \label{lemma:Bnse:ba}
  $\B^{\text{nse}}$ is a pointed semitopological barycentric algebra
  with the topology $\tau_1$, with $\bot \eqdef (0, 0)$ as
  distinguished element.  $\B^{\text{nse}}$ is a pointed
  semitopological barycentric algebra with the topology $\tau_2$;
\end{lemma}
\begin{proof}
  $\C$ is a semitopological cone, hence a pointed semitopological
  barycentric algebra by Lemma~\ref{lemma:bary:in:cone:iff}.  Then
  $\B^{\text{nse}}$, which contains the least element
  $\bot \eqdef (0, 0)$, is a pointed semitopological barycentric
  algebra in the topology $\tau_1$.  It is even a topological one,
  since $\C$ is topological, but we will not need this.

  We turn to $\tau_2$.  We must first show that $\bot$ is least in the
  specialization preordering of $\tau_2$: we show that the only
  $\tau_2$-open neighborhood of $\bot$ is the whole of
  $\B^{\text{nse}}$, or equivalently that every non-empty
  $\tau_2$-closed set $C$ contains $\bot$.  $C$ is $\tau_1$-closed,
  hence downwards-closed in the usual componentwise ordering, and
  since $C$ is non-empty, it must contain $\bot = (0, 0)$.

  We now claim that the function $\_ +_d (f', e')$ is continuous for
  every $d \in [0, 1]$ and $(f', e') \in \B^{\text{nse}}$.  Let $C$ be
  a $\tau_2$-closed set.  Since $\B^{\text{nse}}$ is a semitopological
  barycentric algebra with the topology $\tau_1$ and $C$ is
  $\tau_1$-closed, $(\_ +_d (f', e'))^{-1} (C)$ is $\tau_1$-closed.
  Let $a, b, c \in [0, 1]$ be such that $b+c \leq 1$ and
  $(f, e) \in \B^{\text{nse}}$, with the property that
  $ba. (f_a, e_n) + c.(f, e) \in (\_ +_d (f', e'))^{-1} (C)$ for
  infinitely many values of $n \in \nat$.

  If $c \neq 0$ or $d \neq 1$, then $cd+1-d > 0$, so
  $(ba.(f_a, e_n) + c.(f, e)) +_d (f', e') = bad.(f_a, e_n) + cd.(f,
  e) + (1-d).(f', e') = bad.(f_a, e_n) + (cd+1-d).((f, e) +_{\frac
    {cd} {cd+1-d}} (f', e'))$; then
  $bd + (cd+1-d) = (b+c) d + (1-d) \leq d + (1-d) = 1$, and
  $(ba.(f_a, e_n) + c.(f, e)) +_d (f', e') = bad.(f_a, e_n) +
  (cd+1-d).((f, e) +_{\frac {cd} {cd+1-d}} (f', e')) \in C$ for
  infinitely many values of $C$.  Since $C$ is $\tau_2$-closed,
  $bad.(\one, 0) + (cd+1-d).((f, e) +_{\frac {cd} {cd+1-d}} (f', e'))
  \in C$.  But
  $bad.(\one, 0) + (cd + 1 - d). ((f, e) +_{\frac {cd} {cd+1-d}} (f',
  e')) = (ba.(\one, 0) + c.(f, e)) +_d (f', e')$, so
  $ba.(\one, 0) + c.(f, e) \in (\_ +_d (f', e'))^{-1} (C)$.

  If $c=0$ and $d=1$, then $ba.(f_a, e_n) + c.(f, e) = ba.(f_a, e_n)$,
  and $(ba.(f_a, e_n) + c.(f, e)) +_d (f', e') = ba.(f_a, e_n)$ as
  well.  Since that is in $C$ for infinitely many values of $n$,
  $ba.(\one, 0)$ is in $C$, too, and
  $ba.(\one, 0) = (ba.(\one, 0) +c.(f, e)) +_d (f', e')$, so
  $ba.(\one, 0) + c.(f, e) \in (\_ +_d (f', e'))^{-1} (C)$.

  Finally, we claim that, given fixed elements
  $(f, e), (f', e') \in \B^{\text{nse}}$, the function $F$ that maps
  every $d \in [0, 1]$ to $(f, e) +_d (f', e')$ is continuous, where
  $[0, 1]$ is given the standard topology on the reals.  The inverse
  image of a $\tau_2$-closed subset $C$ of $\B^{\text{nse}}$ is closed
  in $[0, 1]$ because $+_d$ is (even jointly) continuous when
  $\B^{\text{nse}}$ is given the topology $\tau_1$, and $C$ is
  $\tau_1$-closed.  \qed
\end{proof}

\begin{lemma}
  \label{lemma:Bnse:T0}
  The downward closure $\dc (f, e)$ of every point
  $(f, e) \in \B^{\text{nse}}$ in the componentwise ordering is closed
  in the topology $\tau_2$, hence $\dc (f, e)$ is the $\tau_2$-closure
  of $(f, e)$.  In particular, $\B^{\text{nse}}$ is $T_0$ in the
  topology $\tau_2$.
\end{lemma}
\begin{proof}
  For every $(f, e) \in \B^{\text{nse}}$, $\dc (f, e)$ is
  $\tau_1$-closed, since $\leq$ is the specialization ordering of
  $\tau_1$.  Let $a, b, c \in [0, 1]$ be such that $b+c \leq 1$, let
  $(f', e') \in  \B^{\text{nse}}$ and let us assume that $ba.(f_a,
  e_n) + c.(f', e') \in \dc (f, e)$ for infinitely many values of
  $n$.

  Since the elements of $\B^{\text{nse}}$ are of the form
  $\alpha.(\one, 0) + \sum_{i=1}^m \alpha_i \cdot (f_{a_i}, e_{n_i})$
  and each $e_{n_i}$ takes the value $0$ everywhere except at $n_i$,
  the set $E \eqdef \{k \in \nat \mid e (k) \neq 0\}$ is finite.

  However, for every $n \in \nat$, if
  $ba.(f_a, e_n) + c.(f', e') \in \dc (f, e)$, then
  $ba.e_n \leq ba.e_n + c.e' \leq e$, so $ba \leq e (n)$.  Hence
  $ba \leq e (n)$ for infinitely many values of $n$, which entails
  that $E$ is infinite if $ba \neq 0$.  Therefore $ba=0$.  Hence our
  assumption that $ba.(f_a, e_n) + c.(f', e') \in \dc (f, e)$ for
  infinitely many values of $n$ reduces to
  $c.(f', e') \in \dc (f, e)$, from which
  $ba.(\one, 0) + c.(f', e') \in \dc (f, e)$ follows immediately.

  It follows that $\dc (f, e)$ is $\tau_2$-closed.  Every
  $\tau_2$-closed subset $C$ of $\B^{\text{nse}}$ that contains $(f,
  e)$ is $\tau_1$-closed, hence downwards-closed, so $C$ must contain
  $(f, e)$.  It follows that $\dc (f, e)$ is the $\tau_2$-closure of
  $(f, e)$.

  Let $\preceq$ be the specialization ordering of $\B^{\text{nse}}$ in
  the topology $\tau_2$.  Then $(f, e) \preceq (f', e')$ if and only
  if $(f, e)$ is in the $\tau_2$-closure of $(f', e')$, which is
  $\dc (f', e')$ as we have just seen.  Hence
  $(f, e) \preceq (f', e')$ if and only if $(f, e) \leq (f', e')$.

  We have obtained that $\leq$ is the specialization ordering of
  $\tau_2$.  Since it is antisymmetric, the topology $\tau_2$ is
  $T_0$.  \qed
\end{proof}

In the rest of this appendix, we fix $a \in {]0, 1]}$.  For every
$n \in \nat$, for every $b \in {]0, 1]}$, let
\[
  C^b_n \eqdef \{(f, e) \in \B^{\text{nse}} \mid b.f \leq a.f_a \text{
    and } b.e \leq a.e_n\}.
\]
Let also $f_\infty \eqdef a \cdot \one$, $e_\infty \eqdef 0$, and:
\[
  C^b_\infty \eqdef \dc (f_\infty, e_\infty),
\]
where $\dc$ is downward closure with respect to the usual,
componentwise ordering $\leq$.  Finally, let:
\[
  C^b \eqdef \bigcup_{n \in \nat \cup \{\infty\}} C^b_n.
\]
\begin{lemma}
  \label{lemma:Bnse:Cclosed:tau1}
  For every $b \in {]0, 1]}$, $C^b$ is $\tau_1$-closed.
\end{lemma}
\begin{proof}
  Let $(f, e)$ be a point of $\B^{\text{nse}}$ outside $C^b$.  We wish
  to find a $\tau_1$-open neighborhood of $(f, e)$ that is disjoint
  from $C^b$.  We define $E$ as in the proof of
  Lemma~\ref{lemma:Bnse:T0}, as the set of indices $k \in \nat$ such
  that $e (k) \neq 0$; this is a finite set.

  Let $U \eqdef \bigcap_{k \in E} [k > 0]$.  $U$ is a $\tau_1$-open
  neighborhood of $(f, e)$, by definition of $E$.
  \begin{itemize}
  \item $U$ intersects $C^b_\infty$ if and only if $e_\infty (k) > 0$
    for every $k \in E$, if and only if $E$ is empty.
  \item For every $n \in \nat$, $U$ intersects $C^b_n$ if and only if
    there is a pair $(f', e') \in C^b_n$ such that $e' (k) > 0$ for
    every $k \in E$.  If so, then $b.e' \leq a.e_n$, so
    $a.e_n (k) > 0$, and therefore $k=n$.  This entails that
    $E \subseteq \{n\}$.  Conversely, if $E \subseteq \{n\}$, then
    $a.(f_a, e_n)$ is in $C^b_n$ (since $b \leq 1$) and in $U$.
    Therefore $U$ intersects $C^b_n$ if and only if
    $E \subseteq \{n\}$.
  \item It follows that $U$ intersects $C^b$ if and only if $E$ has
    cardinality at most $1$.  In particular, if $E$ contains at least
    two indices, then $U$ is an open neighborhood of $(f, e)$ that is
    disjoint from $C^b$.
  \end{itemize}
  Henceforth, we assume that $E$ contains at most one index.

  If $E$ is empty, then by considering the shape of elements of $\B^{\text{nse}}$,
  $(f, e)$ is equal to $(\alpha.\one, 0)$ for some
  $\alpha \in [0, 1]$.  More generally, let us assume that
  $(f, e) = (\alpha.\one, 0)$, whether $E$ is empty or not.  Since
  $(f, e) \not\in C^b$, $(f, e)$ is not in $C^b_\infty$, so
  $\alpha > a$.  Therefore $[I > a]$ is an open neighborhood of
  $(f, e)$.  $[I > a]$ does not intersect $C^b_\infty$, because
  $\min_{x \in I} f_\infty (x) = a \not> a$.  If $[I > a]$ intersected
  $C^b_n$ for some $n \in \nat$, there would be a pair
  $(f', e') \in \B^{\text{nse}}$ such that $b . f' \leq a.f_a$ and
  $b . e' \leq a.e_n$ and such that $\min_{x \in I} f' (x) > a$.  But by
  taking minima over $x \in I$ on both sides of $b . f' \leq a.f_a$,
  we would have $ba < \min_{x \in I} a.f_a (x) = 0$ (since
  $f_a (0)=0$, equivalently $0 \not\in V_a$), which is impossible.
  Therefore $[I > a]$ does not intersect $C^b$.

  Finally, we deal with the case where $E$ is a one-element set
  $\{k\}$ and $(f, e)$ is not of the form $(\alpha.\one, 0)$.  The
  pair $(f, e)$ is of the form
  $(\alpha.\one + \beta .  f_{a'}, \beta . e_k)$ for some
  $\alpha, \beta \in \Rp$ such that $\alpha+\beta \leq 1$ and some
  $a' \in {]0, 1]}$.  It must be that $\beta > 0$, since $(f, e)$ is
  not of the form $(\alpha.\one, 0)$.  Since $(f, e)$ is not in $C^b$,
  it is not in $C^b_k$, so $b.e = b\beta . e_k \not\leq a.e_k$ or
  $b.f \not\leq a.f_a$.
  \begin{itemize}
  \item In the first case, $e (k) > a/b$, so $[k > a/b]$ is an open
    neighborhood of $(f, e)$.  For every $(f', e') \in [k > a/b]$,
    $e' (k) > a/b$, so we cannot have $b.e' \leq a.e_n$ for any
    $n \in \nat$, and therefore $(f', e')$ is not in $C^b_n$.  We
    cannot have $e' \leq e_\infty$ since $e' (k) > 0$, so $(f', e')$
    is not in $C^b_\infty$ either.  Hence $(f', e')$ is not in $C^b$.
    It follows that $[k > a/b]$ is disjoint from $C^b$.
  \item In the second case, there is a point $x \in I$ such that
    $b.f (x) > a.f_a (x)$, so $[\{x\} > r]$ is an open neighborhood of
    $(f, e)$ where $r \eqdef (a/b).f_a (x)$.  We claim that
    $[\{x\} > r] \cap [k > 0]$ is disjoint from $C^b$.  Let $(f', e')$
    be in the intersection.  It cannot be in $C^b_\infty$, otherwise
    $e' \leq e_\infty = 0$, but $(f', e') \in [k > 0]$, which forces
    $e' (k) > 0$.  If $(f', e')$ is in some $C^b_n$, then
    $b.f' \leq a.f_a$.  But
    $(f', e') \in [\{x\} > r]$, so $f' (x) > r$, equivalently
    $b.f'(x) > a.f_a (x)$, which is impossible.  Hence
    $[\{x\} > r] \cap [k > 0]$ is disjoint from $C^b$.  \qed
  \end{itemize}
\end{proof}

\begin{lemma}
  \label{lemma:Bnse:Cclosed:tau2}
  For every $b \in {]0, 1]}$, $C^b$ is $\tau_2$-closed.
\end{lemma}
\begin{proof}
  We already know that $C^b$ is $\tau_1$-closed by
  Lemma~\ref{lemma:Bnse:Cclosed:tau1}.  Let $a', b', c \in [0, 1]$ be
  such that $b'+c \leq 1$, let $(f, e) \in \B^{\text{nse}}$, and let
  us assume that $b'a'.(f_{a'}, e_n) + c.(f,e) \in C^b$ for infinitely
  many values of $n \in \nat$.  We wish to show that
  $b'a'.(\one, 0) + c.(f,e)$ is in $C^b$.

  Since $C^b$ is downwards-closed, $c.(f, e)$ is in $C^b$.  Hence if
  $b'a'=0$, then $b'a'.(\one, 0) + c.(f,e) = c.(f, e)$ is in $C^b$.
  In the sequel, we assume that $b'a' \neq 0$, namely that $b' > 0$
  and $a' > 0$.

  If for some $n \in \nat$, $b'a'.(f_{a'}, e_n) + c.(f,e)$ is in
  $C^b_\infty$, then $b'a'.e_n + c.e \leq e_\infty = 0$, so $b'a'=0$
  (and $c.e=0$), but this is impossible.  Hence, for each
  $n \in \nat$, $b'a'.(f_{a'}, e_n) + c.(f,e)$ is in $C^b_m$ for some
  $m \in \nat$, so $bb'a'.e_n + bc.e \leq a.e_m$.  Since $b > 0$ and
  $b'a' > 0$,
  $bb'a'.e_n (n) \leq bb'a'.e_n (n) + bc.e (n) \leq a.e_m (n)$ entails
  that $m=n$.  Therefore $b'a'.(f_{a'}, e_n) + c.(f,e) \in C^b_n$ for
  every $n \in \nat$.  In particular,
  $bc.e (m) \leq bb'a'.e_n (m) + bc.e (m) \leq a.e_n (m)$ for every
  $m \in \nat$, so $bc.e (m) = 0$ for every $m \neq n$.  Since that
  holds for every $n \in \nat$, $bc.e = 0$, and therefore $c.e=0$
  (since $b > 0$).  In particular, $b'a'.(f_{a'}, e_n) \in C^b_n$ for
  every $n \in \nat$, so $bb'a'. f_{a'} \leq a.f_a$.  Since
  $b, b', a'> 0$, for every $x \in V_{a'}$,
  $bb'a'. f_{a'} (x) = bb'a' > 0$, so $a.f_a (x) > 0$, and therefore
  $x \in V_a$.  Therefore $V_{a'} \subseteq V_a$, namely
  ${]1-a', 1]} \subseteq {]1-a, 1]}$.  This entails $a' \leq a$, as
  otherwise $1-a$ would be in $]1-a', 1]$ but not in $]1-a, 1]$.
  Hence
  $b'a'.(\one, 0) + c.(f, e) = b'a'.(\one, 0) \leq b'a.(\one, 0) \leq
  (a.\one, 0) \in C^b_\infty \subseteq C$.  \qed
\end{proof}

\begin{lemma}
  \label{lemma:Bnse:notfull}
  Let $0 < a \leq b < 1$.  Then $C^b$ cannot be written as
  $(\alpha \cdot \_)^{-1} (C')$ for any $\tau_2$-closed subset $C'$ of
  $\B^{\text{nse}}$ and for any $\alpha \in {]0, b]}$.
\end{lemma}
\begin{proof}
  Let us imagine that $C^b$ can be written as
  $(\alpha \cdot \_)^{-1} (C')$ for some $\tau_2$-closed subset $C'$
  of $\B^{\text{nse}}$.  For every $n \in \nat$,
  $((a/b) \cdot f_a, (a/b) \cdot e_n)$ is in $\B^{\text{nse}}$, since
  $a \leq b$.  It is in $C^b_n \subseteq C^b$, so
  $\alpha \cdot ((a/b) \cdot f_a, (a/b) \cdot e_n) = (\alpha /b)a
  . (f_a, e_n)$ is in $C'$.
  Since $C'$ is $\tau_2$-closed, and by definition of the $\tau_2$
  topology (with $c\eqdef 0$, and $\alpha/b$ instead of $b$, which is
  legitimate since $\alpha/b + c \leq 1$, as $\alpha \leq b$),
  $(\alpha /b)a.(\one, 0)$ must also be $C'$.  This is the image under
  $\alpha \cdot \_$ of the point $(a/b) . (\one, 0)$, so
  $(a/b).(\one, 0)$ must be in $C^b$.  But:
  \begin{itemize}
  \item $(a/b) . (\one, 0)$ is not in $C^b_\infty$, because $(a/b)
    . \one \not\leq a \cdot \one = f_\infty$, since $b < 1$.
  \item $(a/b) \cdot (\one, 0)$ is not in any $C^b_n$, since otherwise
    we would have $a \cdot \one \leq a.f_a$, but $a \cdot \one (0)=a$
    and $a.f_a (0) = 0$, and $a > 0$.
  \end{itemize}
  Hence we have reached a contradiction, and we must conclude that
  $C^b$ is not of the form $(\alpha \cdot \_)^{-1} (C')$ for any
  $\tau_2$-closed subset $C'$ of $\B^{\text{nse}}$.  \qed
\end{proof}

Hence, for every $\alpha \in {]0, 1[}$, taking
$b \in {[\max (a, \alpha), 1[}$, we obtain that $C^b$ is not the
inverse image of any $\tau_2$-closed set under $\alpha \cdot \_$.  It
follows that the complement of $C^b$ is not the inverse image of any
$\tau_2$-open set under $\alpha \cdot \_$, and therefore
$\alpha \cdot \_$ is not full.  By
Proposition~\ref{prop:bary:alg:embed:strict}, $\B^{\text{nse}}$ is not
strictly embeddable.


\section{A topological barycentric algebra that is not weakly locally
  convex}
\label{sec:topol-baryc-algebra}

Using the inequality $(a+b)^p \leq a^p + b^p$, valid for all
$a, b \in \Rp$, we verify that $\ell_-^p (\nat)$ is a topological
barycentric algebra.

We profit from the fact that $\ell_-^p (\nat)$ in fact has pointwise
addition and scalar multiplication.  Using the inequality
$(a+b)^p \leq a^p + b^p$, it is easy to see that addition is
$1$-Lipschitz, hence jointly continuous.  The map $a \cdot \_$ is
$a^p$-Lipschitz for every $a \in \Rp$.  The slightly more difficult
part is the following.  $\Rp^{op}$ is $\Rp$ with the reverse ordering
$\geq$, and $(\Rp^{op})_\sigma$ is $\Rp^{op}$ with the Scott topology,
in other words it is $\Rp$ with open sets equal to $]-\infty, -a[$,
$a \in \Rp$, plus the empty set and the whole of $\Rp$.
\begin{lemma}
  \label{lemma:lp:vi}
  For every $-\vec x \in \ell^p_- (\nat)$, $a \mapsto -a \cdot \vec x$
  is continuous from $(\Rp^{op})_\sigma$ to $\ell^p_- (\nat)$.
\end{lemma}
\begin{proof}
  We show that the inverse image $I$ of the open ball
  $B_{-\vec y, < r}$ centered at any $-\vec y \in \ell^p_- (\nat)$
  with radius $r > 0$ is open in $(\Rp^{op})_\sigma$.  For every
  $a \in \Rp$, $a \in I$ if and only if
  $d^p (-\vec y, -a \cdot \vec x) < r$, if and only if
  $\sum_{n \in \nat} \max (ax_n - y_n, 0)^p < r$.  The map
  $a \mapsto \sum_{n \in \nat} \max (ax_n - y_n, 0)^p$ is monotonic,
  so $I$ is downwards-closed.  If $I$ is empty, equal to $\Rp$ or of
  the form $[0, a_0[$, then $I$ is open in $(\Rp^{op})_\sigma$ and we
  are done.  Otherwise, $I$ is of the form $[0, a_0]$ for some
  $a_0 \in \Rp$; we claim that this case is impossible.  Since
  $a_0 \in I$, $d^p (-\vec y, -a_0 \cdot \vec x) < r$.  Let
  $\epsilon > 0$ be such that
  $d^p (-\vec y, -a_0 \cdot \vec x) + \epsilon < r$.  For every
  $a > a_0$,
  $d^p (-a_0 \cdot \vec x, -a \cdot \vec x) = \sum_{n \in \nat} \max
  (ax_n - a_0 x_n, 0)^p = (a - a_0)^p \sum_{n \in \nat} x_n^p$.  This
  is where we need the condition $\sum_{n \in \nat} x_n^p < \infty$:
  for $a > a_0$ sufficiently close to $a_0$,
  $(a - a_0)^p \sum_{n \in \nat} x_n^p$ will be smaller than
  $\epsilon$.  Then
  $d^p (-\vec y, -a \cdot \vec x) \leq d^p (-\vec y, -a_0 \cdot \vec
  x) + d^p (-a_0 \cdot \vec x, -a \cdot \vec x) \leq d^p (-\vec y,
  -a_0 \cdot \vec x) + \epsilon < r$, so $a \in I$: contradiction.  We
  conclude that $I$ is Scott-open in $\Rp^{op}$, hence that
  $a \mapsto a \cdot -\vec x$ is continuous.  \qed
\end{proof}

\begin{lemma}
  \label{lemma:lp:vii}
  The cone $\ell^p_- (\nat)$ is a topological barycentric algebra.
\end{lemma}
\begin{proof}
  It is easy to see that $\ell^p_- (\nat)$ is a cone.  In particular,
  it is a barycentric algebra, with $(-\vec x) +_a (-\vec y)$ defined
  as $a \cdot (-\vec x) + (1-a) \cdot (-\vec y)$.

  We claim that $(\Rp^{op})_\sigma$ is a c-space.  (It is even a sober
  c-space, or equivalently a continuous dcpo.  This would be as easy
  to check, but we do not wish to rely on, and therefore to have to
  define, the notion of continuous dcpo.)  Its specialization
  ordering is $\geq$, so the upward closure $\upc s$ of a point $s$ of
  $(\Rp^{op})_\sigma$ is $]-\infty, s]$.  Let $t \in \Rp$ be an
  element of an open subset $U$ of $(\Rp^{op})_\sigma$.  We wish to
  find $s \in \Rp$ and an open subset $V$ of $(\Rp^{op})_\sigma$ such
  that $t \in V \subseteq \upc s \subseteq U$, namely such that
  $t \in V \subseteq {]-\infty, s]}$ and $s \in U$.  If $U$ is the
  whole of $\Rp$, then we take $s \eqdef 0$ and $V \eqdef U$.
  Otherwise, $U = {]-\infty, r[}$ for some $r \in \Rp$ and $t < r$.
  We let $s \eqdef (t+r)/2$ and $V \eqdef {]-\infty, s[}$.

  By Lemma~\ref{lemma:lp:vi}, scalar multiplication
  $(a, x) \mapsto a \cdot x$, seen as a map from
  $(\Rp^{op})_\sigma \times \ell_-^p (\nat)$ to $\ell_-^p (\nat)$, is
  continuous in its first argument, and we have seen that it is
  Lipschitz hence continuous in its second argument.
  Using Ershov's observation \cite[Proposition~2]{Ershov:aspace:hull},
  and since $(\Rp^{op})_\sigma$ is a c-space, scalar multiplication is
  jointly continuous.

  Let $g \colon [0, 1] \to (\Rp^{op})_\sigma \times (\Rp^{op})_\sigma$
  be defined by $g (a) \eqdef (a, 1-a)$, where $[0, 1]$ has its usual
  metric topology.  The inverse image of a basic open subset
  ${[0, s[} \times {[0, t[}$ by $g$ is equal to $]1-t, s[$ if
  $0 \leq 1-t < s \leq 1$, to $[0, s[$ if $1-t < 0 < s \leq 1$, to
  $]-t, 1]$ if $0 \leq 1-t < 1 < s$, to $[0, 1]$ if $1-t < 0$ and
  $1 < s$, and is empty otherwise.  Hence $g$ is continuous.

  Let
  $f' \colon \ell^p_- (\nat) \times (\Rp^{op})_\sigma \times
  (\Rp^{op})_\sigma \times \ell^p_- (\nat) \to \ell^p_- (\nat)$ be the
  function defined by $f' (x, a, b, y) \eqdef a \cdot x + b \cdot y$.
  Since scalar multiplication is jointly continuous, and since
  addition is, too (being $1$-Lipschitz), $f'$ is jointly continuous.
  Now
  $f' \times (\identity \relax \times g \times \identity \relax)
  \colon A \times [0, 1] \times A \to \ell^p_- (\nat)$ is (jointly)
  continuous, and maps $(x, a, y)$ to $x +_a y$.  Therefore
  $\ell^p_- (\nat)$ is a topological barycentric algebra.  \qed
\end{proof}
Note that $\ell^p_- (\nat)$ is not a topological cone, although it is
a cone: the map $a \mapsto -a \cdot \vec x$ is not continuous from
$\Rp$ with its Scott topology to $\ell^p_- (\nat)$, since it is not
even monotonic.

Let $\vec 0$ be the sequence $(0, 0, \cdots)$, $\vec e_n$ be the
sequence with zeros at every position except for a one at position
$n$.  For all $\eta, \epsilon > 0$ such that $\eta < \epsilon$, we
pick $a \in {]0, \epsilon^{1/p}[}$.  Then the distance from $\vec 0$
to $\frac 1 n \sum_{i=1}^n a \cdot \vec e_n$ is
$\sum_{i=1}^n (a/n)^p = a^p n^{1-p}$.  Since $p < 1$, this tends to
$\infty$ as $n$ tends to $\infty$.  Hence, for $n$ large enough,
$\frac 1 n \sum_{i=1}^n a \cdot \vec e_n$ is not in
$B_{\vec 0, <\epsilon}$.  By construction,
$\frac 1 n \sum_{i=1}^n a \cdot \vec e_n$ is in the convex hull
$\conv {B_{\vec 0, <\eta}}$.  It follows that
$\conv {B_{\vec 0, <\eta}}$ is not included in
$B_{\vec 0, <\epsilon}$, whatever $\eta, \epsilon > 0$ are such that
$\eta < \epsilon$.

If $\ell^p_- (\nat)$ were locally convex, then $B_{\vec 0, <\epsilon}$
would contain a convex open neighborhood $U$ of $\vec 0$, which would
itself contain an open ball $B_{\vec 0, <\eta}$ with
$0 < \eta < \epsilon$.  Then
$\conv {B_{\vec 0, <\eta}} \subseteq U \subseteq B_{\vec 0,
  <\epsilon}$, but we have just seen that this is impossible.

\section{A semitopological cone that is not weakly locally convex}
\label{sec:semit-cone-that}

We first verify that $\C$ is a semitopological cone, with what we
called the closed sets being exactly its closed subsets.
\begin{lemma}
  \label{lemma:peck:nonlocconv:i}
  The closed sets indeed form the closed subsets of a topology on
  $\C$.
\end{lemma}
\begin{proof}
  Let ${(C_i)}_{i \in I}$ be an arbitrary family of closed sets as
  per the definition of closed set, and
  $C \eqdef \bigcap_{i \in I} C_i$.  Then $C$ is Scott-closed, and if
  $2a .  \chi_{U_n} +g \in C$ for infinitely many values of $n$, then
  for every $i \in I$, $2a .  \chi_{U_n} + g \in C_i$ for the same
  infinite set of values of $n$.  By assumption, $a . \one + g$ is in
  $C_i$, and this for every $i \in I$, so $a . \one \in C$.  Therefore
  $C$ is closed.

  Let us now suppose that $I$ is finite, and let
  $C \eqdef \bigcup_{i \in I} C_i$.  Then $C$ is Scott-closed, and if
  $2a .  \chi_{U_n} + g \in C$ for infinitely many values of $n$,
  there must be an $i \in I$ such that
  $\{n \in \nat \mid 2a .  \chi_{U_n} + g \in C_i\}$ is infinite:
  otherwise, $\{n \in \nat \mid 2a .  \chi_{U_n} + g \in C_i\}$ would
  be finite for every $i \in I$, and then the finite union
  $\bigcup_{i \in I} \{n \in \nat \mid 2a .  \chi_{U_n} + g \in C_i\}
  = \{n \in \nat \mid 2a .  \chi_{U_n} + g \in C\}$ would be finite, a
  contradiction.  (This is a pigeonhole principle argument.)  Since
  $C_i$ is closed, $a . \one + g$ must be in $C_i$, hence in $C$.
  \qed
\end{proof}

\begin{lemma}
  \label{lemma:peck:nonlocconv:ii}
  With the topology of Lemma~\ref{lemma:peck:nonlocconv:i}, and
  pointwise addition and scalar multiplication, $\C$ is a
  semitopological cone, and the sequence
  ${(2.\chi_{U_n})}_{n \in \nat}$ converges to $\one$.
\end{lemma}
\begin{proof}
  We first show that addition is separately continuous.  Let us
  fix $g \in \C$ and $C$ be a closed set.  Then $\_ + g$ is
  Scott-continuous, so $(\_ + g)^{-1} (C)$ is Scott-closed.  Let
  $a \in \Rp$ and $h \in \C$.  If
  $2a . \chi_{U_n} + h \in (\_ + g)^{-1} (C)$ for infinitely many
  values of $n$, then $2a . \chi_{U_n} + (g+h) \in C$ for infinitely
  many values of $n$, so $a . \one + (g+h) \in C$, and therefore
  $a . \one + h \in (\_ + g)^{-1} (C)$.

  Let us fix $b \in \Rp$.  We claim that $b \cdot \_$ is continuous.
  Let $C$ be a closed set.  The set $(b \cdot \_)^{-1} (C)$ is
  Scott-closed because $C$ and because $b \cdot \_$ is
  Scott-continuous.  Let $a \in \Rp$ and $h \in \C$.  If
  $2a . \chi_{U_n} + h \in (b \cdot \_)^{-1} (C)$ for infinitely many
  values of $n$, then $2ab . \chi_{U_n} + b.h \in C$ for infinitely
  many values of $n$, so $ab . \one + b.h \in C$, namely
  $a . \one + h \in (b \cdot \_)^{-1} (C)$.

  Let us fix $x \in \C$.  We claim that $\_ \cdot x$ is continuous.
  Let $C$ be a closed set.  The set $(\_ \cdot x)^{-1} (C)$ is
  Scott-open in $\Rp$ because $\_ \cdot x$ is Scott-continuous.

  We finally consider the sequence ${(2.\chi_{U_n})}_{n \in \nat}$.
  Let $U$ be an open subset of $\C$ containing $\one$.  Let $C$ be the
  complement of $U$ in $\C$.  By definition of closed sets in $\C$ and
  since $\one = 1.\one + 0 \not\in C$, there cannot be infinitely many
  values of $n$ such that $2\chi_{U_n} \in C$.  Hence there is an
  $n_0 \in \nat$ such that for every $n \geq n_0$,
  $2 \chi_{U_n} \in U$, leading to the conclusion that
  ${(2.\chi_{U_n})}_{n \in \nat}$ converges to $\one$.  \qed
\end{proof}

We recall that $f_n \eqdef \frac 1 n \sum_{i=n}^{2n-1} 2.\chi_{U_i}$
for every $n \in \nat^*$ and that
$f_\infty \eqdef \limsup_{n \in \nat^*} f_n$, namely that
$f_\infty (x) \eqdef \inf_{n \in \nat^*} \sup_{m \in \nat^*, m\geq n}
f_m (x)$ for every $x \in {[0, 1[}$.  $C$ is defined as $\dc \{f_n
\mid n \in \nat^*\} \cup \dc f_\infty$, where $\dc$ is downward
closure with respect to the pointwise ordering, and we aim to show
that $C$ is closed.  This takes several steps.
\begin{lemma}
  \label{lemma:peck:nonlocconv:iv}
  $C$ is Scott-closed with respect to the pointwise ordering on $\C$.
\end{lemma}
\begin{proof}
  We first show that for every directed subset $D$ of $\C$, for every
  finite subset $E$ of $\C$, if $D \subseteq \dc E$ then
  $D \subseteq \dc f$ for some $f \in E$.  This is a well-known fact.
  For completeness, here is a proof.  By contradiction, let us assume
  that $D \not\subseteq \dc f$ for any $f \in E$.  For each $f \in E$,
  we find some $g_f \in D \diff \dc f$.  Since $D$ is directed, there
  is a $g \in D$ such that $g_f \leq g$ for every $f \in E$.  By
  assumption, $g \leq f$ for some $f \in E$, so $g_f \leq g \leq f$,
  contradicting the fact that $g_f$ is not in $\dc f$.

  $C$ is obviously downwards-closed.  Let ${(g_i)}_{i \in I}$ be a
  directed family of elements of $C$ with pointwise supremum $g$, and
  let us assume that every $g_i$ is in $C$.

  If there is a finite subset $E$ of $\nat^* \cup \{\infty\}$ such
  that every $g_i$ is less than or equal to $f_n$ for some $n \in E$,
  then every $g_i$ is less than or equal to a single $f_n$ by our
  preliminary observation, and therefore $g \in \dc f_n \subseteq C$.

  Therefore we assume that for every finite subset $E$ of
  $\nat^* \cup \{\infty\}$, there is an $i \in I$ such that
  $g_i \not\in \dc \{f_n \mid n \in E\}$.  In that case, we claim that
  for every $i \in I$, $g_i \leq f_\infty$, namely that for every
  $i \in I$, for every $n \in \nat^*$,
  $g_i \leq \sup_{m \in \nat^*, m \geq n} f_m$.  We consider the
  finite set $E \eqdef \{1, 2, \cdots, n-1\} \cup \{\infty\}$.  There
  is a $j \in I$ such that $g_j \not\in \dc \{f_n \mid n \in E\}$.  By
  directedness, there is a $k \in I$ such that $g_i, g_j \leq g_k$.
  Since $g_k \leq f_m$ for some $m \in \nat^* \cup \{\infty\}$, it
  must be that $m \not\in E$: if $m \in E$, we would have
  $g_j \leq g_k \leq f_m$, contradicting
  $g_j \not\in \dc \{f_n \mid n \in E\}$.  Hence $m \in \nat^*$ and
  $m \geq n$.  From $g_i \leq g_k \leq f_m$, it follows that
  $g_i \leq \sup_{m \in \nat^*, m \geq n} f_m$, as desired.

  Since $g_i \leq f_\infty$ for every $i \in I$, we obtain that
  $g \leq f_\infty$, hence $g \in C$.  Therefore $C$ is Scott-closed.
  \qed
\end{proof}

We recall that
$U_n \eqdef \bigcup_{i=0}^{2^n-1} \left[\frac {2i} {2^{n+1}}, \frac
  {2i+1} {2^{n+1}}\right[$.
\begin{lemma}
  \label{lemma:peck:nonlocconv:vi-x}
  The following hold:
  \begin{enumerate}
  \item for every $x \in {[0, 1[}$ written in binary as
    $0.x_1 x_2 \cdots = \sum_{i \geq 1} x_i 2^{-i}$, where each $x_i$
    is in $\{0, 1\}$ and the sequence of numbers $x_i$ is not
    eventually all ones, for every $n \in \nat$,
    $\chi_{U_n} (x) = 1-x_{n+1}$,
    $f_n (x) = 2 - \frac 2 n \sum_{i=n+1}^{2n} x_i$ and
    $f_\infty (x) = 2 - 2 \liminf_{n \in \nat^*} \frac 1 n
    \sum_{i=n+1}^{2n} x_i$;
  \item for all $n \in \nat$ and $m \in \nat^*$ such that $n<m$ or
    $n \geq 2m$, the only $a \in \Rp$ such that
    $2a . \chi_{U_n} \leq f_m$ is $0$;
  \item for all $n \in \nat$ and $m \in \nat^*$ such that
    $m \leq n < 2m$, every $a \in \Rp$ such that
    $2a . \chi_{U_n} \leq f_m$ is such that $a \leq 1/m$;
  \item for every $n \in \nat$, the only $a \in \Rp$ such that
    $2a. \chi_{U_n} \leq f_\infty$ is $0$;
  \item $C$ is closed.
  \end{enumerate}
\end{lemma}
\begin{proof}
  1.  We have $\chi_{U_0} (x) = 1-x_1$, $\chi_{U_1} (x) = 1-x_2$, and
  in general $\chi_{U_n} (x) = 1-x_{n+1}$ by design.  Therefore
  $f_n (x) = 2 - \frac 2 n \sum_{i=n+1}^{2n} x_i$ for every
  $n \in \nat$ and
  $f (x) = 2 - 2 \liminf_{n \in \nat^*} \frac 1 n \sum_{i=n+1}^{2n}
  x_i$.

  2.  Relying on item~1, it suffices to apply both sides of the
  inequality $2a . \chi_{U_n} \leq f_m$ to a number $x$ in $U_n$
  (equivalently, such that $x_{n+1}=0$, so that
  $2a . \chi_{U_n} (x) = 2a$) such that $f_m (x)$ is as small as we
  wish.  We pick $x$ so that not only $x_{n+1}=0$, but also $x_{m+1}$,
  \ldots, $x_{2m}$ are equal to $1$: this is possible since $n+1$ is
  not among $\{m+1, \cdots, 2m\}$, since $n < m$ or $n \geq 2m$.  Then
  $f_m (x) = 2 - \frac 2 m \sum_{i=m+1}^{2m} 1 = 0$ by item~1.  We
  obtain that $2a = 2a . \chi_{U_n} (x) \leq f_m (x) = 0$, so $a=0$.

  3.  Similarly, we find $x \in U_n$ (so that $x_{n+1}=0$ and
  $2a . \chi_{U_n} (x)=2a$), and we let $x_i \eqdef 1$ for every
  $i \in \{m+1, \cdots, n, n+2, \cdots, 2m\}$ (necessarily excluding
  $i = n+1$).  Then $f_m (x) = 2 - \frac 2 m (m-1) = \frac 2 m$, so we
  obtain the inequality $2a \leq 2/m$, equivalently $a \leq 1/m$.

  4.  Let us fix $n \in \nat^*$.  We once again pick a particular
  point $x \in U_n$; in particular, we set $x_{n+1} \eqdef 0$.  In
  order to avoid a final infinite sequence of ones, we let
  $x_i \eqdef 0$ for every $i$ of the form $2^k$, $k \in \nat$.  All
  the other values of $x_i$ are set to $1$.

  We claim that
  $\liminf_{m \in \nat^*} \frac 1 m \sum_{i=m+1}^{2m} x_i = 1$.  By
  definition,
  $\frac 1 m \sum_{i=m+1}^{2m} x_i = 1 - \frac {\# \{k \in \nat \mid
    m+1 \leq 2^k \leq 2m\}} m$, where $\#$ means ``number of''.  There
  is at most one natural number $k$ such that $m+1 \leq 2^k \leq 2m$,
  since if $m+1 \leq 2^k$, then $2^{k+1} > 2m$.  Hence
  $\frac 1 m \sum_{i=m+1}^{2m} x_i \geq 1 - \frac 1 m$.  This tends to
  $1$ as $m$ tends to $\infty$, proving the claim.

  By item~1,
  $f_\infty (x) = 2 - 2 \liminf_{m \in \nat^*} \frac 1 m
  \sum_{i=m+1}^{2m} x_i$, which is therefore equal to $0$.  The
  inequality $2a.\chi_{U_n} (x) \leq f_\infty (x)$ then reduces to
  $2a \leq 0$, so $a=0$.

  5.  In light of Lemma~\ref{lemma:peck:nonlocconv:iv}, it remains to
  show that for every $a \in \Rp$ and $g \in \C$ such that
  $2a.\chi_{U_n} + g \in C$ for infinitely many values of
  $n \in \nat$, $a.\one+g$ is also in $C$.

  This is obvious if $a=0$.

  Hence let us assume that $a > 0$ and that $2a.\chi_{U_n} + g \in C$
  for infinitely many values of $n \in \nat$.  We will simply show
  that this case cannot happen.  We fix $m_0 \in \nat^*$ such that
  $a \geq 1/m_0$, namely $m_0 \geq 1/a$.  There are still infinitely
  many values of $n \in \nat$ such that $2a.\chi_{U_n} + g \in C$
  \emph{and} such that $n \geq 2m_0$.  We pick one.  Since
  $2a.\chi_{U_n} + g \in C$ and $C$ is downwards-closed,
  $2a.\chi_{U_n}$ is also in $C$, hence smaller than or equal to $f_m$
  for some $m \in \nat^* \cup \{\infty\}$.  If $m=\infty$, item~4
  would imply $a=0$, which is impossible, so $m \in \nat^*$.  If
  $n < m$ or $n \geq 2m$, then by item~2 we would have $a=0$, which is
  impossible as well.  Therefore $m \in \nat^*$ and $m \leq n < 2m$.
  By item~3, $a \leq 1/m$.  But $2m \geq n+1$, and $n \geq 2m_0$, so
  $2m \geq 2m_0+1$, hence $m > m_0$, and this entails that
  $a < 1/m_0$, contradicting $a \geq 1/m_0$.  \qed
\end{proof}

Let us proceed.
\begin{lemma}
  \label{lemma:peck:nonlocconv:xi}
  $\one$ is not in $C$, so ${(f_n)}_{n \in \nat^*}$ does not converge
  to $\one$ in $\C$.
\end{lemma}
\begin{proof}
  If $\one \in C$, then $\one \leq f_m$ for some $m \in \nat^* \cup
  \{\infty\}$.  Therefore $2a.\chi_{U_n} \leq f_m$ for every $n \in
  \nat^*$, where $a \eqdef 1/2$.
  We use Lemma~\ref{lemma:peck:nonlocconv:vi-x}:
  if $m \in \nat^*$, then
  item~2 applied to any $n \geq 2m$ would imply $a=0$,
  otherwise item~4 would imply $a=0$.  Hence we reach a
  contradiction in both cases, so $\one$ is not in $C$.

  $C$ is a closed set by Lemma~\ref{lemma:peck:nonlocconv:vi-x},
  item~5.  It contains every $f_n$ with $n \in \nat^*$ by definition.
  Hence it must contain all the limits of ${(f_n)}_{n \in \nat}$.
  Since $\one$ is not in $C$, $\one$ cannot be a limit of
  ${(f_n)}_{n \in \nat}$.  \qed
\end{proof}

\begin{lemma}
  \label{lemma:peck:nonlocconv:xii}
  For every open neighborhood $U$ of $\one$ in $\C$, $\conv U$ must
  intersect $C$, so $\C$ is not weakly locally convex.
\end{lemma}
\begin{proof}
  Let $U$ be an open neighborhood of $\one$.  By
  Lemma~\ref{lemma:peck:nonlocconv:ii}, $U$ must contain
  $2 \chi_{U_n}$ for $n$ large enough, say for every $n \geq n_0$.
  Then $f_n$ is a convex combination of $2 \chi_{U_n}$, \ldots,
  $2 \chi_{U_{2n-1}}$, hence must be in $\conv U$.  But $f_n \in C$,
  so $\conv U$ intersects $C$.

  If $\C$ were weakly locally convex, then $V \eqdef \C \diff C$ would
  be an open neighborhood of $\one$ by
  Lemma~\ref{lemma:peck:nonlocconv:xi}, and would therefore contain a
  convex subset $A$, itself containing an open neighborhood $U$ of
  $\one$.  Then $A$ would contain $\conv U$, and we have just seen
  that $\conv U$ would then intersect $C$.  Hence $A$, and therefore
  the even larger set $V$ would also intersect $C$, which is absurd.
  \qed
\end{proof}


\end{document}
